%% file: WB_Paper.tex
\documentclass[11pt,a4paper,reqno]{amsart}

\pdfoutput=1

\usepackage[utf8]{inputenc}
\usepackage{epsfig,amsmath,latexsym,amssymb}
\usepackage{amsthm}
\usepackage{bm}
\usepackage[numbers,sort&compress,longnamesfirst]{natbib}
\usepackage{float}
\usepackage{graphicx}
\usepackage{url}
\usepackage[lmargin=2cm, rmargin=2cm, tmargin=2.5cm, bmargin=2.5cm]{geometry}
\usepackage{color}
\usepackage[citecolor=blue,urlcolor=blue,filecolor=blue]{hyperref}
\usepackage{multirow}
\usepackage{enumitem}
\usepackage{pdflscape}
\usepackage{dsfont}
\usepackage[ruled,vlined]{algorithm2e}
\usepackage{algorithmic}
\usepackage{placeins}
\usepackage{xcolor}
\usepackage{booktabs}
\usepackage{times}
\usepackage{comment}
\usepackage{subcaption}
\usepackage{listings}
\usepackage{centernot}
\usepackage{tikz}
\usepackage{adjustbox}
\usepackage{array}
\usepackage[normalem]{ulem}
\newcolumntype{L}[1]{>{\raggedright\arraybackslash}p{#1}}

\usetikzlibrary{calc}

\newcommand{\PLINEQCONST}{C_{\text{\normalfont{P\L}}}}
\newcommand{\VARINEQCONST}{C_{\text{\normalfont{Var}}}}
\newcommand{\AC}{\text{\normalfont{ac}}}
\newcommand{\LIN}{\text{\normalfont{lin}}}
\newcommand{\ENTROPIC}{\text{\normalfont{entr}}}
\newcommand{\FULL}{\text{\normalfont{full}}}
\newcommand{\LOCAL}{\text{\normalfont{loc}}}
\newcommand{\MINSUB}{\text{\normalfont{min}}}
\newcommand{\MAXSUB}{\text{\normalfont{max}}}
\newcommand{\LBSUB}{\text{\normalfont{LB}}}
\newcommand{\UBSUB}{\text{\normalfont{UB}}}
\newcommand{\STRONGLYCONVEX}{\text{\normalfont{sc}}}

\newcommand{\COPULA}{\text{\normalfont{cop}}}
\newcommand{\SPHERICAL}{\text{\normalfont{sph}}}
\newcommand{\ELLIPTICAL}{\text{\normalfont{ell}}}

\DeclareFontFamily{U}{mathx}{\hyphenchar\font45}
\DeclareFontShape{U}{mathx}{m}{n}{
  <-> mathx10
}{}
\DeclareSymbolFont{mathx}{U}{mathx}{m}{n}
\DeclareMathAccent{\widecheck}{0}{mathx}{"71}

\DeclareFontFamily{U}{mathx}{\hyphenchar\font45}
\DeclareFontShape{U}{mathx}{m}{n}{
      <5> <6> <7> <8> <9> <10>
      <10.95> <12> <14.4> <17.28> <20.74> <24.88>
      mathx10
      }{}
\DeclareSymbolFont{mathx}{U}{mathx}{m}{n}
\DeclareMathSymbol{\bigtimes}{1}{mathx}{"91}

\input{mathsymbols}

\numberwithin{equation}{section}  
\newtheorem{definition}{Definition}[section]

\newtheorem{remark}[definition]{Remark}
\newtheorem{theorem}[definition]{Theorem}
\newtheorem{proposition}[definition]{Proposition}
\newtheorem{corollary}[definition]{Corollary}
\newtheorem{lemma}[definition]{Lemma}

\newtheorem{assumption}[definition]{Assumption}
\newtheorem{setting}[definition]{Setting}

\usepackage{hyperref}
\usepackage{xcolor}
\hypersetup{
	colorlinks,
	linkcolor={red!70!black},
	citecolor={blue!70!black},
	urlcolor={blue!100!black}
}

\DeclareMathOperator*{\argmin}{arg\,min}
\newcommand{\intersects}{\hspace{1pt}\cap\hspace{1pt}}

\numberwithin{equation}{section}

\numberwithin{figure}{section}

\numberwithin{table}{section}

\DontPrintSemicolon

\begin{document}
\title[Provably convergent algorithm for free-support Wasserstein barycenter]{Provably convergent stochastic fixed-point algorithm for free-support Wasserstein barycenter of continuous non-parametric measures}

\author[Z. Chen]{Zeyi Chen}
\author[A. Neufeld]{Ariel Neufeld}
\author[Q. Xiang]{Qikun Xiang}

\address{Area of Decision Sciences, INSEAD, 1 Ayer Rajah Ave, 138676 Singapore}
\email{zeyi.chen@insead.edu}

\address{Division of Mathematical Sciences, Nanyang Technological University, 21 Nanyang Link, 637371 Singapore}
\email{ariel.neufeld@ntu.edu.sg}

\address{Division of Mathematical Sciences, Nanyang Technological University, 21 Nanyang Link, 637371 Singapore}
\email{qikun.xiang@ntu.edu.sg}

\date{\today}

\begin{abstract}
	We develop an estimator-based stochastic fixed-point framework for approximately computing the \mbox{2-Wasserstein} barycenter of continuous, non-parametric probability measures. 
	Notably, we provide the first rigorous convergence analysis for an implementable estimator-based
	stochastic extension of the fixed-point iterative scheme proposed by {\'A}lvarez-Esteban, del Barrio, Cuesta-Albertos, and Matr{\'a}n (2016).
	In particular, we establish almost sure convergence and identify sufficient conditions under which the proposed scheme achieves a geometric convergence rate in the number of iterations, provided that the errors in the approximation steps are suitably controlled.
	We subsequently propose a concrete, provably convergent, and computationally tractable stochastic algorithm that accommodates input measures satisfying Caffarelli-type regularity conditions, 
	which form a dense subset of the Wasserstein space. 
	This algorithm leverages a modified entropic optimal transport map estimator to enable efficient and scalable implementation.
	To facilitate quantitative evaluation, we further propose a novel and efficient procedure for synthetically generating benchmark instances, in which the input measures exhibit non-trivial features and the corresponding barycenters are approximately known.
	Numerical experiments on both synthetic and real-world datasets demonstrate the strong computational efficiency, estimation accuracy, and sampling flexibility of our approach. 


\end{abstract}
\maketitle



\section{Introduction}\label{sec: intro}

In this paper, we consider the problem of approximately computing the 2-Wasserstein barycenter \citep{agueh2011barycenters} of non-parametric probability measures that possess finite second moments and are absolutely continuous with respect to the Lebesgue measure on $\R^d$.
The \textit{$\CW_2$-distance} between two probability measures $\mu, \nu \in \CP_2(\R^d)$ with finite second moments is defined via the following optimal transport (OT) problem (see, e.g., \citep{villani2009optimal}) with squared-Euclidean cost: 
\begin{align}
	\CW_2(\mu, \nu) := \biggl( \inf_{\pi \in \Pi(\mu, \nu)} \int_{\R^d\times\R^d} \|\BIx - \BIy\|^2 \DIFFM{\pi}{\DIFF \BIx, \DIFF \BIy}\biggr)^{\frac{1}{2}},
	\label{eqn: W2-distance}
\end{align}
where $\Pi(\mu, \nu)$ denotes the set of coupling measures with $\mu$ and $\nu$ as their marginals (see Definition~\ref{def: coupling}).\footnote{In this paper, we use the terms \textit{$\CW_2$-distance} and \textit{$\CW_2$-metric} interchangeably; see, e.g., \citep[Theorem~6.9]{villani2009optimal} for the properties of $\CW_2$ as a metric on $\CP_2(\R^d)$.}
Given $K\in\N\cap[2,\infty)$ input measures $\nu_1, \dots, \nu_K \in \CP_2(\R^d)$, and weights $w_1, \dots, w_K \in(0,1)$ satisfying $\sum_{k = 1}^{K} w_k = 1$, 
the \textit{barycenter functional} $V:\CP_2(\R^d)\to\R$ is defined as follows:
\begin{align}
	\label{eqn: V-definition}
	V(\mu) := \sum_{k = 1}^{K} w_k \CW_2(\mu, \nu_k)^2 \qquad \forall \mu\in\CP_2(\R^d).
\end{align}
Then, a probability measure $\bar{\mu} \in \CP_2(\R^d)$ is called a 2-Wasserstein barycenter, or \textit{$\CW_2$-barycenter}, 
of $\nu_1, \dots, \nu_K$ with weights $w_1, \dots, w_K$ if
\begin{align}
	\label{def: W2-barycenter}
	\bar{\mu} \in \argmin_{\mu \in \CP_2(\R^d)} V(\mu).
\end{align}

Due to the appealing geometric structure and statistical properties of the Wasserstein space \citep{panaretos2020invitation}, the Wasserstein barycenter problem arises in widespread applications of combining probabilistic information from heterogeneous data sources. 
Examples include aggregating subset posteriors in Bayesian inferences \citep{srivastava2018scalable,srivastava2015wasp}, mixing color textures in image processing \citep{rabin2012wasserstein}, forming group consensus from expert forecasts in decision analysis \citep{petracou2022decision}, and selecting the reference measure in distributionally robust optimization \citep{rychener2024wasserstein,lau2022wasserstein}, to name but a few.

However, it is well-known that computing the $\CW_2$-barycenter of general probability measures is challenging.
For example, even in the restrictive case with discrete input measures, 
it has been proved by \citet*{altschuler2022NP} that the problem is already NP-hard.
A common approximation strategy is to parametrize the underlying $\CW_2$-barycenter via a discrete measure supported on fixed atoms
and turn the problem into optimizing the histogram weights over a finite-dimensional probability simplex;
see, e.g., \citep[Chapter~6]{peyre2019computational} and the references therein. 
Nevertheless, most such ``fixed-support'' approaches scale poorly to high dimensions due to prohibitive computational burdens, 
and they are unsuitable for scenarios when sampling from the barycenter measure is needed.

Given the aforementioned challenges, our work contributes to the literature of ``free-support'' approaches which do not prescribe any discrete support when approximating the $\CW_2$-barycenter
of general continuous non-parametric probability measures.
At a high level, the goal of this paper is to study implementable stochastic extensions of the prominent fixed-point iterative scheme proposed by \citet*{alvarez2016fixed}. 
Specifically, \citet{alvarez2016fixed} have demonstrated that the \mbox{$\CW_2$-barycenter} of 
absolutely continuous probability measures $\nu_1, \dots, \nu_K\in\CP_{2,\AC}(\R^d)$ can be characterized as a fixed-point of the operator $G: \CP_{2, \AC}(\R^d) \to \CP_{2, \AC}(\R^d)$ defined through the following pushforward operation:\footnote{
	For two closed subsets $\CX,\CY$ of Euclidean spaces and a Borel measurable function $T:\CX\to\CY$, the pushforward of a probability measure $\mu\in\CP(\CX)$ by $T$ is denoted by $T\sharp\mu\in \CP(\CY)$, which is defined via $T \sharp \mu (B) \equiv \mu\circ T^{-1}(B)$ for every Borel set $B\subseteq\CY$.}
\begin{align}
	\label{eqn: G-operator}
	G(\mu) := \big[\textstyle \sum_{k = 1}^K w_k T^\mu_{\nu_k}\big] \sharp \mu \qquad \forall \mu \in \CP_{2, \AC}(\R^d),
\end{align}
where $T^{\mu}_{\nu_k}$ corresponds to Monge's OT map from $\mu$ to $\nu_k$ (see Theorem~\ref{thm: Brenier}).
In particular, the $G$-operator in (\ref{eqn: G-operator}) is continuous with respect to the $\CW_2$-metric \citep[Theorem~3.1]{alvarez2016fixed}, and the following result holds.
\begin{theorem}[Properties of the $G$-operator {\citep[Corollary~3.5 \& Theorem~3.6]{alvarez2016fixed}}]\label{thm: G-property}
	Let $\nu_1,\ldots,\nu_K$ be probability measures with finite second moments that are absolutely continuous with respect to the Lebesgue measure on~$\R^d$,
	i.e., $\nu_1,\ldots,\nu_K\in\CP_{2,\AC}(\R^d)$,
	where for at least one index $k\in\{1,\ldots,K\}$, 
	the density function of $\nu_k$ is essentially bounded.
	Then, the $G$-operator defined in (\ref{eqn: G-operator}) satisfies the following properties.
	\begin{enumerate}[label=(\roman*)]
		\item\label{thms: G-property-fixedpoint}The unique $\CW_2$-barycenter $\bar{\mu} \in \CP_{2, \AC}(\R^d)$ (see Theorem~\ref{thm: unique barycenter}) of $\nu_1, \dots, \nu_K$ with weights $w_1,\ldots,w_K$ is a fixed-point of $G$, i.e., $\bar{\mu} = G(\bar{\mu})$.
		\item\label{thms: G-property-accumulation}For any $\mu_0 \in \CP_{2, \AC}(\R^d)$, the sequence $(\mu_t)_{t \in \N_0}$ generated by the iteration 
		\begin{align}
			\label{eqn: G-iteration}
			\mu_{t} := G(\mu_{t-1}) \qquad \forall t\in\N
		\end{align}
		is tight. 
		Moreover, every accumulation point of the sequence $(\mu_t)_{t \in \N_0}$ with respect to the $\CW_2$-metric is a fixed-point of $G$. 
	\end{enumerate}
\end{theorem}

\begin{algorithm}[t]
	\caption{\bf{Abstract stochastic fixed-point iterative scheme.}}%
	\label{algo: abstract}
	\KwIn{$K\in\N\cap[2,\infty)$ input probability measures $\nu_1, \ldots, \nu_K \in \CP_{2, \AC}(\R^d)$, 
	weights $w_1,\ldots,w_K\in(0,1)$ with $\sum_{k=1}^Kw_k=1$, 
	initial probability measure~$\mu_0 \in \CP_{2, \AC}(\R^d)$.}
	\KwOut{$(\widehat{\mu}_t)_{t \in \N_0}$.}

	\nl\label{alglin: abstract-initialize}Initialize $\widehat{\mu}_0 \leftarrow \mu_0$.

	\nl \For{$k=1,\ldots,K$}{
		\nl\label{alglin: abstract-initial-samplesize}Choose the sample sizes $\widehat{M}_{0,k},\widehat{N}_{0,k}\in\N$
		and the hyperparameter(s) $\widehat{\ITheta}_{0,k}\in\Theta$ 
		for the estimator.
	}

	\nl \For{$t = 1, 2, \dots$}{
		\textbf{[Iteration~$t$]:}\\
		\nl \For{$k=1,\ldots,K$}{

			\nl\label{alglin: abstract-sample1}Randomly generate $\widehat{M}_{t-1,k} \in \N$ independent samples $\{\BIX_{t,k,i}\}_{i=1:\widehat{M}_{t-1,k}}$ from $\widehat{\mu}_{t-1}$.

			\nl\label{alglin: abstract-sample2}Randomly generate $\widehat{N}_{t-1,k} \in \N$ independent samples $\{\BIY_{t,k,j}\}_{j=1:\widehat{N}_{t-1,k}}$ from $\nu_k$.

			\nl\label{alglin: abstract-estimator}Approximate $T^{\widehat{\mu}_{t-1}}_{\nu_k}$ with an estimator $\widehat{T}_{t,k}\approx T^{\widehat{\mu}_{t-1}}_{\nu_k}$ using the samples $\{\BIX_{t,k,i}\}_{i=1:\widehat{M}_{t-1,k}}$ and $\{\BIY_{t,k,j}\}_{j=1:\widehat{N}_{t-1,k}}$
			as well as the hyperparameter(s) $\widehat{\ITheta}_{t-1,k}$.
		}
		\nl\label{alglin: abstract-update}Choose $\widehat{\mu}_{t} \in \CP_{2, \AC}(\R^d)$ using all available information up to iteration~$t$ such that $\widehat{\mu}_{t} \approx \big[\sum_{k = 1}^K w_k \widehat{T}_{t,k}\big] \sharp \widehat{\mu}_{t-1}$.\\

		\nl \For{$k=1,\ldots,K$}{
			\nl\label{alglin: abstract-next-samplesize}Choose the sample sizes 
			$\widehat{M}_{t,k},\widehat{N}_{t,k}\in\N$
			and the hyperparameter(s) $\widehat{\ITheta}_{t,k}\in\Theta$ 
			for the estimator using all available information up to iteration~$t$.
		}
	}
	\nl \Return $\big(\widehat{\mu}_t\big)_{t \in \N_0}$.
\end{algorithm}

The $G$-iteration \eqref{eqn: G-iteration} naturally gives rise to a simple iterative scheme for finding the $\CW_2$-barycenter where one begins with an arbitrary $\mu_0\in\CP_{2,\AC}(\R^d)$ and iterates (\ref{eqn: G-iteration}) to generate $(\mu_t)_{t\in\N_0}$.
It is then guaranteed that $(\mu_t)_{t\in\N_0}$ converges in $\CW_2$ to the $\CW_2$-barycenter of $\nu_1,\ldots,\nu_K$ with weights $w_1,\ldots,w_K$ whenever $G$ has a unique fixed-point. 
In fact, it has been recently shown by \citet*{tanguy2024computing} that this fixed-point method can be generalized to compute barycenters under general transportation costs and possibly non-continuous probability measures. 
Moreover, we highlight that there exists an alternative treatment of the $G$-iteration \eqref{eqn: G-iteration} from the perspective of Wasserstein gradient flows.
In particular, 
\citet*{zemel2019frechet} have shown that this iterative scheme corresponds to a gradient descent scheme minimizing the barycenter functional~$V$ on the Wasserstein space $\big(\CP_2(\R^d),\CW_2\big)$ that leads to the optimal decrement; see also \citep{chewi2020gradient} and \citep{backhoff2025stochastic}. 
In fact, any fixed-point of the $G$-operator is a \textit{Karcher mean} \citep{karcher1977riemannian} of $\nu_1,\ldots,\nu_K\in\CP_{2,\AC}(\R^d)$, which is a notion that features the stationary points of sum-of-squared functionals; 
see our brief summary of these notions at the end of Section~\ref{ssec: preliminary-OT} as well as \citep[Theorem~1 and Corollary~1]{zemel2019frechet}.

However, when $\nu_1,\ldots,\nu_K\in\CP_{2,\AC}(\R^d)$ are general continuous probability measures, the $G$-operator \eqref{eqn: G-operator} corresponds to an impractical ``oracle'' due to the intractability in computing the OT map $T^\mu_{\nu_k}$ exactly.
As a consequence, existing numerical implementations of this framework are either limited to particular parametric measures from the same elliptical family (see, e.g., \citep[Section~4]{alvarez2016fixed}), or carried out via neural network approximations \citep{korotin2022wasserstein} which hinder end-to-end convergence analyses. 

This bottleneck motivates our development of a provably convergent estimator-based stochastic extension of this deterministic fixed-point iterative scheme (or, equivalently, a stochastic approximation counterpart of the aforementioned Wasserstein gradient descent scheme) that supports both rigorous analysis in theory and efficient implementations in practice. 
The intuitive idea of our framework is sketched in Algorithm~\ref{algo: abstract} in its most abstract form.
Specifically, Line~\ref{alglin: abstract-estimator} in Algorithm~\ref{algo: abstract} approximates each true OT map $T^{\widehat{\mu}_{t-1}}_{\nu_k}$ with an estimator $\widehat{T}_{t,k}$ which is a (measurable) function of samples from the source and target measures,
whereas Line~\ref{alglin: abstract-update} thereafter approximates the $G$-operator \eqref{eqn: G-operator} when updating from $\widehat{\mu}_{t-1}$ to $\widehat{\mu}_{t}$.
Notice that letting $\widehat{T}_{t,k}= T^{\widehat{\mu}_{t-1}}_{\nu_k}$ in Line~\ref{alglin: abstract-estimator} and letting $\widehat{\mu}_{t} = \big[\sum_{k = 1}^K w_k \widehat{T}_{t,k}\big] \sharp \widehat{\mu}_{t-1}$ in Line~\ref{alglin: abstract-update} will recover the deterministic $G$-iteration \eqref{eqn: G-iteration}, and one may therefore anticipate that this stochastic extension mimics the behavior of the deterministic system, provided that the approximations therein have controlled errors. 
This paper rigorously confirms this intuition and explicitly discusses \textit{when} and \textit{how} approximation errors are controlled for Algorithm~\ref{algo: abstract} to preserve similar convergence guarantees as in Theorem~\ref{thm: G-property}.

\subsection{Contributions}\label{ssec: contributions}
The main contributions of this paper are summarized as follows.

\begin{enumerate}[label = (\roman*),beginpenalty=10000]
	\item We derive in Section~\ref{sec: fixedpoint-convergence} general conditions under which the stochastic estimator-based iterative scheme in Algorithm~\ref{algo: abstract} converges almost surely to a fixed-point of the $G$-operator in \eqref{eqn: G-operator} (Theorem~\ref{thm: main-convergence}), which coincides with the $\CW_2$-barycenter of the input measures when 
	$G$ has a unique fixed-point.  
	Specifically, we rigorously establish convergence when the approximation errors are suitably controlled in conditional expectation, and provide sufficient conditions that guarantee
	a convergence rate that is geometric in the number of iterations. 
	We also present concrete examples satisfying these conditions (Section~\ref{ssec: main-convergence-conditions}).
	To the best of our knowledge, this is the first work to deliver a rigorous convergence analysis for 
	the estimator-based
	stochastic counterpart of the fixed-point iterative scheme of \citet{alvarez2016fixed}.

	\item 
	In Section~\ref{sec: Caffarelli-setting},
	we develop a computationally tractable stochastic fixed-point algorithm (Algorithm~\ref{algo: concrete}) for approximately computing the $\CW_2$-barycenter
	as a concrete version of Algorithm~\ref{algo: abstract}.
	This algorithm accommodates a broad class of probability measures
	satisfying Caffarelli-type regularity conditions,
	which are sufficiently rich to approximate any probability measure in the Wasserstein space to arbitrary accuracy (Proposition~\ref{prop: Caffarelli-density}). 
	We show that this algorithm satisfies the conditions required for convergence (Theorem~\ref{thm: Caffarelli-convergence}). 
	Moreover, we characterize a class of admissible OT map estimators that ensure convergence (Assumption~\ref{asp: OTmap-estimator}), and explicitly propose a modified entropic OT map estimator that is amenable to efficient implementation driven by Sinkhorn's algorithm (Proposition~\ref{prop: OTmap-estimator-entropic}). 
	A complexity analysis is also provided for our algorithm.

	\item We propose a novel and efficient approach (Algorithm~\ref{algo: generating-problem-instance}) for generating synthetic $\CW_2$-barycenter problem instances in Section~\ref{sec: generation-instance}. 
	Our method constructs input probability measures with non-trivial distributional features whose $\CW_2$-barycenter approximately coincides with a user-specified probability measure that is known a priori (Proposition~\ref{prop: synthetic-generation}). 
	As a result, our framework allows quantitative evaluation of any approximate $\CW_2$-barycenter candidate and facilitates direct comparison across different $\CW_2$-barycenter algorithms, 
	thus providing a practical and flexible benchmarking tool.

	\item In Section~\ref{sec: numerics}, we conduct numerical experiments on both synthetically generated and real-world problem instances, and we benchmark our proposed algorithm against baseline algorithms in the literature. 
	Our results demonstrate the strong computational efficiency of our proposed stochastic fixed-point algorithm, as well as its competitive quality in estimation and flexibility in sampling compared to state-of-the-art methods.
\end{enumerate}

\subsection{Related works}\label{ssec: literature}
We mention in detail three streams of literature that are closely related to our study, which include free-support algorithms for approximating the Wasserstein barycenter, variants of the entropy-regularized Wasserstein barycenter, and various approaches for estimating the OT map.

\subsubsection*{Free-support Wasserstein barycenter algorithms}
Many works in the literature, e.g., 
\citep{cuturi2014fast} and \citep{claici2018stochastic}, among others, have proposed algorithms that incrementally update the support of a discretized approximate barycenter.
While these algorithms do not anticipate the support of the underlying Wasserstein barycenter, and thus belong to free-support algorithms,
their resulting approximate barycenters remain discrete.
Recent years have witnessed an extensive and rapid development of continuous free-support algorithms driven by thriving high-dimensional function approximation tools such as neural networks.
For example, 
\citet*{cohen2020estimating} parametrized the barycenter as the pushforward of another latent measure by a generative neural network in optimizing the barycenter functional; 
\citet*{fan2021scalable} tackled a min-max-min reformulation of the variational problem \eqref{def: W2-barycenter} using input convex neural networks \citep{makkuva2020optimal}; 
\citet*{korotin2022wasserstein} 
parametrized the $\CW_2$-barycenter candidate via a generative neural network
to approximately implement the fixed-point iteration \eqref{eqn: G-iteration}. 
In the meantime, there have emerged numerous studies that considered the Wasserstein gradient of the barycenter functional \citep{zemel2019frechet}, and proposed first-order gradient-descent type methods to compute the Wasserstein barycenter \citep{chewi2020gradient,backhoff2025stochastic,kim2025optimal}; 
readers are referred to \citep[Section~5]{chewi2024statistical} for an overview of Wasserstein gradient flows.

\subsubsection*{Entropy-regularized Wasserstein barycenters}
Entropic optimal transport (EOT) modifies the objective function in \eqref{eqn: W2-distance} by including a penalization term based on the entropy of the coupling.
Its computational advantages have been actively studied since the seminal work by \citet*{cuturi2013sinkhorn}.
Entropy-regularized Wasserstein barycenters can be defined by plugging entropic variants of the Wasserstein distance into the barycenter functional \eqref{eqn: V-definition}.  
For example, 
\citet*{janati2020debiased} promoted the usage of Sinkhorn divergences in formulating regularized Wasserstein barycenter problems since they reduce the biases incurred by naive entropic regularizations \citep{feydy2019interpolating}.
\citet*{chizat2025doubly} recently provided a ``doubly-regularized'' entropic formulation that simultaneously embeds an inner regularization on the optimal transport between the barycenter and the input measures, and an outer regularization term of the barycenter's differential entropy.
In particular, \citet{chizat2025doubly} has provided approximation error bounds of doubly-regularized EOT barycenters to the unregularized Wasserstein barycenter, as well as stability bounds with respect to perturbations of the input measures.
The formulation of \citet{chizat2025doubly} is flexible enough to recover several existing entropy-regularized Wasserstein barycenters, notably 
inner-regularized
barycenters \citep{cuturi2014fast,cuturi2018semidual} and outer-regularized barycenters \citep{bigot2019penalization, carlier2021entropic}; see \citep[Section~1.3]{chizat2025doubly}.
Dedicated numerical algorithms for these variants are presented in the aforementioned references.

\subsubsection*{Estimation of OT maps}
Motivated by the unavailability of tractable exact OT maps between general continuous probability measures,
our stochastic fixed-point algorithm utilizes consistent OT map estimators.
Despite that the estimation of OT maps often incurs difficult computations in light of the difficulty of evaluating the \mbox{$\CW_2$-distance} in high dimensions \citep{tacskesen2023discrete,taskesen2023semi},
diverse types of OT map estimators with desirable statistical and structural properties have been proposed in the literature.
By imposing smoothness assumptions on the ground-truth OT map between measures, 
\citet*{hutter2021minimax} established a universal lower bound on the minimax convergence rate of any data-driven OT map estimator. 
In particular, as discussed in \citep[Appendix~E]{hutter2021minimax}, their assumptions are satisfied when the source and target measures fulfill a collection of regularity conditions, which is provided by Caffarelli's regularity theory; see, e.g.,\@ \citep[Theorem~12.50]{villani2009optimal}.
Subsequently, concrete classes of OT map estimators with provably sharp convergence rates are developed under such Caffarelli-type settings, notably 
the (empirical and smoothed) plug-in estimators \citep{manole2024plugin,deb2021rates}, 
the entropic estimator \citep{pooladian2021entropic}, 
the kernel sum-of-squares estimator \citep{muzellec2021near}, 
among others.
Besides achieving favorable statistical rates, OT map estimators are typically expected to preserve geometric properties of the underlying true OT map, e.g.,\@ as the gradient of a convex function. 
To this end, many structure-preserving estimation schemes have been developed.
For example, \citet*{paty2020regularity} leveraged the convex interpolability framework of \citet*{taylor2017convex} to seek a proxy of the Brenier potential with desired curvature; 
\citet*{curmei2023shape} parametrized the OT map as the gradient of a finite-degree polynomial and leveraged sum-of-squares matrices to enforce the shape constraints; 
\citet*{gonzalez2022gan} deployed Lipschitz-constrained generative adversarial networks for OT map estimation. 

\section{Preliminary notions and results}\label{sec: preliminaries}

\subsection{Notions and notations}\label{ssec: notations}
In the following, we introduce the terminologies and notations that are used throughout this paper. 
We let 
$\lfloor \,\cdot\,\rfloor$, $\lceil \,\cdot\,\rceil$
denote the floor and ceiling functions, respectively, 
i.e., for $a\in\R$,
$\lfloor a\rfloor$ is the greatest integer less than or equal to $a$,
$\lceil a\rceil$ is the least integer greater than or equal to $a$.
We let $a\vee b:=\max\{a,b\}$,
$a\wedge b:=\min\{a,b\}$ 
for all $a,b\in\R$.
All vectors are assumed to be column vectors and are denoted by boldface symbols. 
In particular, for $k\in\N$, $\veczero_k$ and $\vecone_k$ denote the vector in $\R^k$ with all entries equal to zero and one, respectively. 
We denote by $\langle \,\cdot\,, \cdot \,\rangle$ the Euclidean dot product, i.e., $\langle \BIx, \BIy \rangle := \BIx^\TRANSP \BIy$ and 
we denote the $p$-norm by $\|\cdot\|_p$ for any 
$p\in[1,\infty]$.
We use $\|\cdot\|$ without subscript to denote the Euclidean norm for notational simplicity, i.e., $\|\BIx\|:=\|\BIx\|_2=\big(\langle\BIx,\BIx\rangle\big)^{\frac{1}{2}}$. 
Open and closed Euclidean balls centered at $\BIx$ with radius $r>\nobreak0$ are denoted by $B(\BIx,r)$ and $\bar{B}(\BIx,r)$, respectively.
For any $k\in\N$ and any set $\CX \subseteq \R^k$, we let $\clos(\CX)$ denote its closure and let $\boundary(\CX)$ denote its boundary.
Moreover, for $k\in\N$, we let $\BO_k$ denote the $k$-by-$k$ zero matrix, and let $\BI_k$ denote the $k$-by-$k$ identity matrix. 
We use $\mathbb{S}^k$, $\mathbb{S}^k_{+}$, and $\mathbb{S}^k_{++}$ to denote the set of $k$-by-$k$ matrices that are symmetric, symmetric positive semi-definite, and symmetric positive definite, respectively.
For $\BA,\BB\in\mathbb{S}^k$, $\BA\succeq\BB$ denotes $\BA-\BB\in\mathbb{S}^k_+$. 
Furthermore, the smallest and the largest eigenvalues of any $\BA\in\mathbb{S}^k$ are denoted by $e_{\MINSUB}(\BA)$ and $e_{\MAXSUB}(\BA)$, respectively.

For a closed subset $\CX$ of a Euclidean space, we let $\CB(\CX)$ denote the Borel subsets of $\CX$, and let $\CP(\CX)$ denote the set of Borel probability measures on $\CX$, while $\CP_2(\CX)\subseteq\CP(\CX)$ consists of the ones with finite second moments. 
We view $\CP_2(\CX)$ as a Polish space equipped with the $\CW_2$-metric in (\ref{eqn: W2-distance}).
For any $\BIx\in\R^d$,
we use $\delta_{\BIx}$ to denote the Dirac measure at~$\BIx$.
The set $\CP_{2, \AC}(\CX)$ consists of all probability measures in $\CP_2(\CX)$ which are absolutely continuous with respect to the Lebesgue measure. 
When analyzing functions defined on any subset $\CM\subseteq\CP_{2}(\CX)$,
the notion of Borel measurability is defined with respect to the Borel $\sigma$-algebra on $\CM$ generated from the subspace topology inherited from $\big(\CP_2(\CX),\CW_2\big)$.
For any $\mu\in\CP(\CX)$ and any $\CY\in\CB(\CX)$ with $\mu(\CY)>0$, we use $\mu|_{\CY}$ to denote the probability measure formed by truncating $\mu$ to $\CY$, i.e., $\mu|_{\CY}(A):=\frac{\mu(\CY\intersects A)}{\mu(\CY)}$ $\forall A\in\CB(\CX)$.
For closed subsets $\CX,\CY$ of Euclidean spaces,
the pushforward of a probability measure $\mu\in\CP(\CX)$ by a Borel measurable function $T:\CX\to\CY$ is denoted by $T\sharp\mu\in\CP(\CY)$.

Let us also introduce the notations for the following function classes. 
For any $\mu\in\CP(\CX)$ where $\CX\subseteq\R^d$ is closed and 
for any $p\in[1,\infty)$, 
we let $\CL^p(\mu)$ denote the set of Borel measurable functions $f:\CX\to\R$ where $|f|^p$ is $\mu$-integrable.
Moreover, for any Borel measurable function $T:\R^d\to\R^d$ and any probability measure $\mu\in\CP(\R^d)$,
we define
\begin{align*}
	\|T\|_{\CL^2(\mu)}:=\bigg(\int_{\R^d}\big\|T(\BIx)\big\|^2\DIFFM{\mu}{\DIFF\BIx}\bigg)^{\frac{1}{2}}.
\end{align*}
For a bounded open set $\CX\subset\R^d$ and for $q\in\N_0$, $\alpha\in(0,1]$, 
we let $\CC^q(\clos(\CX))$ denote the set of \mbox{$\R$-valued} continuous functions on $\clos(\CX)$ that are $q$-times continuously differentiable on $\CX$, 
let $\CC^{\infty}(\clos(\CX))$ denote the set of infinitely differentiable $\R$-valued functions on $\clos(\CX)$,
and let $\CC^{q,\alpha}(\clos(\CX))$ denote the set of \mbox{$\R$-valued} continuous functions on $\clos(\CX)$ that are $q$-times continuously differentiable on $\CX$ whose $q$-th order partial derivatives are $\alpha$-H{\"o}lder continuous.
In particular, $\CC^{q,\alpha}(\clos(\CX))$ is a Banach space with respect to the following norm (see, e.g., \citep[Theorem~5.1.1]{evans2022partial}):
\begin{align*}
	\|\varphi\|_{\CC^{q,\alpha}(\clos(\CX))}:=\max_{|\Bbeta|\le q}\sup_{\BIx\in\CX}\big\{\big|\partial^{\Bbeta}\varphi(\BIx)\big|\big\}+\max_{|\Bbeta|=q}\sup_{\BIx,\BIy\in\CX,\,\BIx\ne\BIy}\Bigg\{\frac{\big|\partial^{\Bbeta}\varphi(\BIx)-\partial^{\Bbeta}\varphi(\BIy)\big|}{\|\BIx-\BIy\|^{\alpha}}\Bigg\} \qquad\forall \varphi\in \CC^{q,\alpha}(\clos(\CX)),
\end{align*}
where $\Bbeta=(\beta_1,\ldots,\beta_d)\in\N_0^d$ is a multi-index, $|\Bbeta|:=\beta_1+\cdots+\beta_d$, and $\partial^{\Bbeta}\varphi:=\frac{\partial^{|\Bbeta|}\varphi}{\partial x_1^{\beta_1}\cdots\partial x_d^{\beta_d}}$ denotes the partial derivative of $\varphi$ with respect to the multi-index $\Bbeta$. 
We call $\CC^{q,\alpha}(\clos(\CX))$ the set of $(q,\alpha)$-H{\"o}lder functions on $\clos(\CX)$. 
Moreover, we let $\CC^{\LOCAL,q,\alpha}(\R^d)$ denote the set of \mbox{$\R$-valued} functions on $\R^d$ that are $(q,\alpha)$-H{\"o}lder when restricted to the closure of any bounded open set.
We call $\CC^{\LOCAL,q,\alpha}(\R^d)$ the set of locally $(q,\alpha)$-H{\"o}lder functions on $\R^d$. 
Furthermore, we denote by $\CC_{\LIN}(\R^d,\R^d)$ the set of continuous functions from $\R^d$ to $\R^d$ that have at most linear growth, i.e., $T\in\nobreak\CC_{\LIN}(\R^d,\R^d)$ if and only if $T:\R^d\to\R^d$ is continuous and \mbox{$\sup_{\BIx\in\R^d}\Big\{\frac{\|T(\BIx)\|}{1+\|\BIx\|}\Big\}<\infty$}. 
Note that $\CC_{\LIN}(\R^d,\R^d)$ is a Banach space with respect to the norm $\|T\|_{\CC_{\LIN}(\R^d,\,\R^d)}:=\sup_{\BIx\in\R^d}\Big\{\frac{\|T(\BIx)\|}{1+\|\BIx\|}\Big\}$ \mbox{$\forall T\in \CC_{\LIN}(\R^d,\R^d)$}. 
Lastly, 
we let $I_d:\R^d\to\R^d$ denote the identity map on $\R^d$.

We use $\partial \varphi(\BIx)\subseteq\R^d$ to denote 
the subdifferential of any proper, lower semi-continuous (l.s.c.), and convex function $\varphi:\R^d\to\R\cup\{\infty\}$ at $\BIx \in \R^d$.
For $0\le \underline{\lambda}\le \overline{\lambda}\le \infty$, a proper, l.s.c., and convex function $\varphi: \R^d \rightarrow \R \cup \{\infty\}$ is called \textit{$\overline{\lambda}$-smooth} 
(denoted by $\varphi \in \FC_{0, \overline{\lambda}}(\R^d)$) if
\begin{align*}
	\varphi(\BIy) \leq \varphi(\BIx) + \langle\BIg, \BIy - \BIx\rangle + \frac{\overline{\lambda}}{2}\|\BIx - \BIy\|^2 \qquad\forall \BIx,\BIy\in\R^d, \; \forall \BIg\in\partial \varphi(\BIx),
\end{align*}
and is called \textit{$\underline{\lambda}$-strongly convex} 
(denoted by $\varphi \in \FC_{\underline{\lambda}, \infty}(\R^d)$) if
\begin{align*}
	\varphi(\BIy) \geq \varphi(\BIx) + \langle\BIg, \BIy - \BIx\rangle + \frac{\underline{\lambda}}{2}\|\BIx - \BIy\|^2 \qquad \forall \BIx,\BIy\in\R^d, \; \forall \BIg\in\partial \varphi(\BIx).
\end{align*}
It follows from classical results (see, e.g., \citep[Lemma~1.2.3 \& Theorem~2.1.5]{nesterov2004introconvex}) that for $\overline{\lambda}<\infty$, every $\varphi\in\FC_{0,\overline{\lambda}}(\R^d)$ is continuously differentiable on $\R^d$ and $\nabla\varphi$ is $\overline{\lambda}$-Lipschitz continuous. 
For $0\le \underline{\lambda}\le\overline{\lambda}\le \infty$,
we denote by $\FC_{\underline{\lambda}, \overline{\lambda}}(\R^d)$ the collection of proper, l.s.c.\@, and convex functions on $\R^d$ which are $\overline{\lambda}$-smooth and \mbox{$\underline{\lambda}$-strongly} convex.
In particular, $\FC_{0, \infty}(\R^d)$ consists of all proper l.s.c.\@ convex functions on $\R^d$. 
In addition, we denote 
$\FC^{\infty}_{\underline{\lambda},\overline{\lambda}}(\R^d):=\CC^{\infty}(\R^d)\intersects\FC_{\underline{\lambda},\overline{\lambda}}(\R^d)$, 
$\FC^{q}_{\underline{\lambda},\overline{\lambda}}(\R^d):=\CC^{q}(\R^d)\intersects\FC_{\underline{\lambda},\overline{\lambda}}(\R^d)$, 
$\FC^{\LOCAL,q,\alpha}_{\underline{\lambda},\overline{\lambda}}(\R^d):=\CC^{\LOCAL,q,\alpha}(\R^d)\intersects\FC_{\underline{\lambda},\overline{\lambda}}(\R^d)$ for $q\in\N_0$, $\alpha\in(0,1]$.

\subsection{Preliminary results about optimal transport and Wasserstein barycenter}\label{ssec: preliminary-OT}

Many of our discussions in this paper invoke results from the optimal transport theory and properties around the $\CW_2$-distance between probability measures; see, e.g., the books of \citet*{villani2003topics, villani2009optimal}, \citet*{santambrogio2015optimal}, 
\citet*{ambrosio2008gradient}, \citet*{chewi2024statistical}. We start by recalling the notion of couplings. 
\begin{definition}[Coupling]\label{def: coupling}
	Given $m \in \N$ probability measures $\nu_1 \in \CP(\CX_1), \dots, \nu_m \in \CP(\CX_m)$ on closed subsets $\CX_1, \dots, \CX_m$ of $\R^d$, the set of couplings of $\nu_1, \dots, \nu_m$ is denoted by $\Pi(\nu_1, \dots, \nu_m)$, which is defined as
	\begin{align*}
		\Pi(\nu_1, \ldots, \nu_m) := \big\{ \pi \in \CP(\CX_1 \times \cdots \times \CX_m): \text{ the marginal of }\pi \text{ on } \CX_i \text{ is } \nu_i \text{ for }i = 1, \dots, m \big\}.
	\end{align*}
\end{definition}

The minimization problem embedded in the formulation (\ref{eqn: W2-distance}) is known as Kantorovich's OT problem \citep{kantorovich1948problem} with respect to the squared-Euclidean cost, and the infimum is well known to be attained by an optimal coupling; see, e.g., \citep[Theorem~4.1]{villani2009optimal}. In the rest of this paper, the optimality of a coupling is always considered with respect to the squared-Euclidean cost.
The existence of a $\CW_2$-barycenter is shown by \citet*[Proposition~2.3]{agueh2011barycenters}, and there may exist more than one $\CW_2$-barycenter of $\nu_1, \dots, \nu_K$ in general. A sufficient condition to guarantee the uniqueness of the $\CW_2$-barycenter is given as follows.

\begin{theorem}[{\citep[Proposition~3.5 \& Theorem~5.1]{agueh2011barycenters}}]\label{thm: unique barycenter}
	Among $\nu_1, \dots, \nu_K \in \CP_2(\R^d)$, if there exists at least one index $k \in \{1, \dots, K\}$ such that $\nu_k \in \CP_{2, \AC}(\R^d)$, then the $\CW_2$-barycenter $\bar{\mu}$ of $\nu_1, \dots, \nu_K \in \CP_2(\R^d)$ is unique. Moreover, if there exists at least one index $k \in \{1, \dots, K\}$ such that 
	the density function of $\nu_k$ is essentially bounded,
	then the unique $\CW_2$-barycenter $\bar{\mu}$ belongs to $\CP_{2, \AC}(\R^d)$. 
\end{theorem}

Next, let us present Brenier's theorem which characterizes optimal couplings with gradients of convex functions when the source measure $\mu$ belongs to $\CP_{2,\AC}(\R^d)$; see, e.g., \citep[Theorem~2.12]{villani2003topics}.

\begin{theorem}[Brenier's theorem]\label{thm: Brenier}
	Let $\mu\in\CP_{2,\AC}(\R^d)$, $\nu \in \CP_2(\R^d)$. 
	Then, there is a unique optimal coupling $\pi^\star \in \Pi(\mu, \nu)$ that minimizes (\ref{eqn: W2-distance}). 
	Moreover, $\pi\in \Pi(\mu, \nu)$ minimizes (\ref{eqn: W2-distance}) if and only if there exists a proper, l.s.c., and convex function $\varphi^{\mu}_{\nu}:\R^d\to\R\cup\{\infty\}$ such that $\pi = \big[I_d, T^{\mu}_{\nu}\big] \sharp\mu$ where $T^{\mu}_{\nu}=\nabla\varphi^{\mu}_{\nu}$ is the $\mu$-almost everywhere unique gradient of $\varphi^{\mu}_{\nu}$. 
	In this case, it holds that 
	\begin{align*}
		\CW_2(\mu,\nu)^2=\int_{\R^d}\|\BIx\|^2-2\varphi^{\mu}_{\nu}(\BIx)\DIFFM{\mu}{\DIFF\BIx} + \int_{\R^d} \|\BIy\|^2 - 2 \sup_{\BIx\in\R^d}\big\{\langle\BIy,\BIx\rangle-\varphi^{\mu}_{\nu}(\BIx)\big\}\DIFFM{\nu}{\DIFF\BIy},
	\end{align*}
	and $T^{\mu}_{\nu}$ is Monge's OT map from $\mu$ to $\nu$, i.e.,\@ it is the $\mu$-almost everywhere unique optimal solution of the following Monge's OT problem:
	\begin{align*}
		T^{\mu}_{\nu}\in\argmin_{T}\bigg\{\int_{\R^d} \big\|\BIx - T(\BIx)\big\|^2 \DIFFM{\mu}{\DIFF \BIx}:  T: \R^d \rightarrow \R^d \text{ is Borel measurable and }T \sharp \mu = \nu \bigg\}.
	\end{align*}
\end{theorem}
We refer to $\varphi^{\mu}_{\nu}:\R^d\to\R\cup\{\infty\}$ and $T^{\mu}_{\nu}:\R^d\to\R^d$ in Theorem~\ref{thm: Brenier} as the \textit{Brenier potential} from $\mu$ to $\nu$ and the \textit{optimal transport (OT) map} from $\mu$ to $\nu$, respectively. 
In general, the $\mu$-almost everywhere uniqueness of $T^{\mu}_{\nu}$ does not necessarily imply the $\mu$-almost everywhere uniqueness of $\varphi^{\mu}_{\nu}$ even up to an additive constant.
However, $\varphi^{\mu}_{\nu}$ becomes $\mu$-almost everywhere uniquely determined up to an additive constant if $\support(\mu)$ is the closure of a connected open set on which $\mu$ has positive density; see, e.g., \citep[Remark~10.30]{villani2009optimal}.

It is well-known that $\CP_{2,\AC}(\R^d)$ equipped with the $\CW_2$-metric admits a
Riemann-like structure in which 
$\mu(\lambda) := \big[I_d+\lambda(T^{\mu_0}_{\mu_1}-I_d)\big]\sharp\mu_0$ $\forall \lambda\in[0,1]$
acts as the unique geodesic curve between any $\mu_0,\mu_1\in\CP_{2,\AC}(\R^d)$;
see, e.g., \citep[Section~3.2]{zemel2019frechet} and 
\citep[Appendix~A]{chewi2020gradient} for concise summaries of the related notions and results.
This property allows one to characterize 
the Fr{\'e}chet derivative of certain functionals on $\CP_{2,\AC}(\R^d)$.
In particular, for any $\nu\in\CP_{2,\AC}(\R^d)$,
the Fr{\'e}chet derivative of the functional 
$\CP_{2,\AC}(\R^d)\ni \mu\mapsto \frac{1}{2}\CW_2(\mu,\nu)^2\in\R$ at any $\mu\in\CP_{2,\AC}(\R^d)$
is equal to
$-(T^{\mu}_{\nu}- I_d)$ \citep[Corollary~10.2.7]{ambrosio2008gradient}.
Consequently, this yields the following Fr{\'e}chet derivative of the barycenter functional $V$ in (\ref{eqn: V-definition}), denoted by $\boldsymbol{\nabla}V(\cdot)$:
\begin{align*}
	\boldsymbol{\nabla}V(\mu)= 2I_d-2\Bigg(\sum_{k=1}^{K}w_kT^{\mu}_{\nu_k}\Bigg) \qquad \forall \mu\in\CP_{2,\AC}(\R^d).
\end{align*}
$\boldsymbol{\nabla}V(\cdot)$ is called the \textit{Wasserstein gradient} of $V$.
This grants the $G$-iteration (\ref{eqn: G-iteration}) an alternative interpretation 
as a gradient descent step: $\mu_t=G(\mu_{t-1})=\big[I_d-\frac{1}{2}\boldsymbol{\nabla}V(\mu_{t-1})\big]\sharp\mu_{t-1}$ with the optimal step size $\frac{1}{2}$; see \citep[Lemma~2]{zemel2019frechet}.
Additionally, 
observe that $\big\|\boldsymbol{\nabla}V(\mu)\big\|^2_{\CL^2(\mu)}=4\CW_2(\mu,G(\mu))^2$
and $\mu$ is a fixed-point of $G$ if and only if 
$\boldsymbol{\nabla}V(\mu)=\nobreak0$ ($\mu$-almost everywhere). 
Thus, any fixed-point of~$G$ can be interpreted as a stationary point of the barycenter functional~$V$ (also known as the Karcher mean \citep[Corollary~1]{zemel2019frechet}),
i.e., where the Wasserstein gradient vanishes. 
Note, however, that $V$ does not necessarily possess geodesic convexity, 
i.e., there exist $\mu_0,\mu_1\in\CP_{2,\AC}(\R^d)$ where
the condition
$V\big([I_d+\lambda(T^{\mu_0}_{\mu_1}-\nobreak I_d)]\sharp\mu_0\big)\le (1-\nobreak\lambda)V(\mu_0)+\lambda V(\mu_1)$ $\forall \lambda\in[0,1]$ 
is violated;
see \citep[Appendix~B.2]{chewi2020gradient} for a counterexample demonstrating this violation.
The lack of geodesic convexity leads to the difficulty of guaranteeing that $G$ possesses a unique fixed-point.
Thus, convergence analyses of Wasserstein barycenter algorithms that arise from the \mbox{$G$-iteration} (\ref{eqn: G-iteration})
are analogous to 
convergence analyses of gradient descent algorithms minimizing non-convex functions in a Euclidean space.
This feature of our convergence analysis will be discussed further in Section~\ref{ssec: main-convergence-conditions}.

\section{Convergence analysis}\label{sec: fixedpoint-convergence}

\subsection{Statement of the main convergence theorem}\label{ssec: fixedpoint-convergence-statement}

In this section, we work under the following mild regularity assumptions about the input probability measures $\nu_1,\ldots,\nu_K$ as well as 
the OT map estimator $\widehat{T}_{t,k}$
in Line~\ref{alglin: abstract-estimator} of Algorithm~\ref{algo: abstract}.

\begin{assumption}[Regularity assumptions of Algorithm~\ref{algo: abstract}]\label{asp: abstract-regularity}
	There exists at least one index $k\in\{1,\ldots,K\}$
	such that 
	the density function of the $k$-th input measure $\nu_k\in\CP_{2,\AC}(\R^d)$ is essentially bounded.
	The set $\Theta$ is a metric space representing the space of hyperparameter(s) used in the OT map estimator in Line~\ref{alglin: abstract-estimator}.
	Moreover,
	for each $t\in\N$ and for $k=1,\ldots,K$, 
	$\widehat{T}_{t,k}$ belongs to $\CC_{\LIN}(\R^d,\R^d)$,
	and has a Borel dependency on 
	\interfootnotelinepenalty=10000
	$\big(\BIX_{t,k,1},\ldots,\BIX_{t,k,\widehat{M}_{t-1,k}},\allowbreak\BIY_{t,k,1},\ldots,\BIY_{t,k,\widehat{N}_{t-1,k}},\widehat{\ITheta}_{t-1,k}\big)$.\footnote{We say that $\widehat{T}_{t,k}$ 
	has a Borel dependency on 
	$\big(\BIX_{t,k,1},\ldots,\allowbreak\BIX_{t,k,\widehat{M}_{t-1,k}},\allowbreak\BIY_{t,k,1},\ldots,\BIY_{t,k,\widehat{N}_{t-1,k}},\widehat{\ITheta}_{t-1,k}\big)$ 
	if for every 
	$m\in\nobreak\N$
	and every
	$n\in\nobreak\N$
	there exists a Borel measurable function 
	$h_{m,n}:(\R^d)^{m+n}\times\Theta\to\CC_{\LIN}(\R^d,\R^d)$
	such that $\widehat{T}_{t,k}\INDI_{\{\widehat{M}_{t-1,k}=m,\,\widehat{N}_{t-1,k}=n\}}=h_{m,n}\big(\allowbreak\BIX_{t,k,1},\ldots,\allowbreak\BIX_{t,k,m},\allowbreak\BIY_{t,k,1},\ldots,\BIY_{t,k,n},\widehat{\ITheta}_{t-1,k}\big)\INDI_{\{\widehat{M}_{t-1,k}=m,\,\widehat{N}_{t-1,k}=n\}}$
	$\PROB$-almost surely.
	Analogous notions of Borel dependency apply to subsequent arguments.}
\end{assumption}

Let $(\Omega,\CF,\PROB)$ be a probability space on which the random samples 
in Line~\ref{alglin: abstract-sample1} and Line~\ref{alglin: abstract-sample2}
of Algorithm~\ref{algo: abstract}
are defined. 	
Under Assumption~\ref{asp: abstract-regularity},
Algorithm~\ref{algo: abstract} generates an $(\CF_t)_{t\in\N_0}$-adapted stochastic process $(\widehat{\mu}_{t})_{t\in\N_0}$, 
where
\begin{align}
	\CF_{0}&:=\{\emptyset,\Omega\},\nonumber\\
	\CF_{1}&:=\sigma\Bigg(\bigcup_{k=1}^{K}
		\Bigg[
			\Bigg(\bigcup_{i=1}^{\widehat{M}_{0,k}}\sigma(\BIX_{1,k,i})\Bigg)
			\cup 
			\Bigg(\bigcup_{j=1}^{\widehat{N}_{0,k}}\sigma(\BIY_{1,k,j})\Bigg)\Bigg]\Bigg),\label{eqn: filtration}\\
	\CF_{t}&:=\sigma\left(\CF_{t-1}\cup \left\{A\in\CF:\!\!\begin{tabular}{l}
		$A\cap 
			\Big(\bigcap_{k=1}^{K}
				\big(
					\{\widehat{M}_{t-1,k}=m_{t-1,k}\}
					\cap 
					\{\widehat{N}_{t-1,k}=n_{t-1,k}\}
				\big)
			\Big)$\\
		$\quad \in \sigma\Big(
			\bigcup_{k=1}^{K}\Big[
				\big(
					\bigcup_{i=1}^{m_{t-1,k}} \sigma(\BIX_{t,k,i})
				\big)
				\cup 
				\big(
					\bigcup_{j=1}^{n_{t-1,k}} \sigma(\BIY_{t,k,j})
				\big)
			\Big]\Big)$ \\
		$\hspace{142pt}\forall (m_{t-1,k},n_{t-1,k})_{k=1:K}\subset\N$
	\end{tabular}\!\!\right\}\right)\qquad \forall t\ge 2.\nonumber
\end{align}

The goal of this section is to prove the following main convergence theorem.

\begin{theorem}[Convergence analysis of Algorithm~\ref{algo: abstract}]\label{thm: main-convergence}
	Let $(\Omega,\CF,\PROB)$ be a probability space on which the random samples in Line~\ref{alglin: abstract-sample1} and Line~\ref{alglin: abstract-sample2} are defined, 
	let the regularity conditions in Assumption~\ref{asp: abstract-regularity} hold,
	and let $(\CF_t)_{t\in\N_0}$ be defined by (\ref{eqn: filtration}).
	Moreover, 
	let $\bar{\mu}$ denote the unique $\CW_2$-barycenter of $\nu_1,\ldots,\nu_K$ with weights $w_1,\ldots,w_K$.
	Then, whenever the following inequalities are satisfied for some $\beta\in(0,1)$:
	\begin{align}
		\EXP\Big[\big\|\widehat{T}_{t,k}-T^{\widehat{\mu}_{t-1}}_{\nu_k}\big\|^2_{\CL^2(\widehat{\mu}_{t-1})}\Big|\CF_{t-1}\Big] & \le \beta^{t} \qquad \forall 1\le k\le K,\; \forall t\in\N, \label{eqn: main-convergence-cond1}\\
		\EXP\Big[\CW_2\big(\big[{\textstyle\sum_{k=1}^{K}w_k\widehat{T}_{t,k}}\big]\sharp \widehat{\mu}_{t-1},\widehat{\mu}_t\big)^2\Big|\CF_{t-1}\Big] & \le \beta^{t} \qquad \hphantom{\forall 1\le k\le K,\; }\, \forall t\in\N, \label{eqn: main-convergence-cond2}
	\end{align}
	the following statements hold.
	\begin{enumerate}[label=(\roman*),beginpenalty=10000]
		\item\label{thms: main-convergence-tight-as}%
		It holds $\PROB$-almost surely that $(\widehat{\mu}_t)_{t\in\N_0}$ is precompact with respect to the $\CW_2$-metric.
		Moreover, every accumulation point of $(\widehat{\mu}_t)_{t\in\N_0}$ with respect to the $\CW_2$-metric is a fixed-point of $G$. 

		\item\label{thms: main-convergence-as}%
		If $G$ has a unique fixed-point, then $(\widehat{\mu}_t)_{t\in\N_0}$ converges $\PROB$-almost surely in $\CW_2$ to $\bar{\mu}$.

		\item\label{thms: main-convergence-violation}%
		The process $\big(\CW_2(\widehat{\mu}_t,G(\widehat{\mu}_t))^2\big)_{t\in\N_0}$ measuring the violation of the fixed-point property of $(\widehat{\mu}_t)_{t\in\N_0}$
		obeys the following best-iterate sublinear rate in the number of iterations~$t$:
		\begin{align*}
			\EXP\bigg[\min_{0\le s\le t}\big\{\CW_2(\widehat{\mu}_{s},G(\widehat{\mu}_{s}))^2\big\}\bigg] & \le \bigg(V(\widehat{\mu}_0)-V(\bar{\mu})+ \frac{4\beta}{1-\beta}\bigg)(t+1)^{-1} \qquad \forall t\in\N_0.
		\end{align*}

		\item\label{thms: main-convergence-PL}%
		Let us assume in addition that the following expected Polyak--{\L}ojasiewicz inequality holds with respect to some $\PLINEQCONST\in(0,1]$:
		\begin{align}
			\EXP\big[V(\widehat{\mu}_t)\big] - V(\bar{\mu}) & \le \frac{1}{\PLINEQCONST}\EXP\big[\CW_2(\widehat{\mu}_t,G(\widehat{\mu}_t))^2\big] \qquad \forall t\in \N_0.
			\tag{\textsf{Ex-P\L}}
			\label{eqn: main-convergence-condPL}
		\end{align}
		Then, $(\widehat{\mu}_t)_{t\in\N_0}$ converges $\PROB$-almost surely in $\CW_2$ to $\bar{\mu}$,
		and $\big(V(\widehat{\mu}_t)\big)_{t\in\N_0}$ converges in $\CL^1$ to~$V(\bar{\mu})$ with respect to the following rate:
		\begin{align}
			\begin{split}
			\hspace{35pt}\EXP\big[V(\widehat{\mu}_t)\big]-V(\bar{\mu})&\le 
			\big(V(\widehat{\mu}_0) - V(\bar{\mu})\big)(1-\PLINEQCONST)^{t}\\
			&\qquad +
			\begin{cases}
				\frac{4\beta}{|1-\PLINEQCONST-\beta|} \big((1-\PLINEQCONST)\vee \beta\big)^t & \text{if }\beta\ne 1-\PLINEQCONST \\
				4t\beta^{t} & \text{if }\beta= 1-\PLINEQCONST
			\end{cases} \qquad \forall t\in\N_0.
		\end{split}
		\label{eqn: main-convergence-linearrate}
		\end{align}
		In particular, 
		the above convergence rate is geometric in the number of iterations~$t$ because $\EXP\big[V(\widehat{\mu}_t)\big]-V(\bar{\mu})=O\Big(\big((1-\PLINEQCONST)\vee \beta\big)^{(1-\delta) t}\Big)$
		for any $\delta\in(0,1)$.

		\item\label{thms: main-convergence-PL-var}%
		In addition to the assumption of statement~\ref{thms: main-convergence-PL},
		let us assume that the following expected variance inequality holds with respect to some $\VARINEQCONST>0$:
		\begin{align}
			\EXP\big[V(\widehat{\mu}_t)\big] - V(\bar{\mu}) & \ge \VARINEQCONST\EXP\big[\CW_2(\widehat{\mu}_t,\bar{\mu})^2\big] \qquad \forall t\in \N_0.
			\tag{\textsf{Ex-Var}}
			\label{eqn: main-convergence-condvar}
		\end{align}
		Then, $\big(\CW_2(\widehat{\mu}_t,\bar{\mu})^2\big)_{t\in\N_0}$ converges in $\CL^1$ to~0 with respect to the following rate:
		\begin{align*}
			\begin{split}
			\hspace{14pt}\EXP\big[\CW_2(\widehat{\mu}_t,\bar{\mu})^2\big]
			&\le  \frac{1}{\VARINEQCONST}\big(V(\widehat{\mu}_0) - V(\bar{\mu})\big)(1-\PLINEQCONST)^{t}\\
			&\qquad +\begin{cases}
				\frac{4\beta}{\VARINEQCONST|1-\PLINEQCONST-\beta|} \big((1-\PLINEQCONST)\vee \beta\big)^t & \text{if }\beta\ne 1-\PLINEQCONST \\
				\frac{4}{\VARINEQCONST}t\beta^{t} & \text{if }\beta= 1-\PLINEQCONST
			\end{cases} \qquad \forall t\in\N_0.
		\end{split}
		\end{align*}
		In particular, the above convergence rate is geometric in the number of iterations~$t$ because \linebreak$\EXP\big[\CW_2(\widehat{\mu}_t,\bar{\mu})^2\big]= O\Big(\big((1-\PLINEQCONST)\vee \beta\big)^{(1-\delta)t}\Big)$ for any $\delta\in(0,1)$.
	\end{enumerate}
\end{theorem}

Notice that, to attain the desired convergence rates in the number of iterations, one needs increasingly accurate
OT map estimators $(\widehat{T}_{t,k})_{k=1:K}$ as Algorithm~\ref{algo: abstract} progresses, 
as implied by the inequality (\ref{eqn: main-convergence-cond1}) in Theorem~\ref{thm: main-convergence}. 
Accordingly, the sample sizes $(\widehat{M}_{t,k},\widehat{N}_{t,k})_{k=1:K}$ are typically expected to increase with respect to~$t$;
see Proposition~\ref{prop: special-case-1d}\ref{props: special-case-1d-inequalities} in Section~\ref{ssec: main-convergence-conditions} for an example.
In the remainder of this section, 
we will perform an analysis of the decrements of the stochastic process 
$\big(V(\widehat{\mu}_t)\big)_{t\in\N_0}$
in Section~\ref{ssec: decrement},
and we will present the proof of Theorem~\ref{thm: main-convergence}
in Section~\ref{ssec: proof-main-convergence}.
In Section~\ref{ssec: main-convergence-conditions},
we will discuss the various conditions in Theorem~\ref{thm: main-convergence},
including 
the uniqueness of the fixed-point of~$G$,
\eqref{eqn: main-convergence-condPL}, and
\eqref{eqn: main-convergence-condvar}.
In particular, 
we will provide a one-dimensional setting
to show that 
all conditions in Theorem~\ref{thm: main-convergence} including 
(\ref{eqn: main-convergence-cond1}) and
(\ref{eqn: main-convergence-cond2}) 
can be \textit{simultaneously satisfied}.

\subsection{Decrement analysis}\label{ssec: decrement}

Before proving Theorem~\ref{thm: main-convergence},
let us first analyze the decrements of the stochastic process
$\big(V(\widehat{\mu}_t)\big)_{t\in\N_0}$
in the following lemma.

\begin{lemma}[Decrement of the process $\big(V(\widehat{\mu}_t)\big)_{t\in\N_0}$]\label{lem: decrement}
	Let Assumption~\ref{asp: abstract-regularity} hold, 
	let $(\Omega,\CF,\PROB)$ be a probability space on which the random samples in Line~\ref{alglin: abstract-sample1} and Line~\ref{alglin: abstract-sample2} of Algorithm~\ref{algo: abstract} are defined, 
	and let $(\CF_t)_{t\in\N_0}$ be defined by (\ref{eqn: filtration}).
	Moreover, let $V$ be the function defined in (\ref{eqn: V-definition}) and let $G$ be the operator defined in (\ref{eqn: G-operator}). 
	Then, the stochastic process $\bigl(V(\widehat{\mu}_t)\bigr)_{t \in \N_0}$ satisfies
	\begin{align}
		\label{eqn: decrement-pathwise}
			\begin{split}
			V(\widehat{\mu}_{t}) -	V(\widehat{\mu}_{t-1}) 
			&\le {-\CW_2}\big(\widehat{\mu}_{t-1}, G(\widehat{\mu}_{t-1}) \big)^2  
			+ 2 \sum_{k = 1}^{K} w_k \big\|\widehat{T}_{t,k} - T^{\widehat{\mu}_{t-1}}_{\nu_k} \big\|^2_{\CL^2(\widehat{\mu}_{t-1})} \\
			&\qquad + 2 \CW_2\big(\big[{\textstyle\sum_{k=1}^Kw_k}\widehat{T}_{t,k}\big]\sharp\widehat{\mu}_{t-1}, \widehat{\mu}_{t}\big)^2 \hspace{77.7pt}\qquad\qquad\forall t\in\N,\;\PROB\text{-a.s.,}
			\end{split}
	\end{align}
	where all terms on the right-hand side of (\ref{eqn: decrement-pathwise}) are $\PROB$-almost surely finite for all $t\in\N$.
	In particular, taking conditional expectations with respect to $\CF_{t-1}$ on both sides of (\ref{eqn: decrement-pathwise}) yields
	\begin{align}
		\label{eqn: decrement-condexp}
		\begin{split}
			\EXP\big[V(\widehat{\mu}_{t}) \big| \CF_{t-1} \big] -	V(\widehat{\mu}_{t-1}) 
			&\le {-\CW_2} \big(\widehat{\mu}_{t-1}, G(\widehat{\mu}_{t-1}) \big)^2  
			+ 2\sum_{k = 1}^{K}w_k \EXP \Big[ \big\|\widehat{T}_{t,k} - T^{\widehat{\mu}_{t-1}}_{\nu_k} \big\|^2_{\CL^2(\widehat{\mu}_{t-1})} \Big | \CF_{t-1}\Big] \\
			&\qquad + 2 \EXP \Big[\CW_2\big(\big[{\textstyle\sum_{k=1}^Kw_k}\widehat{T}_{t,k}\big]\sharp\widehat{\mu}_{t-1}, \widehat{\mu}_{t}\big)^2 \Big| \CF_{t-1}\Big] \qquad\quad \forall t\in\N,\;\PROB\text{-a.s.}
		\end{split}
	\end{align}
\end{lemma}

\begin{proof}[Proof of Lemma~\ref{lem: decrement}]
	Throughout this proof, let us fix an arbitrary $t\in\N$, denote $\bar{T}^{\widehat{\mu}_{t-1}}:=\sum_{k=1}^Kw_k T^{\widehat{\mu}_{t-1}}_{\nu_k}$, $\bar{T}_{t}:=\sum_{k=1}^Kw_k \widehat{T}_{t,k}$, and denote $\widetilde{\mu}_{t}:=\bar{T}_{t}\sharp\widehat{\mu}_{t-1}$. 
	Since $\widehat{\mu}_{t-1}\in\CP_2(\R^d)$ by Line~\ref{alglin: abstract-initialize} and Line~\ref{alglin: abstract-update} of Algorithm~\ref{algo: abstract},
	and since $\bar{T}_{t}\in\CC_{\LIN}(\R^d,\R^d)$ by Assumption~\ref{asp: abstract-regularity},
	it holds $\PROB$-almost surely that
	\begin{align*}
		\int_{\R^d}\|\BIy\|^2\DIFFM{\widetilde{\mu}_{t}}{\DIFF\BIy}&=\int_{\R^d}\big\|\bar{T}_t(\BIy)\big\|^2\DIFFM{\widehat{\mu}_{t-1}}{\DIFF\BIy}\le \|\bar{T}_t\|_{\CC_{\LIN}(\R^d,\,\R^d)}^2\int_{\R^d}\big(1+\|\BIy\|\big)^2\DIFFM{\widehat{\mu}_{t-1}}{\DIFF\BIy}<\infty,
	\end{align*}
	and hence $\widetilde{\mu}_{t}\in\CP_{2}(\R^d)$ $\PROB$-almost surely.
	Moreover, for $k=1,\ldots,K$,
	we have 
	\begin{align*}
		\big\|\widehat{T}_{t,k}-T^{\widehat{\mu}_{t-1}}_{\nu_k}\big\|^2_{\CL^2(\widehat{\mu}_{t-1})}
		&\le 2\|\widehat{T}_{t,k}\|_{\CC_{\LIN}(\R^d,\,\R^d)}^2\int_{\R^d}\big(1+\|\BIx\|\big)^2\DIFFM{\widehat{\mu}_{t-1}}{\DIFF\BIx}
		+2\int_{\R^d}\|\BIy\|^2\DIFFM{\nu_k}{\DIFF\BIy}<\infty.
	\end{align*}
	Hence, we have shown that all terms on the right-hand side of (\ref{eqn: decrement-pathwise}) are $\PROB$-almost surely finite.

	Our proof uses the following identity, which can be verified directly by expanding both sides:
	\begin{align}
		\label{eqn: decrement-proof-identity}
		\begin{split}
			\sum_{k= 1}^{K}w_k \big\| \BIy - \BIz_k \big\|^2 &= \| \BIy - \bar{\BIz}\|^2 + \sum_{k = 1}^{K}w_k \| \bar{\BIz} - \BIz_k \|^2 \\
			& \qquad \qquad \text{where }\bar{\BIz}:=\sum_{k=1}^Kw_k \BIz_k \qquad \forall \BIy,\BIz_1,\ldots,\BIz_K\in\R^d.
		\end{split}
	\end{align}
	For any $\BIx\in\R^d$, substituting $\BIy \leftarrow \BIx$ and $\BIz_k \leftarrow T^{\widehat{\mu}_{t-1}}_{\nu_k}(\BIx)$ in (\ref{eqn: decrement-proof-identity}) gives us
	\begin{align}
		\label{eqn: decrement-proof-decompose-1}
		\sum_{k = 1}^{K}w_k \big\|\BIx - T_{\nu_k}^{\widehat{\mu}_{t-1}}(\BIx) \big\|^2 
		= \big\|\BIx - \bar{T}^{\widehat{\mu}_{t-1}}(\BIx)\big\|^2 
		+ \sum_{k = 1}^{K}w_k  \big\|\bar{T}^{\widehat{\mu}_{t-1}}(\BIx) - T_{\nu_k}^{\widehat{\mu}_{t-1}}(\BIx) \big\|^2 \qquad \forall \BIx\in\R^d.
	\end{align}
	Moreover, 
	for any $\BIx,\BIy\in\R^d$,
	substituting $\BIy \leftarrow \BIy$ and $\BIz_k \leftarrow T_{\nu_k}^{\widehat{\mu}_{t-1}}(\BIx)$ in (\ref{eqn: decrement-proof-identity}), we obtain
	\begin{align}
		\label{eqn: decrement-proof-decompose-2}
		\begin{split}
			\sum_{k = 1}^{K}w_k \big\|\BIy - T_{\nu_k}^{\widehat{\mu}_{t-1}}(\BIx) \big\|^2 
			&= \big\|\BIy - \bar{T}^{\widehat{\mu}_{t-1}}(\BIx)\big\|^2 
			+ \sum_{k = 1}^{K}w_k \big\|\bar{T}^{\widehat{\mu}_{t-1}}(\BIx) - T_{\nu_k}^{\widehat{\mu}_{t-1}}(\BIx) \big\|^2 \qquad \forall \BIx,\BIy\in\R^d.
		\end{split}
	\end{align}
	Combining (\ref{eqn: decrement-proof-decompose-1}) and (\ref{eqn: decrement-proof-decompose-2}) leads to
	\begin{align}
	\label{eqn: decrement-proof-decompose-comb}
	\begin{split}
		&\left(\sum_{k = 1}^{K}w_k \big\|\BIy - T_{\nu_k}^{\widehat{\mu}_{t-1}}(\BIx) \big\|^2\right)
		- \left(\sum_{k = 1}^{K}w_k\big\|\BIx - T_{\nu_k}^{\widehat{\mu}_{t-1}}(\BIx) \big\|^2\right) \\
		& \qquad = \big\|\BIy- \bar{T}^{\widehat{\mu}_{t-1}}(\BIx)\big\|^2 - \big\|\BIx - \bar{T}^{\widehat{\mu}_{t-1}}(\BIx)\big\|^2\\
		& \qquad \le - \big\|\BIx - \bar{T}^{\widehat{\mu}_{t-1}}(\BIx)\big\|^2 
		+ 2\big\|\bar{T}_{t}(\BIx) - \bar{T}^{\widehat{\mu}_{t-1}}(\BIx)\big\|^2 
		+ 2\big\|\BIy - \bar{T}_{t}(\BIx)\big\|^2   
		\qquad \forall\BIx,\BIy\in\R^d,\; \PROB\text{-a.s.}
	\end{split}
	\end{align}

	In the remainder of this proof, all statements hold in the $\PROB$-almost sure sense, and we will omit ``$\PROB$-a.s.\@'' for ease of notation. 
	Let $\CX=\CY=\CZ:=\R^d$ denote different copies of $\R^d$.
	Subsequently, let $\widehat{\xi}_t\in\Pi(\widetilde{\mu}_t,\widehat{\mu}_t)$ be an optimal coupling of 
	$\widetilde{\mu}_t$ 
	and $\widehat{\mu}_t$,
	let $\widehat{\eta}_t\in\Pi(\widehat{\mu}_{t-1},\widetilde{\mu}_t,\widehat{\mu}_t)\subset\CP(\CX\times\CZ\times\CY)$ 
	be constructed by the gluing lemma (see, e.g., \citep[Lemma~7.6]{villani2003topics})
	with 
	$[I_d,\bar{T}_{t}]\sharp\widehat{\mu}_{t-1}\in\Pi(\widehat{\mu}_{t-1},\widetilde{\mu}_{t})\subset\CP(\CX\times\CZ)$
	and 
	$\widehat{\xi}_t\in\Pi(\widetilde{\mu}_{t},\widehat{\mu}_{t})\subset\CP(\CZ\times\CY)$,
	and let 
	$\widehat{\pi}_{t}\in\Pi(\widehat{\mu}_{t-1},\widehat{\mu}_{t})$ 
	denote the marginal of $\widehat{\eta}_t$ on $\CX\times\CY$.
	In the following, let us examine the integrals of 
	the terms 
	$\sum_{k = 1}^{K}w_k \big\|\BIy - T_{\nu_k}^{\widehat{\mu}_{t-1}}(\BIx) \big\|^2$,
	$\sum_{k = 1}^{K}w_k\big\|\BIx - T_{\nu_k}^{\widehat{\mu}_{t-1}}(\BIx) \big\|^2$,
	$\big\|\BIx - \bar{T}^{\widehat{\mu}_{t-1}}(\BIx)\big\|^2$,
	$\big\|\bar{T}_{t}(\BIx) - \bar{T}^{\widehat{\mu}_{t-1}}(\BIx)\big\|^2$,
	$\big\|\BIy - \bar{T}_{t}(\BIx)\big\|^2$
	with respect to $\widehat{\pi}_t$.

	Firstly, for $k=1,\ldots,K$, 
	let 
	$\widehat{\gamma}_{t,k}\in\Pi(\widehat{\mu}_{t-1},\widehat{\mu}_{t},\nu_{k})\subset\CP_2(\CX\times\CY\times\CZ)$ 
	be constructed by the gluing lemma with 
	$\widehat{\pi}_{t}\in\Pi(\widehat{\mu}_{t-1},\widehat{\mu}_{t})\subset\CP(\CX\times\CY)$
	and 
	$\big[I_d,T_{\nu_k}^{\widehat{\mu}_{t-1}}\big]\sharp\widehat{\mu}_{t-1}\in\Pi(\widehat{\mu}_{t-1},\nu_k)\subset\CP(\CX\times\CZ)$.
	Since the marginal of $\widehat{\gamma}_{t,k}$ on $\CY\times\CZ$
	is a suboptimal coupling of $\widehat{\mu}_t$ and $\nu_k$, it
	then holds that
	\begin{align*}
		\int_{\R^d\times\R^d}\big\|\BIy - T_{\nu_k}^{\widehat{\mu}_{t-1}}(\BIx) \big\|^2 \DIFFM{\widehat{\pi}_{t}}{\DIFF\BIx,\DIFF\BIy}
		&=\int_{\R^d\times\R^d\times\R^d}\|\BIy-\BIz\|^2\DIFFM{\widehat{\gamma}_{t,k}}{\DIFF\BIx,\DIFF\BIy,\DIFF\BIz} \\
		&\ge \CW_2(\widehat{\mu}_{t},\nu_k)^2 
		\hspace{40pt}\qquad \forall 1\le k\le K.
	\end{align*}
	Consequently, we get
	\begin{align}
		\label{eqn: decrement-proof-decompose-integral1}
		\sum_{k = 1}^{K}w_k \int_{\R^d\times\R^d}\big\|\BIy - T_{\nu_k}^{\widehat{\mu}_{t-1}}(\BIx) \big\|^2 \DIFFM{\widehat{\pi}_t}{\DIFF\BIx,\DIFF\BIy}\ge \sum_{k=1}^Kw_k \CW_2(\widehat{\mu}_{t},\nu_k)^2 = V(\widehat{\mu}_{t}).
	\end{align}
	Secondly, since $T^{\widehat{\mu}_{t-1}}_{\nu_k}$ is the OT map from $\widehat{\mu}_{t-1}$ to $\nu_k$ for $k=1,\ldots,K$, we have 
	\begin{align}
		\label{eqn: decrement-proof-decompose-integral2}
		\begin{split}
			\sum_{k = 1}^{K}w_k \int_{\R^d\times\R^d}\big\|\BIx - T_{\nu_k}^{\widehat{\mu}_{t-1}}(\BIx) \big\|^2 \DIFFM{\widehat{\pi}_t}{\DIFF\BIx,\DIFF\BIy}
			&=\sum_{k = 1}^{K}w_k \int_{\R^d}\big\|\BIx - T_{\nu_k}^{\widehat{\mu}_{t-1}}(\BIx) \big\|^2 \DIFFM{\widehat{\mu}_{t-1}}{\DIFF\BIx}\\
			&= \sum_{k=1}^Kw_k \CW_2(\widehat{\mu}_{t-1},\nu_k)^2 = V(\widehat{\mu}_{t-1}).
		\end{split}
	\end{align}
	Thirdly,
	for $k=1,\ldots,K$, 
	let 
	$\varphi^{\widehat{\mu}_{t-1}}_{\nu_k}$ denote the Brenier potential
	from $\widehat{\mu}_{t-1}$ to $\nu_k$,
	which is a proper, l.s.c.,
	and convex function.
	Since 
	$\bar{T}^{\widehat{\mu}_{t-1}}$
	is $\widehat{\mu}_{t-1}$-almost everywhere equal to the gradient of the proper, l.s.c.,
	and convex function $\sum_{k=1}^{K}w_k\varphi^{\widehat{\mu}_{t-1}}_{\nu_k}$,
	it follows from Brenier's theorem (Theorem~\ref{thm: Brenier}) 
	that $\bar{T}^{\widehat{\mu}_{t-1}}$
	is the OT map from $\widehat{\mu}_{t-1}$
	to $\bar{T}^{\widehat{\mu}_{t-1}}\sharp\widehat{\mu}_{t-1}=G(\widehat{\mu}_{t-1})$,
	resulting in
	\begin{align}
		\label{eqn: decrement-proof-decompose-integral3}
		\int_{\R^d\times\R^d}\big\|\BIx - \bar{T}^{\widehat{\mu}_{t-1}}(\BIx) \big\|^2 \DIFFM{\widehat{\pi}_{t}}{\DIFF\BIx,\DIFF\BIy}=\int_{\R^d}\big\|\BIx - \bar{T}^{\widehat{\mu}_{t-1}}(\BIx) \big\|^2 \DIFFM{\widehat{\mu}_{t-1}}{\DIFF\BIx}= \CW_2\big(\widehat{\mu}_{t-1}, G(\widehat{\mu}_{t-1})\big)^2.
	\end{align}
	Fourthly, the convexity of $\R^d\ni\BIz\mapsto\|\BIz\|^2\in\R$ together with Jensen's inequality gives
	\begin{align*}
		\big\|\bar{T}_{t}(\BIx) - \bar{T}^{\widehat{\mu}_{t-1}}(\BIx)\big\|^2&=\left\|\sum_{k=1}^Kw_k \big(\widehat{T}_{t,k}(\BIx)-T^{\widehat{\mu}_{t-1}}_{\nu_k}(\BIx)\big)\right\|^2\le \sum_{k=1}^Kw_k \big\|\widehat{T}_{t,k}(\BIx)-T^{\widehat{\mu}_{t-1}}_{\nu_k}(\BIx)\big\|^2 \quad \forall \BIx\in\R^d,
	\end{align*}
	which results in
	\begin{align}
		\label{eqn: decrement-proof-decompose-integral4}
		\begin{split}
			\int_{\R^d\times\R^d}\big\|\bar{T}_{t}(\BIx) - \bar{T}^{\widehat{\mu}_{t-1}}(\BIx)\big\|^2 \DIFFM{\widehat{\pi}_t}{\DIFF\BIx,\DIFF\BIy}
			&=\int_{\R^d}\big\|\bar{T}_{t}(\BIx) - \bar{T}^{\widehat{\mu}_{t-1}}(\BIx)\big\|^2 \DIFFM{\widehat{\mu}_{t-1}}{\DIFF\BIx}\\
			&\le \sum_{k=1}^Kw_k \int_{\R^d}\big\|\widehat{T}_{t,k}(\BIx)-T^{\widehat{\mu}_{t-1}}_{\nu_k}(\BIx)\big\|^2 \DIFFM{\widehat{\mu}_{t-1}}{\DIFF\BIx}\\
			&=\sum_{k=1}^Kw_k \big\|\widehat{T}_{t,k}-T^{\widehat{\mu}_{t-1}}_{\nu_k}\big\|^2_{\CL^2(\widehat{\mu}_{t-1})}.
		\end{split}
	\end{align}
	Lastly, recall that $\widehat{\xi}_t$
	is an optimal coupling of $\widetilde{\mu}_{t}$
	and $\widehat{\mu}_{t}$,
	which leads to 
	\begin{align}
		\label{eqn: decrement-proof-decompose-integral5}
		\begin{split}
			\int_{\R^d\times\R^d}\big\|\BIy - \bar{T}_{t}(\BIx)\big\|^2 \DIFFM{\widehat{\pi}_t}{\DIFF\BIx,\DIFF\BIy}
			&= \int_{\R^d\times\R^d\times\R^d}\|\BIy-\BIz\|^2 \DIFFM{\widehat{\eta}_{t}}{\DIFF\BIx,\DIFF\BIz,\DIFF\BIy}\\
			&=\int_{\R^d\times\R^d}\|\BIy-\BIz\|^2\DIFFM{\widehat{\xi}_{t}}{\DIFF\BIy,\DIFF\BIz}=\CW_2(\widetilde{\mu}_{t},\widehat{\mu}_{t})^2.
		\end{split}
	\end{align}
	
	Now, integrating both sides of (\ref{eqn: decrement-proof-decompose-comb}) with respect to $\widehat{\pi}_t$ and then combining it with (\ref{eqn: decrement-proof-decompose-integral1})--(\ref{eqn: decrement-proof-decompose-integral5}) completes the proof of (\ref{eqn: decrement-pathwise}).
	Finally, taking conditional expectations with respect to $\CF_{t-1}$ on both sides of (\ref{eqn: decrement-pathwise}) proves~(\ref{eqn: decrement-condexp}).
	The proof is now complete.
\end{proof}

\begin{remark}
	In \citep[Proposition~3.3]{alvarez2016fixed}, the decrement of the sequence $\big(V(\mu_t)\big)_{t\in\N_0}$ in the deterministic fixed-point iteration $\mu_{t}\leftarrow G(\mu_{t-1})$ $\forall t\in\N$ is controlled through the inequality:
	\begin{align}
		\label{eqn: decrement-deterministic}
		V(\mu_{t})-V(\mu_{t-1})\le -\CW_2(\mu_{t-1},G(\mu_{t-1}))^2 \qquad \forall t\in\N.
	\end{align}
	Compared to (\ref{eqn: decrement-deterministic}), 
	the stochastic decrement in (\ref{eqn: decrement-pathwise}) has two additional terms on the right-hand side:
	\begin{itemize}
		\item the term $2\sum_{k = 1}^{K} w_k \big\|\widehat{T}_{t,k} - T^{\widehat{\mu}_{t-1}}_{\nu_k} \big\|^2_{\CL^2(\widehat{\mu}_{t-1})}$
		comes from the inexactness when approximating the true OT map 
		$T_{\nu_k}^{\widehat{\mu}_{t-1}}$
		by the OT map estimator $\widehat{T}_{t,k}$,
		i.e., from the approximation in Line~\ref{alglin: abstract-estimator} of Algorithm~\ref{algo: abstract};
		
		\item the term $2\CW_2\big(\big[{\textstyle\sum_{k=1}^Kw_k}\widehat{T}_{t,k}\big]\sharp\widehat{\mu}_{t-1}, \widehat{\mu}_{t}\big)^2$
		comes from the inexactness when approximating the pushforward
		$\big[{\textstyle\sum_{k=1}^Kw_k}\widehat{T}_{t,k}\big]\sharp\widehat{\mu}_{t-1}$ by $\widehat{\mu}_{t}$,
		i.e., from the approximation in Line~\ref{alglin: abstract-update} of Algorithm~\ref{algo: abstract}.
	\end{itemize}
	Note that Theorem~\ref{thm: main-convergence} assumes that the conditional expectations of these two additional error terms are controlled to 
	decay geometrically with respect to~$t$.
\end{remark}

\subsection{Proof of Theorem~\ref{thm: main-convergence}}\label{ssec: proof-main-convergence}

We are now ready to prove Theorem~\ref{thm: main-convergence}.
Throughout this proof, we denote
$\widetilde{\mu}_{t}:=\big[{\textstyle\sum_{k=1}^Kw_k}\widehat{T}_{t,k}\big]\sharp\widehat{\mu}_{t-1}$ for all $t\in\N$ for notational simplicity.
The proof is divided into the following 8 steps.
\begin{itemize}[beginpenalty=10000]
	\item Step~1: showing that for $k=1,\ldots,K$, it holds $\PROB$-almost surely that 
	$\big\|\widehat{T}_{t,k} - T^{\widehat{\mu}_{t-1}}_{\nu_k} \big\|^2_{\CL^2(\widehat{\mu}_{t-1})}\le \beta^{\frac{t}{2}}$ 
	is true for all but finitely many $t\in\N$.
	
	\item Step~2: showing that it holds $\PROB$-almost surely that 
	$\CW_2(\widetilde{\mu}_{t},\widehat{\mu}_{t})^2\le \beta^{\frac{t}{2}}$ 
	is true for all but finitely many $t\in\N$;
	in particular, 
	$\lim_{t\to\infty}\CW_2(\widetilde{\mu}_{t},\widehat{\mu}_{t})=\nobreak0$ 
	holds $\PROB$-almost surely.
	
	\item Step~3: showing that for $k=1,\ldots,K$, 
	$\lim_{t\to\infty}\CW_2(\widehat{T}_{t,k}\sharp\widehat{\mu}_{t-1},\nu_k)=\nobreak0$
	holds $\PROB$-almost surely.
	
	\item Step~4: proving the $\PROB$-almost sure $\CW_2$-precompactness of $(\widehat{\mu}_{t})_{t\in\N_0}$.
	
	\item Step~5: proving statements~\ref{thms: main-convergence-tight-as} and \ref{thms: main-convergence-as}.
	
	\item Step~6: proving statement~\ref{thms: main-convergence-violation}.
	
	\item Step~7: proving (\ref{eqn: main-convergence-linearrate}).
	
	\item Step~8: proving statements~\ref{thms: main-convergence-PL} and \ref{thms: main-convergence-PL-var}.
	
\end{itemize}

\underline{Step~1.}
Let us fix an arbitrary $k\in\{1,\ldots,K\}$ in this step.
Applying the law of total expectation and Markov's inequality to
(\ref{eqn: main-convergence-cond1}) gives
\begin{align*}
	\PROB\Big[\big\|\widehat{T}_{t,k} - T^{\widehat{\mu}_{t-1}}_{\nu_k} \big\|^2_{\CL^2(\widehat{\mu}_{t-1})}>\beta^{\frac{t}{2}}\Big]
	\le \beta^{-\frac{t}{2}}\EXP\Big[\big\|\widehat{T}_{t,k} - T^{\widehat{\mu}_{t-1}}_{\nu_k} \big\|^2_{\CL^2(\widehat{\mu}_{t-1})}\Big]\le \beta^{\frac{t}{2}} \qquad \forall t\in\N.
\end{align*}
Since 
$\sum_{t=1}^{\infty}\PROB\Big[\big\|\widehat{T}_{t,k} - T^{\widehat{\mu}_{t-1}}_{\nu_k} \big\|^2_{\CL^2(\widehat{\mu}_{t-1})}>\beta^{\frac{t}{2}}\Big]\le \sum_{t=1}^{\infty}\beta^{\frac{t}{2}}<\infty$,
the Borel--Cantelli lemma implies that, $\PROB$-almost surely,
$\big\|\widehat{T}_{t,k} - T^{\widehat{\mu}_{t-1}}_{\nu_k} \big\|^2_{\CL^2(\widehat{\mu}_{t-1})}\le \beta^{\frac{t}{2}}$ 
holds for all but finitely many $t\in\N$.
This completes Step~1.

\underline{Step~2.}
Similar to Step~1, applying the law of total expectation and Markov's inequality to (\ref{eqn: main-convergence-cond2}) yields
\begin{align*}
	\PROB\big[\CW_2(\widetilde{\mu}_{t},\widehat{\mu}_{t})^2>\beta^{\frac{t}{2}}\big]
	\le \beta^{-\frac{t}{2}}\EXP\big[\CW_2(\widetilde{\mu}_{t},\widehat{\mu}_{t})^2\big]\le \beta^{\frac{t}{2}} \qquad \forall t\in\N.
\end{align*}
Since 
$\sum_{t=1}^{\infty}\PROB\big[\CW_2(\widetilde{\mu}_{t},\widehat{\mu}_{t})^2>\beta^{\frac{t}{2}}\big]\le \sum_{t=1}^{\infty}\beta^{\frac{t}{2}}<\infty$,
the Borel--Cantelli lemma implies that, $\PROB$-almost surely,
$\CW_2(\widetilde{\mu}_{t},\widehat{\mu}_{t})^2\le \beta^{\frac{t}{2}}$ 
holds for all but finitely many $t\in\N$.
In particular, 
$\lim_{t\to\infty}\CW_2(\widetilde{\mu}_{t},\widehat{\mu}_{t})=\nobreak0$ 
holds $\PROB$-almost surely.
This completes Step~2.

\underline{Step~3.}
Let us fix an arbitrary $k\in\{1,\ldots,K\}$ here.
Observe that $\big[\widehat{T}_{t,k},T^{\widehat{\mu}_{t-1}}_{\nu_k}\big]\sharp\widehat{\mu}_{t-1}\in\Pi\big(\widehat{T}_{t,k}\sharp\widehat{\mu}_{t-1},\nu_k\big)$.
It thus holds 
$\PROB$-almost surely
that 
\begin{align*}
	\CW_2\big(\widehat{T}_{t,k}\sharp\widehat{\mu}_{t-1},\nu_k\big)^2&\le \int_{\R^d}\big\|\widehat{T}_{t,k}(\BIx)-T^{\widehat{\mu}_{t-1}}_{\nu_k}(\BIx)\big\|^2\DIFFM{\widehat{\mu}_{t-1}}{\DIFF\BIx} = \big\|\widehat{T}_{t,k}-T^{\widehat{\mu}_{t-1}}_{\nu_k}\big\|_{\CL^2(\widehat{\mu}_{t-1})}^2 \qquad \forall t\in\N.
\end{align*}
Subsequently, Step~1 implies that $\lim_{t\to\infty}\CW_2\big(\widehat{T}_{t,k}\sharp\widehat{\mu}_{t-1},\nu_k\big)=\nobreak0$ holds 
$\PROB$-almost surely.
This completes Step~3.

\underline{Step~4.}
In this step, let $\widetilde{\eta}_{t}:=\big[\widehat{T}_{t,1},\ldots,\widehat{T}_{t,K}\big]\sharp\widehat{\mu}_{t-1}\in\Pi\big(\widehat{T}_{t,1}\sharp\widehat{\mu}_{t-1},\ldots,\widehat{T}_{t,K}\sharp\widehat{\mu}_{t-1}\big)\subset\CP\big((\R^d)^K\big)$ for all $t\in\N$.
Moreover,
let $A$ denote the function
$(\R^d)^K\ni(\BIx_1,\ldots,\BIx_K)\mapsto\sum_{k=1}^{K}w_k\BIx_k\in\R^d$.
Notice that $\widetilde{\mu}_t=A\sharp\widetilde{\eta}_t$ for all $t\in\N$. 
By Step~3 
and the equivalence between (i) and (iii) in 
\citep[Theorem~7.12]{villani2003topics},
it holds $\PROB$-almost surely that
$\big(\widehat{T}_{t,k}\sharp\widehat{\mu}_{t-1}\big)_{t\in\N}$
converges weakly to~$\nu_k$ for $k=1,\ldots,K$,
and that 
\begin{align}
	\lim_{t\to\infty} \int_{\R^d}\|\BIy\|^2 \DIFFM{\widehat{T}_{t,k}\sharp\widehat{\mu}_{t-1}}{\DIFF\BIy} = \int_{\R^d}\|\BIy\|^2 \DIFFM{\nu_k}{\DIFF\BIy} \qquad \forall 1\le k\le K,\; 
	\PROB\text{-a.s.}
	\label{eqn: main-convergence-proof-quadconverge}
\end{align}
Subsequently, it follows from Prokhorov's theorem that
$\big(\widehat{T}_{t,k}\sharp\widehat{\mu}_{t-1}\big)_{t\in\N}$ is $\PROB$-almost surely a tight sequence of probability measures for $k=1,\ldots,K$. 
It hence holds $\PROB$-almost surely that 
each marginal of the sequence 
$(\widetilde{\eta}_t)_{t\in\N}$ 
(on each copy of $\R^d$) 
belongs to a tight sequence of probability measures on $\R^d$, 
and it then follows from 
a multi-marginal generalization of 
\citep[Lemma~4.4]{villani2009optimal} that 
$(\widetilde{\eta}_t)_{t\in\N}$ 
is $\PROB$-almost surely a tight sequence of probability measures on $(\R^d)^K$.
Consequently, 
Prokhorov's theorem implies that, $\PROB$-almost surely, 
every subsequence of $(\widetilde{\eta}_t)_{t\in\N}$ 
admits a further subsequence which is weakly convergent. 
Let $(\widetilde{\eta}_{t_i})_{i\in\N}$ be a weakly convergent subsequence of $(\widetilde{\eta}_t)_{t\in\N}$ with weak limit $\widetilde{\eta}_{t_{\infty}}\in\CP((\R^d)^K)$.
It subsequently follows from Step~3 that 
$\widetilde{\eta}_{t_{\infty}}\in\Pi(\nu_1,\ldots,\nu_K)$ holds $\PROB$-almost surely, and hence
(\ref{eqn: main-convergence-proof-quadconverge}) implies
\begin{align}
	\begin{split}
	\lim_{i\to\infty}\int_{(\R^d)^K} \|\BIx\|^2 \DIFFM{\widetilde{\eta}_{t_i}}{\DIFF\BIx}
	&=\lim_{i\to\infty}\int_{(\R^d)^K}{\textstyle\sum_{k=1}^K}\|\BIx_k\|^2\DIFFM{\widetilde{\eta}_{t_i}}{\DIFF \BIx_1,\ldots,\DIFF \BIx_K}\\
	&=\sum_{k=1}^K\lim_{i\to\infty} \int_{\R^d} \|\BIx_k\|^2\DIFFM{\widehat{T}_{t_i,k}\sharp \widehat{\mu}_{t_i-1}}{\DIFF\BIx_k}\\
	&=\sum_{k=1}^K\int_{\R^d}\|\BIx_k\|^2\DIFFM{\nu_k}{\DIFF\BIx_k}=\int_{(\R^d)^K}\|\BIx\|^2\DIFFM{\widetilde{\eta}_{t_{\infty}}}{\DIFF\BIx} \qquad \PROB\text{-a.s.}
	\end{split}
	\label{eqn: main-convergence-proof-W2-convergence}
\end{align}
On the other hand,
it follows from the convexity of 
$\R^d\ni \BIz\mapsto \|\BIz\|^2\in\R$ 
and Jensen's inequality that
\begin{align}
	\big\|A(\BIx)\big\|^2\le \sum_{k=1}^Kw_k \|\BIx_k\|^2\le \|\BIx\|^2 \qquad \forall \BIx=(\BIx_1,\ldots,\BIx_K)\in(\R^d)^K.
	\label{eqn: main-convergence-proof-quadratic-growth}
\end{align}
Combining (\ref{eqn: main-convergence-proof-quadratic-growth}), (\ref{eqn: main-convergence-proof-W2-convergence}), 
and the equivalence between (iii) and (iv) in \citep[Theorem~7.12]{villani2003topics} yields
\begin{align}
	\begin{split}
	\lim_{i\to\infty}\int_{\R^d}\|\BIy\|^2\DIFFM{\widetilde{\mu}_{t_i}}{\DIFF\BIy}
	&= \lim_{i\to\infty} \int_{(\R^d)^K} \big\|A(\BIx)\big\|^2\DIFFM{\widetilde{\eta}_{t_i}}{\DIFF\BIx}\\
	&= \int_{(\R^d)^K}\big\|A(\BIx)\big\|^2\DIFFM{\widetilde{\eta}_{t_{\infty}}}{\DIFF\BIx}
	= \int_{\R^d}\|\BIy\|^2 \DIFFM{A\sharp \widetilde{\eta}_{t_{\infty}}}{\DIFF\BIy}\qquad \PROB\text{-a.s.}
	\end{split}
	\label{eqn: main-convergence-proof-W2-convergence-tildemu}
\end{align}
Moreover, since $(\widetilde{\eta}_{t_i})_{i\in\N}$ converges weakly to $\widetilde{\eta}_{t_{\infty}}$ $\PROB$-almost surely and $A$ is continuous, it holds $\PROB$-almost surely that 
$\lim_{i\to\infty}\int_{\R^d}\psi\DIFFX{\widetilde{\mu}_{t_{i}}}
=\lim_{i\to\infty}\int_{(\R^d)^K}\psi\circ A \DIFFX{\widetilde{\eta}_{t_i}}
=\int_{(\R^d)^K}\psi\circ A\DIFFX{\widetilde{\eta}_{t_{\infty}}}
=\int_{\R^d}\psi \DIFFX{A\sharp \widetilde{\eta}_{t_{\infty}}}$ 
holds for any continuous and bounded function $\psi:\R^d\to\R$, 
which shows that
$(\widetilde{\mu}_{t_i})_{i\in\N}$ converges weakly to $A\sharp\widetilde{\eta}_{t_{\infty}}$ $\PROB$-almost surely.
Now, (\ref{eqn: main-convergence-proof-W2-convergence-tildemu}) 
and the equivalence between (i) and (iii) in \citep[Theorem~7.12]{villani2003topics} 
show that \sloppy{$\lim_{i\to\infty}\CW_2(\widetilde{\mu}_{t_i},A\sharp\widetilde{\eta}_{t_{\infty}})=\nobreak0$} holds $\PROB$-almost surely.
Furthermore, since Step~2 has established that 
$\lim_{t\to\infty}\CW_2(\widetilde{\mu}_t,\widehat{\mu}_t)=\nobreak0$
holds $\PROB$-almost surely,
the above analysis implies that, $\PROB$-almost surely, every subsequence of $(\widehat{\mu}_t)_{t\in\N_0}$ admits a further subsequence which converges with respect to the $\CW_2$-metric, and thus $(\widehat{\mu}_t)_{t\in\N_0}$ is precompact with respect to the $\CW_2$-metric $\PROB$-almost surely. 
Step~4 is now complete.

\underline{Step~5.}
In this step, 
for every $\omega\in\Omega$, 
let us use the notations 
$\widehat{\mu}_{t}^{(\omega)}$, 
$\widetilde{\mu}_t^{(\omega)}$,
$\widehat{T}_{t,k}^{(\omega)}$
to explicitly express the dependencies of the random variables 
$\widehat{\mu}_{t}$, 
$\widetilde{\mu}_t$, 
$\widehat{T}_{t,k}$
on~$\omega$ 
for $k=1,\ldots,K$ and $t\in\N_0$. 
Lemma~\ref{lem: decrement} and Steps~1--4 above 
have shown the existence of 
an $\CF$-measurable set 
$\widetilde{\Omega}\subseteq\Omega$ 
with $\PROB[\widetilde{\Omega}]=1$, which satisfies:
\begin{align}
	\label{eqn: main-convergence-proof-subsetOmega}
	\forall \omega\in\widetilde{\Omega},\;\exists \overline{t}^{(\omega)}\in\N_0,\; \begin{cases}
		\big(\widehat{\mu}_{t}^{(\omega)}\big)_{t\in\N_0} \text{ is precompact with respect to the }\CW_2\text{-metric},\!\!\!\!\!\!\!\!\!\!\!\!\!\!\!\!\!\!\! \\
		\CW_2\big(\widetilde{\mu}_{t}^{(\omega)},\widehat{\mu}_{t}^{(\omega)}\big)^2 \le \beta^{\frac{t}{2}} & \phantom{1\le k\le K,\; }\;\;\,\forall t\ge \overline{t}^{(\omega)}, \\
		\Big\|\widehat{T}_{t,k}^{(\omega)}-T^{\widehat{\mu}_{t-1}^{(\omega)}}_{\nu_k}\Big\|^2_{\CL^2\big(\widehat{\mu}_{t-1}^{(\omega)}\big)} \le \beta^{\frac{t}{2}} & \forall 1\le k\le K,\; \forall t\ge \overline{t}^{(\omega)},\\
		\text{\eqref{eqn: decrement-pathwise} holds with respect to }\omega & \phantom{1\le k\le K,\; }\;\;\,\forall t\ge \overline{t}^{(\omega)}.
	\end{cases}
\end{align}
Let us fix an arbitrary $\omega\in\widetilde{\Omega}$ 
and let $(t_i)_{i\in\N_0}$ be 
an arbitrary subsequence such that 
$\lim_{i\to\infty}\CW_2\big(\widehat{\mu}_{t_i}^{(\omega)},\widehat{\mu}_{t_\infty}^{(\omega)}\big)=\nobreak0$ 
for $\widehat{\mu}_{t_\infty}^{(\omega)}\in\CP_2(\R^d)$.
The continuity of $V$ on $\CP_2(\R^d)$ then implies that $\lim_{i\to\infty}V\big(\widehat{\mu}_{t_i}^{(\omega)}\big)=V\big(\widehat{\mu}_{t_\infty}^{(\omega)}\big)$.
Removing finitely many initial terms from $(t_i)_{i\in\N_0}$ if necessary,
we assume without loss of generality that $t_0\ge \overline{t}^{(\omega)}$. 
For each $i\in\N_0$, 
summing (\ref{eqn: decrement-pathwise}) over 
$t\leftarrow t_i+1,t_i+2,\ldots,t_{i+1}$
and using the properties in (\ref{eqn: main-convergence-proof-subsetOmega}) lead to
\begin{align*}
	V\big(\widehat{\mu}_{t_{i+1}}^{(\omega)}\big)-V\big(\widehat{\mu}_{t_i}^{(\omega)}\big) 
	&=\sum_{s=t_i+1}^{t_{i+1}} V\big(\widehat{\mu}_{s}^{(\omega)}\big) - V\big(\widehat{\mu}_{s-1}^{(\omega)}\big) \\
	&\leq {-\left(\sum_{s=t_i+1}^{t_{i+1}}\CW_2\Big(\widehat{\mu}_{s-1}^{(\omega)}, G\big(\widehat{\mu}_{s-1}^{(\omega)}\big) \Big)^2\right)}
	+ 2\left(\sum_{s=t_i+1}^{t_{i+1}} \sum_{k = 1}^{K}w_k \Big\|\widehat{T}_{s,k}^{(\omega)} - T^{\widehat{\mu}_{s-1}^{(\omega)}}_{\nu_k} \Big\|^2_{\CL^2\big(\widehat{\mu}_{s-1}^{(\omega)}\big)} \right) \\
	&\qquad + 2\left( \sum_{s=t_i+1}^{t_{i+1}} \CW_2\big(\widetilde{\mu}_{s}^{(\omega)}, \widehat{\mu}_{s}^{(\omega)}\big)^2 \right)\allowdisplaybreaks\\
	&\le {-\left(\sum_{s=t_i+1}^{t_{i+1}}\CW_2\Big(\widehat{\mu}_{s-1}^{(\omega)}, G\big(\widehat{\mu}_{s-1}^{(\omega)}\big) \Big)^2\right)} + \left( \sum_{s=t_i+1}^{t_{i+1}} 4\beta^{\frac{s}{2}}\right)\\
	&\le - \CW_2\Big(\widehat{\mu}_{t_i}^{(\omega)}, G\big(\widehat{\mu}_{t_i}^{(\omega)}\big) \Big)^2  
	+ \frac{4\beta^{\frac{t_i+1}{2}}}{1-\beta^{\frac{1}{2}}}
	\hspace{120pt} \qquad\qquad\forall i\in\N_0.
\end{align*}
Rearranging the terms above leads to
\begin{align*}
	\CW_2\Big(\widehat{\mu}_{t_i}^{(\omega)}, G\big(\widehat{\mu}_{t_i}^{(\omega)}\big) \Big)^2 \le \Big|V\big(\widehat{\mu}_{t_{i+1}}^{(\omega)}\big)-V\big(\widehat{\mu}_{t_i}^{(\omega)}\big)\Big| + \frac{4\beta^{\frac{t_i+1}{2}}}{1-\beta^{\frac{1}{2}}} \qquad\forall i\in\N_0,
\end{align*}
and we subsequently get
\begin{align*}
	\limsup_{i\to\infty}\CW_2\Big(\widehat{\mu}_{t_i}^{(\omega)}, G\big(\widehat{\mu}_{t_i}^{(\omega)}\big) \Big)^2 \le \limsup_{i\to\infty}\Big|V\big(\widehat{\mu}_{t_{i+1}}^{(\omega)}\big)-V\big(\widehat{\mu}_{t_i}^{(\omega)}\big)\Big| + \limsup_{i\to\infty}\frac{4\beta^{\frac{t_i+1}{2}}}{1-\beta^{\frac{1}{2}}}=0.
\end{align*}
This shows that $\lim_{i\to\infty}\CW_2\Big(\widehat{\mu}_{t_i}^{(\omega)},G\big(\widehat{\mu}_{t_i}^{(\omega)}\big)\Big)=\nobreak0$,
and we hence get
$\lim_{i\to\infty}\CW_2\Big(G\big(\widehat{\mu}_{t_i}^{(\omega)}\big),\widehat{\mu}_{t_{\infty}}^{(\omega)}\Big)=\nobreak0$.

Next, by Assumption~\ref{asp: abstract-regularity} and by the symmetry among $\nu_1,\ldots,\nu_K$,
let us assume without loss of generality that
the density function $f_{\nu_1}$ of $\nu_1$ is essentially bounded.
Subsequently, for any $\mu\in\CP_{2,\AC}(\R^d)$, the analysis in \citep[Remark~3.2]{alvarez2016fixed} demonstrates that the density function $f_{G(\mu)}$ of $G(\mu)\in\CP_{2,\AC}(\R^d)$ satisfies
\begin{align*}
	\esssup\big\{f_{G(\mu)}(\BIx)\big\}\le w_1^{-d} \esssup\big\{f_{\nu_1}(\BIx)\big\}<\infty.
\end{align*}
Consequently, 
by the Portmanteau theorem,
it holds for every open set $E\subseteq\R^d$ that
\begin{align*}
	\widehat{\mu}_{t_\infty}^{(\omega)}(E)\le \liminf_{i\to\infty} G\big(\widehat{\mu}_{t_i}^{(\omega)}\big)(E)\le w_1^{-d} \esssup\big\{f_{\nu_1}(\BIx)\big\}\lebesgue(E),
\end{align*}
where $\lebesgue$ denotes the Lebesgue measure on $\R^d$. 
It thus follows that $\widehat{\mu}_{t_\infty}^{(\omega)}\in\CP_{2,\AC}(\R^d)$.
Now, the continuity of the operator 
$\CP_{2,\AC}(\R^d)\ni\mu \mapsto G(\mu)\in\CP_{2,\AC}(\R^d)$ 
in \citep[Theorem~3.1]{alvarez2016fixed} implies that 
$\CW_2\Big(\widehat{\mu}_{t_\infty}^{(\omega)}, G\big(\widehat{\mu}_{t_\infty}^{(\omega)}\big) \Big)
\le \liminf_{i\to\infty} 
\Big[\CW_2\Big(G\big(\widehat{\mu}_{t_i}^{(\omega)}\big),\widehat{\mu}_{t_{\infty}}^{(\omega)}\Big)
+ \CW_2\Big(G\big(\widehat{\mu}_{t_i}^{(\omega)}\big), G\big(\widehat{\mu}_{t_\infty}^{(\omega)}\big) \Big)\Big]=\nobreak0$, 
which shows that $\widehat{\mu}_{t_{\infty}}^{(\omega)}$ is a fixed-point of~$G$. 
Since $\PROB[\widetilde{\Omega}]=1$, 
it therefore holds $\PROB$-almost surely that every $\CW_2$-accumulation point of $(\widehat{\mu}_t)_{t\in\N_0}$ is a fixed-point of~$G$. 
We have thus completed the proof of statement~\ref{thms: main-convergence-tight-as}.
In the case where $G$ has 
a unique fixed-point $\bar{\mu}\in\CP_{2,\AC}(\R^d)$, 
then statement~\ref{thms: main-convergence-tight-as} implies that, $\PROB$-almost surely, every $\CW_2$-accumulation point of $(\widehat{\mu}_t)_{t\in\N_0}$ is equal to $\bar{\mu}$.
Therefore, $(\widehat{\mu}_t)_{t\in\N_0}$ converges $\PROB$-almost surely in $\CW_2$ to $\bar{\mu}$.
The proof of statement~\ref{thms: main-convergence-as} is now complete. 

\underline{Step~6.}
For any $t\in\N_0$,
summing over both sides of 
(\ref{eqn: decrement-condexp}) in Lemma~\ref{lem: decrement}
with respect to $t\leftarrow 1,2,\ldots,t+1$, 
applying (\ref{eqn: main-convergence-cond1}), (\ref{eqn: main-convergence-cond2}),
and then taking the expectations on both sides yield
\begin{align}
	\label{eqn: main-convergence-proof-telescope}
	\begin{split}
		\EXP\big[V(\widehat{\mu}_{t+1})\big] - \EXP\big[V(\widehat{\mu}_0)\big]
		&= \sum_{s=1}^{t+1} \Big(\EXP\big[V(\widehat{\mu}_s)\big] - \EXP\big[V(\widehat{\mu}_{s-1})\big] \Big)\\
		&\le \sum_{s=1}^{t+1} \Big({-\EXP}\big[\CW_2\big(\widehat{\mu}_{s-1},G(\widehat{\mu}_{s-1})\big)^2\big]+4\beta^{s} \Big)\\
		&\le \Bigg({-\sum_{s=1}^{t+1}}\EXP\big[\CW_2\big(\widehat{\mu}_{s-1},G(\widehat{\mu}_{s-1})\big)^2\big]\Bigg) + \frac{4\beta}{1-\beta} \qquad \forall t\in\N_0.
	\end{split}
\end{align}
Since $V(\widehat{\mu}_0)$ is $\CF_0$-measurable 
and 
$V(\widehat{\mu}_{t+1})\ge V(\bar{\mu})$ holds $\PROB$-almost surely for all $t\in\N_0$,
(\ref{eqn: main-convergence-proof-telescope}) implies that
\begin{align*}
	\sum_{s=1}^{t+1}\EXP\big[\CW_2\big(\widehat{\mu}_{s-1},G(\widehat{\mu}_{s-1})\big)^2\big]\le V(\widehat{\mu}_0) - V(\bar{\mu}) + \frac{4\beta}{1-\beta} \qquad \forall t\in\N_0.
\end{align*}
In particular,
we get
\begin{align*}
	(t+1)\EXP\bigg[\min_{0\le s\le t}\big\{\CW_2\big(\widehat{\mu}_{s},G(\widehat{\mu}_{s})\big)^2\big\}\bigg]
	&\le \sum_{s=1}^{t+1}\EXP\big[\CW_2\big(\widehat{\mu}_{s-1},G(\widehat{\mu}_{s-1})\big)^2\big]\le V(\widehat{\mu}_0) - V(\bar{\mu}) + \frac{4\beta}{1-\beta} \quad \forall t\in\N_0,
\end{align*}
which proves statement~\ref{thms: main-convergence-violation}.

\underline{Step~7.}
Let us now assume in addition that
\eqref{eqn: main-convergence-condPL}
holds with respect to $\PLINEQCONST\in(0,1]$.
Observe that (\ref{eqn: main-convergence-linearrate}) holds when $t=\nobreak0$.
Thus, we only consider $t\in\N$ in the following.
Applying (\ref{eqn: main-convergence-cond1}), (\ref{eqn: main-convergence-cond2}) to (\ref{eqn: decrement-condexp}) and then taking expectations on both sides leads to
\begin{align*}
	\EXP\big[V(\widehat{\mu}_{t})\big] - \EXP\big[V(\widehat{\mu}_{t-1})\big] 
	&\le -\EXP\big[\CW_2\big(\widehat{\mu}_{t-1},G(\widehat{\mu}_{t-1})\big)^2\big] + 4\beta^{t} \;\qquad \forall t\in\N.
\end{align*}
Substituting \eqref{eqn: main-convergence-condPL} into the inequality above yields
\begin{align*}
	\EXP\big[V(\widehat{\mu}_{t})\big] - \EXP\big[V(\widehat{\mu}_{t-1})\big] 
	&\le \PLINEQCONST\Big(V(\bar{\mu})- \EXP\big[V(\widehat{\mu}_{t-1})\big]\Big) + 4\beta^{t} \qquad \forall t\in\N.
\end{align*}
Rearranging the terms above, we get
\begin{align*}
	\EXP\big[V(\widehat{\mu}_{t})\big] - V(\bar{\mu})
	&\le (1-\PLINEQCONST)\Big(\EXP\big[V(\widehat{\mu}_{t-1})\big] - V(\bar{\mu})\Big) + 4\beta^{t} \;\qquad \forall t\in\N.
\end{align*}
Subsequently, we iteratively substitute this inequality into itself $t-1$~times to obtain
\begin{align}
	\label{eqn: main-convergence-proof-PLiter}
	\begin{split}
		\EXP\big[V(\widehat{\mu}_{t})\big] - V(\bar{\mu})
		&\le (1-\PLINEQCONST)^2\Big(\EXP\big[V(\widehat{\mu}_{t-2})\big] - V(\bar{\mu})\Big) + 4(1-\PLINEQCONST)\beta^{t-1} + 4\beta^{t}\\
		&\le \cdots\\
		&\le (1-\PLINEQCONST)^t\big(V(\widehat{\mu}_{0}) - V(\bar{\mu})\big)
		+ 4\beta\sum_{s=0}^{t-1}(1-\PLINEQCONST)^{s}\beta^{t-1-s} \qquad \forall t\in\N.
	\end{split}
\end{align}
In the case where $\beta<1-\PLINEQCONST$, we have
\begin{align}
	\label{eqn: main-convergence-proof-additionalerror-lt}
	\begin{split}
		\sum_{s=0}^{t-1}(1-\PLINEQCONST)^{s}\beta^{t-1-s}
		&= \sum_{s=0}^{t-1}\beta^{s}(1-\PLINEQCONST)^{t-1-s}\\
		&= (1-\PLINEQCONST)^{t-1}\sum_{s=0}^{t-1}\bigg(\frac{\beta}{1-\PLINEQCONST}\bigg)^{s} < \frac{(1-\PLINEQCONST)^{t}}{(1-\PLINEQCONST)-\beta} \qquad \forall t\in\N.
	\end{split}
\end{align}
In the case where $\beta>1-\PLINEQCONST$, we have
\begin{align}
	\label{eqn: main-convergence-proof-additionalerror-gt}
	\begin{split}
		\sum_{s=0}^{t-1}(1-\PLINEQCONST)^{s}\beta^{t-1-s}
		&=\beta^{t-1}\sum_{s=0}^{t-1}\bigg(\frac{1-\PLINEQCONST}{\beta}\bigg)^s
		\le \frac{\beta^t}{\beta - (1-\PLINEQCONST)} \qquad \forall t\in\N.
	\end{split}
\end{align}
Lastly, in the case where $\beta=1-\PLINEQCONST$, we have
\begin{align}
	\label{eqn: main-convergence-proof-additionalerror-eq}
	\begin{split}
		\sum_{s=0}^{t-1}(1-\PLINEQCONST)^{s}\beta^{t-1-s}
		= t\beta^{t-1} \qquad \forall t\in\N.
	\end{split}
\end{align}
Combining (\ref{eqn: main-convergence-proof-PLiter})--(\ref{eqn: main-convergence-proof-additionalerror-eq}) proves (\ref{eqn: main-convergence-linearrate}) and completes Step~7.

\underline{Step~8.}
To prove statement~\ref{thms: main-convergence-PL},
observe that, for any $\delta\in(0,1)$, 
$\frac{4\beta}{|1-\PLINEQCONST-\beta|}\big((1-\PLINEQCONST)\vee \beta\big)^{t}=O\Big(\big((1-\nobreak\PLINEQCONST)\vee\nobreak \beta\big)^{(1-\delta)t}\Big)$
if $\beta\ne1-\nobreak\PLINEQCONST$
and 
$4t\beta^{t}=O\Big(\big((1-\PLINEQCONST)\vee \beta\big)^{(1-\delta)t}\Big)$
if $\beta=1-\nobreak\PLINEQCONST$.
Thus,
it remains to show that $\lim_{t\to\infty}\CW_2(\widehat{\mu}_t,\bar{\mu})=\nobreak0$ $\PROB$-almost surely.
To that end,
let us combine (\ref{eqn: main-convergence-linearrate}) and Markov's inequality to derive the following inequality:
\begin{align}
	\PROB\Big[V(\widehat{\mu}_t)-V(\bar{\mu}) > \big((1-\PLINEQCONST)\vee \beta\big)^{\frac{t}{2}}\Big] &\le \big((1-\PLINEQCONST)\vee \beta\big)^{-\frac{t}{2}}\EXP\big[V(\widehat{\mu}_t)-V(\bar{\mu})\big]\nonumber\\
	&\le \big(V(\widehat{\mu}_0)-V(\bar{\mu})\big)\big((1-\PLINEQCONST)\vee \beta\big)^{\frac{t}{2}} 
	\label{eqn: main-convergence-proof-PLtoas}\\
	&\qquad + \begin{cases}
		\frac{4\beta}{|1-\PLINEQCONST-\beta|}\big((1-\PLINEQCONST)\vee \beta\big)^{\frac{t}{2}} & \text{if }\beta\ne 1-\PLINEQCONST\\
		4t\beta^{\frac{t}{2}} & \text{if }\beta = 1-\PLINEQCONST
	\end{cases} \qquad\!\! \forall t\in\N_0.\nonumber
\end{align}
Since both
$\sum_{t=0}^{\infty}\big((1-\PLINEQCONST)\vee \beta\big)^{\frac{t}{2}}$
and 
$\sum_{t=0}^{\infty}t\beta^{\frac{t}{2}}$
are convergent,
we get 
$\sum_{t=0}^{\infty}\PROB\Big[V(\widehat{\mu}_t)-V(\bar{\mu}) > \big((1-\nobreak\PLINEQCONST)\vee\nobreak \beta\big)^{\frac{t}{2}}\Big]<\infty$,
and thus the Borel--Cantelli lemma implies that,
$\PROB$-almost surely,
$V(\widehat{\mu}_t)-V(\bar{\mu})\le \big((1-\PLINEQCONST)\vee \beta\big)^{\frac{t}{2}}$ 
holds for all but finitely many $t\in\N_0$.
In particular,
$\lim_{t\to\infty}V(\widehat{\mu}_t)=V(\bar{\mu})$ holds $\PROB$-almost surely.
Combining this with statement~\ref{thms: main-convergence-tight-as},
it follows that $\PROB$-almost surely,
every accumulation point of $(\widehat{\mu}_t)_{t\in\N_0}$ minimizes~$V$,
and is thus a $\CW_2$-barycenter of $\nu_1,\ldots,\nu_K$ with weights $w_1,\ldots,w_K$.
Since $\bar{\mu}$ is the unique \mbox{$\CW_2$-barycenter} of $\nu_1,\ldots,\nu_K$ with weights $w_1,\ldots,w_K$,
we get $\lim_{t\to\infty}\CW_2(\widehat{\mu}_t,\bar{\mu})=\nobreak0$ $\PROB$-almost surely,
completing the proof of statement~\ref{thms: main-convergence-PL}.
Finally, directly combining \eqref{eqn: main-convergence-condvar} and statement~\ref{thms: main-convergence-PL} proves statement~\ref{thms: main-convergence-PL-var}.
The proof of Theorem~\ref{thm: main-convergence} is now complete.
\qed

\subsection{Discussions about the conditions in Theorem~\ref{thm: main-convergence}}\label{ssec: main-convergence-conditions}

In this subsection, 
we discuss the following conditions in Theorem~\ref{thm: main-convergence}:
the uniqueness of the fixed-point of the $G$-operator,
the expected Polyak--{\L}ojasiewicz inequality \eqref{eqn: main-convergence-condPL},
and the expected variance inequality \eqref{eqn: main-convergence-condvar}.
We will discuss sufficient conditions in the literature to guarantee them.
At the end of the subsection, we will present a 
one-dimensional setting under which these three conditions 
along with the inequalities (\ref{eqn: main-convergence-cond1}) and 
(\ref{eqn: main-convergence-cond2}) 
in Theorem~\ref{thm: main-convergence}
can be \textit{simultaneously satisfied}.

The operator~$G$ in (\ref{eqn: G-operator}) does not always have a unique fixed-point for general input measures $\nu_1,\ldots,\nu_K\in\CP_{2,\AC}(\R^d)$; see, e.g., Example~3.1 of \citep{alvarez2016fixed} for a concrete counterexample. 
As discussed in Section~\ref{ssec: preliminary-OT},
this non-uniqueness essentially stems from the lack of geodesic convexity of the barycenter functional~$V$.
The following conditions for a fixed-point of $G$ to be the unique $\CW_2$-barycenter are provided by \citet[Theorem~2 \& Remark~1]{zemel2019frechet}.

\begin{proposition}[Optimality criteria for the fixed-points of $G$ {\citep{zemel2019frechet}}]\label{prop: suff-cond-unique-fixedpoint}
	Let $\CX\subseteq\R^d$ be open and convex.
	For $k=1,\ldots,K$,
	let $\nu_k\in\CP_{2,\AC}(\R^d)$ satisfy 
	$\support(\nu_k)=\clos(\CX)$,
	and that 
	the density function $f_{\nu_k}$ of $\nu_k$
	is bounded and strictly positive on $\CX$.
	Then, a fixed-point $\mu\in\CP_{2,\AC}(\R^d)$ of $G$ is the unique $\CW_2$-barycenter of $\nu_1,\ldots,\nu_K$ with weights $w_1,\ldots,w_K$ provided that one of the following conditions is satisfied:
	\begin{enumerate}[label=(\Roman*),beginpenalty=10000]
		\item $\CX=\R^d$, the density function $f_{\mu}$ of $\mu$ is bounded and strictly positive, 
		and $f_{\mu},f_{\nu_1},\ldots,f_{\nu_K}\in\CC^{\LOCAL,0,\alpha}(\R^d)$ for some $\alpha\in(0,1]$;

		\item $\CX$ is bounded, $\support(\mu)=\clos(\CX)$, the density function $f_{\mu}$ of $\mu$ is bounded,
		and $f_{\mu},f_{\nu_1},\ldots,f_{\nu_K}$ are all bounded away from zero on $\CX$.
	\end{enumerate}
\end{proposition}

The above optimality criteria provide conditions on the 
accumulation points of $(\widehat{\mu}_t)_{t\in\N_0}$ produced by Algorithm~\ref{algo: abstract}
to rule out those fixed-points of $G$ that are not the $\CW_2$-barycenter $\bar{\mu}$.
However, guaranteeing the convergence of $(\widehat{\mu}_t)_{t\in\N_0}$
to $\bar{\mu}$ requires checking that every fixed-point of $G$ satisfies an optimality criterion in Proposition~\ref{prop: suff-cond-unique-fixedpoint},
resulting in a condition that is difficult to verify in practice.
Even though \citet[Remark~2]{zemel2019frechet}
have conjectured that $G$ has a unique fixed-point when only $\nu_1,\ldots,\nu_K$ satisfy the respective conditions stated in Proposition~\ref{prop: suff-cond-unique-fixedpoint},
the only setting we are aware of which guarantees the uniqueness of the fixed-point of~$G$ over the entire space $\CP_{2, \AC}(\R^d)$ is the one-dimensional setting \citep[Section~6.1]{backhoff2025stochastic}, i.e., when $d=1$,
whereas 
sufficient conditions to guarantee the uniqueness of the fixed-point of~$G$ for non-parametric $\nu_1,\ldots,\nu_K$ in $d\ge 2$ dimensions remains an open problem to the best of our knowledge.

To circumvent this difficulty when analyzing the convergence rate of a sequence $(\mu_t)_{t\in\N_0}\subset\CP_{2,\AC}(\R^d)$ produced by some (deterministic) iterative Wasserstein barycenter algorithm,
one strategy is to restrict attention to a suitably chosen subset $\CM\subset\CP_{2,\AC}(\R^d)$ 
where $\bar{\mu}\in\CM$,
and to establish or assume that the 
Polyak--{\L}ojasiewicz (P{\L}) inequality:
$V(\mu) - V(\bar{\mu}) \le \frac{1}{\PLINEQCONST}\CW_2(\mu,G(\mu))^2$ 
for all $\mu\in\CM$ 
with respect to some constant $\PLINEQCONST\in(0,1]$.
Subsequently, as long as one can ensure that 
$\mu_t\in\CM$ in each iteration~$t$,
the convergence guarantee of $\big(V(\mu_t)\big)_{t\in\N_0}$ to $V(\bar{\mu})$ 
at the geometric rate $O\big((1-\PLINEQCONST)^t\big)$ in~$t$
ensues.
For example,
\citet*{chewi2020gradient} have adopted this strategy to 
establish a convergence rate of the $G$-iteration (\ref{eqn: G-iteration})
that is geometric in~$t$
in the case where $\nu_1,\ldots,\nu_K$ all belong to the same elliptical family (see Proposition~\ref{prop: suff-cond-PL}\ref{propc: suff-cond-PL-elliptical} below for the details of their setting), 
and \citet{montesuma2025computing} have recently analyzed the convergence of a Wasserstein gradient algorithm for computing the $\CW_2$-barycenter of empirical measures under 
the assumption of the P{\L} inequality without explicitly justifying when it holds. 
We remark that the 
assumption of the 
P{\L} inequality
is common in the non-convex optimization literature for analyzing 
the convergence rates of gradient-based algorithms; see, e.g., \citep{polyak1963gradient}.

Under our stochastic fixed-point iterative scheme in Algorithm~\ref{algo: abstract}, 
the desired P{\L} inequality
manifests in the expectation form presented in \eqref{eqn: main-convergence-condPL}.
In the following, we present four specific settings in the literature under which 
\eqref{eqn: main-convergence-condPL} holds, which are direct consequences of the discussion of
\citet*[Section~6]{backhoff2025stochastic}.

\begin{proposition}[Settings where \eqref{eqn: main-convergence-condPL} holds \citep{backhoff2025stochastic}]\label{prop: suff-cond-PL}
	The expected Polyak--{\L}ojasiewicz inequality
	\eqref{eqn: main-convergence-condPL}
	holds under the following four settings.
	\begin{enumerate}[label=(\roman*),beginpenalty=10000]
		\item\label{propc: suff-cond-PL-1d}
		\emph{(One-dimensional measures; see, e.g., \citep[Proposition~7.14]{chewi2024statistical})}
		When $d=1$ and $\nu_1,\ldots,\nu_K\in\CP_{2,\AC}(\R)$,
		\eqref{eqn: main-convergence-condPL} holds with respect to 
		$\PLINEQCONST\leftarrow 1$.
		
		\item\label{propc: suff-cond-PL-commcopula}
		\emph{(Measures sharing a common copula; see \citep[Theorem~2.9]{cuesta1993optimal})}
		Let $F:[0,1]^d\to[0,1]$ be a copula,
		and let $\CM_{\COPULA,F}$ denote the collection of all $\mu\in\CP_{2,\AC}(\R^d)$ with copula $F$ (see, e.g., \citep[Definition~5.1 \& Definition~5.4]{mcneil2015quantitative} for the relevant definitions).
		Then, when $\nu_1,\ldots,\nu_K\in\CM_{\COPULA,F}$
		and $(\widehat{\mu}_t)_{t\in\N_0}\subset\CM_{\COPULA,F}$,
		\eqref{eqn: main-convergence-condPL} holds with respect to 
		$\PLINEQCONST\leftarrow 1$.
		
		\item\label{propc: suff-cond-PL-sphereequiv}
		\emph{(Spherically equivalent measures; see \citep[Theorem~3.2]{cuesta1993optimal})}
		Let $\eta\in\CP_{2,\AC}(\R^d)$ and let $\CM_{\SPHERICAL,\eta}\subset \CP_{2,\AC}(\R^d)$ be defined as follows:
		\begin{align*}
			\hspace{30pt}\CM_{\SPHERICAL,\eta}:=\left\{T\sharp\eta: 
			\begin{tabular}{l}
				$T(\BIx)={\textstyle\frac{\alpha(\|\BIx\|)}{\|\BIx\|}\BIx}\; \forall \BIx\in\R^d\setminus\{\veczero_d\},\; T(\veczero_d)=\veczero_d,\;\alpha:\R_+\to\R_+$ \\
				$\text{is continuous and increasing, and satisfies }\lim_{z\to\infty}{\textstyle\frac{\alpha(z)}{z}}<\infty$
			\end{tabular}\!\!\right\}.
		\end{align*}
		Then, when $\nu_1,\ldots,\nu_K\in\CM_{\SPHERICAL,\eta}$
		and $(\widehat{\mu}_t)_{t\in\N_0}\subset\CM_{\SPHERICAL,\eta}$,
		\eqref{eqn: main-convergence-condPL} holds with respect to 
		$\PLINEQCONST\leftarrow 1$.
		
		\item\label{propc: suff-cond-PL-elliptical}
		\emph{(Measures belonging to the same elliptical family; see \citep[Theorem~19]{chewi2020gradient})}
		Let $\eta\in\CP_{2,\AC}(\R^d)$ be spherical,
		that is, 
		$U\sharp\eta=\eta$ for any transformation 
		$\R^d\ni\BIx\mapsto U(\BIx):=\BU\BIx\in\R^d$
		where $\BU\in\R^{d\times d}$ is an orthonormal matrix
		(i.e., $\BU\BU^\TRANSP=\BU^\TRANSP\BU=\BI_d$).
		Moreover, let the covariance matrix of $\eta$ be $\BI_d$,
		let $\zeta,\overline{\lambda}\in(0,\infty)$,
		and let $\CM_{\ELLIPTICAL,\eta,\zeta,\overline{\lambda}}\subset \CP_{2,\AC}(\R^d)$ 
		be defined as follows:
		\begin{align*}
			\hspace{40pt}
			\CM_{\ELLIPTICAL,\eta,\zeta,\overline{\lambda}}:=\Big\{T\sharp\eta: T(\BIx)=\BA\BIx+\BIb \; \forall \BIx\in\R^d,\; \BA\in\mathbb{S}^{d}_{++}, \; \det(\BA)\ge \zeta, \;\BA \preceq \overline{\lambda}\BI_d,\; \BIb\in\R^d\Big\}.
		\end{align*}
		Then, when $\nu_1,\ldots,\nu_K\in\CM_{\ELLIPTICAL,\eta,\zeta,\overline{\lambda}}$
		and $(\widehat{\mu}_t)_{t\in\N_0}\subset\CM_{\ELLIPTICAL,\eta,\zeta,\overline{\lambda}}$,
		\eqref{eqn: main-convergence-condPL} holds with respect to 
		$\PLINEQCONST\leftarrow \nobreak\frac{1}{4}\zeta^{4}\overline{\lambda}^{-4d}$.
	\end{enumerate}
\end{proposition}
Note that in the first three settings above, 
\eqref{eqn: main-convergence-condPL} holds with respect to 
$\PLINEQCONST\leftarrow 1$
because of the property that 
$T^{\mu_1}_{\mu_2}\circ T^{\mu_2}_{\mu_3}=T^{\mu_1}_{\mu_3}$ holds
for all 
$\mu_1,\mu_2,\mu_3\in\CP_{2,\AC}(\R)$,
for all 
$\mu_1,\mu_2,\mu_3\in\CM_{\COPULA,F}$,
and for all
$\mu_1,\mu_2,\mu_3\in\CM_{\SPHERICAL,\eta}$,
respectively.
Essentially, this is due to the fact that 
$\big(\CP_{2,\AC}(\R),\CW_2\big)$,
$\big(\CM_{\COPULA,F},\CW_2\big)$,
and
$\big(\CM_{\SPHERICAL,\eta},\CW_2\big)$
can all be isometrically embedded into a Hilbert space;
see, e.g., \citep[Proposition~7.14]{chewi2024statistical} and 
\citep[Section~6.2 \& Section~6.3]{backhoff2025stochastic}.

Moreover, the expected variance inequality \eqref{eqn: main-convergence-condvar} bounds the expected $\CW_2$-distance between 
the measures $(\widehat{\mu}_{t})_{t\in\N_0}$ produced by Algorithm~\ref{algo: abstract}
and the true barycenter $\bar{\mu}$
using the expected optimality gap.
We present below a sufficient condition for 
\eqref{eqn: main-convergence-condvar} to hold, 
which is a direct consequence of \citep[Theorem~6]{chewi2020gradient}.
\begin{proposition}[Sufficient conditions for \eqref{eqn: main-convergence-condvar}]\label{prop: suff-cond-var}
	Let $\nu_1,\ldots,\nu_K\in\CP_{2,\AC}(\R^d)$,
	where there exists at least one index $k\in\{1,\ldots,K\}$
	such that 
	the density function of $\nu_k$ is essentially bounded.
	Let $\bar{\mu}\in\CP_{2,\AC}(\R^d)$ denote the unique $\CW_2$-barycenter of $\nu_1,\ldots,\nu_K$ with weights $w_1,\ldots,w_K$,
	and let $\varphi^{\bar{\mu}}_{\nu_k}:\R^d\to\R\cup\{\infty\}$ denote the Brenier potential from $\bar{\mu}$ to $\nu_k$ for $k=1,\ldots,K$.
	Then, under the assumption that there exist $\underline{\lambda}_1,\ldots,\underline{\lambda}_K\in\R_+$ 
	satisfying 
	$\varphi^{\bar{\mu}}_{\nu_k}\in\FC_{\underline{\lambda}_k,\infty}(\R^d)$ $\forall 1\le k\le K$
	and 
	$\underline{\lambda}:=\sum_{k=1}^{K}w_k\underline{\lambda}_k>0$,
	\eqref{eqn: main-convergence-condvar} holds with respect to $\VARINEQCONST\leftarrow \underline{\lambda}$.
\end{proposition}

Let us also remark that 
under the settings of Proposition~\ref{prop: suff-cond-PL}\ref{propc: suff-cond-PL-1d}--\ref{propc: suff-cond-PL-sphereequiv}, 
one can show that \eqref{eqn: main-convergence-condvar} holds with respect to
$\VARINEQCONST\leftarrow 1$ 
using the aforementioned property 
that 
$\big(\CP_{2,\AC}(\R),\CW_2\big)$,
$\big(\CM_{\COPULA,F},\CW_2\big)$,
and
$\big(\CM_{\SPHERICAL,\eta},\CW_2\big)$
can be isometrically embedded into a Hilbert space.
Under the setting of Proposition~\ref{prop: suff-cond-PL}\ref{propc: suff-cond-PL-elliptical}, 
\citet*{chewi2020gradient} used Proposition~\ref{prop: suff-cond-var} to show that 
\eqref{eqn: main-convergence-condvar} holds with respect to $\VARINEQCONST\leftarrow \zeta^2\overline{\lambda}^{-2d}$.

Lastly,
we present in the following proposition a one-dimensional setting 
under which $G$ has a unique fixed-point,
both \eqref{eqn: main-convergence-condPL}
and \eqref{eqn: main-convergence-condvar} are satisfied,
and the inequalities (\ref{eqn: main-convergence-cond1}) and (\ref{eqn: main-convergence-cond2}) in Theorem~\ref{thm: main-convergence}
are also satisfied.
Thus, this setting yields all the properties of the output
$(\widehat{\mu}_t)_{t\in\N_0}$ of Algorithm~\ref{algo: abstract}
in the conclusions of statements~\ref{thms: main-convergence-tight-as}--\ref{thms: main-convergence-PL-var} in Theorem~\ref{thm: main-convergence}.

\begin{proposition}[A one-dimensional setting with kernel-based estimators]\label{prop: special-case-1d}
	Let $d=1$, 
	let $(\underline{a}_k,\overline{a}_k)_{k=1:K}$ satisfy 
	$-\infty<\underline{a}_k<\overline{a}_k<\infty$ $\forall 1\le k\le K$,
	let $\underline{a}:=\sum_{k=1}^{K}w_k\underline{a}_k$, 
	$\overline{a}:=\sum_{k=1}^{K}w_k\overline{a}_k$,
	let $\mu_0\in\CP_{2,\AC}(\R)$ satisfy 
	$\support(\mu_0)=[\underline{a},\overline{a}]$,
	and let 
	$(\nu_k)_{k=1:K}\subset\CP_{2,\AC}(\R)$ satisfy 
	$\support(\nu_k)=[\underline{a}_k,\overline{a}_k]$ $\forall 1\le k\le K$.
	Let $f_{\mu_0}$
	denote the density function of $\mu_0$,
	where we assume that 
	there exists $\zeta_0\ge\nobreak1$ satisfying 
	$\zeta_0^{-1}\le f_{\mu_0}(x)\le \zeta_0$ $\forall x\in[\underline{a},\overline{a}]$.
	For $k=1,\ldots,K$, 
	let $f_{\nu_k}$ denote the density function of $\nu_k$,
	where we assume that 
	there exists $\zeta_k\ge\nobreak1$ satisfying 
	$\zeta^{-1}_k\le f_{\nu_k}(x)\le \zeta_{k}$ $\forall x\in\nobreak[\underline{a}_k,\overline{a}_k]$.
	Moreover, let 
	$\kappa_0,\kappa_1,\ldots,\kappa_K:\R\to(0,\infty)$ be continuous and satisfy
	$\kappa_k(x)\ge \kappa_k(x')$ whenever $|x'|\ge \nobreak|x|$,
	$\int_{\R}\kappa_k(x)\DIFFX{x}=\nobreak1$,
	$\int_{\R}x^2\kappa_k(x)\DIFFX{x}<\nobreak\infty$
	$\forall 0\le k\le K$.
	With the above assumptions and notions,
	we let $(X_{t,k,i})_{i=1:\widehat{M}_{t-1,k}}$ and
	$(Y_{t,k,j})_{j=1:\widehat{N}_{t-1,k}}$
	denote the samples from $\widehat{\mu}_{t-1}$ and $\nu_k$
	generated in Line~\ref{alglin: abstract-sample1} and Line~\ref{alglin: abstract-sample2} of Algorithm~\ref{algo: abstract},
	respectively 
	(we remove the boldface in their notations because they are scalars rather than vectors when $d=1$),
	which are random variables on a probability space $(\Omega,\CF,\PROB)$
	that are jointly independent conditional on $\CF_{t-1}$ defined in (\ref{eqn: filtration}).
	Furthermore, we let $\Theta:=(0,\infty)^2$,
	and let $\bar{\mu}$ denote the unique $\CW_2$-barycenter of $\nu_1,\ldots,\nu_K$ with weights $w_1,\ldots,w_K$.
	Subsequently, for each iteration~$t\in\N$ and for $k=1,\ldots,K$,
	let us denote $\widehat{\ITheta}_{t-1,k}=(h_{t-1,k,0},h_{t-1,k,1})\in\nobreak\Theta$,
	and define the functions 
	$(\widehat{F}_{t,k,0}:[\underline{a},\overline{a}]\to[0,1])_{k=1:K}$,
	$(\widehat{F}_{t,k,1}:[\underline{a}_k,\overline{a}_k]\to[0,1])_{k=1:K}$ 
	as well as the OT map estimators 
	$(\widehat{T}_{t,k}:[\underline{a},\overline{a}]\to[\underline{a}_k,\overline{a}_k])_{k=1:K}$ in Line~\ref{alglin: abstract-estimator} of Algorithm~\ref{algo: abstract}
	as follows:
	\begin{align*}
		\widehat{F}_{t,k,0}(x)&:= \frac{1}{\widehat{M}_{t-1,k}}\sum_{i=1}^{\widehat{M}_{t-1,k}}\frac{{\displaystyle\int_{\underline{a}}^{x}}\kappa_0\Big(\frac{z-X_{t,k,i}}{h_{t-1,k,0}}\Big)\DIFFX{z}}{{\displaystyle\int_{\underline{a}}^{\overline{a}}}\kappa_0\Big(\frac{z-X_{t,k,i}}{h_{t-1,k,0}}\Big)\DIFFX{z}}
		\hspace{10.4pt}\qquad \forall x\in [\underline{a},\overline{a}],\;\forall 1\le k\le K, \allowdisplaybreaks\\
		\widehat{F}_{t,k,1}(x)&:= \frac{1}{\widehat{N}_{t-1,k}}\sum_{j=1}^{\widehat{N}_{t-1,k}}\frac{{\displaystyle\int_{\underline{a}_k}^{x}}\kappa_k\Big(\frac{z-Y_{t,k,j}}{h_{t-1,k,1}}\Big)\DIFFX{z}}{{\displaystyle\int_{\underline{a}_k}^{\overline{a}_k}}\kappa_k\Big(\frac{z-Y_{t,k,j}}{h_{t-1,k,1}}\Big)\DIFFX{z}}
		\qquad \forall x\in [\underline{a}_k,\overline{a}_k],\; \forall 1\le k\le K,\allowdisplaybreaks\\
		\widehat{T}_{t,k}(x)&:= \begin{cases}
			\underline{a}_k & 
			\hspace{26.3pt}\qquad\forall x\in(-\infty,\underline{a}), \\
			\widehat{F}_{t,k,1}^{-1}\big(\widehat{F}_{t,k,0}(x)\big) & 
			\hspace{42.4pt}\qquad\forall x\in[\underline{a},\overline{a}], \\
			\overline{a}_k & 
			\hspace{34.8pt}\qquad\forall x\in(\overline{a},\infty),
		\end{cases}
		\hspace{32.5pt}\forall 1\le k\le K.
	\end{align*}
	Additionally, let us set 
	$\widehat{\mu}_{t}:=\big[\sum_{k = 1}^K w_k \widehat{T}_{t,k}\big] \sharp \widehat{\mu}_{t-1}$
	in Line~\ref{alglin: abstract-update} of Algorithm~\ref{algo: abstract}.
	Then, the following statements hold.
	\begin{enumerate}[label=(\roman*),beginpenalty=10000]
		\item\label{props: special-case-1d-regularity}
		All conditions in Assumption~\ref{asp: abstract-regularity} are satisfied.
		
		\item\label{props: special-case-1d-inequalities}
		For any $\beta\in(0,1)$ and for $t\in\N$,
		one can choose 
		$\widehat{M}_{t-1,k}=O(\beta^{-2t})$,
		$\widehat{N}_{t-1,k}=O(\beta^{-2t})$,
		$h_{t-1,k,0}=O\big(\beta^{\frac{t}{2}}\big)$,
		$h_{t-1,k,1}=O\big(\beta^{\frac{t}{2}}\big)$
		for $k=1,\ldots,K$
		such that 
		both (\ref{eqn: main-convergence-cond1})
		and (\ref{eqn: main-convergence-cond2})
		are satisfied.
		Note that the constants omitted by the big-$O$ notations 
		here depend on $\widehat{\mu}_{t-1}$, $(\nu_k)_{k=1:K}$, 
		$\big(\kappa_{k}(\cdot)\big)_{k=0:K}$
		and do not depend on $\beta^t$.
		
		\item\label{props: special-case-1d-uniqueness}
		$G$ has a unique fixed-point.
		
		\item\label{props: special-case-1d-PL}
		The expected Polyak--{\L}ojasiewicz inequality \eqref{eqn: main-convergence-condPL} holds with respect to 
		$\PLINEQCONST\leftarrow 1$.
		
		\item\label{props: special-case-1d-var}
		The expected variance inequality \eqref{eqn: main-convergence-condvar}
		holds with respect to 
		$\VARINEQCONST\leftarrow 1$.
	\end{enumerate}
	Consequently, 
	all conditions in Theorem~\ref{thm: main-convergence}
	are simultaneously satisfied.
	In this case,
	the conclusions of statements~\ref{thms: main-convergence-tight-as}--\ref{thms: main-convergence-PL-var} in Theorem~\ref{thm: main-convergence} hold,
	namely, it holds that
	$(\widehat{\mu}_t)_{t\in\N_0}$ converges 
	$\PROB$-almost surely in $\CW_2$ to~$\bar{\mu}$,
	$\EXP\big[V(\widehat{\mu}_t)\big]-V(\bar{\mu})\le 4\beta^{t}$ $\forall t\in\N$,
	and 
	$\EXP\big[\CW_2(\widehat{\mu}_t,\bar{\mu})^2\big]\le\nobreak 4\beta^{t}$ $\forall t\in\N$.
\end{proposition}

\begin{proof}[Proof of Proposition~\ref{prop: special-case-1d}]
	See Appendix~\ref{sapx: proof-main-convergence-conditions}.
\end{proof}

\section{A computationally tractable setting under Caffarelli-type conditions}\label{sec: Caffarelli-setting}

\subsection{Concrete setting}\label{ssec: Caffarelli-setting-concrete}

We have developed in Theorem~\ref{thm: main-convergence} sufficient conditions
to guarantee the convergence of Algorithm~\ref{algo: abstract}.
Specifically, the $\PROB$-almost sure convergence of 
$(\widehat{\mu}_t)_{t\in\N_0}$ requires both
(\ref{eqn: main-convergence-cond1})
and (\ref{eqn: main-convergence-cond2}) to be satisfied.
While we have discussed in 
Proposition~\ref{prop: special-case-1d} 
that 
(\ref{eqn: main-convergence-cond1})
and (\ref{eqn: main-convergence-cond2})
can be satisfied in the one-dimensional setting,
the goal of this section is to develop a setting involving a wider class of 
non-parametric input probability measures that are possibly multi-dimensional,
and to develop a concrete and computationally tractable procedure
to carry out the approximations in 
Line~\ref{alglin: abstract-estimator} and
Line~\ref{alglin: abstract-update} 
of Algorithm~\ref{algo: abstract}.
Our setting will be presented in Setting~\ref{sett: Caffarelli}
and our concrete procedure will be presented in Algorithm~\ref{algo: concrete}.

Most importantly,
we show that when the sample sizes
$(\widehat{M}_{t,k})_{k=1:K,\,t\in\N_0}$, $(\widehat{N}_{t,k})_{k=1:K,\,t\in\N_0}$ 
as well as the additional hyperparameters in the concrete procedure are suitably chosen, 
both (\ref{eqn: main-convergence-cond1})
and (\ref{eqn: main-convergence-cond2}) can be satisfied,
allowing us to prove its convergence in Theorem~\ref{thm: Caffarelli-convergence}.
To this end, observe that (\ref{eqn: main-convergence-cond1})
requires the convergence rates of the approximation errors of the OT map estimators $(\widehat{T}_{t,k})_{k=1:K,\,t\in\N}$.
Establishing the convergence rates of OT map estimators from $\mu$ to $\nu$ requires not only regularities of $\mu$ and $\nu$, but also regularity of the true OT map $T^{\mu}_{\nu}$
(and hence also the Brenier potential $\varphi^{\mu}_{\nu}$).
Typically, one would require 
$\varphi^{\mu}_{\nu}\in\CC^{2}(\support(\mu))$ 
and
$\lambda_{\LBSUB}\BI_d\preceq\nabla^2\varphi^{\mu}_{\nu}(\BIx)\preceq\lambda_{\UBSUB}\BI_d$ $\forall \BIx\in\support(\mu)$
for $0<\nobreak\lambda_{\LBSUB}\le\lambda_{\UBSUB}<\nobreak\infty$;
see, e.g., \citep{hutter2021minimax,muzellec2021near,ghosal2022multivariate,manole2024plugin,pooladian2021entropic}.
Therefore, our concrete setting is inspired by Caffarelli's global regularity theory for Brenier potentials.
The foundations of the regularity theory of OT maps were developed
by \citet*{caffarelli1990localization, caffarelli1990interior, caffarelli1991some, caffarelli1992regularity, caffarelli1992boundary, caffarelli1996boundary}
in a series of studies under suitable geometric assumptions on the supports and densities of the measures. 
Here, we only partially report these results as phrased in \citep[Theorem~12.50(iii)]{villani2009optimal}.

\begin{theorem}[Caffarelli's global regularity theory]\label{thm: Caffarelli}
	Let $q\in\N_0$, and
	let $\CX_\mu$ and $\CX_\nu$ 
	be two bounded open sets in $\R^d$ that 
	both have $\CC^{q+2}$-boundaries and are both uniformly convex.\footnote{A set $\CX \subset \R^d$ is said to have $\CC^p$-boundary with $p \in [1, \infty)$ if $\boundary(\CX)$ is locally the graph of a $\CC^p$-function,
	and is said to be uniformly convex if its second fundamental form is uniformly positive on the whole of $\boundary(\CX)$; see \citep[page~317]{villani2009optimal}.}
	Let $\mu\in\CP_{2,\AC}(\R^d)$ be concentrated on $\CX_{\mu}$ and 
	let $\nu\in\CP_{2,\AC}(\R^d)$ be concentrated on $\CX_{\nu}$, 
	i.e., \sloppy{$\mu(\R^d\backslash \CX_{\mu}) =\nu(\R^d\backslash\CX_{\nu})=\nobreak0$}.
	Suppose that for $\alpha\in(0,1)$, 
	$f_{\mu}\in\CC^{q,\alpha}\big(\clos(\CX_{\mu})\big)$, $f_{\nu}\in\CC^{q,\alpha}\big(\clos(\CX_{\nu})\big)$ are the density functions of $\mu$, $\nu$ with respect to the Lebesgue measure, respectively. 
	Moreover, suppose that there exists $\zeta\ge1$ such that $\zeta^{-1}\le f_{\mu}(\BIx)\le \zeta$ $\forall \BIx\in\clos(\CX_{\mu})$ and that $\zeta^{-1}\le f_{\nu}(\BIx)\le\zeta$ $\forall \BIx\in\clos(\CX_{\nu})$. 
	Then, the Brenier potential $\varphi^{\mu}_{\nu}$ belongs to $\CC^{q+2,\alpha}\big(\clos(\CX_{\mu})\big)$.
\end{theorem}

Our setting works with the following classes of 
\textit{admissible probability measures} satisfying a set of conditions tailored to Caffarelli's global regularity theory.

\begin{definition}[Caffarelli-type admissible probability measures]\label{def: admissible-measures}
	For any $q\in\N_0$, 
	let $\CS^q(\R^d)$ denote the collection of subsets of $\R^d$ defined as follows:
	\begin{align*}
		\CS^q(\R^d):=\big\{\clos(\CY):\CY\subset\R^d \text{ is non-empty, open, bounded, has a }\CC^{q+2} \text{-boundary, and is uniformly convex}\big\}.
	\end{align*}
	Let $\CM^q(\R^d)$ denote the following class of probability measures on $\R^d$:
	\begin{align*}
		\CM^q(\R^d):=\left\{\mu\in\CP_{2,\AC}(\R^d) : \begin{tabular}{l}
			$\support(\mu)\in\CS^q(\R^d),\; \exists\alpha\in(0,1),\;\exists \zeta\ge1,\;\exists f_{\mu}\in\CC^{q,\alpha}(\support(\mu)),$ \\
			$\zeta^{-1}\le f_{\mu}(\BIx)\le \zeta\;\forall \BIx\in\support(\mu),\;f_{\mu}$ is the density function of $\mu$
			\end{tabular}
		\right\}.
	\end{align*}
	We will refer to $\CM^q(\R^d)$ as the class of Caffarelli-type $q$-admissible compactly supported probability measures.
	Moreover, let $\CM^q_{\FULL}(\R^d)$ denote the following class of probability measures on $\R^d$:
	\begin{align*}
		\CM^q_{\FULL}(\R^d):=\left\{\rho\in\CP_{2,\AC}(\R^d): \begin{tabular}{l}
			$\support(\rho)=\R^d,\; \exists\alpha\in(0,1),\;\exists f_{\rho}\in\CC^{\LOCAL,q,\alpha}(\R^d),$ \\
			$f_{\rho}(\BIx)> 0\;\forall \BIx\in\R^d,\;f_{\rho}$ is the density function of $\rho$
			\end{tabular}
		\right\}. 
	\end{align*}
	We will refer to $\CM^q_{\FULL}(\R^d)$ as the class of Caffarelli-type $q$-admissible fully supported probability measures. 
\end{definition}

While Caffarelli-type $q$-admissible compactly supported probability measures $\CM^q(\R^d)$ are required to satisfy stringent conditions,
they are in fact \textit{dense} in $\CP_2(\R^d)$ with respect to the $\CW_2$-metric for any $q\in\N_0$.
Therefore, it is not overly restrictive to consider input measures $\nu_1,\ldots,\nu_K\in\CM^q(\R^d)$ in our algorithm.
In the following, 
we present a \textit{constructive} proof of the density property of 
$\CM^q(\R^d)$ for the sake of completeness.

\begin{proposition}[$\CM^q(\R^d)$ is dense in $(\CP_2(\R^d),\CW_2)$]\label{prop: Caffarelli-density}
	Let $\nu\in\CP_2(\R^d)$, $q\in\N_0$, and $\epsilon>0$ be arbitrary.
	Let $\eta\in\CP_{2,\AC}(\R^d)$ be the law of a $d$-dimensional Gaussian random variable with mean vector $\veczero_d$ and covariance matrix $\frac{\epsilon^2}{4d}\BI_d$.
	Moreover, let $\rho:=\eta * \nu\in\CP_2(\R^d)$ be the convolution of $\eta$ and $\nu$, defined as follows 
	(see, e.g., \citep[Definition~3.9.8, Vol.~I]{bogachev2007measure}):
	\begin{align*}
		\eta * \nu(E) := \int_{\R^d}\int_{\R^d}\INDI_{E}(\BIx+\BIy)\DIFFM{\eta}{\DIFF\BIx}\DIFFM{\nu}{\DIFF\BIy} \qquad \forall E\in\CB(\R^d).
	\end{align*}
	Subsequently, let $r>0$ be sufficiently large such that 
	\begin{align*}
		\rho\big(\bar{B}(\veczero_d,r)\big)&\ge \bigg(1+\frac{\epsilon^2}{16\int_{\R^d}\|\BIx\|^2\DIFFM{\rho}{\DIFF\BIx}}\bigg)^{-1}, \qquad \int_{\R^d}\|\BIx\|^2\INDI_{\R^d\setminus \bar{B}(\veczero_d,r)}(\BIx)\DIFFM{\rho}{\DIFF\BIx} \le \frac{\epsilon^2}{16}.
	\end{align*}
	Then, it holds that 
	$\rho|_{\bar{B}(\veczero_d,r)}\in\CM^q(\R^d)$
	and 
	$\CW_2\big(\rho|_{\bar{B}(\veczero_d,r)},\nu\big)\le\nobreak\epsilon$.
\end{proposition}

\begin{proof}[Proof of Proposition~\ref{prop: Caffarelli-density}]
	See Appendix~\ref{sapx: proof-Caffarelli-setting-concrete}.
\end{proof}

For any $q\in\N_0$, $\mu,\nu\in\CM^q(\R^d)$,
the following lemma utilizes Theorem~\ref{thm: Caffarelli}
to derive the desired regularity properties of $\varphi^{\mu}_{\nu}$
for the consistent estimation of~$T^{\mu}_{\nu}$; see also \citep[Lemma~2]{manole2024plugin} \& \citep[Corollary~3.2]{gigli2011on}.

\begin{lemma}[Curvature properties of $\varphi^{\mu}_{\nu}$]\label{lem: curvature}
	Let $q\in\N_0$, $\mu,\nu\in\CM^q(\R^d)$ be arbitrary, and let 
	$\varphi^{\mu}_{\nu}:\R^d\to\R\cup\{\infty\}$ be the Brenier potential from $\mu$ to $\nu$ 
	(that is unique {$\mu$-almost everywhere} up to the addition of an arbitrary constant by \citep[Remark~10.30]{villani2009optimal}). 
	Then, 
	$\varphi^{\mu}_{\nu}\in\CC^{q+2,\alpha}(\support(\mu))$ 
	for some $\alpha\in(0,1)$,
	and there exist $0<\lambda_{\LBSUB}\le \lambda_{\UBSUB}<\infty$ such that $\lambda_{\LBSUB}\BI_d\preceq \nabla^2 \varphi^{\mu}_{\nu}(\BIx)\preceq \lambda_{\UBSUB}\BI_d$ for all $\BIx\in\support(\mu)$.
\end{lemma}

\begin{proof}[Proof of Lemma~\ref{lem: curvature}]
	See Appendix~\ref{sapx: proof-Caffarelli-setting-concrete}.
\end{proof}

Next, let us introduce the notion of \textit{admissible OT map estimators},
which are assumed to possess measurability, shape, growth, and consistency properties 
to guarantee that the probability measures $(\widehat{\mu}_t)_{t\in\N_0}$
remain in $\CM^q(\R^d)$ throughout Algorithm~\ref{algo: concrete},
and to ensure that the inequalities 
(\ref{eqn: main-convergence-cond1}) and 
(\ref{eqn: main-convergence-cond2}) 
can be satisfied in Algorithm~\ref{algo: concrete}.

\begin{assumption}[Admissible OT map estimator]\label{asp: OTmap-estimator}
	Let $(\Omega,\CF,\PROB)$ be a probability space,
	let $q\in\N_0$, 
	and let $\mu,\nu\in\CM^q(\R^d)$.
	For any $m\in\N$ and $n\in\N$, 
	let $\BIX_1,\ldots,\BIX_m,\BIY_1,\ldots,\BIY_n:\Omega\to\R^d$ be independent random variables such that 
	the law of each $\BIX_i$ is $\mu$
	and the law of each $\BIY_j$ is $\nu$,
	i.e., $\BIX_i\sharp\PROB=\mu$ $\forall 1\le i\le m$,
	$\BIY_j\sharp\PROB=\nu$ $\forall 1\le j\le n$.
	Let $\Theta$ be a metric space, where each $\theta\in\Theta$ denotes the hyperparameter(s) that may, 
	for example, represent the extent of smoothing/regularization 
	(see Section~\ref{ssec: entropic-estimator} for details about the hyperparameters in a concrete OT map estimator). 
	Subsequently, 
	for any $\theta\in\Theta$, 
	let $\widehat{T}^{\mu,m}_{\nu,n}[\BIX_1,\ldots,\BIX_m,\BIY_1,\ldots,\BIY_n, \theta]$ 
	estimate the OT map $T^{\mu}_{\nu}$ from $\mu$ to $\nu$ based on the samples 
	$\BIX_1,\ldots,\BIX_m$ from~$\mu$ 
	and the samples $\BIY_1,\ldots,\BIY_n$ from~$\nu$.
	For notational simplicity, we often make the dependency of this estimated OT map on the samples implicit and use $\widehat{T}^{\mu,m}_{\nu,n}[\theta](\BIx) \in \R^d$ to denote $\widehat{T}^{\mu,m}_{\nu,n}[\BIX_1,\ldots,\BIX_m,\BIY_1,\ldots,\BIY_n,\theta]$ evaluated at $\BIx\in\R^d$.

	We assume further that $\widehat{T}^{\mu,m}_{\nu,n}[\theta]$ satisfies the following conditions.
	\begin{enumerate}[label=(\Roman*), beginpenalty=10000]
		\item\label{asps: OTmap-estimator-measurability}
		\textbf{Measurability:}
		$\widehat{T}^{\mu,m}_{\nu,n}[\BIX_1,\ldots,\BIX_m,\BIY_1,\ldots,\BIY_n, \theta]$ belongs to $\CC_{\LIN}(\R^d,\R^d)$ and
		has a Borel dependency on 
		$(\BIX_1,\ldots,\BIX_m,\allowbreak\BIY_1,\ldots,\BIY_n,\allowbreak \theta) \in (\R^{d})^{m+n} \times \Theta$.

		\item\label{asps: OTmap-estimator-shape}
		\textbf{Shape:}
		there exist minimum sample sizes
		$\underline{m}\in\N$ and $\underline{n}\in\N$ 
		that do not depend on $\mu$ and~$\nu$,
		and
		there exist $\alpha(\mu,\nu,m,n,\allowbreak\BIX_1,\ldots,\BIX_m,\allowbreak\BIY_1,\ldots,\BIY_n,\theta)\in\R_+$ 
		and
		$\underline{\lambda}(\mu,\nu,m,n,\allowbreak\BIX_1,\ldots,\BIX_m,\allowbreak\BIY_1,\ldots,\BIY_n,\theta)\in\R_+$,
		abbreviated to $\alpha$ and $\underline{\lambda}$, 
		both having Borel dependencies on 
		$(\mu,\nu,m,n,\allowbreak\BIX_1,\ldots,\BIX_m,\allowbreak\BIY_1,\ldots,\BIY_n,\theta)$, 
		such that 
		whenever $m\ge \underline{m}$ and $n\ge \underline{n}$,
		it holds $\PROB$-almost surely that
		$\alpha\in(0,1)$,
		$\underline{\lambda}>0$,
		and
		$\widehat{T}^{\mu,m}_{\nu,n}[\theta]=\nabla \widehat{\varphi}^{\mu,m}_{\nu,n}[\theta]$ for a function $\widehat{\varphi}^{\mu,m}_{\nu,n}[\theta]\in\FC_{\underline{\lambda},\infty}^{\LOCAL,q+2,\alpha}(\R^d)$.
		
		\item\label{asps: OTmap-estimator-growth}
		\textbf{Growth:}
		with the same $\underline{m},\underline{n}\in\N$ in \ref{asps: OTmap-estimator-shape},
		there exist $u_0(\nu),u_1(\nu)\in\R_+$ 
		that have Borel dependencies only on $\nu$
		such that
		\begin{align*}
			\hspace{30pt}\EXP\Big[\big\|\widehat{T}^{\mu,m}_{\nu,n}[\theta](\BIx) - \widehat{T}^{\mu,m}_{\nu,n}[\theta](\veczero_d)\big\|^2\Big]\le u_0(\nu)+u_1(\nu)\|\BIx\|^2 \qquad\forall m\ge \underline{m},\; \forall n\ge \underline{n},\; \forall \theta\in\Theta.
		\end{align*}
		
		\item\label{asps: OTmap-estimator-consistency}
		\textbf{Consistency:}
		with the same $\underline{m},\underline{n}\in\N$ in \ref{asps: OTmap-estimator-shape},
		for any $\epsilon>0$, 
		there exist 
		$\overline{m}(\mu,\nu,\epsilon)\in\N\cap[\underline{m},\infty)$,
		$\overline{n}(\mu,\nu,\epsilon)\in\N\cap[\underline{n},\infty)$ 
		that have Borel dependencies on $(\mu,\nu,\epsilon)$,
		and there exists
		$\widetilde{\theta}(\mu,\nu, m, n,\epsilon)\in\Theta$
		that has a Borel dependency on $(\mu,\nu,m,n,\epsilon)$, 
		such that
		\begin{align*}
			\hspace{30pt}\EXP\Big[\big\|\widehat{T}^{\mu,m}_{\nu,n}\big[\widetilde{\theta}(\mu,\nu, m, n,\epsilon)\big]-T^{\mu}_{\nu}\big\|^2_{\CL^2(\mu)}\Big]\le \epsilon \qquad \forall m\ge \overline{m}(\mu,\nu,\epsilon),\;\forall n\ge \overline{n}(\mu,\nu,\epsilon),\; \forall \epsilon>0.
		\end{align*}
	\end{enumerate}
\end{assumption}
We present in Section~\ref{ssec: entropic-estimator} a concrete example of an
admissible OT map estimator satisfying Assumption~\ref{asp: OTmap-estimator}.
Note that the consistency condition in Assumption~\ref{asp: OTmap-estimator}\ref{asps: OTmap-estimator-consistency} is possible due to 
the curvature properties of $T^{\mu}_{\nu}=\nabla\varphi^{\mu}_{\nu}$ in Lemma~\ref{lem: curvature}.

With an admissible OT map estimator $\widehat{T}^{\mu,m}_{\nu,n}[\cdot]$ that satisfies Assumption~\ref{asp: OTmap-estimator},
we are now ready to present the concrete implementation of Algorithm~\ref{algo: abstract} in 
Algorithm~\ref{algo: concrete}.
We present the details of the inputs of Algorithm~\ref{algo: concrete} in the following setting.

\begin{algorithm}[t]
	\caption{\bf{Computationally tractable stochastic fixed-point iterative scheme under Setting~\ref{sett: Caffarelli}}.}\label{algo: concrete}
	\KwIn{$q\in\N_0$, $K\in\N\cap[2,\infty)$ input probability measures $\nu_1, \ldots, \nu_K \in \CM^q(\R^d)$, 
	weights $w_1,\ldots,w_K\in(0,1)$ with $\sum_{k=1}^Kw_k=1$, 
	initial probability measure~$\rho_0 \in \CM^q_{\FULL}(\R^d)$,\linebreak
	family of increasing sets $(\CX_r)_{r \in \N}$, OT map estimator $\widehat{T}^{\mu,m}_{\nu,n}[\cdot]$,
	$\overline{m}(\,\cdot\,,\cdot\,,\cdot\,)$,
	$\overline{n}(\,\cdot\,,\cdot\,,\cdot\,)$,
	$\widetilde{\theta}(\,\cdot\,,\cdot\,,\cdot\,,\cdot\,,\cdot\,)$,
	$\overline{r}_1(\,\cdot\,,\cdot\,)$,
	$\overline{r}_2(\,\cdot\,,\ldots,\cdot\,)$, 
	$\beta\in(0,1)$;
	see Setting~\ref{sett: Caffarelli} for details.}
	\KwOut{$(\widehat{\mu}_t)_{t \in \N_0}$.}

	\nl\label{alglin: concrete-initialize-full}Initialize $\widehat{\rho}_0 \leftarrow \rho_0$.

	\nl\label{alglin: concrete-initialize-radius}Set $\widehat{R}_0\leftarrow \overline{r}_2\big(\widehat{\rho}_0,\nu_1,\ldots,\nu_K,\frac{1}{4}\beta\big)$.
	
	\nl\label{alglin: concrete-initialize-truncate}Initialize $\widehat{\mu}_0\leftarrow \widehat{\rho}_0|_{\CX_{\widehat{R}_0}}$.

	\nl \For{$k=1,\ldots,K$}{
		\nl\label{alglin: concrete-initialize-samplesize}Set $\widehat{M}_{0,k}\leftarrow \overline{m}(\widehat{\mu}_0,\nu_k,\beta)$, 
		$\widehat{N}_{0,k}\leftarrow \overline{n}(\widehat{\mu}_0,\nu_k,\beta)$, 
		$\widehat{\ITheta}_{0,k}\leftarrow \widetilde{\theta}(\widehat{\mu}_0,\nu_k,\widehat{M}_{0,k},\widehat{N}_{0,k},\beta)$.
	}

	\nl\label{alglin: concrete-iteration}\For{$t = 1, 2, \ldots$}{
		\textbf{[Iteration~$t$]:}

		\nl \For{$k=1,\ldots,K$}{
			\nl\label{alglin: concrete-sample1}Randomly generate $\widehat{M}_{t-1,k}$ independent samples $\{\BIX_{t,k,i}\}_{i=1:\widehat{M}_{t-1,k}}$ from $\widehat{\mu}_{t-1}$. 

			\nl\label{alglin: concrete-sample2}Randomly generate $\widehat{N}_{t-1,k}$ independent samples $\{\BIY_{t,k,j}\}_{j=1:\widehat{N}_{t-1,k}}$ from $\nu_k$. 

			\nl\label{alglin: concrete-estimator}$\widehat{T}_{t,k}\leftarrow \widehat{T}^{\widehat{\mu}_{t-1},\widehat{M}_{t-1,k}}_{\nu_k,\widehat{N}_{t-1,k}}\big[\BIX_{t,k,1},\ldots,\BIX_{t,k,\widehat{M}_{t-1,k}},\BIY_{t,k,1},\ldots,\BIY_{t,k,\widehat{N}_{t-1,k}},\widehat{\ITheta}_{t-1,k}\big]$.
		}

		\nl\label{alglin: concrete-pushforward}Update $\widehat{\rho}_{t}\leftarrow \big[\sum_{k=1}^Kw_k\widehat{T}_{t,k}\big]\sharp\widehat{\rho}_{t-1}$. 

		\nl\label{alglin: concrete-radius}Set $\widehat{R}_t\leftarrow \overline{r}_1\big(\widehat{\rho}_{t},\frac{1}{4}\beta^{t}\big) \vee \overline{r}_2\big(\widehat{\rho}_{t},\nu_1,\ldots,\nu_K,\frac{1}{4}\beta^{t+1}\big)$.

		\nl\label{alglin: concrete-truncation}Update $\widehat{\mu}_t\leftarrow \widehat{\rho}_t|_{\CX_{\widehat{R}_t}}$. 

		\nl \For{$k=1,\ldots,K$}{
			\nl\label{alglin: concrete-samplesize}Set $\widehat{M}_{t,k}\leftarrow\overline{m}(\widehat{\mu}_t,\nu_k,\beta^{t+1})$, 
			$\widehat{N}_{t,k}\leftarrow\overline{n}(\widehat{\mu}_t,\nu_k,\beta^{t+1})$, 
			$\widehat{\ITheta}_{t,k}\leftarrow\widetilde{\theta}(\widehat{\mu}_t,\nu_k,\widehat{M}_{t,k},\widehat{N}_{t,k},\beta^{t+1})$.
		}
	}
	\nl \Return $(\widehat{\mu}_t)_{t \in \N_0}$.
\end{algorithm}

\begin{setting}[Inputs of Algorithm~\ref{algo: concrete}]\label{sett: Caffarelli}
The following list explains the inputs of Algorithm~\ref{algo: concrete}
and their required properties.
\begin{enumerate}[label=(S\arabic*),beginpenalty=10000]
	\item\label{setts: Caffarelli-inputs-measures}%
	$q\in\N_0$;
	the $K\in\N\cap[2,\infty)$ input measures $\nu_1,\ldots,\nu_K$ belong to $\CM^q(\R^d)$;
	the weights $w_1,\ldots,w_K\in(0,1)$ satisfy $\sum_{k=1}^{K}w_k=\nobreak1$.

	\item\label{setts: Caffarelli-initial-measure}%
	$\rho_0\in\CM^q_{\FULL}(\R^d)$ is the initial probability measure.
	
	\item\label{setts: Caffarelli-family-increasing}%
	$(\CX_r)_{r\in\N}$ is an infinite sequence of subsets of $\R^d$ satisfying 
	$\CX_r\in\CS^q(\R^d)$, $\CX_{r+1}\supseteq\CX_{r}$ $\forall r\in\N$,
	as well as $\bigcup_{r\in\N}\CX_r=\R^d$.
	We call $(\CX_{r})_{r\in\N}$ a family of increasing sets.
	
	\item\label{setts: Caffarelli-OTestimator}%
	$\widehat{T}^{\mu,m}_{\nu,n}[\cdot]$ is an admissible OT map estimator satisfying Assumption~\ref{asp: OTmap-estimator} 
	with respect to~$q$,
	where the associated $u_0(\cdot)$, $u_1(\cdot)$
	are given by Assumption~\ref{asp: OTmap-estimator}\ref{asps: OTmap-estimator-growth}, 
	and the associated 
	$\overline{m}(\,\cdot\,,\cdot\,,\cdot\,)$,
	$\overline{n}(\,\cdot\,,\cdot\,,\cdot\,)$,
	$\widetilde{\theta}(\,\cdot\,,\cdot\,,\cdot\,,\cdot\,,\cdot\,)$
	are given by Assumption~\ref{asp: OTmap-estimator}\ref{asps: OTmap-estimator-consistency}.
	
	\item\label{setts: Caffarelli-truncation}%
	$\overline{r}_1(\,\cdot\,,\cdot\,)$ and
	$\overline{r}_2(\,\cdot\,,\ldots,\cdot\,)$
	possess the following properties.
	\begin{enumerate}[label=(S5.\Roman*),beginpenalty=10000]
		\item\label{setts: Caffarelli-truncation1}%
		For any $\rho\in\CM^q_{\FULL}(\R^d)$ and any $\epsilon>0$,
		there exists $\overline{r}_1(\rho,\epsilon)\in\N$
		that has a Borel dependency on $(\rho,\epsilon)$
		such that $\CW_2(\rho|_{\CX_r},\rho)^2\le\nobreak\epsilon$
		$\forall r\ge\overline{r}_1(\rho,\epsilon)$;
		see Lemma~\ref{lem: regularity-truncation}\ref{lems: regularity-truncation-radius} for the existence of such
		$\overline{r}_1(\,\cdot\,,\cdot\,)$,
		and see equation~(\ref{eqn: regularity-truncation-proof-r1-def}) in Appendix~\ref{sapx: proof-Caffarelli-setting-concrete} for its explicit expression.
		
		\item\label{setts: Caffarelli-truncation2}%
		For any $\rho\in\CM^q_{\FULL}(\R^d)$ and for any $\epsilon>0$,
		there exists $\overline{r}_2(\rho,\nu_1,\ldots,\nu_K,\epsilon)\in\N$
		that has a Borel dependency on $(\rho,\nu_1,\ldots,\nu_K,\epsilon)$
		such that, for any $r\ge \overline{r}_2(\rho,\nu_1,\ldots,\nu_K,\epsilon)$,
		and any $(m_k)_{k=1:K}\subset\N\cap[\underline{m},\infty)$,
		$(n_k)_{k=1:K}\subset\N\cap[\underline{n},\infty)$
		($\underline{m},\underline{n}\in\N$ are the minimum sample sizes of $\widehat{T}^{\mu,m}_{\nu,n}[\cdot]$ in Assumption~\ref{asp: OTmap-estimator}\ref{asps: OTmap-estimator-shape}),
		$(\theta_k)_{k=1:K}\subset\Theta$,
		it holds that
		$\dot{\mu}_r:=\rho|_{\CX_r}$ and $\dot{T}_r:=\sum_{k=1}^{K}w_k\widehat{T}^{\dot{\mu}_r,m_k}_{\nu_k,n_k}[\theta_k]$
		satisfy
		$\EXP\big[\CW_2(\dot{T}_r\sharp\dot{\mu}_r,\dot{T}_r\sharp\rho)^2\big]\le\nobreak\epsilon$;
		see Lemma~\ref{lem: regularity-truncation}\ref{lems: regularity-truncation-radius} for the existence of such
		$\overline{r}_2(\,\cdot\,,\ldots,\cdot\,)$,
		and see equation~(\ref{eqn: regularity-truncation-proof-r2-def}) in Appendix~\ref{sapx: proof-Caffarelli-setting-concrete} for its explicit expression.
	\end{enumerate}

	\item\label{setts: Caffarelli-beta}%
	$\beta\in(0,1)$ is a constant controlling the convergence of the error in Algorithm~\ref{algo: concrete}; 
	see Theorem~\ref{thm: Caffarelli-convergence} and Theorem~\ref{thm: main-convergence}.
\end{enumerate}	
\end{setting}

A concrete example of a family of increasing sets $(\CX_r)_{r\in\N}$ that 
satisfies the condition \ref{setts: Caffarelli-family-increasing} is $\big(\bar{B}(\veczero_d,r)\big)_{r\in\N}$.
Similarly, a family of increasing ellipsoids in $\R^d$ also satisfies \ref{setts: Caffarelli-family-increasing}.

Under Setting~\ref{sett: Caffarelli},
the shape condition of the OT map estimator 
$\widehat{T}^{\mu,m}_{\nu,n}[\cdot]$ in Assumption~\ref{asp: OTmap-estimator}\ref{asps: OTmap-estimator-shape} allows us to preserve the regularity properties of 
$(\widehat{\rho}_t)_{t\in\N_0}$ and $(\widehat{\mu}_{t})_{t\in\N_0}$ 
throughout Algorithm~\ref{algo: concrete}.
This is stated in statements~\ref{lems: regularity-truncation-preserve} and \ref{lems: regularity-truncation-pushforward-preserve} of Lemma~\ref{lem: regularity-truncation} below.
Moreover, statement~\ref{lems: regularity-truncation-radius} of Lemma~\ref{lem: regularity-truncation}
uses the growth condition of the OT map estimator 
$\widehat{T}^{\mu,m}_{\nu,n}[\cdot]$ in Assumption~\ref{asp: OTmap-estimator}\ref{asps: OTmap-estimator-growth}
to show that
$\overline{r}_1(\,\cdot\,,\cdot\,)$ and
$\overline{r}_2(\,\cdot\,,\ldots,\cdot\,)$
in \ref{setts: Caffarelli-truncation} of Setting~\ref{sett: Caffarelli}
indeed exist.

\begin{lemma}\label{lem: regularity-truncation}
	Under \ref{setts: Caffarelli-inputs-measures}--\ref{setts: Caffarelli-OTestimator}, the following statements hold.
	\begin{enumerate}[label=(\roman*),beginpenalty=10000]
		\item\label{lems: regularity-truncation-preserve}%
		For any $\rho\in\CM_{\FULL}^q(\R^d)$ and $\CX\in\CS^q(\R^d)$, it holds that $\rho|_{\CX}\in\CM^q(\R^d)$.
		
		\item\label{lems: regularity-truncation-pushforward-preserve}%
		For any $\rho\in\CM^q_{\FULL}(\R^d)$, $\mu\in\CM^q(\R^d)$,
		and for any 
		$(m_k)_{k=1:K}\subset\N\cap[\underline{m},\infty)$,
		$(n_k)_{k=1:K}\subset\N\cap[\underline{n},\infty)$
		($\underline{m},\underline{n}\in\N$ are the minimum sample sizes of $\widehat{T}^{\mu,m}_{\nu,n}[\cdot]$ in Assumption~\ref{asp: OTmap-estimator}\ref{asps: OTmap-estimator-shape}),
		$(\theta_k)_{k=1:K}\subset\Theta$,
		it holds $\PROB$-almost surely that 
		$\bar{T}:=\sum_{k=1}^{K}w_k\widehat{T}^{\mu,m_k}_{\nu_k,n_k}[\theta_k]$ satisfies 
		$\bar{T}\sharp\rho\in\CM^q_{\FULL}(\R^d)$.
		
		\item\label{lems: regularity-truncation-radius}%
		One can explicitly construct $\overline{r}_1(\,\cdot\,,\cdot\,)$ and
		$\overline{r}_2(\,\cdot\,,\ldots,\cdot\,)$
		to possess the required properties in \ref{setts: Caffarelli-truncation};
		see equations~(\ref{eqn: regularity-truncation-proof-r1-def}) and (\ref{eqn: regularity-truncation-proof-r2-def}) in the proof for their explicit expressions.
	\end{enumerate}
\end{lemma}

\begin{proof}[Proof of Lemma~\ref{lem: regularity-truncation}]
	See Appendix~\ref{sapx: proof-Caffarelli-setting-concrete}.
\end{proof}

Using Lemma~\ref{lem: regularity-truncation}\ref{lems: regularity-truncation-preserve} and Lemma~\ref{lem: regularity-truncation}\ref{lems: regularity-truncation-pushforward-preserve},
one can inductively check that 
$\widehat{\rho}_{t}\in\CM^q_{\FULL}(\R^d)$
and 
$\widehat{\mu}_{t}\in\CM^q(\R^d)$
hold $\PROB$-almost surely 
for all $t\in\N_0$ throughout Algorithm~\ref{algo: concrete},
and hence $\widehat{T}_{t,k}$ in Line~\ref{alglin: concrete-estimator} of Algorithm~\ref{algo: concrete} is well-defined.

\begin{remark}
	In Algorithm~\ref{algo: concrete}, 
	rather than directly updating 
	$\widehat{\mu}_{t-1}$ to $\widehat{\mu}_{t}\leftarrow\big[\sum_{k=1}^Kw_k\widehat{T}_{t,k}\big]\sharp\widehat{\mu}_{t-1}$, 
	we first apply the pushforward by 
	$\big[\sum_{k=1}^Kw_k\widehat{T}_{t,k}\big]$ to 
	$\widehat{\rho}_{t-1}\in\CM^q_{\FULL}(\R^d)$ 
	in Line~\ref{alglin: concrete-pushforward} 
	to obtain $\widehat{\rho}_{t}\in\CM^q_{\FULL}(\R^d)$, 
	and then truncate $\widehat{\rho}_{t}$ to $\CX_{\widehat{R}_{t}}$ to get $\widehat{\mu}_{t}$ in Line~\ref{alglin: concrete-truncation}. 
	The truncation step guarantees that $\widehat{\mu}_{t}\in\CM^q(\R^d)$ 
	so that the consistency condition of the OT map estimator in Assumption~\ref{asp: OTmap-estimator}\ref{asps: OTmap-estimator-consistency} can be satisfied (see our results and discussions in Section~\ref{ssec: entropic-estimator}). 
	Note that the support of $\big[\sum_{k=1}^Kw_k\widehat{T}_{t,k}\big]\sharp\widehat{\mu}_{t-1}$ 
	does not necessarily belong to $\CS^q(\R^d)$; 
	specifically, the uniform convexity condition may fail. 
\end{remark}

\subsection{Convergence guarantee and computational tractability of Algorithm~\ref{algo: concrete}}\label{ssec: Caffarelli-convergence-proof}

We are now ready to present the main result of this section.

\begin{theorem}[Convergence guarantee of Algorithm~\ref{algo: concrete}]\label{thm: Caffarelli-convergence}
	Under Setting~\ref{sett: Caffarelli}, 
	let $(\Omega,\CF,\PROB)$ be a probability space on which the random samples in
	Line~\ref{alglin: concrete-sample1} and Line~\ref{alglin: concrete-sample2} of Algorithm~\ref{algo: concrete} are defined,
	and let $(\CF_t)_{t\in\N_0}$ be defined by (\ref{eqn: filtration}).
	Moreover, let $\bar{\mu}$ denote the unique $\CW_2$-barycenter of $\nu_1,\ldots,\nu_K$ with weights $w_1,\ldots,w_K$.
	Then, the following statements hold.
	\begin{enumerate}[label=(\roman*),beginpenalty=10000]
		\item\label{thms: Caffarelli-convergence-regularity}%
		All conditions in Assumption~\ref{asp: abstract-regularity} are satisfied.

		\item\label{thms: Caffarelli-convergence-measurability}%
		Both $(\widehat{\rho}_{t})_{t\in\N_0}$ and $(\widehat{\mu}_{t})_{t\in\N_0}$ in Algorithm~\ref{algo: concrete} are $(\CF_t)_{t\in\N_0}$-adapted stochastic processes.
		
		\item\label{thms: Caffarelli-convergence-inequalities}%
		Both inequalities (\ref{eqn: main-convergence-cond1}) and (\ref{eqn: main-convergence-cond2}) in Theorem~\ref{thm: main-convergence} are satisfied with respect to $\beta\in(0,1)$ in the inputs of Algorithm~\ref{algo: concrete}.
	\end{enumerate}
	Consequently, statements~\ref{thms: main-convergence-tight-as}--\ref{thms: main-convergence-PL-var} 
	in Theorem~\ref{thm: main-convergence}
	hold with respect to the output $(\widehat{\mu}_t)_{t\in\N_0}$ of Algorithm~\ref{algo: concrete}.
\end{theorem}

\begin{proof}[Proof of Theorem~\ref{thm: Caffarelli-convergence}]
	We will defer some measurability-related steps of this proof to Lemma~\ref{lem: concrete-measurability} in Appendix~\ref{sapx: proof-Caffarelli-convergence-proof}.
	Let us first prove statement~\ref{thms: Caffarelli-convergence-regularity}.
	Since $\nu_1,\ldots,\nu_K\in\CM^q(\R^d)$ by the condition \ref{setts: Caffarelli-inputs-measures} in Setting~\ref{sett: Caffarelli},
	it follows from the definition of $\CM^q(\R^d)$ in Definition~\ref{def: admissible-measures} that 
	$\nu_1,\ldots,\nu_K$ all have essentially bounded density functions.
	Moreover, 
	Line~\ref{alglin: concrete-estimator} and
	the measurability condition of $\widehat{T}^{\mu,m}_{\nu,n}[\BIX_1,\ldots,\BIX_m,\BIY_1,\ldots,\BIY_n,\theta]$ in Assumption~\ref{asp: OTmap-estimator}\ref{asps: OTmap-estimator-measurability}
	imply that 
	$\widehat{T}_{t,k}$ has a Borel dependency on 
	$(\BIX_{t,k,1},\ldots,\BIX_{t,k,\widehat{M}_{t-1,k}},\allowbreak\BIY_{t,k,1},\ldots,\BIY_{t,k,\widehat{N}_{t-1,k}},\widehat{\ITheta}_{t-1,k})$
	for $k=1,\ldots,K$ and for all $t\in\N$.
	This proves statement~\ref{thms: Caffarelli-convergence-regularity}.

	To prove statement~\ref{thms: Caffarelli-convergence-measurability}, 
	observe that $\widehat{\rho}_0$ is $\CF_0$-measurable
	by Line~\ref{alglin: concrete-initialize-full}.
	Then,
	Line~\ref{alglin: concrete-initialize-radius} and
	the property of $\overline{r}_2(\,\cdot\,,\ldots,\cdot\,)$ in the condition \ref{setts: Caffarelli-truncation2} imply that
	$\widehat{R}_0$ is $\CF_0$-measurable,
	and thus
	Line~\ref{alglin: concrete-initialize-truncate} and Lemma~\ref{lem: concrete-measurability}\ref{lems: concrete-measurability-truncation} guarantee that $\widehat{\mu}_0$ is $\CF_0$-measurable.
	Next, let us assume that $\widehat{\rho}_{t-1}$ and $\widehat{\mu}_{t-1}$ are $\CF_{t-1}$-measurable for some $t\in\N$.
	Hence,
	Line~\ref{alglin: concrete-initialize-samplesize},
	Line~\ref{alglin: concrete-samplesize}, and
	the properties of 
	$\overline{m}(\,\cdot\,,\cdot\,,\cdot\,)$,
	$\overline{n}(\,\cdot\,,\cdot\,,\cdot\,)$,
	$\widetilde{\theta}(\,\cdot\,,\cdot\,,\cdot\,,\cdot\,,\cdot\,)$
	in Assumption~\ref{asp: OTmap-estimator}\ref{asps: OTmap-estimator-consistency} ensure that 
	$\widehat{M}_{t-1,k}$, $\widehat{N}_{t-1,k}$, $\widehat{\ITheta}_{t-1,k}$ 
	are all $\CF_{t-1}$-measurable
	for $k=1,\ldots,K$,
	and 
	the Borel dependency of $\widehat{T}_{t,k}$ on 
	$(\BIX_{t,k,1},\ldots,\BIX_{t,k,\widehat{M}_{t-1,k}},\allowbreak\BIY_{t,k,1},\ldots,\BIY_{t,k,\widehat{N}_{t-1,k}},\widehat{\ITheta}_{t-1,k})$ 
	then yields that 
	$(\widehat{T}_{t,k})_{k=1:K}$ are $\CF_{t}$-measurable.
	Consequently,
	Line~\ref{alglin: concrete-pushforward} and Lemma~\ref{lem: concrete-measurability}\ref{lems: concrete-measurability-pushforward} 
	guarantee that $\widehat{\rho}_t$ is $\CF_{t}$-measurable.
	Lastly,
	since Line~\ref{alglin: concrete-radius} and the properties of
	$\overline{r}_1(\,\cdot\,,\cdot\,)$ and
	$\overline{r}_2(\,\cdot\,,\ldots,\cdot\,)$
	in the condition \ref{setts: Caffarelli-truncation} imply that 
	$\widehat{R}_{t}$ is $\CF_t$-measurable, 
	we get from Line~\ref{alglin: concrete-truncation} and Lemma~\ref{lem: concrete-measurability}\ref{lems: concrete-measurability-truncation} that
	$\widehat{\mu}_t$ is $\CF_t$-measurable,
	and statement~\ref{thms: Caffarelli-convergence-measurability} then follows from induction.

	It remains to prove statement~\ref{thms: Caffarelli-convergence-inequalities}.
	To that end, let us fix an arbitrary $t\in\N$.
	For $k=1,\ldots,K$, 
	substituting
	$\mu\leftarrow\widehat{\mu}_{t-1}$,
	$\nu\leftarrow\nu_k$,
	$m\leftarrow\widehat{M}_{t-1,k}$,
	$n\leftarrow\widehat{N}_{t-1,k}$,
	$\epsilon\leftarrow \beta^{t}$
	into the consistency condition in Assumption~\ref{asp: OTmap-estimator}\ref{asps: OTmap-estimator-consistency},
	we get from Line~\ref{alglin: concrete-initialize-samplesize},
	Line~\ref{alglin: concrete-samplesize}, and 
	Line~\ref{alglin: concrete-estimator} 
	that 
	\begin{align*}
		\EXP\Big[\big\|\widehat{T}_{t,k}-T^{\widehat{\mu}_{t-1}}_{\nu_k}\big\|_{\CL^2(\widehat{\mu}_{t-1})}^2\Big|\CF_{t-1}\Big]\le \beta^{t} \qquad \forall 1\le k\le K.
	\end{align*}
	This shows that (\ref{eqn: main-convergence-cond1}) is satisfied.
	To establish (\ref{eqn: main-convergence-cond2}),
	let us denote $\bar{T}_t:=\sum_{k=1}^{K}w_k\widehat{T}_{t,k}$.
	Firstly, observe that 
	Line~\ref{alglin: concrete-radius} guarantees 
	$\widehat{R}_{t}\ge \overline{r}_1\big(\widehat{\rho}_{t},\frac{1}{4}\beta^{t}\big)$
	$\PROB$-almost surely,
	and that Line~\ref{alglin: concrete-truncation} sets 
	$\widehat{\mu}_t=\widehat{\rho}_t|_{\CX_{\widehat{R}_{t}}}$.
	Hence, applying the 
	property of $\overline{r}_1(\,\cdot\,,\cdot\,)$ in the
	condition \ref{setts: Caffarelli-truncation1} to 
	$\rho\leftarrow\widehat{\rho}_t$,
	$r\leftarrow\widehat{R}_{t}$,
	$\epsilon\leftarrow \frac{1}{4}\beta^{t}$ yields
	\begin{align}
		\label{eqn: Caffarelli-convergence-proof-truncineq1}
		\EXP\big[\CW_2(\widehat{\mu}_{t},\widehat{\rho}_{t})^2\big|\CF_{t-1}\big]
		&=\EXP\Big[\CW_2\big(\widehat{\rho}_{t}|_{\CX_{\widehat{R}_t}},\widehat{\rho}_{t}\big)^2\Big|\CF_{t-1}\Big]\le \frac{1}{4}\beta^{t}.
	\end{align}
	Secondly, observe that Line~\ref{alglin: concrete-initialize-radius} and Line~\ref{alglin: concrete-radius} guarantee 
	$\widehat{R}_{t-1}\ge\overline{r}_2\big(\widehat{\rho}_{t-1},\nu_1,\ldots,\nu_K,\frac{1}{4}\beta^{t}\big)$
	$\PROB$-almost surely,
	Line~\ref{alglin: concrete-pushforward} sets $\widehat{\rho}_{t}=\bar{T}_{t}\sharp\widehat{\rho}_{t-1}$,
	and that 
	Line~\ref{alglin: concrete-truncation} sets $\widehat{\mu}_{t-1}=\widehat{\rho}_{t-1}|_{\CX_{\widehat{R}_{t-1}}}$.
	Thus, 
	applying 
	the property of
	$\overline{r}_2(\,\cdot\,,\ldots,\cdot\,)$ in
	the condition \ref{setts: Caffarelli-truncation2} to
	$\rho\leftarrow\widehat{\rho}_{t-1}$,
	$r\leftarrow \widehat{R}_{t-1}$,
	$\dot{\mu}_r\leftarrow \widehat{\rho}_{t-1}|_{\CX_{\widehat{R}_{t-1}}}=\widehat{\mu}_{t-1}$,
	$\widehat{T}^{\dot{\mu}_{r},m_k}_{\nu_k,n_k}[\theta_k]\leftarrow \widehat{T}_{t,k}$ $\forall 1\le k\le K$,
	$\dot{T}_r\leftarrow\bar{T}_{t}$,
	$\epsilon\leftarrow\frac{1}{4}\beta^t$
	leads to
	\begin{align}
		\label{eqn: Caffarelli-convergence-proof-truncineq2}
		\EXP\big[\CW_2(\bar{T}_{t}\sharp\widehat{\mu}_{t-1},\widehat{\rho}_{t})^2\big|\CF_{t-1}\big]
		&=\EXP\big[\CW_2(\bar{T}_{t}\sharp\widehat{\mu}_{t-1},\bar{T}_{t}\sharp\widehat{\rho}_{t-1})^2\big|\CF_{t-1}\big]
		\le \frac{1}{4}\beta^{t}.
	\end{align}
	Finally, combining
	(\ref{eqn: Caffarelli-convergence-proof-truncineq1}) and (\ref{eqn: Caffarelli-convergence-proof-truncineq2}) proves that 
	\begin{align*}
		\EXP\big[\CW_2(\bar{T}_{t}\sharp\widehat{\mu}_{t-1},\widehat{\mu}_{t})^2\big|\CF_{t-1}\big]
		&\le 2\EXP\big[\CW_2(\bar{T}_{t}\sharp\widehat{\mu}_{t-1},\widehat{\rho}_{t})^2\big|\CF_{t-1}\big] 
		+ 2\EXP\big[\CW_2(\widehat{\mu}_{t},\widehat{\rho}_{t})^2\big|\CF_{t-1}\big]
		\le \beta^{t} \qquad \forall t\in\N.
	\end{align*}
	The proof is now complete.
\end{proof}

\begin{remark}\label{rmk: radius-samplesize-ge}
	Observe that the convergence guarantee of Algorithm~\ref{algo: concrete}
	in Theorem~\ref{thm: Caffarelli-convergence}
	remains valid as long as one requires
	$\widehat{R}_0\ge \overline{r}_2\big(\widehat{\rho}_0,\nu_1,\ldots,\nu_K,\frac{1}{4}\beta\big)$
	in Line~\ref{alglin: concrete-initialize-radius}, 
	$\widehat{M}_{0,k}\ge \overline{m}(\widehat{\mu}_0,\nu_k,\beta)$, 
	$\widehat{N}_{0,k}\ge \overline{n}(\widehat{\mu}_0,\nu_k,\beta)$
	in Line~\ref{alglin: concrete-initialize-samplesize}, 
	$\widehat{R}_t\ge \overline{r}_1\big(\widehat{\rho}_{t},\frac{1}{4}\beta^{t}\big) \vee \overline{r}_2\big(\widehat{\rho}_{t},\nu_1,\ldots,\nu_K,\frac{1}{4}\beta^{t+1}\big)$
	in Line~\ref{alglin: concrete-radius}, 
	and
	$\widehat{M}_{t,k}\ge\overline{m}(\widehat{\mu}_t,\nu_k,\beta^{t+1})$, 
	$\widehat{N}_{t,k}\ge\overline{n}(\widehat{\mu}_t,\nu_k,\beta^{t+1})$
	in Line~\ref{alglin: concrete-samplesize}. 
\end{remark}

In Algorithm~\ref{algo: concrete}, let us assume that: 
\begin{itemize}
	\item independent random samples from $\nu_1,\ldots,\nu_K$, and $\rho_0$ can be efficiently generated; 
	
	\item the OT map estimator $\widehat{T}^{\mu,m}_{\nu,n}[\theta]$ can be tractably computed and $\widehat{T}^{\mu,m}_{\nu,n}[\theta](\BIx)$ can be tractably evaluated at any point $\BIx\in\R^d$; 
	
	\item checking whether an arbitrary point $\BIx\in\R^d$ belongs to $\CX_{r}$ is computationally tractable for all $r\in\N$.
\end{itemize}
Under these assumptions, Algorithm~\ref{algo: concrete} is computationally tractable. 
Indeed, for $t\in\N$, a random sample from $\widehat{\mu}_{t-1}$ can be generated in Line~\ref{alglin: concrete-sample1} by \textit{rejection sampling}. 
Specifically, one first generates a random sample $\BIX\in\nobreak\R^d$ from $\rho_0$ and evaluates the composition $\widehat{\BIX}:=\big[\sum_{k=1}^K w_k\widehat{T}_{t-1,k}\big]\circ \cdots \circ \big[\sum_{k=1}^K w_k\widehat{T}_{1,k}\big](\BIX)$.
This sample $\widehat{\BIX}$ is subsequently accepted if $\widehat{\BIX}\in \CX_{\widehat{R}_{t-1}}$. 
Otherwise, this generation process is repeated until the sample $\widehat{\BIX}$ is accepted. 
The computational tractability and complexity of our specific choice of OT map estimator is discussed in Section~\ref{ssec: entropic-estimator}.

\begin{remark}[Distributed implementation of Algorithm~\ref{algo: concrete}]\label{remark: distributed}
	We would like to remark that Algorithm~\ref{algo: concrete} 
	allows for implementation in a distributed and parallel computing environment, 
	which can be appealing in terms of computational efficiency. 
	Suppose that there are $K$ agents where $K\in\N$ is large, 
	and the $k$-th agent has local access to a probability measure $\nu_k\in\CM^q(\R^d)$, for $k=1,\ldots,K$.
	The $\CW_2$-barycenter problem with input measures $\nu_1,\ldots,\nu_K$ and weights $w_1,\ldots,w_K$ is to be solved by a central coordinator who can communicate with the $K$ agents.
	In each iteration~$t\in\N$, the coordinator first generates independent samples 
	$\{\BIX_{t,k,i}\}_{i=1:\widehat{M}_{t-1,k},\,k=1:K}$ from $\widehat{\mu}_{t-1}$ and 
	releases the subcollection of samples $\{\BIX_{t,k,i}\}_{i=1:\widehat{M}_{t-1,k}}$ to agent~$k$ 
	(Line~\ref{alglin: concrete-sample1}), for $k=1,\ldots,K$.
	Each agent~$k$ then generates independent samples $\{\BIY_{t,k,j}\}_{j=1:\widehat{N}_{t-1,k}}$ from $\nu_k$ 
	(Line~\ref{alglin: concrete-sample2}) 
	and uses $\{\BIX_{t,k,i}\}_{i=1:\widehat{M}_{t-1,k}}$ and 
	$\{\BIY_{t,k,j}\}_{j=1:\widehat{N}_{t-1,k}}$ 
	to compute an admissible OT map estimator $\widehat{T}_{t,k}$ (Line~\ref{alglin: concrete-estimator}).
	Subsequently, in order for the coordinator to generate independent samples from $\widehat{\mu}_{t}$, 
	a large number of independent samples 
	$\{\widetilde{\BIX}_{t,i}\}_{i=1:M}$ from $\widehat{\rho}_{t-1}$ are generated and sent to all $K$~agents.
	Upon receiving $\{\widetilde{\BIX}_{t,i}\}_{i=1:M}$ from the coordinator, agent~$k$ evaluates $\big\{\widehat{T}_{t,k}(\widetilde{\BIX}_{t,i})\big\}_{i=1:M}$ and sends it back to the coordinator.
	The coordinator can then generate independent samples from $\widehat{\mu}_{t}$ using the weighted sums $\big\{\sum_{k=1}^Kw_k\widehat{T}_{t,k}(\widetilde{\BIX}_{t,i})\big\}_{i=1:M}$ 
	(Line~\ref{alglin: concrete-pushforward}) followed by the rejection sampling procedure described above.
\end{remark}

\begin{remark}
	In establishing convergence rates for the OT map estimators $(\widehat{T}_{t,k})_{k=1:K}$, our analysis 
	follows existing studies such as \citep{hutter2021minimax,muzellec2021near,ghosal2022multivariate,manole2024plugin,pooladian2021entropic} and
	requires global quantitative smoothness and global Hessian bounds for the Brenier potentials $(\varphi^{\widehat{\mu}_{t-1}}_{\nu_k})_{k=1:K}$ at each iteration of Algorithm~\ref{algo: concrete}. 
	We remark that there exist broader regularity results in the literature that establish interior, partial, Sobolev, or scale-dependent regularity properties of the Brenier potential $\varphi^{\mu}_{\nu}$ under weaker assumptions on $\mu$ and $\nu$; see, e.g., \citep{caffarelli1990interior, caffarelli1992regularity, figalli2010partial, philippis2015partial, goldman2020variational, philippis2013note, philippis2013regularity, chen2017partial,jabin2026holder}. 
	However, these alternative regularity results do not provide the global quantitative regularity guarantees required by our approach. 
	Developing convergence rates for OT map estimators under broader regularity regimes beyond the Caffarelli-type setting constitutes a promising avenue for future research.
\end{remark}

\subsection{Modified entropic OT map estimator}\label{ssec: entropic-estimator}

As stated in Setting~\ref{sett: Caffarelli}
and Theorem~\ref{thm: Caffarelli-convergence}, 
the convergence of Algorithm~\ref{algo: concrete} depends crucially on 
the measurability, shape, growth, and consistency properties of
the OT map estimator $\widehat{T}^{\mu,m}_{\nu,n}[\theta]$ required by Assumption~\ref{asp: OTmap-estimator}.
In this subsection, we consider two 
Caffarelli-type
admissible compactly supported probability measures 
and present a concrete example of an OT map estimator that satisfies Assumption~\ref{asp: OTmap-estimator}, 
which is a modified version of the entropic OT map estimator of \citet*{pooladian2021entropic}. 
This estimator is explicitly constructed via numerically solving an entropic OT problem,
followed by barycentric projection \citep[Definition~5.4.2]{ambrosio2008gradient} as well as a strong convexity modification to ensure the desired curvature property.

Before we present our modified entropic OT map estimator,
let us first recall the definition of the entropic OT problem
and Sinkhorn's algorithm for numerically solving it.
Let $\mu,\nu\in\CP_2(\R^d)$ 
have compact supports
and let 
$\mu\otimes\nu\in\CP_2(\R^d\times\R^d)$ denote the product measure of $\mu$ and $\nu$.
The entropic OT (EOT) problem
between $\mu$ and $\nu$ with respect to the cost function 
$\R^d\times\R^d\ni(\BIx,\BIy)\mapsto -\langle\BIx,\BIy\rangle\in\R$ 
and regularization parameter $\gamma>0$ is given by
\begin{align}
	\inf_{\pi\in\Pi(\mu,\nu)} \int_{\R^d\times\R^d}-\langle\BIx,\BIy\rangle\DIFFM{\pi}{\DIFF\BIx,\DIFF\BIy}+\gamma\mathrm{KL}(\pi|\mu\otimes\nu),
	\label{eqn: EOT}
\end{align}
where $\mathrm{KL}(\pi|\mu\otimes\nu)$ denotes the Kullback--Leibler divergence
between $\pi$ and $\mu\otimes\nu$ defined by
\begin{align*}
	\mathrm{KL}(\pi|\mu\otimes\nu):=\begin{cases}
		\int_{\R^d\times\R^d}\log\big(\frac{\DIFF\pi}{\DIFF\mu\otimes\nu}\big)\DIFFX{\pi} & \text{if }\pi \ll \mu\otimes\nu \\
		\infty & \text{if }\pi \centernot\ll \mu\otimes\nu
	\end{cases} \qquad \forall \pi\in\CP_2(\R^d\times\R^d).
\end{align*}
Note that the EOT problem in (\ref{eqn: EOT}) is equivalent to the EOT problem 
with the squared-Euclidean cost function 
$\R^d\times\nobreak\R^d\ni(\BIx,\BIy)\mapsto \frac{1}{2}\|\BIx-\nobreak\BIy\|^2\in\R$ up to a shift by a constant.
In particular, they yield the same optimal solutions.
The above EOT problem is an approximation of the unregularized OT problem (i.e., when $\gamma=0$),
and admits the following dual problem (see, e.g., \citep[equation~(8)]{pooladian2021entropic}):
\begin{align}
	\label{eqn: EOT-dual}
	\begin{split}
	\sup_{f\in\CL^1(\mu),\,g\in\CL^1(\nu)} &\int_{\R^d}f\DIFFX{\mu} + \int_{\R^d}g\DIFFX{\nu} + \gamma - \gamma \int_{\R^d\times\R^d}\exp\Big({\textstyle\frac{f(\BIx)+g(\BIy)+\langle\BIx,\BIy\rangle}{\gamma}}\Big)\DIFFM{\mu\otimes\nu}{\DIFF\BIx,\DIFF\BIy}.
	\end{split}
\end{align}
One can show that (\ref{eqn: EOT-dual}) admits maximizers and any maximizer 
$(f^\star,g^\star)$ of (\ref{eqn: EOT-dual}) satisfies the following system of equations:
\begin{align*}
	\begin{cases}
		\displaystyle f(\BIx) = -\gamma\log\bigg(\int_{\R^d}\exp\Big({\textstyle\frac{g(\BIy)+\langle\BIx,\BIy\rangle}{\gamma}}\Big)\DIFFM{\nu}{\DIFF\BIy}\bigg) & \text{for }\mu\text{-a.e.}\; \BIx\in\R^d, 
		\vspace{6pt}\\
		\displaystyle g(\BIy) = -\gamma\log\bigg(\int_{\R^d}\exp\Big({\textstyle\frac{f(\BIx)+\langle\BIx,\BIy\rangle}{\gamma}}\Big)\DIFFM{\mu}{\DIFF\BIx}\bigg) & \text{for }\nu\text{-a.e.}\; \BIy\in\R^d.
	\end{cases}
\end{align*}
Inspired by the optimality conditions above, 
the celebrated Sinkhorn's algorithm \citep{sinkhorn1964relationship,cuturi2013sinkhorn} 
initializes $\widehat{g}^{(\gamma,0)}:=0$ and computes
$\big(\widehat{f}^{(\gamma,l)}\big)_{l\in\N}$, $\big(\widehat{g}^{(\gamma,l)}\big)_{l\in\N}$ 
via the following Sinkhorn iteration:
\begin{align}
	\label{eqn: Sinkhorn-general}
	\begin{cases}
		\displaystyle \widehat{f}^{(\gamma,l)}(\BIx) := -\gamma\log\bigg(\int_{\R^d}\exp\Big({\textstyle\frac{\widehat{g}^{(\gamma,l-1)}(\BIy)+\langle\BIx,\BIy\rangle}{\gamma}}\Big)\DIFFM{\nu}{\DIFF\BIy}\bigg) & \forall\BIx\in\R^d
		\vspace{6pt}\\
		\displaystyle \widehat{g}^{(\gamma,l)}(\BIy) := -\gamma\log\bigg(\int_{\R^d}\exp\Big({\textstyle\frac{\widehat{f}^{(\gamma,l)}(\BIx)+\langle\BIx,\BIy\rangle}{\gamma}}\Big)\DIFFM{\mu}{\DIFF\BIx}\bigg) & \forall \BIy\in\R^d
	\end{cases} \qquad \forall l\in\N.
\end{align}
$\big(\widehat{f}^{(\gamma,l)}\big)_{l\in\N}$ and $\big(\widehat{g}^{(\gamma,l)}\big)_{l\in\N}$ are guaranteed to converge to a maximizer $(f^\star,g^\star)$ of (\ref{eqn: EOT-dual}) under mild conditions; see, e.g., \citep{chizat2026sharper, dvurechensky2018computational, franklin1989on, ghosal2024on}.

When 
$\mu=\frac{1}{m}\sum_{i=1}^{m}\delta_{\BIX_i}$ and 
$\nu=\frac{1}{n}\sum_{j=1}^{n}\delta_{\BIY_j}$ 
are 
empirical measures with sample sizes
$m,n\in\nobreak\N$, 
and samples $(\BIX_i)_{i=1:m}\subset\nobreak\R^d$, $(\BIY_j)_{j=1:n}\subset\nobreak\R^d$,
(\ref{eqn: EOT-dual}) can be parametrized into the following maximization problem over scalar-valued variables $(f_i)_{i=1:m}$ and $(g_j)_{j=1:n}$:
\begin{align}
	\label{eqn: EOT-dual-discrete}
	\maximize_{(f_i),\,(g_j)} \quad & \Bigg(\frac{1}{m}\sum_{i=1}^{m}f_i\Bigg)
	+ \Bigg(\frac{1}{n}\sum_{j=1}^{n}g_j\Bigg) 
	+ \gamma
	- \Bigg(\frac{\gamma}{mn}\sum_{i=1}^{m}\sum_{j=1}^{n}\exp\Big({\textstyle\frac{f_i+g_j+\langle\BIX_i,\BIY_j\rangle}{\gamma}}\Big)\Bigg),
\end{align}
and Sinkhorn's algorithm operates on scalar-valued variables 
$\big(\widehat{f}^{(\gamma,l)}_i\big)_{i=1:m,\,l\in\N}$,
$\big(\widehat{g}^{(\gamma,l)}_j\big)_{j=1:n,\,l\in\N}$
through the initialization $\widehat{g}^{(\gamma,0)}_j:=0$ $\forall 1\le j\le n$
and the Sinkhorn iteration:
\begin{align}
	\label{eqn: Sinkhorn-empirical}
	\begin{cases}
		\widehat{f}^{(\gamma,l)}_i := -\gamma\log\Big(\frac{1}{n}\sum_{j=1}^{n}\exp\Big({\textstyle\frac{\widehat{g}^{(\gamma,l-1)}_j+\langle\BIX_i,\BIY_j\rangle}{\gamma}}\Big)\Big) & \forall 1\le i\le m
		\vspace{4pt}\\
		\widehat{g}^{(\gamma,l)}_j := -\gamma\log\Big(\frac{1}{m}\sum_{i=1}^{m}\exp\Big({\textstyle\frac{\widehat{f}^{(\gamma,l)}_i+\langle\BIX_i,\BIY_j\rangle}{\gamma}}\Big)\Big) & \forall 1\le j\le n
	\end{cases} \qquad \forall l\in\N.
\end{align}
Let us define the matrix $\BK\in\R^{m\times n}$
with entries $\BK_{i,j}:=\exp\Big(\frac{\langle\BIX_i,\BIY_j\rangle}{\gamma}\Big)$
$\forall 1\le i\le m$, $\forall 1\le\nobreak j\le\nobreak n$.
After vectorizing
the scalar-valued variables 
$\big(\widehat{f}^{(\gamma,l)}_i\big)_{i=1:m}$,
$\big(\widehat{g}^{(\gamma,l)}_j\big)_{j=1:n}$,
into
$\widehat{\BIf}^{(\gamma,l)}\leftarrow\big(\widehat{f}^{(\gamma,l)}_1,\ldots,\widehat{f}^{(\gamma,l)}_m\big)^\TRANSP\in\R^m$,
$\widehat{\BIg}^{(\gamma,l)}\leftarrow\big(\widehat{g}^{(\gamma,l)}_1,\ldots,\widehat{g}^{(\gamma,l)}_n\big)^\TRANSP\in\R^n$
and then applying the reparametrizations 
$\widehat{\BIu}^{(\gamma,l)}\leftarrow \exp\Big(\frac{\widehat{\BIf}^{(\gamma,l)}}{\gamma}\Big)$,
$\widehat{\BIv}^{(\gamma,l)}\leftarrow \exp\Big(\frac{\widehat{\BIg}^{(\gamma,l)}}{\gamma}\Big)$
for each $l\in\N$
(here $\exp(\cdot)$ denotes entry-wise exponentiation of a vector),
Sinkhorn's algorithm can be expressed via the following reparametrized and vectorized form,
where one begins with
$\widehat{\BIv}^{(\gamma,0)}:=\vecone_{n}$ 
and performs the following iterative updates of
$(\widehat{\BIu}^{(\gamma,l)},\widehat{\BIv}^{(\gamma,l)})_{l\in\N}$:
\begin{align}
	\label{eqn: Sinkhorn-matrix-scaling}
	\begin{cases}
		\widehat{\BIu}^{(\gamma,l)}:= \diag\big(\BK\widehat{\BIv}^{(\gamma,l-1)}\big)^{-1}n\vecone_{m} \\
		\widehat{\BIv}^{(\gamma,l)}\hspace{0.8pt}:= \diag\big(\BK^\TRANSP\widehat{\BIu}^{(\gamma,l)}\big)^{-1}m\vecone_{n}
	\end{cases} \qquad \forall l\in\N.
\end{align}
The iteration (\ref{eqn: Sinkhorn-matrix-scaling}) is also known as the 
iterative proportional fitting procedure and the RAS algorithm;
see, e.g., \citep[Remark~4.5]{peyre2019computational} for historical remarks.

Next, let us present in Proposition~\ref{prop: OTmap-estimator-entropic} our \textit{modified entropic OT map estimator},
which is based on performing $\overline{l}\in\N$ Sinkhorn iterations 
(\ref{eqn: Sinkhorn-empirical})
followed by the barycentric projection step as in the entropic OT map estimator of \citet{pooladian2021entropic}.
Compared to \citep{pooladian2021entropic},
our modification lies in the addition of an extra term 
$T_{\STRONGLYCONVEX}[\psi](\cdot)$
to the estimator to guarantee the strong convexity condition in Assumption~\ref{asp: OTmap-estimator}\ref{asps: OTmap-estimator-shape}.
Since $T_{\STRONGLYCONVEX}[\psi](\cdot)$ is only non-zero on 
$\R^d\setminus \bar{B}(\veczero_d,\psi)$ for each $\psi>0$, it
does not affect the value of the estimator in $\support(\mu)$
when $\psi>0$ is sufficiently large such that $\support(\mu)\subseteq \bar{B}(\veczero_d,\psi)$.
For the sake of notational simplicity,
we will omit $\mu,\nu,m,n$ in the estimator notations 
in our subsequent analyses
and denote our modified entropic OT map estimator by $\widehat{T}_{\ENTROPIC}[\theta]$
with $\theta=(\gamma,\overline{l},\psi)$,
where $\gamma>0$ is the regularization parameter, 
$\overline{l}\in\N$ is the number of Sinkhorn iterations,
and $\psi>0$ is the radius that determines the strong convexity modification term $T_{\STRONGLYCONVEX}[\psi](\cdot)$.
Nonetheless, $m$ and $n$ will always be understood as the numbers of samples from $\mu$ and $\nu$, respectively.

\begin{proposition}[Modified entropic OT map estimator]\label{prop: OTmap-estimator-entropic}
	Let $(\Omega,\CF,\PROB)$ be a probability space,
	let $q\in\N$, 
	let $\mu,\nu\in\CM^q(\R^d)$,
	and let 
	$R_{\mu}:=\inf\big\{r>0:\support(\mu)\subseteq\bar{B}(\veczero_d,r)\big\}$,
	$R_{\nu}:=\inf\big\{r>0:\support(\nu)\subseteq\bar{B}(\veczero_d,r)\big\}$.
	For any $m\in\N$ and $n\in\N$, 
	let $\BIX_1,\ldots,\BIX_m,\BIY_1,\ldots,\BIY_n:\Omega\to\R^d$ be independent random variables such that 
	the law of each $\BIX_i$ is $\mu$
	and the law of each $\BIY_j$ is $\nu$,
	i.e., $\BIX_i\sharp\PROB=\mu$ $\forall 1\le i\le m$,
	$\BIY_j\sharp\PROB=\nu$ $\forall 1\le j\le n$.
	Let $\Theta:=(0,\infty)\times\N\times(0,\infty)$.
	For any $\theta=(\gamma,\overline{l},\psi)\in\Theta$,
	let 
	$\widehat{T}_{\ENTROPIC}[\theta]:\R^d\to\R^d$ be constructed through the following three steps.
	\begin{enumerate}[beginpenalty=10000,label=(\arabic*)]
		\item \textbf{Sinkhorn step.}
		One initializes $\widehat{g}^{(\gamma,0)}_j:=0$ $\forall 1\le j\le n$
		and then iteratively computes 
		$\big(\widehat{f}^{(\gamma,l)}_i\big)_{i=1:m}$,
		$\big(\widehat{g}^{(\gamma,l)}_j\big)_{j=1:n}$
		for $l=1,2,\ldots,\overline{l}$ via the Sinkhorn iteration (\ref{eqn: Sinkhorn-empirical}) (or via the reparametrized and vectorized form in (\ref{eqn: Sinkhorn-matrix-scaling})).
		
		\item \textbf{Barycentric projection step.} 
		Let $\widetilde{T}_{\ENTROPIC}[\gamma,\overline{l}]:\R^d\to\R^d$ be defined as follows:
		\begin{align}
			\widetilde{T}_{\ENTROPIC}[\gamma,\overline{l}](\BIx)&:= \frac{\sum_{j=1}^n \exp\Big(\textstyle\frac{\widehat{g}^{(\gamma,\overline{l})}_j+\langle\BIY_j,\BIx\rangle}{\gamma}\Big) \BIY_j}{\sum_{j=1}^n \exp\Big(\textstyle\frac{\widehat{g}^{(\gamma,\overline{l})}_j+\langle\BIY_j,\BIx\rangle}{\gamma}\Big)} \qquad \forall \BIx\in\R^d.
			\label{eqn: OTmap-estimator-baryproj}
		\end{align}
		In particular,
		$\widetilde{T}_{\ENTROPIC}[\gamma,\overline{l}]$
		is the entropic OT map estimator of \citet{pooladian2021entropic}.

		\item \textbf{Strong convexity modification step.}
		Let $T_{\STRONGLYCONVEX}[\psi]:\R^d\to\R^d$ be defined as follows:
		\begin{align*}
			T_{\STRONGLYCONVEX}[\psi](\BIx):=\begin{cases}
				\exp\Big({-\frac{1}{\|\BIx\|^2-\psi^2}}\Big)\BIx & \forall \BIx\in\R^d \setminus \bar{B}(\veczero_d, \psi), \\
				\veczero_d & \forall \BIx\in \bar{B}(\veczero_d, \psi),
			\end{cases}
		\end{align*}
		and then let $\widehat{T}_{\ENTROPIC}[\gamma,\overline{l},\psi]:\R^d\to\R^d$
		be defined as follows:
		\begin{align*}
			\widehat{T}_{\ENTROPIC}[\gamma,\overline{l},\psi](\BIx):=\widetilde{T}_{\ENTROPIC}[\gamma,\overline{l}](\BIx) + T_{\STRONGLYCONVEX}[\psi](\BIx) \qquad \forall \BIx\in\R^d.
		\end{align*}
	\end{enumerate}
	Then, $\widehat{T}_{\ENTROPIC}[\theta]$ satisfies the following statements.
	\begin{enumerate}[label=(\roman*), beginpenalty=10000]
		\item\label{props: OTmap-estimator-entropic-Borel}%
		\textbf{Measurability:} 
		$\widehat{T}_{\ENTROPIC}[\theta]$ belongs to $\CC_{\LIN}(\R^d,\R^d)$ and has a Borel dependency on $(\BIX_1,\ldots,\BIX_m,\allowbreak\BIY_1,\ldots,\BIY_n,\theta)$,
		that is, $\widehat{T}_{\ENTROPIC}[\theta]$ satisfies Assumption~\ref{asp: OTmap-estimator}\ref{asps: OTmap-estimator-measurability}.

		\item\label{props: OTmap-estimator-entropic-shape}%
		\textbf{Shape:} 
		define
		\begin{align*}
			\hspace{37pt}\underline{\lambda}&:=\Bigg[\frac{e_{\MINSUB}(\widehat{\BSigma})}{\gamma}\exp\bigg({-\frac{2R_{\nu}(R_{\mu}+2\psi)}{\gamma}}\bigg)\Bigg] \wedge \exp\bigg({-\frac{1}{3\psi^2}}\bigg),
		\end{align*}
		where
		$\widehat{\BSigma}:=
		\big(\frac{1}{n}\sum_{j=1}^n\BIY_j\BIY_j^\TRANSP\big) 
		- \big({\frac{1}{n}\sum_{j=1}^n}\BIY_j\big)\big({\frac{1}{n}\sum_{j=1}^n}\BIY_j\big)^\TRANSP$.
		Then, $\underline{\lambda}$ 
		has a Borel dependency on 
		$(\mu,\nu,m,n,\allowbreak\BIX_1,\ldots,\BIX_m,\allowbreak\BIY_1,\ldots,\BIY_n,\gamma,\overline{l},\psi)$.
		Moreover, for any $\theta=(\gamma,\overline{l},\psi)\in\Theta$,
		whenever $n\ge \nobreak d+\nobreak1$,
		it holds $\PROB$-almost surely that 
		$\underline{\lambda}>0$ and there exists 
		$\widehat{\varphi}_{\ENTROPIC}[\theta]\in\FC^{\infty}_{\underline{\lambda},\infty}(\R^d)$ with 
		$\nabla\widehat{\varphi}_{\ENTROPIC}[\theta]=\widehat{T}_{\ENTROPIC}[\theta]$.
		In particular, $\widehat{T}_{\ENTROPIC}[\theta]$ satisfies Assumption~\ref{asp: OTmap-estimator}\ref{asps: OTmap-estimator-shape} 
		with respect to 
		$\underline{m}\leftarrow\nobreak 1$,
		$\underline{n}\leftarrow\nobreak d+\nobreak1$,
		$\underline{\lambda}$ as defined above,
		and any $\alpha\in(0,1)$.
	
		\item\label{props: OTmap-estimator-entropic-growth}%
		\textbf{Growth:} for any $\theta\in\Theta$,
		it holds $\PROB$-almost surely that $\big\|\widehat{T}_{\ENTROPIC}[\theta](\BIx)-\widehat{T}_{\ENTROPIC}[\theta](\veczero_d)\big\|^2\le 8R_{\nu}^2 + 2\|\BIx\|^2$ $\forall\BIx\in\R^d$.
		In particular, $\widehat{T}_{\ENTROPIC}[\theta]$ satisfies Assumption~\ref{asp: OTmap-estimator}\ref{asps: OTmap-estimator-growth} with respect to 
		$u_0(\nu)\leftarrow 8R_{\nu}^2$, 
		$u_1(\nu)\leftarrow 2$. 
		
		\item\label{props: OTmap-estimator-entropic-consistency}%
		\textbf{Consistency:} 
		the following bound holds:
		\begin{align}
			\begin{split}
				\hspace{37pt}\EXP\Big[\big\|\widehat{T}_{\ENTROPIC}[\gamma,\overline{l},\psi]-T^{\mu}_{\nu}\big\|_{\CL^2(\mu)}^2\Big] 
				&\le C_{\ENTROPIC}(\mu,\nu) \Big[\gamma^{-\frac{d}{2}}\big(\log(m)m^{-\frac{1}{2}}+\log(n)n^{-\frac{1}{2}}\big) \\
				& \qquad\qquad\qquad\;\; + \gamma^{\frac{\overline{\alpha}(\mu,\nu)}{2}} + \gamma^{-\frac{3}{2}}\overline{l}^{-\frac{1}{2}}\log(m\vee n) \Big] \\
				& \hspace{-15pt} \forall m,n\in\N\cap[2,\infty),\; \forall \gamma\in \big(0,\overline{\gamma}(\mu,\nu)\big],\; \forall \overline{l}\in\N,\; \forall \psi\in [R_{\mu},\infty),
			\end{split}
			\label{eqn: OTmap-estimator-error}
		\end{align}
		where $\overline{\gamma}(\mu,\nu)\in(0,\infty)$, $\overline{\alpha}(\mu,\nu)\in(3,4]$, and $C_{\ENTROPIC}(\mu,\nu)\in(0,\infty)$
		are terms constructed in the proof;
		see equations~(\ref{eqn: OTmap-estimator-entropic-proof-consistency-alpha-def}), (\ref{eqn: OTmap-estimator-entropic-proof-consistency-Centr-def})
		in Appendix~\ref{sapx: proof-entropic-estimator},
		as well as the proofs of \citep[Theorem~4 \& Theorem~7]{pooladian2021entropic}
		for their explicit expressions.
		Let us define 
		\begin{align*}
			\overline{m}(\mu,\nu,\epsilon)&:= \min\Bigg\{m\in\N : \!\begin{tabular}{l} 
				$m\ge \overline{\gamma}(\mu,\nu)^{-(\overline{\alpha}(\mu,\nu)+d)} \vee 2,$ \\
				$\widetilde{m}^{-\frac{\overline{\alpha}(\mu,\nu)}{2(\overline{\alpha}(\mu,\nu)+d)}}\big(2\log(\widetilde{m})+1\big)\le \textstyle\frac{\epsilon}{2C_{\ENTROPIC}(\mu,\nu)}\; \forall \widetilde{m}\ge m$
			\end{tabular}\!\Bigg\} \hspace{-100pt}
			& \forall \epsilon>0,\allowdisplaybreaks\\
			\overline{n}(\mu,\nu,\epsilon)&:= \overline{m}(\mu,\nu,\epsilon)\vee (d+1) 
			& \forall \epsilon>0,\allowdisplaybreaks\\
			\hspace{34pt}\widetilde{\gamma}(\mu,\nu,m,n,\epsilon)&:= (m\wedge n)^{-\frac{1}{\overline{\alpha}(\mu,\nu)+d}}
			&\forall m\in\N,\; \forall n\in\N, \; \forall \epsilon>0,\allowdisplaybreaks\\
			\widetilde{\overline{l}}(\mu,\nu,m,n,\epsilon)&:=\!\bigg\lceil(m\vee n)^{\frac{\overline{\alpha}(\mu,\nu)+3}{\overline{\alpha}(\mu,\nu)+d}}\bigg\rceil
			&\quad \forall m\in\N,\; \forall n\in\N, \; \forall \epsilon>0,\allowdisplaybreaks\\
			\widetilde{\psi}(\mu,\nu,m,n,\epsilon)&:=R_{\mu}
			&\quad \forall m\in\N,\; \forall n\in\N, \; \forall \epsilon>0,\allowdisplaybreaks\\
			\widetilde{\theta}(\mu,\nu,m,n,\epsilon)&:=\big(\widetilde{\gamma}(\mu,\nu,m,n,\epsilon),\,\widetilde{\overline{l}}(\mu,\nu,m,n,\epsilon),\,\widetilde{\psi}(\mu,\nu,m,n,\epsilon)\big)
			&\forall m\in\N,\; \forall n\in\N, \; \forall \epsilon>0.
		\end{align*}
		Then, 
		$\overline{m}(\mu,\nu,\epsilon)$, 
		$\overline{n}(\mu,\nu,\epsilon)$ have Borel dependencies on 
		$(\mu,\nu,\epsilon)$,
		$\widetilde{\theta}(\mu,\nu,m,n,\epsilon)$ has a Borel dependency on 
		$(\mu,\nu,m,n,\epsilon)$,
		and it holds that
		\begin{align}
			\hspace{37pt}\EXP\Big[\big\|\widehat{T}_{\ENTROPIC}[\widetilde{\theta}(\mu,\nu,m,n,\epsilon)]-T^{\mu}_{\nu}\big\|_{\CL^2(\mu)}^2\Big]&\le \epsilon \qquad \forall m\ge \overline{m}(\mu,\nu,\epsilon), \; \forall n\ge \overline{n}(\mu,\nu,\epsilon),\; \forall \epsilon>0.
			\label{eqn: OTmap-estimator-error-control}
		\end{align}
		In particular, $\widehat{T}_{\ENTROPIC}[\theta]$ satisfies Assumption~\ref{asp: OTmap-estimator}\ref{asps: OTmap-estimator-consistency} with respect to 
		$\overline{m}(\mu,\nu,\epsilon)$, 
		$\overline{n}(\mu,\nu,\epsilon)$,
		and 
		$\widetilde{\theta}(\mu,\nu,\allowbreak m,n,\epsilon)$.
	\end{enumerate}
\end{proposition}

\begin{proof}[Proof of Proposition~\ref{prop: OTmap-estimator-entropic}]
	See Appendix~\ref{sapx: proof-entropic-estimator}.
\end{proof}

Since
$T_{\STRONGLYCONVEX}[\psi](\BIx)=\veczero_d$ for $\mu$-almost every $\BIx\in\R^d$ whenever $\support(\mu)\subseteq \bar{B}(\veczero_{d},\psi)$,
$T_{\STRONGLYCONVEX}[\psi]$ 
does not contribute to the approximation error 
$\EXP\Big[\big\|\widehat{T}_{\ENTROPIC}[\gamma,\overline{l},\psi]-T^{\mu}_{\nu}\big\|_{\CL^2(\mu)}^2\Big]$
when $\psi>0$ is sufficiently large.
Thus, we get
$\EXP\Big[\big\|\widehat{T}_{\ENTROPIC}[\gamma,\overline{l},\psi]-T^{\mu}_{\nu}\big\|_{\CL^2(\mu)}^2\Big]
=
\EXP\Big[\big\|\widetilde{T}_{\ENTROPIC}[\gamma,\overline{l}]-T^{\mu}_{\nu}\big\|_{\CL^2(\mu)}^2\Big]$
for sufficiently large $\psi>0$,
and the error bound \eqref{eqn: OTmap-estimator-error} stems from 
the approximation error of the entropic OT map estimator $\widetilde{T}_{\ENTROPIC}[\gamma,\overline{l}]$ in
\citep[Theorem~4 \& Theorem~7]{pooladian2021entropic}.

\begin{remark}
	We remark that the setting of \citep[Theorem~4 \& Theorem~7]{pooladian2021entropic} requires 
	the source and target measures to have the same support,
	i.e., $\support(\mu)=\support(\nu)$.
	However,
	an inspection shows that
	this assumption is in fact not needed for any step in the proofs of \citep[Theorem~4 \& Theorem~7]{pooladian2021entropic},
	and one can derive the bound (\ref{eqn: OTmap-estimator-error})
	by following essentially the same steps as in the proofs of \citep[Theorem~4 \& Theorem~7]{pooladian2021entropic},
	up to replacing $P_{n}$ with $P_{m}$ where $m$ and $n$ are not necessarily equal in the proofs of \citep[Proposition~5 \& Theorem~7]{pooladian2021entropic}.
\end{remark}

Now that Proposition~\ref{prop: OTmap-estimator-entropic} has shown that 
$\widehat{T}_{\ENTROPIC}[\theta]$
satisfies Assumption~\ref{asp: OTmap-estimator}, 
letting $\widehat{T}^{\mu,m}_{\nu,n}[\theta]\leftarrow \widehat{T}_{\ENTROPIC}[\theta]$ 
in Setting~\ref{sett: Caffarelli} and Algorithm~\ref{algo: concrete} 
leads to the convergence properties of the output $(\widehat{\mu}_t)_{t\in\N_0}$ stated in Theorem~\ref{thm: main-convergence}.
This is summarized in the following corollary. 

\begin{corollary}[Convergence of Algorithm~\ref{algo: concrete} with the modified entropic OT map estimator]\label{cor: fixedpoint-convergence-entropic}
	Let the inputs of Algorithm~\ref{algo: concrete} satisfy Setting~\ref{sett: Caffarelli} with respect to $q\in\N$, 
	where 
	the OT map estimator $\widehat{T}^{\mu,m}_{\nu,n}[\cdot]$ is given by
	the modified entropic OT map estimator $\widehat{T}_{\ENTROPIC}[\cdot]$ defined in Proposition~\ref{prop: OTmap-estimator-entropic},
	and let $u_0(\cdot)$, $u_1(\cdot)$,
	$\overline{m}(\,\cdot\,,\cdot\,,\cdot\,)$,
	$\overline{n}(\,\cdot\,,\cdot\,,\cdot\,)$,
	$\widetilde{\theta}(\,\cdot\,,\cdot\,,\cdot\,,\cdot\,,\cdot\,)$
	be defined as in Proposition~\ref{prop: OTmap-estimator-entropic}.
	Let $(\Omega,\CF,\PROB)$ be a probability space on which the random samples in
	Line~\ref{alglin: concrete-sample1} and Line~\ref{alglin: concrete-sample2} of Algorithm~\ref{algo: concrete} are defined,
	and let $(\CF_t)_{t\in\N_0}$ be defined by (\ref{eqn: filtration}).
	Moreover, let $\bar{\mu}$ denote the unique \mbox{$\CW_2$-barycenter} of $\nu_1,\ldots,\nu_K$ with weights $w_1,\ldots,w_K$.
	Then, the conclusions of statements~\ref{thms: main-convergence-tight-as}--\ref{thms: main-convergence-PL-var} 
	in Theorem~\ref{thm: main-convergence}
	hold with respect to the output $(\widehat{\mu}_t)_{t\in\N_0}$ of Algorithm~\ref{algo: concrete}.
\end{corollary}

In the following, let us analyze the computational complexity of Algorithm~\ref{algo: concrete} when equipped with the modified entropic OT map estimator $\widehat{T}_{\ENTROPIC}[\cdot]$.
Let us denote 
$\widehat{\ITheta}_{t-1,k}=\big(\widehat{\IGamma}_{t-1,k},\widehat{\overline{L}}_{t-1,k},\widehat{\IPsi}_{t-1,k}\big)$ for $k=1,\ldots,K$ in each iteration~$t$ of Algorithm~\ref{algo: concrete}.
Line~\ref{alglin: concrete-estimator}
carries out the Sinkhorn step of the estimator in Proposition~\ref{prop: OTmap-estimator-entropic},
which consists of 
the computation of the pair-wise costs 
$\big({-\langle\BIX_{t,k,i},\BIY_{t,k,j}\rangle}\big)_{i=1:\widehat{M}_{t-1,k},\,j=1:\widehat{N}_{t-1,k}}$
followed by
$\widehat{\overline{L}}_{t-1,k}$~Sinkhorn iterations (\ref{eqn: Sinkhorn-empirical})
with respect to 
$m\leftarrow\widehat{M}_{t-1,k}$,
$n\leftarrow\widehat{N}_{t-1,k}$,
$(\BIX_i)_{i=1:m}\leftarrow (\BIX_{t,k,i})_{i=1:\widehat{M}_{t-1,k}}$,
$(\BIY_j)_{j=1:n}\leftarrow (\BIY_{t,k,j})_{j=1:\widehat{N}_{t-1,k}}$.
Consequently, the computational complexity of Line~\ref{alglin: concrete-estimator} is 
$O\big(\widehat{M}_{t-1,k}\widehat{N}_{t-1,k}(d+\widehat{\overline{L}}_{t-1,k})\big)$.
Line~\ref{alglin: concrete-estimator} then stores
the computed values of 
$\Big(\widehat{g}^{(\widehat{\IGamma}_{t-1,k},\widehat{\overline{L}}_{t-1,k})}_{j}\Big)_{j=1:\widehat{N}_{t-1,k}}$
along with the samples $(\BIY_{t,k,j})_{j=1:\widehat{N}_{t-1,k}}$,
which will be subsequently used whenever the estimated OT map
$\widehat{T}_{t,k}$ is evaluated.
As discussed in Section~\ref{ssec: Caffarelli-convergence-proof}, 
the generation of samples $\{\BIX_{t,k,i}\}_{i=1:\widehat{M}_{t-1,k}}$ from $\widehat{\mu}_{t-1}$ in
Line~\ref{alglin: concrete-sample1}
can be implemented by first generating independent samples 
from the initial probability measure $\rho_0$,
followed by evaluating the function
$\big[\sum_{k=1}^K w_k\widehat{T}_{t-1,k}\big]\circ \cdots \circ \big[\sum_{k=1}^K w_k\widehat{T}_{1,k}\big](\BIX)$ 
for each sample $\BIX$,
and subsequently rejecting those samples that do not belong to $\CX_{\widehat{R}_{t-1}}$.
Let $S_{\rho_0}$ denote the computation cost of generating each random sample from $\rho_0$.
Each evaluation of each estimated OT map 
$\widehat{T}_{s,k}$ where $s\in\{1,\ldots,t-1\}$ is 
done via the barycentric projection step 
and the strong convexity modification step in Proposition~\ref{prop: OTmap-estimator-entropic},
which incur a computational cost of
$O(\widehat{N}_{s-1,k}d)$.
Thus, the computational complexity of generating each sample from 
$\widehat{\rho}_{t-1}$ is 
$O\Big(S_{\rho_0}+d\sum_{s=1}^{t-1}\sum_{k=1}^{K}\widehat{N}_{s-1,k}\Big)$.
Since the average proportion of samples accepted in the rejection sampling is 
$\widehat{\rho}_{t-1}(\CX_{\widehat{R}_{t-1}})$,
the average case computational complexity of 
Line~\ref{alglin: concrete-sample1} is 
$O\Big(\frac{\widehat{M}_{t-1,k}}{\widehat{\rho}_{t-1}(\CX_{\widehat{R}_{t-1}})}\Big(S_{\rho_0}+d\sum_{s=1}^{t-1}\sum_{k'=1}^{K}\widehat{N}_{s-1,k'}\Big)\Big)$.

If 
$S_{\rho_0}$ is polynomial in the dimension~$d$ (e.g., $S_{\rho_0}=O(d^2)$ if $\rho_0$ is a $d$-dimensional Gaussian measure)
and if the truncation indices $(\widehat{R}_{t})_{t\in\N_0}$ 
are large enough such that $\widehat{\rho}_{t}(\CX_{\widehat{R}_{t}})$ is close to~1 for all $t\in\N_0$,
then the computational complexities above are polynomial in the dimension~$d$ 
for fixed
sample sizes $(\widehat{M}_{t,k},\widehat{N}_{t,k})_{k=1:K,\,t\in\N_0}$
and fixed hyperparameters 
$(\widehat{\IGamma}_{t,k},\widehat{\overline{L}}_{t,k},\widehat{\IPsi}_{t,k})_{k=1:K,\,t\in\N_0}$.
Nevertheless,
we remark that the sample sizes required to control the estimation error 
of the modified entropic OT map estimator 
in Proposition~\ref{prop: OTmap-estimator-entropic}\ref{props: OTmap-estimator-entropic-consistency}
scale exponentially in~$d$.
This is a consequence of 
the estimation error of the entropic OT map estimator 
in \citep[Theorem~4 \& Theorem~7]{pooladian2021entropic},
and is fundamentally constrained by
the statistical estimation rate of OT maps 
between
general non-parametric probability measures;
see, e.g., the lower bound for the estimation error in \citep[Theorem~6]{hutter2021minimax}.

Finally, we remark that the terms 
$\overline{r}_1(\rho,\epsilon)$,
$\overline{r}_2(\rho,\nu_1,\ldots,\nu_K,\epsilon)$,
$\overline{m}(\mu,\nu,\epsilon)$, 
$\overline{n}(\mu,\nu,\epsilon)$,
$\widetilde{\theta}(\mu,\nu,m,n,\epsilon)$
in Setting~\ref{sett: Caffarelli}, 
though admitting explicit characterizations, 
rely on certain characteristics of the probability measures therein that are a priori unknown for computation.
However, we provide
in Section~\ref{ssec: numerics-additional-comments}
practical guidelines for implementing our algorithm,
including practical choices of these terms
that have been shown to be effective in our numerical experiments.

\begin{remark}[An alternative choice of the admissible OT map estimator]\label{rmk: alternative-estimators}
Apart from the modified entropic OT map estimator introduced in this subsection, 
an alternative choice of admissible OT map estimator which is also
computationally tractable is the convex least squares estimator of \citet*[Proposition~16]{manole2024plugin}, although it needs to be appropriately modified to possess the strong convexity and differentiability properties in Assumption~\ref{asp: OTmap-estimator}\ref{asps: OTmap-estimator-shape} (e.g.,\@ by imposing shape constraints).
However, in this paper, we choose to focus on the entropic OT map estimator due to its superior computational efficiency. 
\end{remark}

\begin{remark}[Sensitivity of Algorithm~\ref{algo: concrete} to suboptimal Karcher means]
	As discussed in Section~\ref{ssec: preliminary-OT} and in Section~\ref{ssec: main-convergence-conditions}, there may exist more than one fixed-point of the $G$-operator, as the barycenter functional~$V$ is not geodesically convex in general and may attain multiple Karcher means (i.e., stationary points). 
	Therefore, 
	it is theoretically possible for
	Algorithm~\ref{algo: concrete} to be stuck around a suboptimal Karcher mean indefinitely, 
	and thus fail to converge to the underlying $\CW_2$-barycenter.
	To examine how sensitive Algorithm~\ref{algo: concrete} is to Karcher means other than the $\CW_2$-barycenter, we consider a simple problem instance with multiple Karcher means that was provided by \citet[Example~2.2]{backhoff2025stochastic}. 
	Specifically, we initialize Algorithm~\ref{algo: concrete} around a suboptimal Karcher mean and detect whether (samples of) the measures computed in the iterations of Algorithm~\ref{algo: concrete} move towards the optimal Karcher mean that is the $\CW_2$-barycenter. 
	We have empirically observed that the algorithm exhibits a degree of robustness: when provided with appropriately scaled regularization parameters, the iterates tend to escape suboptimal fixed-points and converge towards the barycenter in a handful of iterations.
	We remark that our observation is consistent with prior studies revealing the important role of regularization in overcoming suboptimal solutions; see, for example, \citep[Section~6]{chizat2025doubly}.
	Details of this numerical instance are provided in Appendix~\ref{apx: multiple-karcher-means}.
	Again, we emphasize that this non-uniqueness issue reflects an inherent theoretical limitation shared by all methods based on the fixed-point iterative scheme in \cite{alvarez2016fixed} or Wasserstein gradient flows for computing Wasserstein barycenters.
\end{remark}

\section{Synthetic problem instance generation: a novel algorithm}\label{sec: generation-instance}
To enable quantitative evaluation and comparison across Wasserstein barycenter algorithms, we introduce in this section a novel algorithm for generating synthetic problem instances with general continuous non-parametric input measures
whose ground-truth \mbox{$\CW_2$-barycenter} is (approximately) known.
Our construction is partially inspired by the properties of entropic OT map estimators studied in Section~\ref{ssec: entropic-estimator}.
Compared to existing instance-generating schemes in the literature, our algorithm offers desirable properties in terms of both computational efficiency and the quality of the generated instances.
The proposed approach is used in one of the numerical experiments presented in Section~\ref{sec: numerics}. 

\subsection{Motivation}\label{ssec: generation-instance-motivation}
A major challenge when quantitatively evaluating any numerical method that approximates the \mbox{$\CW_2$-barycenter} is the lack of access to the ground-truth barycenter.
In many existing studies, the input measures used in empirical experiments are restricted to specific parametric families of distributions, most notably elliptical distributions (see, e.g., \citep[Definition~3.26]{mcneil2015quantitative}), 
for which the ground-truth $\CW_2$-barycenter can be efficiently approximated to very high precision \citep{alvarez2016fixed}.
To evaluate approximate $\CW_2$-barycenters computed by various algorithms for non-elliptical measures, 
a common practice is to conduct experiments on low-dimensional imaging datasets and visually assess images generated from the approximate $\CW_2$-barycenters. 
However, such approaches rely purely on human judgement and lack explicit numerical evidence without access to the ground-truth barycenters. 

It is therefore important to seek problem instances with non-parametric free-support input measures where the ground-truth $\CW_2$-barycenter is a priori known or approximately known, 
such that quantitative inspections of empirical approximation errors can be conducted.
To this end, \citet{korotin2022wasserstein} proposed a method of generating input measures using an initial measure which ends up being exactly the $\CW_2$-barycenter, via exploiting the convexity and congruency properties inherited by the Brenier potentials (see \citep[Section~5]{korotin2022wasserstein}). Although their method serves as a reasonable benchmark in many computer vision and imaging applications, the conjugacy operation therein incurs computational burdens, and the constructed congruent functions suffer from limited curvatures. As a consequence, the resulting input measures exhibit little differences in distributional features and resemble pushforwards of the initial measure under certain close-to-affine transformations, which hinders the generalizability of the generated problem instance.

As such, our goal is to provide a flexible algorithm for synthetically generating problem instances with the following desirable properties:  
\begin{enumerate}[label = (P\arabic*),beginpenalty=10000]
	\item\label{property: ground-truth}the ground-truth barycenter is \textit{a priori} known, or at least approximately known, thus can be used for evaluating the efficacy of any candidate Wasserstein barycenter algorithm;
	\item\label{property: input-structure}the input measures should allow fairly general non-parametric forms and exhibit non-trivial and distinct distributional features; 
	\item\label{property: computational-efficiency}both the construction of the input measures and the subsequent sampling are computationally efficient in practice.
\end{enumerate}

\subsection{Algorithm specifics}\label{ssec: generation-instance-specifics}
We are now ready to present in detail our synthetic generation algorithm in Algorithm~\ref{algo: generating-problem-instance}. 
The inputs of Algorithm~\ref{algo: generating-problem-instance} are specified in Setting~\ref{sett: config-generation}, 
after which we elaborate on the intuition and the underlying theoretical rationale.

\begin{algorithm}[t]
	\KwIn{
		see Setting~\ref{sett: config-generation}.}
	\KwOut{$(\nu_k,T_k)_{k = 1: K}$, $V_{\MINSUB}$.}
	\nl\For{$\widetilde{k}=1,\ldots,\widetilde{K}$}{
	\nl\label{alglin: synthetic-generation-entropic1}%
	Define $\eta_{\widetilde{k},j}(\BIx):=\frac{\exp\Big({\textstyle\frac{g_{\widetilde{k},j} + \langle\BIy_{\widetilde{k},j},\BIx\rangle}{\gamma_{\widetilde{k}}}}\Big)}{\sum_{j'=1}^{n_{\widetilde{k}}}\exp\Big({\textstyle\frac{g_{\widetilde{k},j'} + \langle\BIy_{\widetilde{k},j'},\BIx\rangle}{\gamma_{\widetilde{k}}}}\Big)}$ $\forall\BIx\in\R^d$, $\forall 1\le j\le n_{\widetilde{k}}$. \\

	\nl\label{alglin: synthetic-generation-entropic2}%
	Define $\Beta_{\widetilde{k}}(\BIx):=\big(\eta_{\widetilde{k},1}(\BIx),\ldots,\eta_{\widetilde{k},{n_{\widetilde{k}}}}(\BIx)\big)^\TRANSP\in\R^{n_{\widetilde{k}}}$ $\forall\BIx\in\R^d$. \\

	\nl\label{alglin: synthetic-generation-entropic3}%
	Define $\BY_{\widetilde{k}}:=\left(\begin{smallmatrix}
			| & | &  & | \\
			\BIy_{\widetilde{k},1} & \BIy_{\widetilde{k},2} & \cdots & \BIy_{\widetilde{k},n_{\widetilde{k}}} \\
			| & | &  & |
	  		\end{smallmatrix}\right)\in\R^{d\times n_{\widetilde{k}}}$.\\
	
	\nl\label{alglin: synthetic-generation-smoothness}%
	Choose 
	$\overline{\lambda}_{\widetilde{k}} \ge \frac{1}{\gamma_{\widetilde{k}}}\sup_{\BIx\in\R^d}\Big\{e_{\MAXSUB}\Big(\BY_{\widetilde{k}}\big(\diag(\Beta_{\widetilde{k}}(\BIx))-\Beta_{\widetilde{k}}(\BIx)\Beta_{\widetilde{k}}(\BIx)^\TRANSP\big)\BY_{\widetilde{k}}^\TRANSP\Big)\Big\} + 2\underline{\lambda}_{\widetilde{k}}$.\\
	
	\nl\label{alglin: synthetic-generation-auxiliary-map}%
	Define
	$\widetilde{T}_{\widetilde{k}}(\BIx):=\Big(\sum_{j=1}^{n_{\widetilde{k}}}\eta_{\widetilde{k},j}(\BIx)\BIy_{\widetilde{k},j}\Big) + \underline{\lambda}_{\widetilde{k}}\BIx$ $\forall\BIx\in\R^d$ and
	$\widetilde{T}_{-\widetilde{k}}(\BIx):= \overline{\lambda}_{\widetilde{k}}\BIx - \widetilde{T}_{\widetilde{k}}(\BIx)$ $\forall\BIx\in\R^d$.\\
	}

	\nl \For{$\widetilde{k}=1,\ldots,\widetilde{K}$}{
		\nl\label{alglin: synthetic-generation-multiplier}%
		Set
		$\beta_{-\widetilde{k}} \leftarrow (1 - \xi) \alpha_{\widetilde{k}}\Big(\sum_{\widetilde{k}' = 1}^{\widetilde{K}}w_{\Phi(-\widetilde{k}')}\alpha_{\widetilde{k}'}  \overline{\lambda}_{\widetilde{k}'}\Big)^{-1} \in (0, \infty)$, 
		$\beta_{\widetilde{k}} \leftarrow \frac{w_{\Phi(- \widetilde{k})}}{w_{\Phi(\widetilde{k})}} \beta_{-\widetilde{k}}\in (0, \infty)$.\\
	}


	\nl\For{$k=1,\ldots,K$}{

	\nl\label{alglin: synthetic-generation-maps}%
	Define 
	$T_k(\BIx):=\Big(\sum_{i\in\Phi^{-1}(k)}\beta_i\widetilde{T}_{i}(\BIx)\Big) + \xi (\BA_k\BIx + \BIb_k)$ $\forall \BIx\in\R^d$. \\

	\nl\If{$\mathtt{TRUNCATE} = \mathtt{True}$}{
		\nl\label{alglin: synthetic-generation-truncated}%
		Set $\nu_k\leftarrow (T_k\sharp\bar{\mu})|_{\CY_{k}}$.\\
	}

	\nl\Else{
		\nl\label{alglin: synthetic-generation-pushforwards}%
		Set $\nu_k\leftarrow T_k\sharp\bar{\mu}$.\\
	}

	}

	\nl\label{alglin: synthetic-generation-LB}
	$V_{\MINSUB}\leftarrow \int_{\R^d}\sum_{k=1}^{K}w_k\big\|\BIx-T_k(\BIx)\big\|^2\DIFFM{\bar{\mu}}{\DIFF\BIx}$.

	\nl \Return $(\nu_k, T_k)_{k = 1: K}$, $V_{\MINSUB}$.\\
	\caption{\textbf{Synthetic generation of Wasserstein barycenter problem instance.}}\label{algo: generating-problem-instance}
\end{algorithm}

\begin{setting}[Inputs of Algorithm~\ref{algo: generating-problem-instance}]\label{sett: config-generation}
	Let $\bar{\mu}\in\CP_{2,\AC}(\R^d)$,
	let $K\in\N \intersects [2,\infty)$, and let 
	$w_1,\ldots,w_K\in(0,1)$ satisfy $\sum_{k=1}^Kw_k=1$.
	Let $\widetilde{K}\in \N \intersects [2, \infty)$ satisfy $2\widetilde{K}\ge K$, 
	and for $\widetilde{k}=1,\ldots,\widetilde{K}$, 
	let 
		$\underline{\lambda}_{\widetilde{k}}> 0$, 
		$\gamma_{\widetilde{k}}>0$, 
		$\alpha_{\widetilde{k}}>0$,
		$n_{\widetilde{k}}\in\N$, 
		$(g_{\widetilde{k},j})_{j=1:n_{\widetilde{k}}}\subset\R$, and $(\BIy_{\widetilde{k},j})_{j=1:n_{\widetilde{k}}}\subset\R^d$.
	Moreover, let $\Phi:\{1,\ldots, \widetilde{K},\allowbreak-1,\ldots,-\widetilde{K}\}\to\{1,\ldots,K\}$ be a surjective map.
	Furthermore, let $(\BA_k)_{k=1:K} \subset \mathbb{S}_{++}^d, (\BIb_k)_{k = 1:K}\subset\R^d$ satisfy $\sum_{k=1}^K w_k \BA_k= \BI_{d}$ and
	$\sum_{k=1}^K w_k\BIb_k = \veczero_d$.
	Lastly, let $\xi \in [0, 1)$,
	let $\mathtt{TRUNCATE} \in \{\mathtt{True}, \mathtt{False}\}$ be a Boolean variable,
	and let $(\CY_k)_{k=1:K}$ 
	be closed subsets of $\R^d$.
\end{setting}

Let us now provide the motivation and explanation of each operation carried out in Algorithm~\ref{algo: generating-problem-instance}.
Regarding the property~\ref{property: ground-truth}, Algorithm~\ref{algo: generating-problem-instance} is similar to the method of \citet{korotin2022wasserstein} in 
generating problem instances out of a user-specified measure $\bar{\mu}\in\CP_{2,\AC}(\R^d)$ as the underlying \mbox{$\CW_2$-barycenter}.
Regarding the property~\ref{property: input-structure}, the input measures $\nu_1,\ldots,\nu_K$ of our generated problem instances will be characterized via the pushforwards of $\bar{\mu}$ by several tailored transport maps $T_1,\ldots,T_K:\R^d\to\R^d$ (which are defined in Line~\ref{alglin: synthetic-generation-maps}), subject to a possible truncation operation indicated by the Boolean variable $\mathtt{TRUNCATE}\in\{\mathtt{True}, \mathtt{False}\}$ given in Setting~\ref{sett: config-generation}; we defer a detailed discussion regarding when to use $\mathtt{TRUNCATE}=\mathtt{True}$ and $\mathtt{TRUNCATE}=\mathtt{False}$
to after Proposition~\ref{prop: synthetic-generation}.

Specifically, our construction of the transport maps $T_1,\ldots,T_K:\R^d\to\R^d$ involves building $2 \widetilde{K}\in\N$ auxiliary transport maps, namely $(\widetilde{T}_{-\widetilde{k}}, \widetilde{T}_{\widetilde{k}})_{\widetilde{k} = 1: \widetilde{K}}$,
for some $2\widetilde{K}\ge K$; see Lines~\ref{alglin: synthetic-generation-entropic1}--\ref{alglin: synthetic-generation-auxiliary-map}.
Intuitively, 
for $\widetilde{k}=1,\ldots,\widetilde{K}$,
Lines~\ref{alglin: synthetic-generation-entropic1},
\ref{alglin: synthetic-generation-entropic2},
\ref{alglin: synthetic-generation-auxiliary-map}
build 
$\widetilde{T}_{\widetilde{k}}$ upon the configurations $(g_{\widetilde{k},j},\BIy_{\widetilde{k},j})_{j=1:n_{\widetilde{k}}}$ in Setting~\ref{sett: config-generation}
via 
a barycentric projection analogous to the construction of the entropic OT map estimator in Proposition~\ref{prop: OTmap-estimator-entropic}, up to an additional linear term $\underline{\lambda}_{\widetilde{k}}\BIx$ guaranteeing that $\widetilde{T}_{\widetilde{k}}$ is the gradient of a $\underline{\lambda}_{\widetilde{k}}$-strongly convex function.
One can therefore customize the shapes of the pushforward measures $(\widetilde{T}_{\widetilde{k}} \sharp \bar{\mu})_{{\widetilde{k} = 1: \widetilde{K}}}$ by appropriately specifying $(g_{\widetilde{k},j},\BIy_{\widetilde{k},j})_{j=1:n_{\widetilde{k}}}$. 
For instance, suppose that $(\BIy_{\widetilde{k},j})_{j=1:n_{\widetilde{k}}}$ are independent samples from an auxiliary probability measure $\varkappa_{\widetilde{k}}\in\CP_{2,\AC}(\R^d)$, 
and suppose that $(g_{\widetilde{k},j})_{j=1:n_{\widetilde{k}}}$ are given by an approximately optimal solution of the corresponding dual entropic OT problem 
(\ref{eqn: EOT-dual-discrete})
between the empirical versions of $\bar{\mu}$ and $\varkappa_{\widetilde{k}}$ computed by Sinkhorn's algorithm.
Then, the pushforward measure $\widetilde{T}_{\widetilde{k}} \sharp \bar{\mu}$ is expected to inherit the distributional features of~$\varkappa_{\widetilde{k}}$.

On the other hand, 
for $\widetilde{k}=1,\ldots,\widetilde{K}$,
the construction of $\overline{\lambda}_{\widetilde{k}}>0$
and $\widetilde{T}_{-\widetilde{k}}$
in Lines~\ref{alglin: synthetic-generation-entropic3}--\ref{alglin: synthetic-generation-auxiliary-map}
ensures that 
$\widetilde{T}_{-\widetilde{k}}$ is the gradient of a $\underline{\lambda}_{\widetilde{k}}$-strongly convex function
and that 
$\widetilde{T}_{-\widetilde{k}}(\BIx)+\widetilde{T}_{\widetilde{k}}(\BIx)=\overline{\lambda}_{\widetilde{k}}\BIx$ for all $\BIx\in\nobreak\R^d$.
This guarantees that the auxiliary transport maps 
$(\widetilde{T}_{-\widetilde{k}}, \widetilde{T}_{\widetilde{k}})_{\widetilde{k} = 1: \widetilde{K}}$
possess the desirable regularity properties that are essential for generating input measures with the user-specified (approximate) $\CW_2$-barycenter $\bar{\mu}$
(see the proof of Proposition~\ref{prop: synthetic-generation}).

Subsequently, for $k = 1, \dots, K$, the map $T_k$ is defined as a combination performed on $(\widetilde{T}_{-\widetilde{k}}, \widetilde{T}_{\widetilde{k}})_{\widetilde{k} = 1: \widetilde{K}}$
in Line~\ref{alglin: synthetic-generation-maps},
specified by the surjective map $\Phi$ in Setting~\ref{sett: config-generation} and the coefficients $(\beta_{-\widetilde{k}}, \beta_{\widetilde{k}})_{\widetilde{k}=1:\widetilde{K}}$ computed in Line~\ref{alglin: synthetic-generation-multiplier}, 
and subject to an additive affine term $\xi(\BA_k\BIx+\BIb_k)$ in Line~\ref{alglin: synthetic-generation-maps}. 
In this way, $(T_k)_{k = 1: K}$ may partially preserve the features of $(\widetilde{T}_{\widetilde{k}})_{\widetilde{k} = 1: \widetilde{K}}$ and give rise to input measures $\nu_1,\ldots,\nu_K$ with non-trivial distributional features,
thus achieving the desirable property~\ref{property: input-structure}.

The following proposition confirms that the property~\ref{property: ground-truth} is indeed satisfied by any problem instance
generated from Algorithm~\ref{algo: generating-problem-instance}, 
either with or without the truncation operation.

\begin{proposition}[Synthetic generation of $\CW_2$-barycenter problem instance via Algorithm~\ref{algo: generating-problem-instance}]\label{prop: synthetic-generation}
	Let the inputs of Algorithm~\ref{algo: generating-problem-instance} satisfy Setting~\ref{sett: config-generation}, and let 
	$(\nu_k, T_k)_{k = 1: K}, V_{\MINSUB}$
	be the outputs of Algorithm~\ref{algo: generating-problem-instance}.
	Then, in the case where $\mathtt{TRUNCATE}=\mathtt{False}$,
	the following statements hold.
	\begin{enumerate}[label=(\roman*)]
		\item\label{props: synthetic-generation-notruncation}
		$\nu_k\in\CP_{2,\AC}(\R^d)$ for $k=1,\ldots,K$, and $\bar{\mu}$ is the unique $\CW_2$-barycenter of $\nu_1,\ldots,\nu_K$ with weights $w_1,\ldots,w_K$.

		\item\label{props: synthetic-generation-lowerbound}
		$V_{\MINSUB}=V(\bar{\mu})$, where 
		$V$~is defined in (\ref{eqn: V-definition}) with respect to
		the weights $w_1,\ldots,w_K$ in Setting~\ref{sett: config-generation}.
	\end{enumerate}
	In the case where $\mathtt{TRUNCATE}=\mathtt{True}$,
	let us assume in addition that $\bar{\mu}\in\CM^q_{\FULL}(\R^d)$ for $q\in\N_0$, and that $\CY_1,\ldots,\CY_K\in\CS^q(\R^d)$. 
	Then, the following statements hold.
	\begin{enumerate}[label=(\roman*)]
	\setcounter{enumi}{2}
		\item\label{props: synthetic-generation-truncation}
		$\nu_k\in\CM^q(\R^d)$ for $k=1,\ldots,K$.

		\item\label{props: synthetic-generation-truncation-V-value} 
		For any $\CY_1,\ldots,\CY_K\in\CS^q(\R^d)$,
		one can explicitly construct a term $\epsilon(\CY_1,\ldots,\CY_K)\in(0,\infty)$ such that 
		$V(\bar{\mu})\le \inf_{\mu\in\CP_2(\R^d)}\big\{V(\mu)\big\}+\epsilon(\CY_1,\ldots,\CY_K)$
		and 
		$\Big|V_{\MINSUB}-\inf_{\mu\in\CP_2(\R^d)}\big\{V(\mu)\big\}\Big|
		\le \epsilon(\CY_1,\ldots,\CY_K)$;
		see equation (\ref{eqn: synthetic-generation-truncation-error}) in the proof for the explicit definition of $\epsilon(\CY_1,\ldots,\CY_K)$.

		\item\label{props: synthetic-generation-error}
		For $k=1,\ldots,K$, let $(\CY_{k,r})_{r\in\N}$ be a family of increasing sets satisfying $\CY_{k,r}\in\CS^q(\R^d)$, $\CY_{k, r+1}\supseteq\CY_{k, r}$ $\forall r\in\N$,
		as well as $\bigcup_{r\in\N}\CY_{k,r}=\R^d$.
		Then, 
		the term $\epsilon(\,\cdot\,,\ldots,\cdot\,)$ in statement~\ref{props: synthetic-generation-truncation-V-value} satisfies
		$\lim_{r\to\infty}\epsilon(\CY_{1,r},\ldots,\CY_{K,r})=\nobreak0$.
	\end{enumerate}
\end{proposition}

\begin{proof}[Proof of Proposition~\ref{prop: synthetic-generation}]
	See Appendix~\ref{sapx: proof-generation-instance}.
\end{proof}

Therefore, 
in the case where $\mathtt{TRUNCATE} = \mathtt{False}$,
Proposition~\ref{prop: synthetic-generation}\ref{props: synthetic-generation-notruncation} 
and 
Proposition~\ref{prop: synthetic-generation}\ref{props: synthetic-generation-lowerbound}
enable one to evaluate any approximate $\CW_2$-barycenter $\widehat{\mu}$ of the constructed $\nu_1, \dots, \nu_K$ by 
(approximately) computing $V(\widehat{\mu})$ and comparing it with $V_{\MINSUB}=V(\bar{\mu})$, or by
(approximately) computing $\CW_2(\widehat{\mu},\bar{\mu})$.
However, when $\mathtt{TRUNCATE} = \mathtt{False}$, 
the constructed $\nu_1, \dots, \nu_K$ are not necessarily Caffarelli-type admissible probability measures satisfying \ref{setts: Caffarelli-inputs-measures} in Setting~\ref{sett: Caffarelli}, which prevents us from establishing the convergence guarantee of Algorithm~\ref{algo: concrete} as presented in Theorem~\ref{thm: Caffarelli-convergence}. 
In fact, the pushforwards of $\bar{\mu}$ by $T_1, \dots, T_K$ may fail to belong to $\CM^q(\R^d)$ even if $\bar{\mu} \in \CM^q(\R^d)$ is imposed, since $\support\big(T_{1}\sharp\bar{\mu}\big),\ldots,\support\big(T_{K}\sharp\bar{\mu}\big)$ are not necessarily convex.

To this end, Proposition~\ref{prop: synthetic-generation}\ref{props: synthetic-generation-truncation} shows that 
for any $\bar{\mu}\in \CM^q_{\FULL}(\R^d)$,
Algorithm~\ref{algo: generating-problem-instance} with 
$\mathtt{TRUNCATE}=\mathtt{True}$ 
constructs $\nu_1, \dots, \nu_K\in\CM^q(\R^d)$ 
which satisfy \ref{setts: Caffarelli-inputs-measures} in 
Setting~\ref{sett: Caffarelli}.
Moreover, Proposition~\ref{prop: synthetic-generation}\ref{props: synthetic-generation-truncation-V-value} indicates that 
$\bar{\mu}$ approximately solves the minimization problem $\inf_{\mu\in\CP_2(\R^d)}\big\{V(\mu)\big\}$ characterizing the \mbox{$\CW_2$-barycenter},
where the suboptimality is bounded by $\epsilon(\CY_1,\ldots,\CY_K)$.
Additionally, $V_{\MINSUB}$ is approximately equal to 
the minimum value of~$V$ where their absolute difference is also bounded by 
$\epsilon(\CY_1,\ldots,\CY_K)$.
Furthermore, given families of increasing sets $(\CY_{1,r})_{r\in\N},\ldots,(\CY_{K,r})_{r\in\N}$ 
that satisfy the properties in \ref{setts: Caffarelli-family-increasing}
and with $\CY_k\leftarrow \CY_{k,r}$ for $k=1,\ldots,K$, 
Proposition~\ref{prop: synthetic-generation}\ref{props: synthetic-generation-error} shows that the suboptimality gap can be controlled to be arbitrarily small by choosing sufficiently large $r\in\N$.
That is, when $r\in\N$ is large, 
$\bar{\mu}$ is a highly accurate approximate $\CW_2$-barycenter of the generated measures $\nu_1,\ldots,\nu_K$ with weights $w_1,\ldots,w_K$, 
and $V_{\MINSUB}$ can be treated as an approximate lower bound for the barycenter functional~$V$ when one 
quantitatively analyzes the empirical approximation error of any $\CW_2$-barycenter algorithm using the generated measures $\nu_1,\ldots,\nu_K$.

\begin{remark}
	Under the setting of Proposition~\ref{prop: synthetic-generation}\ref{props: synthetic-generation-error} where $(\CY_{1,r})_{r\in\N},\ldots,(\CY_{K,r})_{r\in\N}$ are families of increasing sets, 
	the result of \citet*[Proposition~6]{legouic2017existence} about the stability of Wasserstein barycenter can be used to conclude that 
	the unique $\CW_2$-barycenter of $\nu_1,\ldots,\nu_K$ generated by Algorithm~\ref{algo: generating-problem-instance} with $\mathtt{TRUNCATE}= \mathtt{True}$ and with respect to $\CY_k\leftarrow\CY_{k,r}$ for $k=1,\ldots,K$ 
	converges in $\CW_2$ to~$\bar{\mu}$
	as $r\to\infty$. 
	However, we are unable to get any quantitative bound on their $\CW_2$-distance due to the non-compactness of $\support\big(T_1 \sharp \bar{\mu}\big),\ldots,\support\big(T_K \sharp \bar{\mu}\big)$;
	see the discussion of \citet{carlier2024quantitative}
	about the difficulties of establishing quantitative stability of Wasserstein barycenter.
\end{remark}

Finally, on the computational side, we show that Algorithm~\ref{algo: generating-problem-instance} possesses the desirable property \ref{property: computational-efficiency}.
Assuming that independent samples from 
$\bar{\mu}\in\CP_{2,\AC}(\R^d)$ in the inputs of Algorithm~\ref{algo: generating-problem-instance}
can be efficiently generated, one can efficiently generate independent samples from $\nu_1,\ldots,\nu_K$ in the outputs of Algorithm~\ref{algo: generating-problem-instance} as follows.
For each $k\in\{1,\ldots,K\}$,
one first generates 
$M \in \N$ independent samples $\{\BIZ_i\}_{i=1:M}$ from $\bar{\mu}$ and then
computes their images $\big\{T_k(\BIZ_i)\big\}_{i=1:M}$ under $T_k$.
Subsequently, in the case where $\mathtt{TRUNCATE}=\mathtt{False}$, 
$\big\{T_k(\BIZ_i)\big\}_{i=1:M}$ are $M$ independent samples from $\nu_k$.
In the case where $\mathtt{TRUNCATE}=\mathtt{True}$,
one performs an extra rejection step where 
the samples in
$\big\{T_k(\BIZ_i)\big\}_{i=1:M}$
that do not belong to $\CY_k$ are rejected and discarded.
In this way, 
$\big\{T_k(\BIZ_i)\big\}_{i=1:M}\intersects \CY_k$ are independent samples from~$\nu_k$.

\section{Numerical experiments}\label{sec: numerics}
In this section, we evaluate the performance of our proposed stochastic fixed-point algorithm (Algorithm~\ref{algo: concrete}) driven by the modified entropic OT map estimator in Proposition~\ref{prop: OTmap-estimator-entropic}. 
Our numerical experiments include problem instances generated by our synthetic generation algorithm (Algorithm~\ref{algo: generating-problem-instance}) as well as a subset posterior aggregation problem using an external dataset.
The $\mathtt{Python}$ implementation of our proposed algorithm and all code for our numerical experiments can be accessed at the GitHub repository accompanying this paper: \url{https://github.com/CHENZeyi1101/WB_Algo}.
We further compare our algorithm with several free-support algorithms in the literature for approximating the Wasserstein barycenter.
These benchmarks include the ones developed by \citet*{cuturi2014fast}, \citet*{li2020continuous}, \citet*{fan2021scalable}, \citet*{korotin2022wasserstein}, and \citet*{kim2025optimal};
among them, the algorithm of \citet{kim2025optimal} is grid-based and only supports two-dimensional instances.
A brief description of these benchmark algorithms, together with their empirical run-time and computational budgets, is provided in Table~\ref{tab: benchmark-summary} in Appendix~\ref{sapx: benchmark-descriptions}.\footnote{We measure run-time on an Apple M1~Pro CPU (10 cores, 16\,GB RAM) running macOS, with no GPU acceleration, using Python~3.11 with $\mathtt{POT}$~v0.9.6 \citep{flamary2021pot}, $\mathtt{geomloss}$~v0.2.6 \citep{Feydy2019Fast}, and $\mathtt{PyTorch}$~v2.6. When applying benchmark algorithms to our experiments, we preserve the model hyperparameters specified in their respective source code.}

\subsection{Experiment~1: synthetic problem instances}\label{ssec: exp-synthetic-problem-instances}
In this experiment, we use Algorithm~\ref{algo: generating-problem-instance} in Section~\ref{sec: generation-instance} to synthetically generate problem instances with non-parametric continuous input measures $\nu_1, \dots, \nu_K$ from a pre-specified probability measure $\bar{\mu}$. 
To examine the performance of our algorithm in both low- and high-dimensional regimes, we generate two problem instances with $d = 2$ and $d = 10$, respectively, which are referred to as [SG-2d] and [SG-10d] in the rest of this section.

\subsubsection*{Experimental setup}
For both [SG-2d] and [SG-10d], the approximate ground-truth $\CW_2$-barycenter $\bar{\mu}$ is specified as a mixture of 5 Gaussian measures supported on $\R^d$ with randomly generated mean vectors and covariance matrices.
Therefore, it holds that $\bar{\mu} \in \CM^q_{\FULL}(\R^d)$ for all $q\in\N_0$.
When using Algorithm~\ref{algo: generating-problem-instance},
we set $K=5$, $\widetilde{K}=5$ in [SG-2d]
and set $K=10$, $\widetilde{K}=10$ in [SG-10d].
For both instances, equal weights $w_1 = \cdots = w_K = \frac{1}{K}$ are considered, and $(n_{\widetilde{k}})_{\widetilde{k} = 1: \widetilde{K}}$ are also set to be equal. 
Regarding $(g_{\widetilde{k},j},\BIy_{\widetilde{k},j})_{j=1:n_{\widetilde{k}}}$ for each $\widetilde{k} = 1, \dots, \widetilde{K}$, we consider auxiliary probability measures $\varkappa_1, \dots, \varkappa_{\widetilde{K}}\in \CP_{2, \AC}(\R^d)$ all being mixtures of 5 Gaussian measures
(recall our discussions in Section~\ref{ssec: generation-instance-specifics} about the role of $\varkappa_1, \dots, \varkappa_{\widetilde{K}}$).
Then,
for $\widetilde{k}=1,\ldots,\widetilde{K}$,
$(\BIx_{j})_{j=1:n_{\widetilde{k}}}$ are $n_{\widetilde{k}}$ independent samples drawn from $\bar{\mu}$, 
whereas $(\BIy_{\widetilde{k},j})_{j=1:n_{\widetilde{k}}}$ are $n_{\widetilde{k}}$ independent samples drawn from $\varkappa_{\widetilde{k}}$.
Subsequently, $(g_{\widetilde{k},j})_{j=1:n_{\widetilde{k}}}$ are 
computed by Sinkhorn's algorithm that approximately solves the dual EOT problem 
(\ref{eqn: EOT-dual-discrete})
between the empirical measures 
$\frac{1}{n_{\widetilde{k}}} \sum_{j = 1}^{n_{\widetilde{k}}} \delta_{\BIx_j}$ and 
$\frac{1}{n_{\widetilde{k}}} \sum_{j = 1}^{n_{\widetilde{k}}} \delta_{\BIy_{\widetilde{k}, j}}$. 
The remaining input variables, i.e., 
$\big(\underline{\lambda}_{\widetilde{k}}, \gamma_{\widetilde{k}},\alpha_{\widetilde{k}}\big)_{\widetilde{k}=1:\widetilde{K}}$, $\Phi$, $(\BA_k, \BIb_k)_{k = 1:K}$, $\xi$,
are all configured to satisfy Setting~\ref{sett: config-generation}; 
see 
the concrete configurations in
Table~\ref{tab: instance-config} in Appendix~\ref{sapx: exp-details-synthetic-generation} as well as in
our code.

To further ensure that $\nu_1,\ldots,\nu_K$ are Caffarelli-type admissible input measures satisfying Setting~\ref{sett: Caffarelli}, we set $\mathtt{TRUNCATE}=\mathtt{True}$ in the inputs of Algorithm~\ref{algo: generating-problem-instance} in both [SG-2d] and [SG-10d].
Specifically, 
for $k =\nobreak 1, \dots, K$, 
we set $\CY_k=\bar{B}(\veczero_d,r)$ 
for large~$r>0$ to truncate the pushforward measures in Line~\ref{alglin: synthetic-generation-truncated} of Algorithm~\ref{algo: generating-problem-instance}.
Consequently, Proposition~\ref{prop: synthetic-generation}\ref{props: synthetic-generation-truncation} guarantees $\nu_1, \dots, \nu_K \in \CM^q(\R^d)$ for all $q\in\N_0$.
Hence, [SG-2d] and [SG-10d] satisfy Setting~\ref{sett: Caffarelli} by construction, and are therefore fully covered by the convergence guarantee of Algorithm~\ref{algo: concrete} established in Theorem~\ref{thm: Caffarelli-convergence}.
Their probability density functions are visualized in Figure~\ref{fig: input-kde-pdf} via kernel density estimation (KDE), each based on $10^4$ independent samples.
Subsequently, the value of $V_{\MINSUB}$ in Line~\ref{alglin: synthetic-generation-LB} of Algorithm~\ref{algo: generating-problem-instance}
is computed via Monte Carlo integration using $10^7$ independent samples from $\bar{\mu}$.

\begin{figure}[t]
    \centering
    \includegraphics[width=0.8\textwidth]{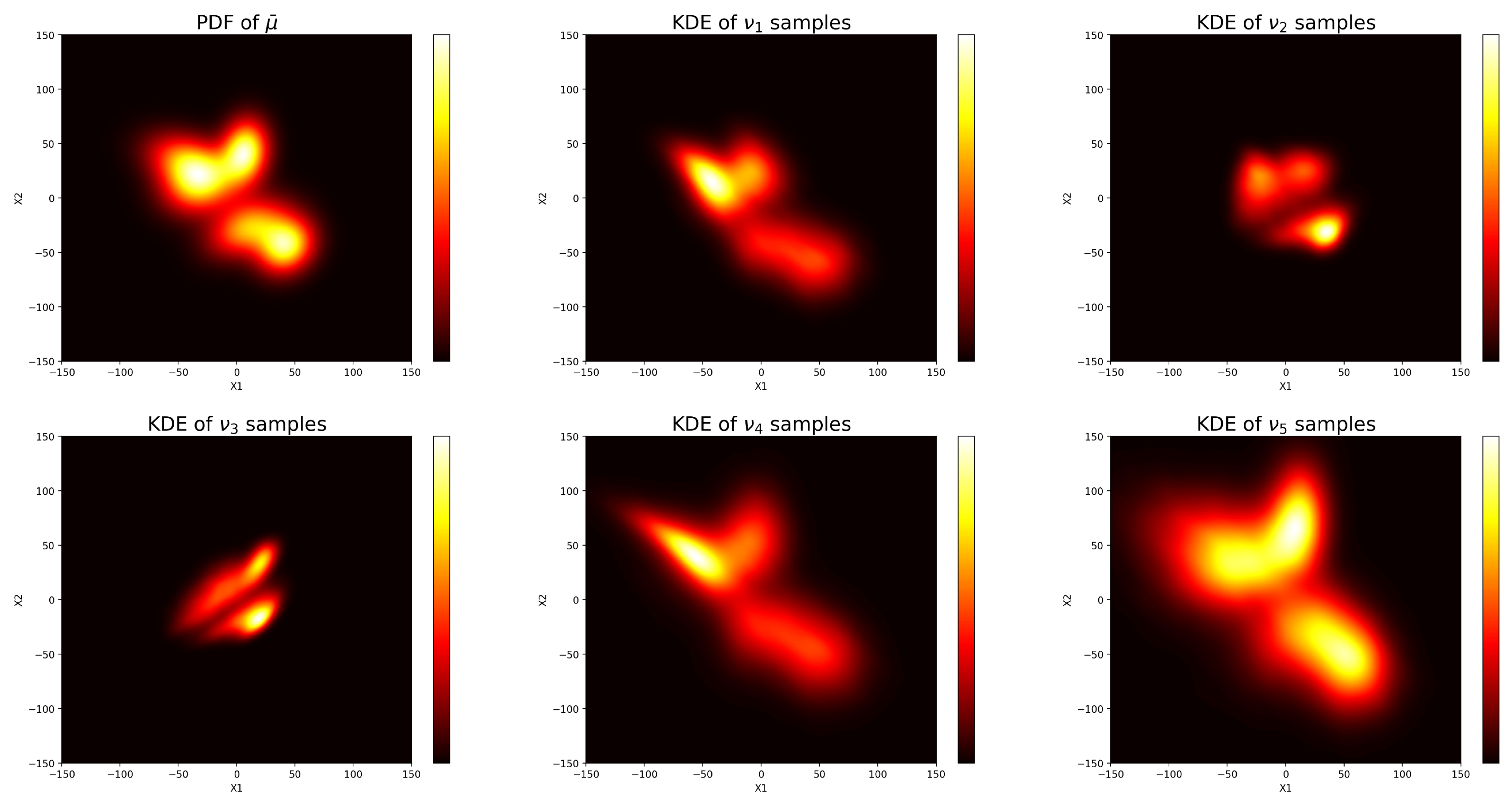}
    \caption{Probability density function of $\bar{\mu}$ and the KDEs of the input measures $\nu_1, \dots, \nu_5$ in [SG-2d].}\label{fig: input-kde-pdf}
\end{figure}

\begin{figure}[t]
    \centering
    \begin{subfigure}{\textwidth}
        \centering
        \begin{minipage}{0.48\textwidth}
            \centering
            \includegraphics[width=\linewidth]{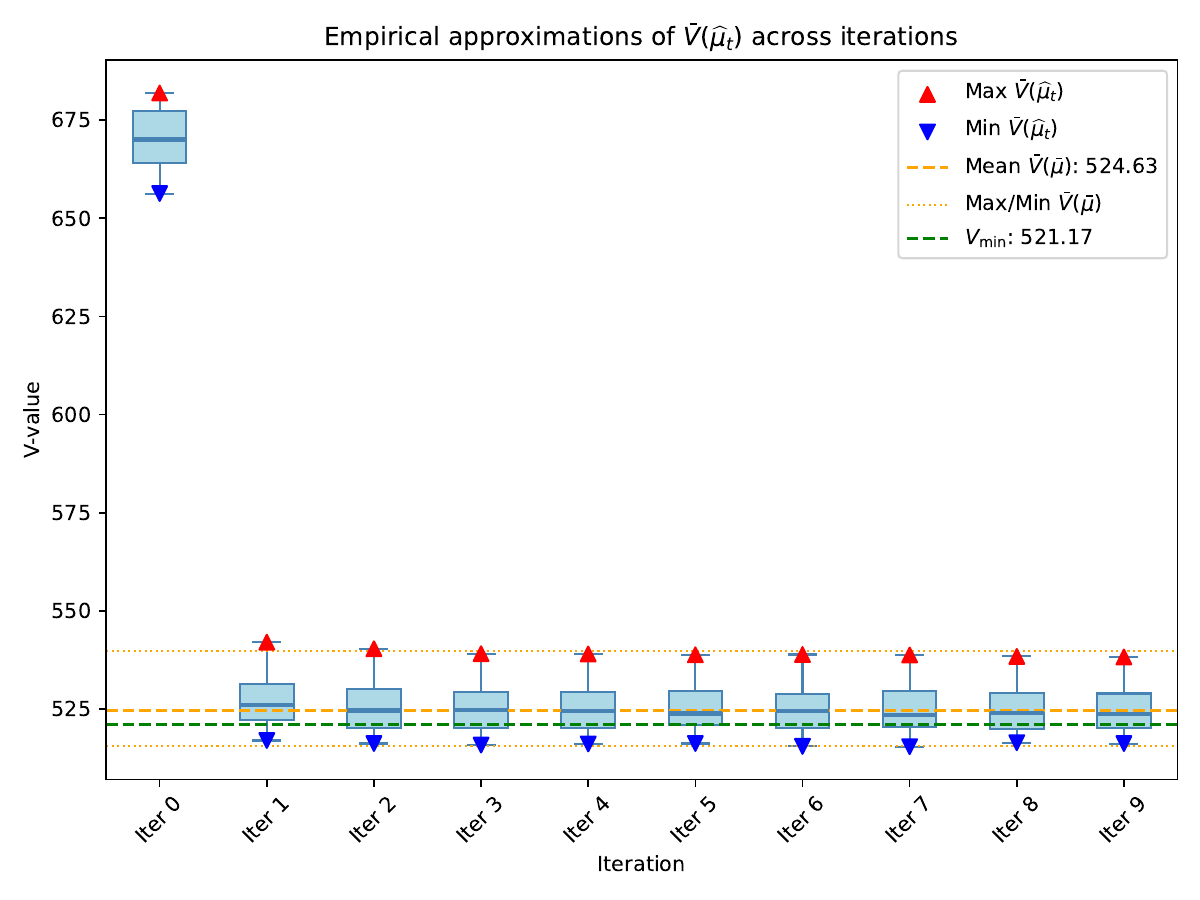}
        \end{minipage}
        \begin{minipage}{0.48\textwidth}
            \centering
            \includegraphics[width=\linewidth]{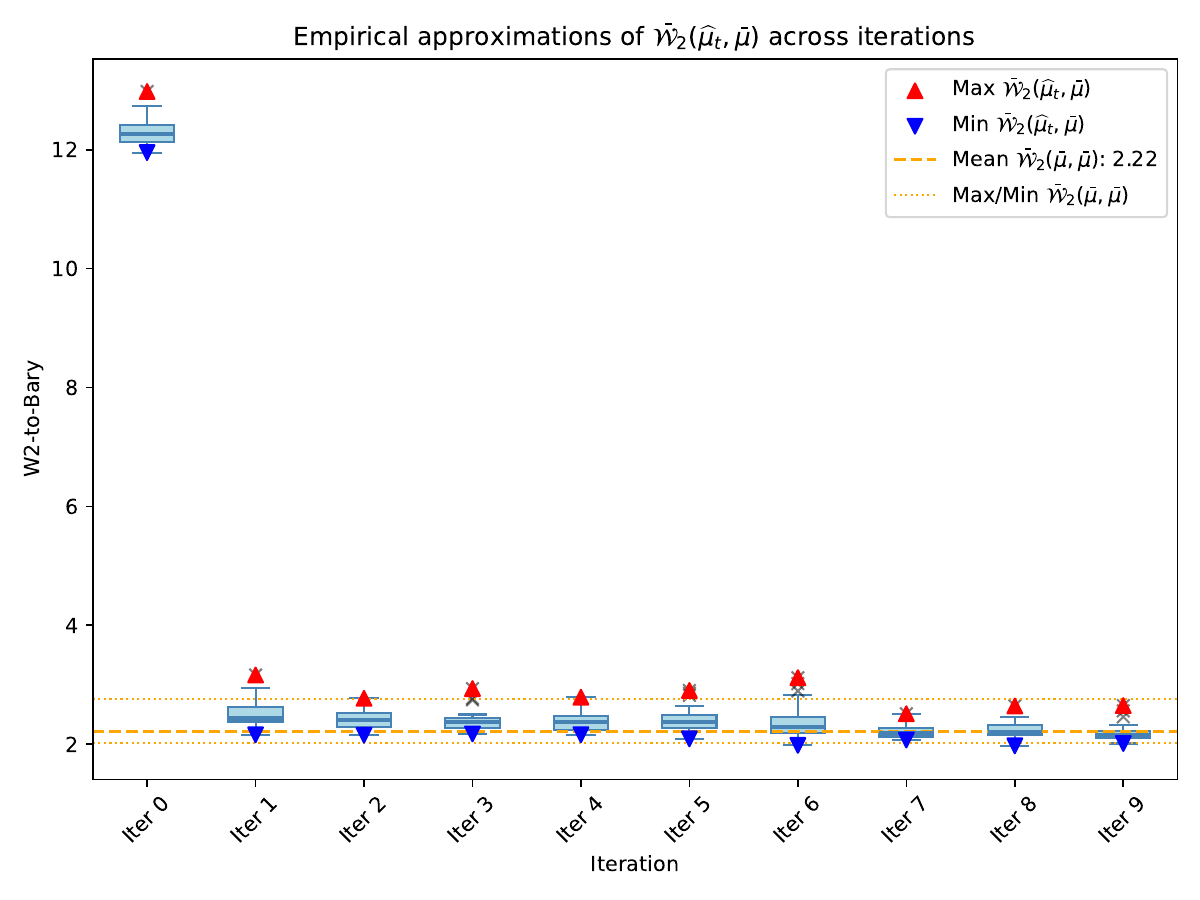}
        \end{minipage}
        \caption{Empirical evaluations of $(\widehat{\mu}_t)_{t = 0:9}$ in [SG-2d].}\label{fig: stochastic-FP-SG-2d}
    \end{subfigure}
    \vspace{0.5em}
    \begin{subfigure}{\textwidth}
        \centering
        \begin{minipage}{0.48\textwidth}
            \centering
            \includegraphics[width=\linewidth]{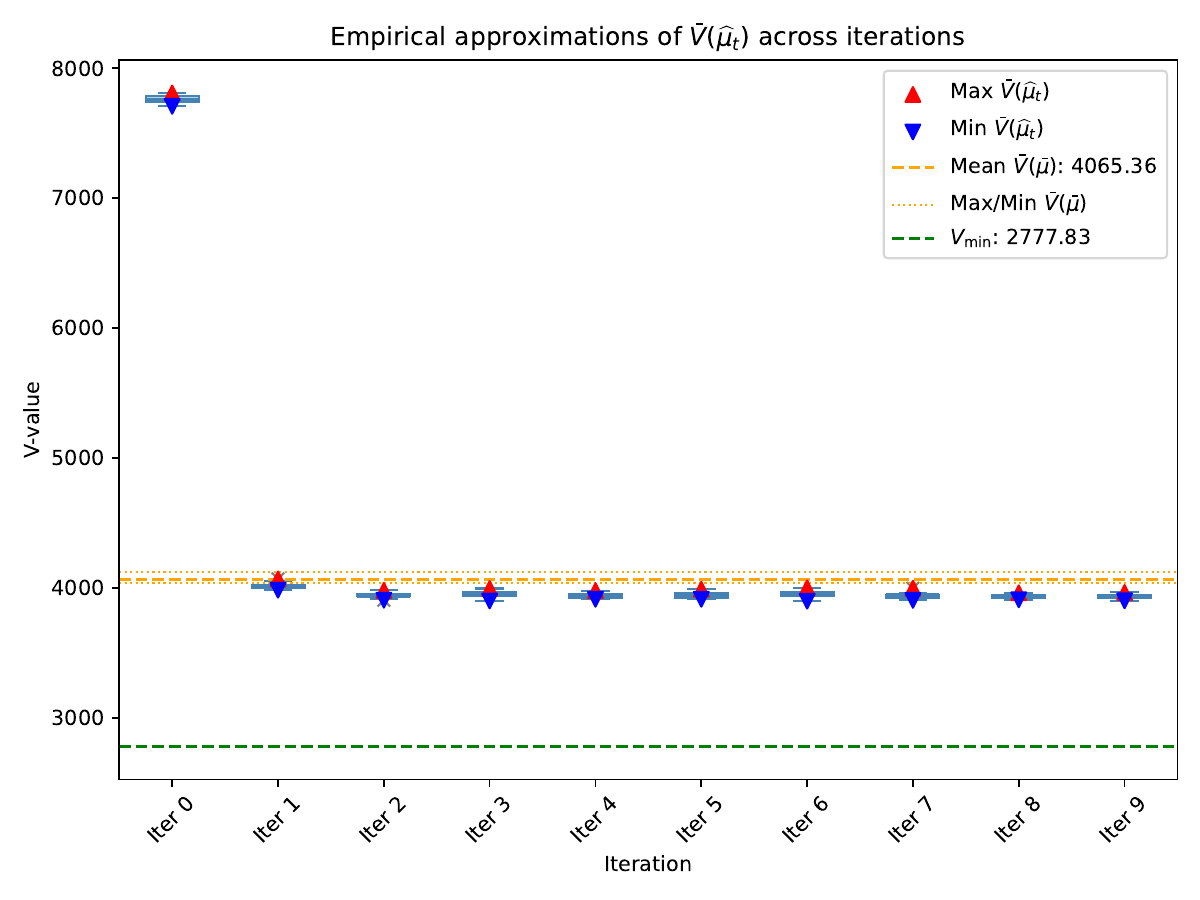}
        \end{minipage}
        \begin{minipage}{0.48\textwidth}
            \centering
            \includegraphics[width=\linewidth]{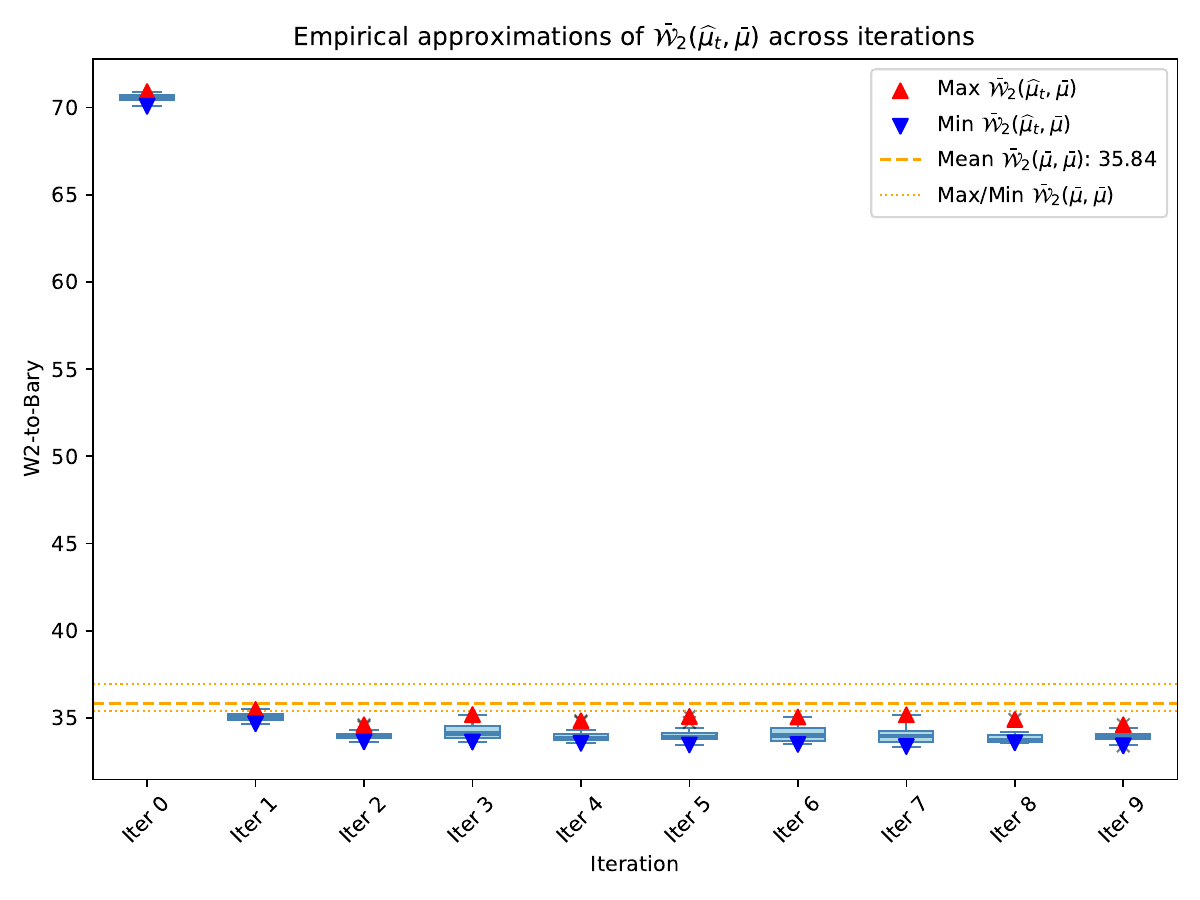}
        \end{minipage}
        \caption{Empirical evaluations of $(\widehat{\mu}_t)_{t = 0:9}$ in [SG-10d].}\label{fig: stochastic-FP-SG-10d}
    \end{subfigure}
    \caption{Box plots showing the empirical evaluations of $(\widehat{\mu}_t)_{t = 0:9}$ computed by Algorithm~\ref{algo: concrete} in [SG-2d] and [SG-10d], respectively. 
	\textbf{Left}: 
	values of $(\bar{V}(\widehat{\mu}_t))_{t = 0:9}$.
	\textbf{Right}: values of $(\bar{\CW}_2(\widehat{\mu}_t, \bar{\mu}))_{t = 0:9}$.}\label{fig: stochastic-FP-SG}
\end{figure}

We adopt two metrics to quantitatively evaluate the qualities of approximate barycenter candidates computed by our algorithm 
and the aforementioned benchmark algorithms.
For every approximate barycenter $\widehat{\mu}\in\CP_{2}(\R^d)$,
we approximately compute $V(\widehat{\mu})$
and $\CW_2(\widehat{\mu},\bar{\mu})$,
where we approximate the $\CW_2$-distance between each pair of probability measures 
via empirical approximation with $10^4$ independent samples,
that is, we generate $10^4$ samples from both the source and the target probability measures and then compute the \mbox{$\CW_2$-distance} between the resulting empirical measures.
This empirical approximation is repeated 20 times,
and we denote the 
average values by
$\bar{V}(\widehat{\mu})$ and
$\bar{\CW}_2(\widehat{\mu},\bar{\mu})$, respectively.
The value of $\bar{V}(\widehat{\mu})$ is
compared with $V_{\MINSUB}$ in the outputs of Algorithm~\ref{algo: generating-problem-instance}.
Besides $V_{\MINSUB}$,
we additionally compute $\bar{V}(\bar{\mu})$ as the empirical approximation of $V(\bar{\mu})$.
Despite that Proposition~\ref{prop: synthetic-generation}\ref{props: synthetic-generation-truncation-V-value} and Proposition~\ref{prop: synthetic-generation}\ref{props: synthetic-generation-error}
suggest
$\big|V(\bar{\mu})-V_{\MINSUB}\big|$ to be small,
the aforementioned empirical approximations of $\CW_2$-distances incur positive biases
making $\bar{V}(\bar{\mu})$ larger than $V_{\MINSUB}$ in this experiment,
which is particularly noticeable in higher dimensions; 
see, e.g., the bottom panels of Figure~\ref{fig: stochastic-FP-SG}.
Because of this positive bias, 
we also compute 
$\bar{\CW}_2(\bar{\mu},\bar{\mu})$ via empirical approximation
to put the empirically approximated error of $\bar{\CW}_2(\widehat{\mu},\bar{\mu})$ into perspective 
in order to fully gauge the quality of each approximate barycenter candidate $\widehat{\mu}$.

\subsubsection*{Result analysis}
Figure~\ref{fig: stochastic-FP-SG} presents the empirical performance of $(\widehat{\mu}_t)_{t = 0:9}$ computed by Algorithm~\ref{algo: concrete} in both [SG-2d] and [SG-10d], which includes the evaluation of $(\bar{V}(\widehat{\mu}_t))_{t = 0:9}$ and $(\bar{\CW}_2(\widehat{\mu}_t, \bar{\mu}))_{t = 0:9}$ across iterations. 
Specifically, each box plot summarizes the distribution of the evaluated values across 20 empirical approximations at a given iteration (including median, interquartile range, and whiskers), with overlaid markers for the minimum and maximum.
We also plot a yellow dashed line in each panel showing $\bar{V}(\bar{\mu})$ or $\bar{\CW}_2(\bar{\mu}, \bar{\mu})$ for quantitative comparison, 
each accompanied by dotted lines below and above representing the respective minimum and maximum of the empirical approximations.
Furthermore, the computed value of $V_{\MINSUB}$ in each instance is shown as a green dashed line in the left panels of Figure~\ref{fig: stochastic-FP-SG}.

It is observed that both $\bar{V}(\widehat{\mu}_t)$ and $\bar{\CW}_2(\widehat{\mu}_t, \bar{\mu})$ witness a sharp descent to near-optimal values that match the corresponding values attained by $\bar{\mu}$ (indicated by the yellow dashed lines) after a single iteration, which demonstrates that Algorithm~\ref{algo: concrete} efficiently converges to a near-optimal candidate measure for approximating the $\CW_2$-barycenter of $\nu_1, \dots, \nu_K$.
In particular, $\bar{V}(\widehat{\mu}_t)$ and $\bar{\CW}_2(\widehat{\mu}_t, \bar{\mu})$ remain close to $\bar{V}(\bar{\mu})$ and $\bar{\CW}_2(\bar{\mu}, \bar{\mu})$, respectively, for $t = 1, \dots, 9$, which demonstrates the stability of Algorithm~\ref{algo: concrete}.
The differences between $\bar{V}(\bar{\mu})$ and $V_{\MINSUB}$ (i.e., the gaps between the green and yellow dashed lines in the top-left and bottom-left panels of Figure~\ref{fig: stochastic-FP-SG}), as well as the values of $\bar{\CW}_2(\bar{\mu}, \bar{\mu})$, 
reflect the positive estimation biases incurred when estimating the $\CW_2$-distances between continuous measures using their empirical counterparts.

\begin{table}[t]
    \centering
    \begin{adjustbox}{max width=0.65\textwidth}
    \begin{tabular}{lrrrr}
        \hline
        \rule{0pt}{2.5ex}\multirow{2}{*}{\textbf{Algorithm}} & \multicolumn{2}{c}{\textbf{$\bar{V}(\widehat{\mu})$}} & \multicolumn{2}{c}{\textbf{$\bar{\mathcal{W}}_2(\widehat{\mu}, \bar{\mu})$}} \\
        \cmidrule(lr){2-3} \cmidrule(lr){4-5}
        & \textbf{Mean} & \textbf{(Min, Max)} & \textbf{Mean} & \textbf{(Min, Max)} \\
        \hline
        \multicolumn{5}{l}{\textbf{Problem instance: [SG-2d]}} \\
        \hline
        Algorithm~\ref{algo: concrete} (ours) & 524.66 & (516.21, 538.27) & 2.20 & (2.01, 2.65) \\
        \citet{fan2021scalable} & 526.85 & (518.04, 540.93) & 2.67 & (2.40, 2.93) \\
        \citet{korotin2022wasserstein} & 525.09 & (516.10, 538.44) & 2.37 & (2.06, 2.92) \\
        \citet{li2020continuous} & 696.46 & (676.47, 719.15) & 13.26 & (12.74, 13.85) \\
        \citet{cuturi2014fast} & 525.02 & (515.81, 539.19) & 2.31 & (2.15, 2.56) \\
        \citet{kim2025optimal} & 542.96 & (527.99, 557.62) & 4.79 & (3.72, 6.20) \\
        \hline
        \multicolumn{5}{l}{\textbf{Problem instance: [SG-10d]}} \\
        \hline
        Algorithm~\ref{algo: concrete} (ours) & 3935.74 & (3900.13, 3964.96) & 33.94 & (33.39, 34.61) \\
        \citet{fan2021scalable} & 4262.36 & (4226.30, 4284.14) & 38.49 & (38.30, 38.83) \\
        \citet{korotin2022wasserstein} & 15221.22 & (15157.87, 15278.74) & 111.49 & (110.80, 111.92) \\
        \citet{li2020continuous} & 4676.26 & (4621.78, 4728.71) & 43.44 & (42.87, 43.87) \\
        \citet{cuturi2014fast} & 3906.16 & (3872.13, 3937.60) & 33.48 & (33.09, 34.08) \\
        \hline
    \end{tabular}
    \end{adjustbox}
    \caption{Performances of Wasserstein barycenter algorithms in Experiments [SG-2d] and [SG-10d]. In [SG-2d], 
	$V_{\MINSUB}=521.17$,
	$\bar{V}(\bar{\mu}) = 524.63$, $\bar{\CW}_2(\bar{\mu}, \bar{\mu}) = 2.22$. 
	In [SG-10d], 
	$V_{\MINSUB}=2777.83$,
	$\bar{V}(\bar{\mu}) = 4065.36$,
	$\bar{\CW}_2(\bar{\mu}, \bar{\mu}) = 35.84$.}
	\label{tab: results-SG}
\end{table}

To further demonstrate the efficacy of our algorithm, we compare its performance with other benchmark Wasserstein barycenter algorithms on the same instances of [SG-2d] and [SG-10d]. 
Table~\ref{tab: results-SG} presents the evaluation results for both $\bar{V}(\widehat{\mu})$ and $\bar{\CW}_2(\widehat{\mu}, \bar{\mu})$ across different algorithms, where all reported scores are computed using 
the approximate $\CW_2$-barycenter candidate obtained at the last iteration of each algorithm.
The results show that our Algorithm~\ref{algo: concrete}, together with the algorithm of \citet{cuturi2014fast}, consistently performs well across the two problem instances: both algorithms achieve low values of $\bar{V}(\widehat{\mu})$ and $\bar{\CW}_2(\widehat{\mu},\bar{\mu})$. 
We defer a detailed discussion on the relation between our algorithm and the algorithm of \citet{cuturi2014fast} to Remark~\ref{remark: CD14}. 
Among the other benchmarks, the neural network based algorithms proposed by \citet{fan2021scalable} and \citet{korotin2022wasserstein} are also competitive in [SG-2d], yet they do not demonstrate strong performances in [SG-10d]. 
In particular, the substantial loss suffered by the algorithm of \citet{korotin2022wasserstein} in the latter instance seems to imply inferior parametrization by the generative neural network therein.  
Moreover, we emphasize that the algorithm of \citet{li2020continuous}, in fact, focuses on a regularized $\CW_2$-barycenter biased from the $\CW_2$-barycenter we study (see Table~\ref{tab: benchmark-summary} in Appendix~\ref{sapx: benchmark-descriptions}), which explains its relatively inferior scores in both metrics. 

\subsection{Experiment~2: subset posterior aggregation}\label{ssec: exp-subset-posterior-aggregation}
In large-scale Bayesian inference, 
one approach for scaling the inference task to very large datasets
is to split 
a large dataset into 
several subsets and produce a collection of subset posterior distributions (e.g.,\@ via Markov chain Monte Carlo methods), followed by an aggregation scheme to approximate the full-data posterior. 
Wasserstein barycenter provides a principled approach to perform this aggregation \citep{srivastava2015wasp}. 
Theoretical results show that the (equally-weighted) $\CW_2$-barycenter of subset posteriors asymptotically converges to the full-data posterior under appropriate conditions \citep{srivastava2018scalable}.
Therefore, one may consider the $\CW_2$-barycenter of subset posteriors as 
a proxy which accurately approximates the full-data posterior in large-data regimes.

In this experiment, we utilize an open-source dataset on hourly and daily count of rental bikes between years 2011 and 2012 in Washington, D.C.,
with covariates covering 8 features on weather and seasonal information.\footnote{See \url{https://archive.ics.uci.edu/dataset/275/bike+sharing+dataset}.} 
We focus on a Bayesian Poisson regression task for predicting hourly bike rentals that has been considered in previous works \citep{li2020continuous, fan2021scalable} for evaluating Wasserstein barycenter algorithms.
A summary of the dataset is provided in Table~\ref{tab:bike_summary} in Appendix~\ref{sapx: exp-details-bike-sharing}.
Regression coefficients (except the intercept term) are treated as an \mbox{8-dimensional} vector $\Bvartheta := (\vartheta_1, \dots, \vartheta_8) \in \R^8$, and the Markov chain Monte Carlo sampling of the posteriors is implemented using the $\mathtt{Stan}$ library \citep{carpenter2017stan}.
We refer to this experiment as [BS-8d].
Unlike [SG-2d] and [SG-10d] in Experiment~1, the subset posterior measures $\nu_1,\ldots,\nu_5$ here are 
exogenously generated from data 
without truncation,
and thus they do not necessarily satisfy Setting~\ref{sett: Caffarelli}.
As Algorithm~\ref{algo: concrete} is completely sample-driven, it can still be used to compute an approximate $\CW_2$-barycenter;
[BS-8d] is therefore intended primarily as an empirical demonstration of the applicability of Algorithm~\ref{algo: concrete} beyond the exact assumptions of Theorem~\ref{thm: Caffarelli-convergence} for convergence guarantee.

\subsubsection*{Experimental setup}
As mentioned, 
the $\CW_2$-barycenter of the subset posteriors of the regression coefficients serves as an accurate approximation of
the full-data posterior,
which we denote by $\bar{\mu}_{\FULL} \in \CP_{2, \AC}(\R^8)$.
In this experiment, we instead compute the full-data posterior 
$\bar{\mu}_{\FULL}$ and use it as an accurate approximation of the ground-truth $\CW_2$-barycenter.
To obtain the subset posterior distributions, we follow \citet{li2020continuous} to randomly split the data into $K = 5$ subsets which are equally-sized up to rounding.
The 5 subset posterior distributions, denoted by $\nu_1, \dots, \nu_5 \in \CP_{2, \AC}(\R^8)$, are therefore considered as the input measures of the problem instance in this experiment. 
A brief overview of the Bayesian Poisson regression model and the definition of the subset posteriors are provided in Appendix~\ref{sapx: exp-details-bike-sharing}.

We adopt the two metrics $\bar{V}(\widehat{\mu})$ and $\bar{\CW}_2(\widehat{\mu},\bar{\mu}_{\FULL})$, where $\bar{V}(\cdot)$ and $\bar{\CW}_2(\,\cdot\,, \cdot\,)$ have been introduced in Experiment~1, to quantitatively evaluate the quality of each approximate barycenter candidate $\widehat{\mu}$ computed by our algorithm 
or by the aforementioned benchmark algorithms.
We compare them with $\bar{V}(\bar{\mu}_{\FULL})$ and $\bar{\CW}_2(\bar{\mu}_{\FULL}, \bar{\mu}_{\FULL})$ accordingly.
Note that in contrast to Experiment~1,
we do not have access to an accurate approximate lower bound $V_{\MINSUB}$ for the barycenter functional in this experiment.

\begin{figure}[t]
    \centering
    \begin{subfigure}{0.45\textwidth}
        \centering
        \includegraphics[width=\linewidth]{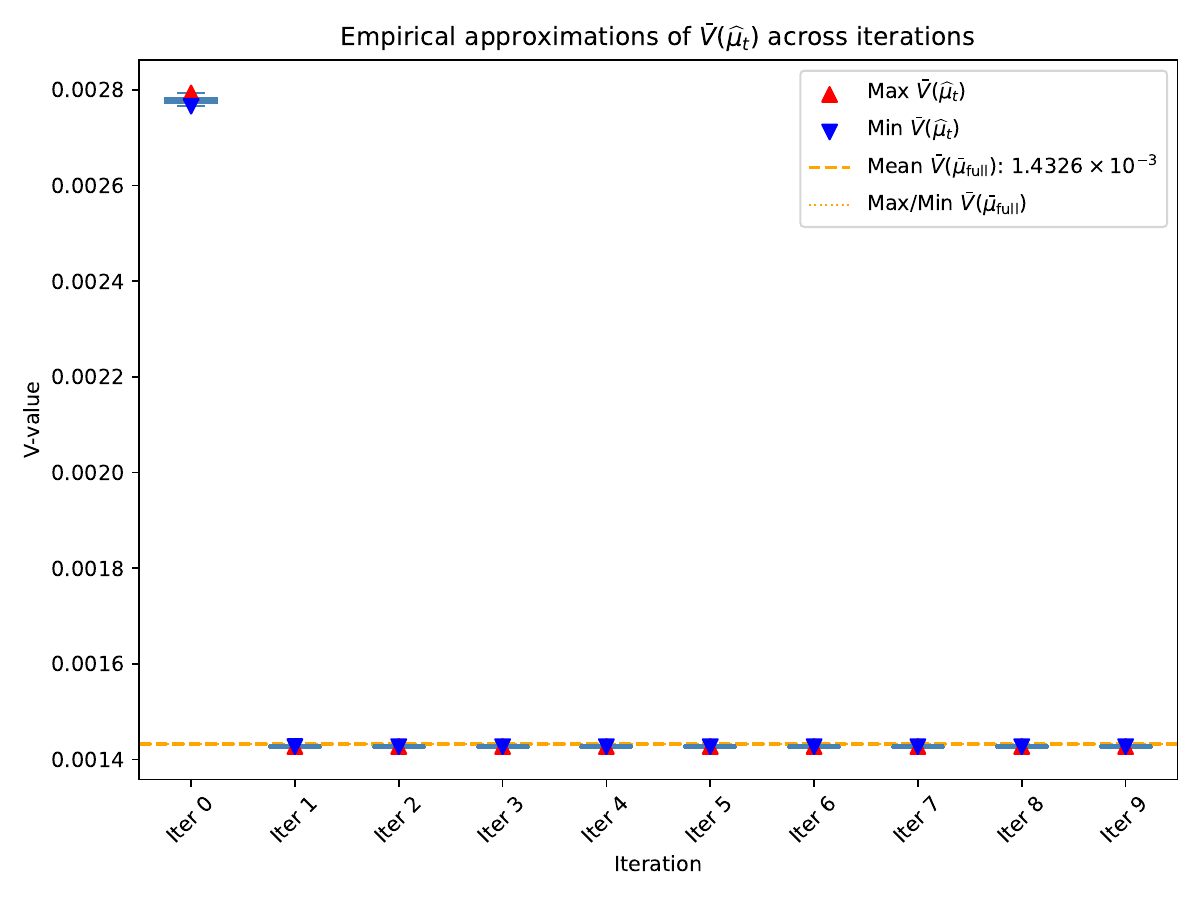}
    \end{subfigure}
    \begin{subfigure}{0.45\textwidth}
        \centering
        \includegraphics[width=\linewidth]{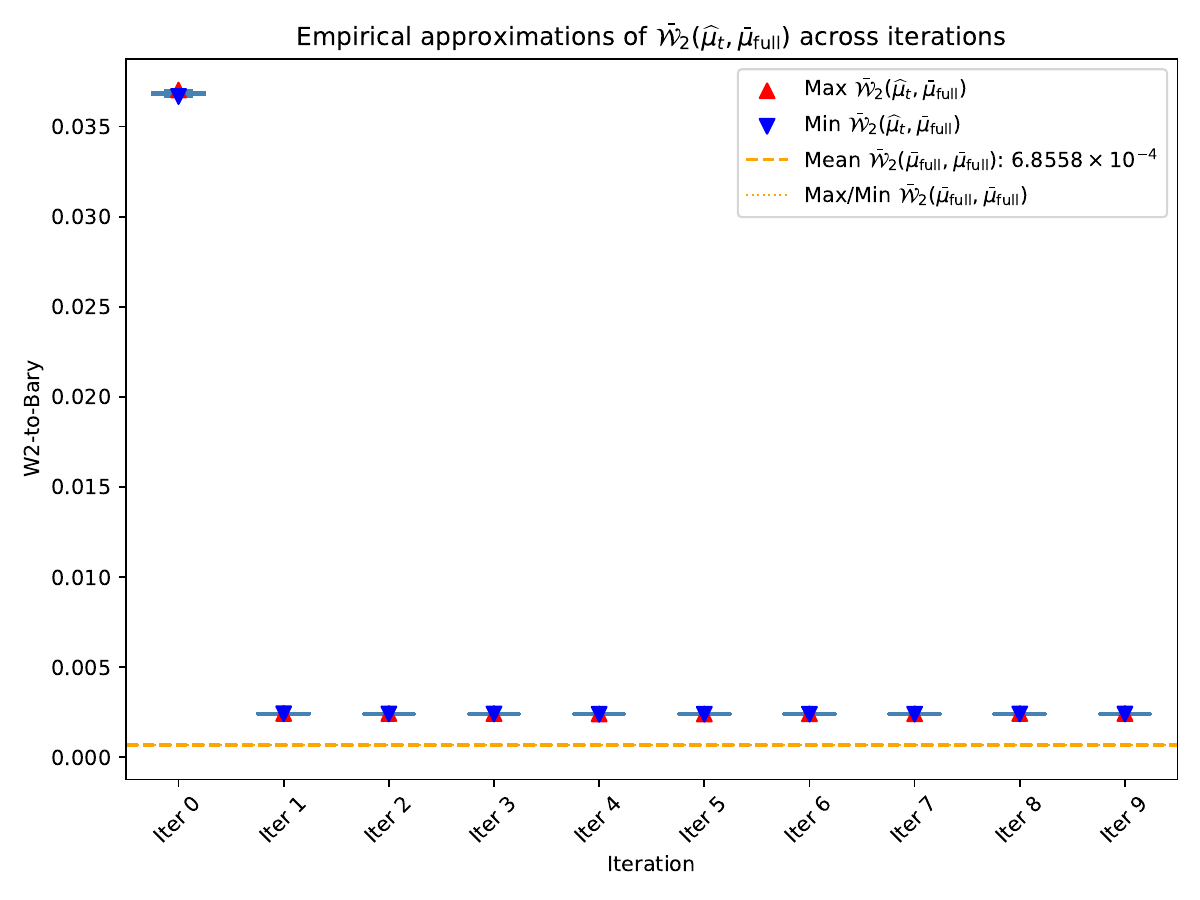}
    \end{subfigure}
    \caption{Box plots showing the empirical evaluations of $(\widehat{\mu}_t)_{t = 0:9}$ computed by Algorithm~\ref{algo: concrete} in [BS-8d]. \textbf{Left}: 
	values of $(\bar{V}(\widehat{\mu}_t))_{t = 0:9}$.
	\textbf{Right}: values of $(\bar{\CW}_2(\widehat{\mu}_t, \bar{\mu}_{\FULL}))_{t = 0:9}$.}\label{fig: stochastic-FP-BS-8d}
\end{figure}

\subsubsection*{Result analysis}
Figure~\ref{fig: stochastic-FP-BS-8d} presents the empirical performance of $(\widehat{\mu}_t)_{t = 0:9}$ computed by Algorithm~\ref{algo: concrete}, which includes the evaluation of $(\bar{V}(\widehat{\mu}_t))_{t = 0:9}$ and $(\bar{\CW}_2(\widehat{\mu}_t, \bar{\mu}_{\FULL}))_{t = 0:9}$ across iterations. 
The elements in the figure are the same as the ones in Figure~\ref{fig: stochastic-FP-SG} for Experiment~1.
It is observed that the evaluated values in both metrics witness a sharp descent to near-optimal values after a single iteration, which demonstrates the strong performance of our algorithm in approximating the $\CW_2$-barycenter of $\nu_1, \dots, \nu_5$.
In particular, $\bar{V}(\widehat{\mu}_t)$ and $\bar{\CW}_2(\widehat{\mu}_t, \bar{\mu}_{\FULL})$ remain close to $\bar{V}(\bar{\mu}_{\FULL})$ and $\bar{\CW}_2(\bar{\mu}_{\FULL}, \bar{\mu}_{\FULL})$, respectively, for $t = 1, \dots, 9$.
The gap ($\approx 1.74 \times 10^{-3}$) between $\bar{\CW}_2(\bar{\mu}_{\FULL}, \bar{\mu}_{\FULL})$ and $\bar{\CW}_2(\widehat{\mu}_t, \bar{\mu}_{\FULL})$ after convergence, which is shown in the right panel of Figure~\ref{fig: stochastic-FP-BS-8d}, is also observed when evaluating other benchmark algorithms (see Table~\ref{tab: results-BS-8d}). Therefore, it likely captures the approximation error of the full-data posterior $\bar{\mu}_{\FULL}$ to the exact $\CW_2$-barycenter of the subset posteriors.

Table~\ref{tab: results-BS-8d} presents the empirical evaluations for both $\bar{V}(\widehat{\mu})$ and $\bar{\CW}_2(\widehat{\mu}, \bar{\mu}_{\FULL})$ across different algorithms in \mbox{[BS-8d]}, where all reported scores are computed using 
the approximate $\CW_2$-barycenter candidate obtained at the last iteration of
each algorithm.
We observe that, except for the algorithm of \citet{li2020continuous}, all other algorithms have achieved comparable performances in approximating the underlying $\CW_2$-barycenter, with Algorithm~\ref{algo: concrete} and 
the algorithm of \citet{cuturi2014fast} providing slightly better scores in both metrics.
This is consistent with the numerical results obtained in Experiment~1.

\begin{table}[t]
    \centering
    \begin{adjustbox}{max width=0.7\textwidth}
    \begin{tabular}{lrrrr}
        \hline
        \rule{0pt}{2.5ex}\multirow{2}{*}{\textbf{Algorithm}} & \multicolumn{2}{c}{\textbf{$\bar{V}(\widehat{\mu})$ ($\times 10^{-3}$)}} & \multicolumn{2}{c}{\textbf{$\bar{\mathcal{W}}_2(\widehat{\mu}, \bar{\mu}_{\FULL})$ ($\times 10^{-3}$)}} \\
        \cmidrule(lr){2-3} \cmidrule(lr){4-5}
        & \textbf{Mean} & \textbf{(Min, Max)} & \textbf{Mean} & \textbf{(Min, Max)} \\
        \hline
        Algorithm~\ref{algo: concrete} (ours) & 1.427217 & (1.426985, 1.427468) & 2.415107 & (2.399016, 2.430971) \\
        \citet{fan2021scalable} & 1.429260 & (1.429048, 1.429543) & 2.798973 & (2.770246, 2.821899) \\
        \citet{korotin2022wasserstein} & 1.430236 & (1.429998, 1.430493) & 2.743783 & (2.731388, 2.758314) \\
        \citet{li2020continuous} & 4.669375 & (4.621498, 4.749356) & 5.660722 & (5.618368, 5.729884) \\
        \citet{cuturi2014fast} & 1.427236 & (1.427002, 1.427487) & 2.410925 & (2.401364, 2.419805) \\
        \hline
    \end{tabular}
    \end{adjustbox}
    \caption{Performances of Wasserstein barycenter algorithms in Experiment [BS-8d], where $\bar{V}(\bar{\mu}_{\FULL}) = 1.4326 \times 10^{-3}$ and $\bar{\CW}_2(\bar{\mu}_{\FULL}, \bar{\mu}_{\FULL}) = 6.8558 \times 10^{-4}$.}\label{tab: results-BS-8d}
\end{table}

\subsection{Practical guidelines and additional comments}\label{ssec: numerics-additional-comments}
We conclude this section by presenting a list of practical guidelines for implementing our proposed stochastic fixed-point algorithm, as well as its advantages when compared with other methods.
We also discuss in Remark~\ref{remark: CD14} the connection between our algorithm and the algorithm of \citet{cuturi2014fast}, given their comparable performances in our numerical experiments.

\subsubsection*{Practical guidelines}
We present below a list of practical considerations for implementing Algorithm~\ref{algo: concrete} with the modified entropic OT map estimator, which are based on empirical observations from our experiments.
	\begin{itemize}
		\item \emph{About the choice of the truncation set $\CX_{\widehat{R}_{t}}$.} 
		Despite that the truncation step
		in Line~\ref{alglin: concrete-truncation} of Algorithm~\ref{algo: concrete}
		is required to guarantee the theoretical convergence properties of Algorithm~\ref{algo: concrete} in Theorem~\ref{thm: Caffarelli-convergence},
		our empirical observations suggest that the choice of $\CX_{\widehat{R}_{t}}$ has little effect on the performance of our algorithm. 
		Therefore, in practice, one could set the truncation indices $(\widehat{R}_t)_{t \in \N_0}$ to be sufficiently large such that the effect of rejection sampling is negligible; see Remark~\ref{rmk: radius-samplesize-ge} and our discussion in Section~\ref{ssec: Caffarelli-convergence-proof}. 
		
		\item \emph{About the implementation of Sinkhorn's algorithm.}
		Sinkhorn's algorithm can be efficiently implemented using modern software libraries for optimal transport, such as the Python Optimal Transport  ($\mathtt{POT}$) library \citep{flamary2021pot}, the Optimal Transport Tools ($\mathtt{OTT}$) library \citep{cuturi2022optimal}, and the $\mathtt{geomloss}$ library \citep{feydy2020geometric}.
		Moreover, it admits highly parallelizable implementations on GPUs; see, e.g., \citep[Section~4.3]{peyre2019computational} for a detailed exposition on the numerical aspects of Sinkhorn's algorithm, and \citep[Section~3.3]{feydy2020geometric} for state-of-the-art techniques in speeding up Sinkhorn's algorithm. 
		In our numerical experiments, we utilized the $\mathtt{geomloss}$ library \citep{feydy2020geometric} to implement the Sinkhorn step. 

		\item \emph{About the choices of the sample sizes $\widehat{M}_{t,k},\widehat{N}_{t,k}$.}
		Since the convergence of Algorithm~\ref{algo: concrete} requires sufficiently large sample sizes, in practice one could always choose the sample sizes as large as permitted by the computational capacity of the hardware and software used to implement Sinkhorn's algorithm.\footnote{For example, the implementation of Sinkhorn's algorithm in the $\mathtt{geomloss}$ library performs well with ${\sim\!10^5}$ samples per measure; see \citep[page~125]{feydy2020geometric}.}
		Moreover, one may consider gradually increasing the sample sizes over iterations in order to reduce the computational time spent in earlier iterations. 

		\item \emph{About the choices of the hyperparameters $\widehat{\ITheta}_{t,k}$.}
		The definition of the entropic OT problem (\ref{eqn: EOT}) suggests that 
		the entropic regularization parameter should roughly scale quadratically with respect to the support diameters of the source and target measures.
		Moreover, we remark that the performance of our algorithm may deteriorate when the regularization parameter is set very close to zero relative to the scale of the measures' supports, due to the numerical instability of Sinkhorn's algorithm in such cases (see, e.g.,\@ \citep[Section~4.4]{peyre2019computational}).

		\item \emph{About the strong convexity modification step.}
		The addition of the modification term $T_{\STRONGLYCONVEX}[\psi]$ in Proposition~\ref{prop: OTmap-estimator-entropic}
		is required for theoretical purposes, as it guarantees that the resulting entropic OT map estimator satisfies the required shape condition to be admissible within our framework. 
    	From a practical perspective, this modification has a negligible impact: when the truncation indices $(\widehat{R}_t)_{t\in\N_0}$ across iterations are chosen sufficiently large, the additional term does not materially affect the empirical behavior of the algorithm.

		\item \emph{About the termination criteria.}
		Empirical observations from our numerical experiments revealed that our algorithm typically requires only a handful of iterations to achieve near-optimality.
		In practice, a maximum of approximately 10 iterations is typically sufficient for Algorithm~\ref{algo: concrete}.
	\end{itemize}

\subsubsection*{Additional comments}
We would like to highlight two advantages of our proposed algorithm in terms of its computational efficiency and interpretability. 

First, regarding computational efficiency, our algorithm has been empirically witnessed to produce near-optimal solutions within a small number of iterations across our numerical experiments.
In particular, a single outer iteration of Algorithm~\ref{algo: concrete}, which already suffices for near-optimal accuracy, completes in well under 3 hours across all three problem instances; by contrast, the neural network based benchmarks require roughly 6 to 46 hours to complete their training; see Table~\ref{tab: benchmark-summary} in Appendix~\ref{sapx: benchmark-descriptions} for the complete run-time and the computational-budget comparison.
Moreover, while our algorithm can be executed without the need for high-performance computing hardware, it permits hardware acceleration and distributed computations for potential numerical improvements since the OT map estimation is driven by the efficient Sinkhorn's algorithm.

Second, in contrast to many prevalent Wasserstein barycenter algorithms that heavily rely on generative neural networks, our estimator-based algorithm provides ease in parametrizing the underlying Brenier potentials and OT maps. 
In particular, it circumvents the need to tune sophisticated hyperparameters and
avoids the model over-parametrization issue that is potentially present in neural networks.
In fact, the composition of the weighted sum of the OT map estimators, namely $\big[\sum_{k=1}^K w_k\widehat{T}_{t,k}\big]\circ \cdots \circ \big[\sum_{k=1}^K w_k\widehat{T}_{1,k}\big](\cdot)$, can be considered as a structure-aware generative model for approximating the $\CW_2$-barycenter, providing better interpretability.

\begin{remark}\label{remark: CD14}
	One may observe that the algorithm of \citet{cuturi2014fast}, which in our experiments was implemented via the built-in function provided in the $\mathtt{POT}$ library\footnote{$\mathtt{ot.lp.free\_support\_barycenter}$ in \url{https://pythonot.github.io/gen_modules/ot.lp.html}} and was executed with $10^4$ support atoms, 
	has also been competitive across problem instances. 
	In fact, without the line search and weight optimization steps, it has been recognized that each iteration of the algorithm of \citet{cuturi2014fast} essentially amounts to a discrete implementation of the fixed-point iteration in \eqref{eqn: G-iteration} proposed by \citet{alvarez2016fixed}.
	Therefore, it can be considered as a heuristic counterpart of our stochastic fixed-point algorithm without rigorous convergence guarantees.
	However, the algorithm scales poorly with the support size and computes only a single discrete probability measure as an approximate $\CW_2$-barycenter.
	Thus, it prohibits direct sampling from a continuous approximate $\CW_2$-barycenter, which limits its usage in many practical applications.
	In terms of empirical run-time, the algorithm of \citet{cuturi2014fast} with $10^4$ support atoms and a single, already near-optimal outer iteration of Algorithm~\ref{algo: concrete} are broadly comparable across our three problem instances: Algorithm~\ref{algo: concrete} is faster in [SG-2d] and [SG-10d] ($0.11$\,h and $1.50$\,h, versus $5.85$\,h and $4.35$\,h, respectively), but slower in [BS-8d] ($2.81$\,h versus $1.60$\,h); see Table~\ref{tab: benchmark-summary} in Appendix~\ref{sapx: benchmark-descriptions} for details.
\end{remark}

\section{Conclusion and limitations}\label{sec: conclusion-limitation-guideline}

In this paper, we have developed an estimator-driven stochastic fixed-point framework for approximately computing the Wasserstein barycenter of continuous, non-parametric probability measures. 
Under the assumption that the errors of the approximation steps in the scheme are suitably controlled,
we have rigorously established its almost sure convergence and identified 
sufficient conditions to guarantee a convergence rate that is geometric in the number of iterations. 
Building on this foundation, we have subsequently proposed a provably convergent and computationally tractable stochastic algorithm that admits input measures 
satisfying Caffarelli-type regularity conditions,
together with a modified entropic OT map estimator that is computationally efficient. 
Moreover, we have further developed a novel procedure that synthetically generates benchmark instances with approximately known barycenters to enable quantitative comparison across algorithms. 
Lastly, we have performed
numerical experiments on both synthetic and real-world datasets to demonstrate the strong computational efficiency, estimation accuracy, and sampling flexibility of our approach.

Overall, our framework bridges theoretical guarantees and practical implementation, while also leaving several important unanswered questions for future research. 
First, our analysis does not fully characterize conditions under which the $G$-operator admits a unique fixed-point, which is closely related to the geodesic non-convexity of the Wasserstein barycenter functional. 
Identifying verifiable and practically meaningful sufficient conditions for uniqueness remains a longstanding and fundamental open problem.
Second, the proposed algorithm (Algorithm~\ref{algo: concrete}) should be viewed as one concrete instance within the broader stochastic fixed-point framework in the template of Algorithm~\ref{algo: abstract}. 
Developing more efficient, adaptive, or instance-specific implementations that better exploit problem structures possibly beyond our Caffarelli-type setting constitutes another promising direction for future research.

\section*{Acknowledgements}
ZC gratefully acknowledges the financial support from the INSEAD PhD Fellowship, and the support from NTU Singapore under the URECA Undergraduate Research Programme. 
AN and QX gratefully acknowledge the financial support by the MOE AcRF Tier~2 Grant \textit{MOE-T2EP20222-0013}.

\appendix

\section{Omitted proofs}\label{apx: omitted-proofs}

\subsection{Omitted proofs in Section~\ref{ssec: main-convergence-conditions}}\label{sapx: proof-main-convergence-conditions}

Before proving Proposition~\ref{prop: special-case-1d},
let us first establish the following intermediate results.

\begin{lemma}\label{lem: KDE-properties}
	Let $\nu\in\CP_{2,\AC}(\R)$ satisfy $\support(\nu)=[\underline{a},\overline{a}]$ for $-\infty<\underline{a}<\overline{a}<\infty$,
	and let 
	$f_{\nu}$ denote the density function of $\nu$, 
	where we assume that 
	there exists $\zeta\ge 1$ such that 
	$\zeta^{-1}\le f_{\nu}(x)\le \zeta$ $\forall x\in [\underline{a},\overline{a}]$.
	Let $F_{\nu}:[\underline{a},\overline{a}]\to[0,1]$ denote the distribution function of $\nu$, i.e.,
	$F_{\nu}(x):=\int_{\underline{a}}^{x}f_{\nu}(z)\DIFFX{z}$ $\forall x\in [\underline{a},\overline{a}]$.
	Subsequently, let $(\Omega,\CF,\PROB)$ be a probability space.
	For any $n\in\N$,
	let 
	$X_1,\ldots,X_n:\Omega\to\R$ be independent and identically distributed random variables with law $\nu$, i.e., 
	$X_i\sharp\PROB=\nobreak\nu$ $\forall 1\le i\le n$.
	Moreover, let $\kappa:\R\to(0,\infty)$ be continuous and satisfy 
	$\kappa(x)\ge\kappa(x')$ whenever $|x'|\ge |x|$,
	$\int_{\R}\kappa(x)\DIFFX{x}=\nobreak 1$,
	$\int_{\R}x^2\kappa(x)\DIFFX{x}<\nobreak\infty$,
	and let us define $F_{\widetilde{\nu}_{n,h}}:[\underline{a},\overline{a}]\to[0,1]$ for all $h>0$ as follows:
	\begin{align*}
		F_{\widetilde{\nu}_{n,h}}(x):= \frac{1}{n}\sum_{i=1}^{n}
		\frac{{\displaystyle\int_{\underline{a}}^{x}}\kappa\Big(\frac{z-X_i}{h}\Big)\DIFFX{z}}{{\displaystyle\int_{\underline{a}}^{\overline{a}}}\kappa\Big(\frac{z-X_i}{h}\Big)\DIFFX{z}}
		\qquad \forall x\in [\underline{a},\overline{a}],\; \forall h>0.
	\end{align*}
	Then, the following statements hold.
	\begin{enumerate}[label=(\roman*),beginpenalty=10000]
		\item\label{lems: KDE-properties-regularity}
		For any $n\in\N$ and any $h>0$, 
		it holds $\PROB$-almost surely that 
		$F_{\widetilde{\nu}_{n,h}}:[\underline{a},\overline{a}]\to[0,1]$ is a diffeomorphism.
		
		\item\label{lems: KDE-properties-convexity}
		For any $n\in\N$ and any $h>0$,
		it holds $\PROB$-almost surely that
		there exist $0<\underline{\lambda}\le\overline{\lambda}<\infty$ 
		satisfying $\underline{\lambda}\le F_{\widetilde{\nu}_{n,h}}'(x)\le \overline{\lambda}$ $\forall x\in [\underline{a},\overline{a}]$.

		\item\label{lems: KDE-properties-invcdf}
		The following bound holds as $n\to\infty$, $h\to 0$:
		\begin{align*}
			\EXP\bigg[\int_{0}^{1}\big(F_{\widetilde{\nu}_{n,h}}^{-1}(u)-F_{\nu}^{-1}(u)\big)^2\DIFFX{u}\bigg]
			&=O\big(n^{-\frac{1}{2}}+h^2\big).
		\end{align*}
		
		\item\label{lems: KDE-properties-concentrate}
		There exists a $\sigma\big((X_i)_{i=1:n}\big)$-measurable subset 
		$E_{n,h}\subset\Omega$ that depends on $n$ and $h$, 
		which satisfies:
		\begin{align}
			1-\PROB[E_{n,h}]
			&=O\big(n^{-1}h^{-2}\big),\label{eqn: KDE-properties-prob} \\
			\EXP\Bigg[\sup\Bigg\{\frac{\big(F^{-1}_{\widetilde{\nu}_{n,h}}(u)-F^{-1}_{\widetilde{\nu}_{n,h}}(u')\big)^2}{(u-u')^2}:u,u'\in[0,1],\; u\ne u'\Bigg\}\INDI_{E_{n,h}}\Bigg]
			&=O(1),\label{eqn: KDE-properties-densmin}\\
			\EXP\bigg[\int_{[\underline{a},\overline{a}]}\big(F_{\widetilde{\nu}_{n,h}}(x)-F_{\nu}(x)\big)^2\DIFFM{\nu}{\DIFF x}\INDI_{E_{n,h}}\bigg]
			&=O\big(n^{-1}+h^2\big),\label{eqn: KDE-properties-cdf}
		\end{align}
		as $n\to\infty$, $h\to\nobreak0$.
	\end{enumerate}
	All constant terms omitted by the big-$O$ notations in the above statements depend on 
	$\underline{a}$, $\overline{a}$, $\zeta$, $\kappa(\cdot)$ and 
	do not depend on $n$ and $h$.
\end{lemma}

\begin{proof}[Proof of Lemma~\ref{lem: KDE-properties}]
	Before we begin,
	note that the property that $\kappa(x)\ge\kappa(x')$ whenever $|x'|\ge|x|$
	forces $\kappa(\cdot)$ to be symmetric around~0,
	i.e.,
	$\kappa(-x)=\kappa(x)$ $\forall x\in\R$.
	Let us define $\CK(x):=\int_{-\infty}^x\kappa(z)\DIFFX{z}$ $\forall x\in\R$.
	Observe that $\CK$ is strictly increasing, and that
	$\CK(x)\in(0,1)$, $\CK(x)+\CK(-x)=\nobreak1$ $\forall x\in\R$.
	Because of the property that $\int_{\R}x^2\kappa(x)\DIFFX{x}<\infty$,
	we get 
	\begin{align}
		\int_{0}^{\infty}{x}\big(1-\CK(x)\big)\DIFFX{x}
		=\int_{0}^{\infty}\int_{0}^{\infty}x\kappa(z)\INDI_{\{x\le z\}}\DIFFX{z}\DIFFX{x}
		=\int_{0}^{\infty}\frac{z^2}{2}\kappa(z)\DIFFX{z}<\infty, 
		\label{eqn: KDE-properties-proof-kernel-pos1}
	\end{align}
	and we get via a similar derivation that
	\begin{align}
		\int_{-\infty}^{0}{-x}\CK(x)\DIFFX{x}&<\infty. \label{eqn: KDE-properties-proof-kernel-neg1}
	\end{align}
	Moreover, observe that 
	the function $(0,\infty)\ni h\mapsto \min_{x\in[\underline{a},\overline{a}]}\Big\{\CK\big(\frac{\overline{a}-x}{h}\big)-\CK\big(\frac{\underline{a}-x}{h}\big)\Big\}\in(0,\infty)$ is non-increasing,
	and thus we have by the property $\CK(x)+\CK(-x)=1$ $\forall x\in\R$ that
	\begin{align*}
		\min_{x\in[\underline{a},\overline{a}]}\Bigg\{\CK\bigg(\frac{\overline{a}-x}{h}\bigg)-\CK\bigg(\frac{\underline{a}-x}{h}\bigg)\Bigg\} 
		&\ge \min_{x\in[\underline{a},\overline{a}]}\Bigg\{\CK\bigg(\frac{2(\overline{a}-x)}{\overline{a}-\underline{a}}\bigg)-\CK\bigg(\frac{2(\underline{a}-x)}{\overline{a}-\underline{a}}\bigg)\Bigg\} \\
		&= \min_{y\in[0,1]}\big\{\CK(y)+\CK(2-y) - 1\big\}
		\qquad \forall h\in\big(0,{\textstyle\frac{\overline{a}-\underline{a}}{2}}\big].
	\end{align*}
	Subsequently, the property that $\kappa(x)\ge\kappa(x')$ whenever $|x'|\ge|x|$
	implies that 
	$[0,1]\ni y\mapsto \CK(y)+\CK(2-\nobreak y) -\nobreak 1\in\R$
	is non-decreasing on $[0,1]$,
	and hence
	\begin{align}
		\min_{x\in[\underline{a},\overline{a}]}\Bigg\{\CK\bigg(\frac{\overline{a}-x}{h}\bigg)-\CK\bigg(\frac{\underline{a}-x}{h}\bigg)\Bigg\}\ge \CK(2)-\CK(0)>0 
		\qquad \forall h\in\big(0,{\textstyle\frac{\overline{a}-\underline{a}}{2}}\big].
		\label{eqn: KDE-properties-proof-kernel-denominator}
	\end{align}

	Throughout this proof,
	for any $n\in\N$,
	let $\widehat{\nu}_n:=\frac{1}{n}\sum_{i=1}^{n}\delta_{X_i}\in\CP_2(\R)$ denote the empirical measure arising from the samples $(X_i)_{i=1:n}$,
	and let us define $\overline{\nu}_{n,h},\widetilde{\nu}_{n,h}\in\CP_{2,\AC}(\R)$
	for any $n\in\N$, $h>0$ through their density functions
	\begin{align*}
		f_{\overline{\nu}_{n,h}}(x)&:=\frac{1}{nh}\sum_{i=1}^{n}\kappa\bigg(\frac{x-X_i}{h}\bigg)
		\hspace{63.9pt}\qquad \forall x\in\R,\; \forall n\in\N,\; \forall h>0,\\
		f_{\widetilde{\nu}_{n,h}}(x)&:=\frac{1}{nh}\sum_{i=1}^{n}\frac{\kappa\Big(\frac{x-X_i}{h}\Big)}{\CK\Big(\frac{\overline{a}-X_i}{h}\Big) - \CK\Big(\frac{\underline{a}-X_i}{h}\Big)}
		\qquad \forall x\in[\underline{a},\overline{a}],\; \forall n\in\N, \;\forall h>0.
	\end{align*}
	Using the properties that $0<\CK\big(\frac{\underline{a}-X_i}{h}\big)<\CK\big(\frac{\overline{a}-X_i}{h}\big)<1$ $\forall h>0$, $\forall 1\le i\le n$ $\PROB$-almost surely,
	it holds that
	\begin{align}
		f_{\widetilde{\nu}_{n,h}}(x)\ge f_{\overline{\nu}_{n,h}}(x)>\nobreak0
		\qquad\forall x\in[\underline{a},\overline{a}],\; \forall n\in\N,\; \forall h>0,\; \PROB\text{-a.s.}
		\label{eqn: KDE-properties-proof-density-bounds}
	\end{align}
	Next, let 
	$F_{\overline{\nu}_{n,h}}:\R\to[0,1]$
	denote the distribution function of $\overline{\nu}_{n,h}$,
	i.e.,
	\begin{align}
		F_{\overline{\nu}_{n,h}}(x)
		:=\int_{-\infty}^{x}f_{\overline{\nu}_{n,h}}(z)\DIFFX{z}
		=\frac{1}{n}\sum_{i=1}^{n}\CK\bigg(\frac{x-X_i}{h}\bigg) \qquad \forall x\in\R,\; \forall n\in\N,\; \forall h>0,
		\label{eqn: KDE-properties-proof-untruncated-cdf}
	\end{align}
	and observe that 
	$F_{\widetilde{\nu}_{n,h}}$ is 
	the distribution function of $\widetilde{\nu}_{n,h}$ for all $n\in\N$ and all $h>\nobreak0$ 
	$\PROB$-almost surely.
	For any $n\in\N$ and for any $h>\nobreak 0$,
	since $F_{\widetilde{\nu}_{n,h}}'=f_{\widetilde{\nu}_{n,h}}$
	is continuous and positive on $[\underline{a},\overline{a}]$,
	it holds $\PROB$-almost surely that
	$0<\nobreak \inf_{x\in[\underline{a},\overline{a}]}\big\{F_{\widetilde{\nu}_{n,h}}'(x)\big\}\le\sup_{x\in[\underline{a},\overline{a}]}\big\{F_{\widetilde{\nu}_{n,h}}'(x)\big\}<\nobreak\infty$,
	and thus $F_{\widetilde{\nu}_{n,h}}$ is 
	a diffeomorphism.
	This completes the proofs of statements~\ref{lems: KDE-properties-regularity} and \ref{lems: KDE-properties-convexity}.

	To prove statement~\ref{lems: KDE-properties-invcdf},
	let us
	fix arbitrary $n\in\N$ and $h\in\big(0,\frac{\overline{a}-\underline{a}}{2}\big]$,
	and
	apply a classical result about one-dimensional OT
	(see, e.g., \citep[Proposition~1.18]{chewi2024statistical})
	to get 
	\begin{align}
		\EXP\bigg[\int_{0}^{1}\big(F_{\widetilde{\nu}_{n,h}}^{-1}(u)-F_{\nu}^{-1}(u)\big)^2\DIFFX{u}\bigg]
		&= \EXP\big[\CW_2(\widetilde{\nu}_{n,h},\nu)^2\big]
		\le 2\EXP\big[\CW_2(\widehat{\nu}_n,\widetilde{\nu}_{n,h})^2\big]
		+2\EXP\big[\CW_2(\widehat{\nu}_n,\nu)^2\big].
		\label{eqn: KDE-properties-proof-invcdf-decomp}
	\end{align}
	On the one hand,
	for $i=1,\ldots,n$,
	let $\eta_{n,h,i}\in\CP_{2,\AC}(\R)$ be defined through its density function $f_{\eta_{n,h,i}}$ below:
	\begin{align*}
		f_{\eta_{n,h,i}}(x):=\frac{\kappa\Big(\frac{x-X_i}{h}\Big)}{h\Big(\CK\Big(\frac{\overline{a}-X_i}{h}\Big) - \CK\Big(\frac{\underline{a}-X_i}{h}\Big)\Big)} \qquad \forall x\in[\underline{a},\overline{a}],\; \forall 1\le i\le n.
	\end{align*}
	Using (\ref{eqn: KDE-properties-proof-kernel-denominator}), it thus holds that 
	\begin{align}
		\begin{split}
			\int_{[\underline{a},\overline{a}]}(x-X_i)^2\DIFFM{\eta_{n,h,i}}{\DIFF x}
			&=\frac{1}{h\Big(\CK\Big(\frac{\overline{a}-X_i}{h}\Big) - \CK\Big(\frac{\underline{a}-X_i}{h}\Big)\Big)}\int_{\underline{a}}^{\overline{a}} (x-X_i)^2 \kappa\bigg(\frac{x-X_i}{h}\bigg)\DIFFX{x}\\
			&\le \frac{h^2}{\CK(2)-\CK(0)}\int_{-\infty}^{\infty}y^2\kappa(y)\DIFFX{y}=O(h^2) \qquad \forall 1\le i\le n,\; \PROB\text{-a.s.}
		\end{split}
		\label{eqn: KDE-properties-proof-invcdf-step1-component}
	\end{align}
	Next, let us define 
	$\pi_{n,h}:=\frac{1}{n}\sum_{i=1}^{n}\delta_{X_i}\otimes \eta_{n,h,i}\in\CP(\R\times\R)$,
	where $\delta_{X_i}\otimes \eta_{n,h,i}\in\CP(\R\times\R)$ denotes the product measure of $\delta_{X_i}$ and $\eta_{n,h,i}$.
	One checks that 
	$\pi_{n,h}\in\Pi(\widehat{\nu}_n,\widetilde{\nu}_{n,h})$,
	and it therefore holds by (\ref{eqn: KDE-properties-proof-invcdf-step1-component}) that
	\begin{align}
		\begin{split}
			\CW_2(\widehat{\nu}_n,\widetilde{\nu}_{n,h})^2
			&\le \int_{[\underline{a},\overline{a}]\times[\underline{a},\overline{a}]}(x-y)^2\DIFFM{\pi_{n,h}}{\DIFF x,\DIFF y}\\
			&=\frac{1}{n}\sum_{i=1}^n \int_{[\underline{a},\overline{a}]}(X_i-y)^2\DIFFM{\eta_{n,h,i}}{\DIFF y} 
			=O(h^2) \qquad \PROB\text{-a.s.}
		\end{split}
		\label{eqn: KDE-properties-proof-invcdf-step1}
	\end{align}
	On the other hand,
	\citep[Theorem~1]{fournier2015rate} 
	shows that 
	\begin{align}
		\EXP\big[\CW_2(\widehat{\nu}_n,\nu)^2\big]=O\big(n^{-\frac{1}{2}}\big),
		\label{eqn: KDE-properties-proof-invcdf-step2}
	\end{align}
	and thus statement~\ref{lems: KDE-properties-invcdf} follows from combining
	(\ref{eqn: KDE-properties-proof-invcdf-decomp}),
	(\ref{eqn: KDE-properties-proof-invcdf-step1}), and 
	(\ref{eqn: KDE-properties-proof-invcdf-step2}).

	It remains to prove statement~\ref{lems: KDE-properties-concentrate}.
	To that end, let us fix arbitrary $n\in\N$ and $h\in\big(0,\frac{\overline{a}-\underline{a}}{2}\big]$,
	define $J_h:=\big\lfloor\frac{\overline{a}-\underline{a}}{h}\big\rfloor$,
	and define 
	$I_j:=\big[\underline{a}+\frac{j-1}{J_h}(\overline{a}-\underline{a}),\underline{a}+\frac{j}{J_h}(\overline{a}-\underline{a})\big]$
	$\forall 1\le j\le J_h$.
	We thus get 
	\begin{align}
		\frac{\overline{a}-\underline{a}}{h}
		\ge J_h 
		> \frac{\overline{a}-\underline{a}}{h} - 1
		\ge \frac{\overline{a}-\underline{a}}{h} - \frac{\overline{a}-\underline{a}}{2h}
		= \frac{\overline{a}-\underline{a}}{2h}.
		\label{eqn: KDE-properties-proof-concentrate-binnum}
	\end{align}
	Moreover, since $\zeta^{-1}\le f_{\nu}(x)\le \zeta$ $\forall x\in [\underline{a},\overline{a}]$,
	we get 
	\begin{align}
		\zeta^{-1}h
		\le \frac{\zeta^{-1}(\overline{a}-\underline{a})}{J_h}
		\le \nu(I_j)
		\le \frac{\zeta(\overline{a}-\underline{a})}{J_h}
		< 2\zeta h
		\qquad \forall 1\le j\le J_h.
		\label{eqn: KDE-properties-proof-concentrate-binsize}
	\end{align}
	In the following, let us define the random variables 
	$Z_j:=\frac{1}{n}\sum_{i=1}^{n}\INDI_{I_j}(X_i)$ $\forall 1\le j\le J_h$,
	and define the set 
	$E_{n,h}\in\sigma\big((X_i)_{i=1:n}\big)$ as follows:
	\begin{align*}
		E_{n,h}:=\bigcap_{j=1}^{J_h}\Big\{{\textstyle\frac{1}{2}}\zeta^{-1}h<Z_j<3\zeta h\Big\}.
	\end{align*}
	It follows from (\ref{eqn: KDE-properties-proof-concentrate-binsize})
	that 
	$\EXP[Z_j]=\nu(I_j)\in\big[\zeta^{-1}h,2\zeta h\big)$,
	$\VAR[Z_j]=\frac{1}{n}\nu(I_j)\big(1-\nu(I_j)\big)<\frac{2\zeta h}{n}$
	for $j=\nobreak1,\ldots,J_h$,
	which yield 
	$\big\{Z_j\le {\textstyle\frac{1}{2}\zeta^{-1} h}\big\}
	\subseteq \Big\{Z_j\le \EXP[Z_j]-\sqrt{{\textstyle\frac{nh}{8\zeta^3}}\VAR[Z_j]}\Big\}$ and 
	$\big\{Z_j\ge {\textstyle 3\zeta h}\big\}
	\subseteq \Big\{Z_j\ge \EXP[Z_j]+\sqrt{{\textstyle\frac{nh}{8\zeta^3}}\VAR[Z_j]}\Big\}$
	for $j=1,\ldots,J_h$.
	Using these relations, 
	we can bound $1-\PROB[E_{n,h}]$
	via Chebyshev's inequality and (\ref{eqn: KDE-properties-proof-concentrate-binnum}) as follows:
	\begin{align*}
		1-\PROB[E_{n,h}]
		&\le \sum_{j=1}^{J_h}\Big(\PROB\Big[Z_j\le {\textstyle\frac{1}{2}\zeta^{-1} h}\Big] + \PROB\Big[Z_j\ge {\textstyle 3\zeta h}\Big]\Big)
		\le \sum_{j=1}^{J_h}\PROB\bigg[\big|Z_j-\EXP[Z_j]\big|\ge \sqrt{{\textstyle\frac{nh}{8\zeta^3}}\VAR[Z_j]}\bigg]\\
		&\le \frac{8\zeta^3J_h}{nh}
		\le \frac{8\zeta^3 (\overline{a}-\underline{a})}{nh^2}
		=O\big(n^{-1}h^{-2}\big).
	\end{align*}
	This proves (\ref{eqn: KDE-properties-prob}).

	To prove (\ref{eqn: KDE-properties-densmin}),
	observe from (\ref{eqn: KDE-properties-proof-density-bounds}) that 
	\begin{align}
		\big(F_{\widetilde{\nu}_{n,h}}^{-1}\big)'(u) 
		= \frac{1}{f_{\widetilde{\nu}_{n,h}}\big(F_{\widetilde{\nu}_{n,h}}^{-1}(u)\big)}
		\le \bigg(\inf_{x\in[\underline{a},\overline{a}]}\big\{f_{\overline{\nu}_{n,h}}(x)\big\}\bigg)^{-1} \quad \forall u \in [0,1],\; \PROB\text{-a.s.}
		\label{eqn: KDE-properties-proof-densmin-derivative}
	\end{align}
	Moreover, since (\ref{eqn: KDE-properties-proof-concentrate-binnum})
	guarantees 
	$|x-y|\le \frac{\overline{a}-\underline{a}}{J_h}<2h$
	$\forall x,y\in I_j$, $\forall 1\le j\le J_h$,
	and since 
	$\kappa(\cdot)$ is positive everywhere and satisfies 
	$\kappa(x)\ge \kappa(x')$ whenever $|x'|\ge|x|$,
	we can derive the following lower bound for 
	$f_{\overline{\nu}_{n,h}}$:
	\begin{align*}
		\inf_{x\in I_j}\big\{f_{\overline{\nu}_{n,h}}(x)\big\}
		&\ge \frac{1}{nh}\sum_{i=1}^{n}\INDI_{I_j}(X_i)\inf_{x\in I_j}\bigg\{\kappa\bigg(\frac{x-X_i}{h}\bigg)\bigg\} \\
		&\ge \frac{\kappa(2)}{nh}\sum_{i=1}^{n}\INDI_{I_j}(X_i)=\frac{\kappa(2)Z_j}{h} 
		\qquad\forall 1\le j\le J_h,\; \PROB\text{-a.s.}
	\end{align*}
	Consequently, we get from (\ref{eqn: KDE-properties-proof-densmin-derivative}) and the mean value theorem that 
	\begin{align*}
		\hspace{20pt}&\hspace{-20pt}\sup\Bigg\{\frac{\big(F^{-1}_{\widetilde{\nu}_{n,h}}(u)-F^{-1}_{\widetilde{\nu}_{n,h}}(u')\big)^2}{(u-u')^2}:u,u'\in[0,1],\; u\ne u'\Bigg\}\INDI_{E_{n,h}}\\
		&\le \bigg(\inf_{x\in[\underline{a},\overline{a}]}\big\{f_{\overline{\nu}_{n,h}}(x)\big\}\bigg)^{-2}\INDI_{E_{n,h}}
		\le \bigg(\min_{1\le j\le J_h}\bigg\{\frac{\kappa(2)Z_j}{h}\bigg\}\bigg)^{-2}\INDI_{E_{n,h}}\le \frac{4\zeta^2}{\kappa(2)^2}
		\qquad \PROB\text{-a.s.}
	\end{align*}
	Taking expectations on both sides of the above inequality proves 
	(\ref{eqn: KDE-properties-densmin}).

	To prove 
	(\ref{eqn: KDE-properties-cdf}),
	let us first bound the left-hand side of (\ref{eqn: KDE-properties-cdf})
	by Fubini's theorem and a bias-variance decomposition:
	\begin{align}
		\begin{split}
			\hspace{20pt}&\hspace{-20pt}
			\EXP\bigg[\int_{[\underline{a},\overline{a}]}\big(F_{\widetilde{\nu}_{n,h}}(x)-F_{\nu}(x)\big)^2\DIFFM{\nu}{\DIFF x}\INDI_{E_{n,h}}\bigg]\\
			&\le 2\EXP\bigg[\int_{\underline{a}}^{\overline{a}}\big(F_{\overline{\nu}_{n,h}}(x)-F_{\nu}(x)\big)^2f_{\nu}(x)\DIFFX{x}\bigg]
			+ 2\EXP\bigg[\sup_{x\in[\underline{a},\overline{a}]}\Big\{\big(F_{\widetilde{\nu}_{n,h}}(x)-F_{\overline{\nu}_{n,h}}(x)\big)^2\Big\}\INDI_{E_{n,h}}\bigg]\\
			&= 2\int_{\underline{a}}^{\overline{a}}\Big(\EXP\big[F_{\overline{\nu}_{n,h}}(x)\big]-F_{\nu}(x)\Big)^2f_{\nu}(x)
			\DIFFX{x}
			+ 2\int_{\underline{a}}^{\overline{a}}\Big(\EXP\big[F_{\overline{\nu}_{n,h}}(x)^2\big]- \EXP\big[F_{\overline{\nu}_{n,h}}(x)\big]^2\Big)f_{\nu}(x)
			\DIFFX{x} \\
			&\qquad + 2\EXP\bigg[\sup_{x\in[\underline{a},\overline{a}]}\Big\{\big(F_{\widetilde{\nu}_{n,h}}(x)-F_{\overline{\nu}_{n,h}}(x)\big)^2\Big\}\INDI_{E_{n,h}}\bigg].
		\end{split}
		\label{eqn: KDE-properties-proof-cdf-biasvar}
	\end{align}
	In the following, we will bound the three resulting terms on the right-hand side of (\ref{eqn: KDE-properties-proof-cdf-biasvar}) separately.

	Firstly,
	using (\ref{eqn: KDE-properties-proof-untruncated-cdf}) 
	and integration by parts,
	we get 
	\begin{align}
		\begin{split}
			\EXP\big[F_{\overline{\nu}_{n,h}}(x)\big]
			&=\EXP\bigg[\CK\bigg(\frac{x-X_1}{h}\bigg)\bigg]
			=\int_{\underline{a}}^{\overline{a}}\CK\bigg(\frac{x-z}{h}\bigg)f_{\nu}(z)\DIFFX{z}\\
			&= \CK\bigg(\frac{x-\overline{a}}{h}\bigg) 
			+ \frac{1}{h}\int_{\underline{a}}^{\overline{a}}\kappa\bigg(\frac{x-z}{h}\bigg)F_{\nu}(z)\DIFFX{z}\\
			&= \CK\bigg(\frac{x-\overline{a}}{h}\bigg) 
			+ \int_{\frac{x-\overline{a}}{h}}^{\frac{x-\underline{a}}{h}}\kappa(y)F_{\nu}(x-hy)\DIFFX{y}
			\qquad \forall x\in[\underline{a},\overline{a}].
		\end{split}
		\label{eqn: KDE-properties-proof-cdf-bias-step1}
	\end{align}
	Moreover, observe that 
	\begin{align}
		\begin{split}
			F_{\nu}(x)
			&=\CK\bigg(\frac{x-\overline{a}}{h}\bigg)F_{\nu}(x)
				+ \bigg(1-\CK\bigg(\frac{x-\underline{a}}{h}\bigg)\bigg)F_{\nu}(x)
				+ \bigg(\CK\bigg(\frac{x-\underline{a}}{h}\bigg)-\CK\bigg(\frac{x-\overline{a}}{h}\bigg)\bigg)F_{\nu}(x)\\
			&=\CK\bigg(\frac{x-\overline{a}}{h}\bigg)F_{\nu}(x)
				+ \bigg(1-\CK\bigg(\frac{x-\underline{a}}{h}\bigg)\bigg)F_{\nu}(x)
				+ \int_{\frac{x-\overline{a}}{h}}^{\frac{x-\underline{a}}{h}}\kappa(y)F_{\nu}(x) \DIFFX{y}
			\qquad \forall x\in[\underline{a},\overline{a}].
		\end{split}
		\label{eqn: KDE-properties-proof-cdf-bias-step2}
	\end{align}
	Combining (\ref{eqn: KDE-properties-proof-cdf-bias-step1})
	and (\ref{eqn: KDE-properties-proof-cdf-bias-step2}),
	it thus holds that 
	\begin{align}
		\begin{split}
			\Big|\EXP\big[F_{\overline{\nu}_{n,h}}(x)\big] - F_{\nu}(x)\Big|
			&\le \CK\bigg(\frac{x-\overline{a}}{h}\bigg)\big(1 - F_{\nu}(x)\big)
			+ \bigg(1-\CK\bigg(\frac{x-\underline{a}}{h}\bigg)\bigg)F_{\nu}(x) \\
			&\qquad + \int_{\frac{x-\overline{a}}{h}}^{\frac{x-\underline{a}}{h}}\kappa(y)\big|F_{\nu}(x-hy)-F_{\nu}(x)\big|\DIFFX{y} 
			\qquad \forall x\in[\underline{a},\overline{a}].
		\end{split}
		\label{eqn: KDE-properties-proof-cdf-bias-step3}
	\end{align}
	Using the properties that 
	$\zeta^{-1}\le f_{\nu}(x)\le \zeta$ $\forall x\in[\underline{a},\overline{a}]$,
	$F_{\nu}(\underline{a})=0$,
	$F_{\nu}(\overline{a})=1$,
	and using 
	(\ref{eqn: KDE-properties-proof-kernel-pos1}), 
	(\ref{eqn: KDE-properties-proof-kernel-neg1}),
	we get 
	\begin{align}
		\begin{split}
			\int_{\underline{a}}^{\overline{a}}\CK\bigg(\frac{x-\overline{a}}{h}\bigg)\big(1 - F_{\nu}(x)\big)f_{\nu}(x)\DIFFX{x}
			&\le \zeta \int_{\underline{a}}^{\overline{a}}\CK\bigg(\frac{x-\overline{a}}{h}\bigg)\big(1 - F_{\nu}(x)\big)\DIFFX{x}\\
			&= \zeta h \int_{\frac{\underline{a}-\overline{a}}{h}}^{0}\CK(y)\big(F_{\nu}(\overline{a}) - F_{\nu}(\overline{a}+hy)\big)\DIFFX{y}\\
			&\le \zeta^2h^2\int_{-\infty}^{0}{-y}\CK(y)\DIFFX{y}
			=O(h^2),
		\end{split}\label{eqn: KDE-properties-proof-cdf-bias-step4}\allowdisplaybreaks\\
		\begin{split}
			\int_{\underline{a}}^{\overline{a}}\bigg(1-\CK\bigg(\frac{x-\underline{a}}{h}\bigg)\bigg)F_{\nu}(x)f_{\nu}(x) \DIFFX{x}
			&\le \zeta \int_{\underline{a}}^{\overline{a}}\bigg(1-\CK\bigg(\frac{x-\underline{a}}{h}\bigg)\bigg)F_{\nu}(x) \DIFFX{x}\\
			&= \zeta h\int_{0}^{\frac{\overline{a}-\underline{a}}{h}}\big(1-\CK(y)\big)\big(F_{\nu}(\underline{a}+hy)-F_{\nu}(\underline{a})\big) \DIFFX{y}\\
			&\le \zeta^2h^2\int_{0}^{\infty}y\big(1-\CK(y)\big)\DIFFX{y}
			=O(h^2),
		\end{split}\label{eqn: KDE-properties-proof-cdf-bias-step5}\allowdisplaybreaks\\
		\begin{split}
			\int_{\frac{x-\overline{a}}{h}}^{\frac{x-\underline{a}}{h}}\kappa(y)\big|F_{\nu}(x-hy)-F_{\nu}(x)\big|\DIFFX{y}
			&\le \zeta h\int_{-\infty}^{\infty}|y|\kappa(y)\DIFFX{y}
			=O(h) 
			\qquad\qquad \forall x\in[\underline{a},\overline{a}].
		\end{split}\label{eqn: KDE-properties-proof-cdf-bias-step6}
	\end{align}
	Subsequently, 
	using the properties that 
	$\CK\big(\frac{x-\overline{a}}{h}\big),
	\CK\big(\frac{x-\underline{a}}{h}\big),
	F_{\nu}(x)\in(0,1)$ $\forall x\in[\underline{a},\overline{a}]$,
	and combining 
	(\ref{eqn: KDE-properties-proof-cdf-bias-step3})--(\ref{eqn: KDE-properties-proof-cdf-bias-step6}),
	we get
	\begin{align}
		\begin{split}
			\hspace{20pt}&\hspace{-20pt}\!\int_{\underline{a}}^{\overline{a}}\Big(\EXP\big[F_{\overline{\nu}_{n,h}}(x)\big]-F_{\nu}(x)\Big)^2f_{\nu}(x)\DIFFX{x}\label{eqn: KDE-properties-proof-cdf-bias}\\
			&\le 3\int_{\underline{a}}^{\overline{a}}\CK\bigg(\frac{x-\overline{a}}{h}\bigg)^2\big(1 - F_{\nu}(x)\big)^2f_{\nu}(x)\DIFFX{x} 
			+ 3\int_{\underline{a}}^{\overline{a}}\bigg(1-\CK\bigg(\frac{x-\underline{a}}{h}\bigg)\bigg)^2F_{\nu}(x)^2 f_{\nu}(x) \DIFFX{x} \\
			&\qquad+ 3\int_{\underline{a}}^{\overline{a}}\bigg(\int_{\frac{x-\overline{a}}{h}}^{\frac{x-\underline{a}}{h}}\kappa(y)\big|F_{\nu}(x-hy)-F_{\nu}(x)\big|\DIFFX{y}\bigg)^2f_{\nu}(x)\DIFFX{x}\\
			&\le 3\int_{\underline{a}}^{\overline{a}}\CK\bigg(\frac{x-\overline{a}}{h}\bigg)\big(1 - F_{\nu}(x)\big)f_{\nu}(x)\DIFFX{x} 
			+ 3\int_{\underline{a}}^{\overline{a}}\bigg(1-\CK\bigg(\frac{x-\underline{a}}{h}\bigg)\bigg)F_{\nu}(x) f_{\nu}(x) \DIFFX{x} \\
			&\qquad+ 3\int_{\underline{a}}^{\overline{a}}\bigg(\int_{\frac{x-\overline{a}}{h}}^{\frac{x-\underline{a}}{h}}\kappa(y)\big|F_{\nu}(x-hy)-F_{\nu}(x)\big|\DIFFX{y}\bigg)^2f_{\nu}(x)\DIFFX{x}\\
			&= O(h^2).
		\end{split}
	\end{align}

	Secondly,
	notice that 
	\begin{align*}
		\EXP\big[F_{\overline{\nu}_{n,h}}(x)^2\big]- \EXP\big[F_{\overline{\nu}_{n,h}}(x)\big]^2
		&=\frac{1}{n}\Bigg(\EXP\bigg[\CK\bigg(\frac{x-X_1}{h}\bigg)^2\bigg]-\EXP\bigg[\CK\bigg(\frac{x-X_1}{h}\bigg)\bigg]^2\Bigg)
		\le \frac{1}{n},
	\end{align*}
	and we hence get
	\begin{align}
		\int_{\underline{a}}^{\overline{a}}\Big(\EXP\big[F_{\overline{\nu}_{n,h}}(x)^2\big]- \EXP\big[F_{\overline{\nu}_{n,h}}(x)\big]^2\Big)f_{\nu}(x) \DIFFX{x}
		= O\big(n^{-1}\big).
		\label{eqn: KDE-properties-proof-cdf-var}
	\end{align}

	Thirdly,
	observe that 
	\begin{align*}
		F_{\widetilde{\nu}_{n,h}}(x)&=\frac{1}{n}\sum_{i=1}^{n}\frac{\CK\Big(\frac{x-X_i}{h}\Big)-\CK\Big(\frac{\underline{a}-X_i}{h}\Big)}{\CK\Big(\frac{\overline{a}-X_i}{h}\Big)-\CK\Big(\frac{\underline{a}-X_i}{h}\Big)} \qquad \forall x\in[\underline{a},\overline{a}], \; \PROB\text{-a.s.},
	\end{align*}
	and it hence follows from (\ref{eqn: KDE-properties-proof-kernel-denominator})
	that 
	\begin{align}
		\begin{split}
			\hspace{15pt}&\hspace{-15pt}
			\sup_{x\in[\underline{a},\overline{a}]}\Big\{\big|F_{\widetilde{\nu}_{n,h}}(x)-F_{\overline{\nu}_{n,h}}(x)\big|\Big\}\\
			&\le \frac{1}{n}\sum_{i=1}^n \!\left[\left(\frac{1}{\CK\Big(\frac{\overline{a}-X_i}{h}\Big)-\CK\Big(\frac{\underline{a}-X_i}{h}\Big)}-1\right)\sup_{x\in[\underline{a},\overline{a}]}\Bigg\{\CK\bigg(\frac{x-X_i}{h}\bigg)\Bigg\}
			+ \frac{\CK\Big(\frac{\underline{a}-X_i}{h}\Big)}{\CK\Big(\frac{\overline{a}-X_i}{h}\Big)-\CK\Big(\frac{\underline{a}-X_i}{h}\Big)} \right]\\
			&\le \frac{2}{\big(\CK(2)-\CK(0)\big)n}\sum_{i=1}^{n}\bigg[\CK\bigg(\frac{\underline{a}-X_i}{h}\bigg) + \bigg(1-\CK\bigg(\frac{\overline{a}-X_i}{h}\bigg)\bigg)\bigg]
			\qquad\qquad\qquad\qquad\qquad\; \PROB\text{-a.s.}
		\end{split}
		\label{eqn: KDE-properties-proof-cdf-extrabias-uniform}
	\end{align}
	For $j=1,\ldots,J_h$,
	and for all $x\in I_j$,
	it holds by (\ref{eqn: KDE-properties-proof-concentrate-binnum}) that 
	$\frac{\underline{a}-x}{h}\le -\frac{\overline{a}-\underline{a}}{J_h h}(j-1)\le -(j-1)$
	and 
	$\frac{\overline{a}-x}{h}\ge \frac{\overline{a}-\underline{a}}{J_h h}(J_h-j)\ge J_h-j$.
	Therefore, on the one hand, we have
	\begin{align*}
		\frac{1}{n}\sum_{i=1}^{n}\CK\bigg(\frac{\underline{a}-X_i}{h}\bigg)
		&\le\frac{1}{n}\sum_{i=1}^{n}\sum_{j=1}^{J_h}\INDI_{I_j}(X_i)\CK\bigg(\frac{\underline{a}-X_i}{h}\bigg)\\
		&\le \frac{1}{n}\sum_{i=1}^{n}\sum_{j=1}^{J_h}\INDI_{I_j}(X_i)\CK\big({-(j-1)}\big)
		=\sum_{j=1}^{J_h}Z_j\CK\big({-(j-1)}\big)
		\qquad \PROB\text{-a.s.},
	\end{align*}
	which implies that 
	\begin{align}
		\Bigg(\frac{1}{n}\sum_{i=1}^{n}\CK\bigg(\frac{\underline{a}-X_i}{h}\bigg)\Bigg)\INDI_{E_{n,h}}
		&\le 3\zeta h\sum_{j=1}^{J_h}\CK\big({-(j-1)}\big)= 3\zeta h \int_{-\infty}^{0}{\textstyle\sum_{j=1}^{J_h}}\INDI_{(-\infty,{-(j-1)}]}(z)\kappa(z)\DIFFX{z}\nonumber\\
		&\le 3\zeta h\int_{-\infty}^{0}\big(|z|+1\big)\kappa(z)\DIFFX{z}
		=O(h)
		\qquad\qquad\qquad\qquad\qquad \PROB\text{-a.s.}
		\label{eqn: KDE-properties-proof-cdf-extrabias-sum1}
	\end{align}
	On the other hand, 
	we have
	\begin{align*}
		\frac{1}{n}\sum_{i=1}^{n}\bigg(1-\CK\bigg(\frac{\overline{a}-X_i}{h}\bigg)\bigg)
		&\le \frac{1}{n}\sum_{i=1}^{n}\sum_{j=1}^{J_h}\INDI_{I_j}(X_i)\bigg(1-\CK\bigg(\frac{\overline{a}-X_i}{h}\bigg)\bigg)\\
		&\le \frac{1}{n}\sum_{i=1}^{n}\sum_{j=1}^{J_h}\INDI_{I_j}(X_i)\big(1-\CK(J_h-j)\big)
		=\sum_{j=1}^{J_h}Z_j\big(1-\CK(J_h-j)\big)
		\qquad \PROB\text{-a.s.},
	\end{align*}
	which yields
	\begin{align}
		\Bigg(\frac{1}{n}\sum_{i=1}^{n}\bigg(1-\CK\bigg(\frac{\overline{a}-X_i}{h}\bigg)\bigg)\Bigg)\INDI_{E_{n,h}}
		&\le 3\zeta h\sum_{j=1}^{J_h}\big(1-\CK(J_h-j)\big)
		=3\zeta h \int_{0}^{\infty}{\textstyle\sum_{j=1}^{J_h}}\INDI_{[J_h-j,\infty)}(z)\kappa(z)\DIFFX{z} \nonumber\\
		&\le 3\zeta h\int_{0}^{\infty}\big(|z|+1\big)\kappa(z)\DIFFX{z}
		=O(h) 
		\qquad\qquad\qquad\quad\;\; \PROB\text{-a.s.}
		\label{eqn: KDE-properties-proof-cdf-extrabias-sum2}
	\end{align}
	Subsequently, combining 
	(\ref{eqn: KDE-properties-proof-cdf-extrabias-uniform})--(\ref{eqn: KDE-properties-proof-cdf-extrabias-sum2})
	leads to
	\begin{align}
		\sup_{x\in[\underline{a},\overline{a}]}\Big\{\big|F_{\widetilde{\nu}_{n,h}}(x)-F_{\overline{\nu}_{n,h}}(x)\big|\Big\}\INDI_{E_{n,h}}
		=O(h)
		\qquad \PROB\text{-a.s.}
		\label{eqn: KDE-properties-proof-cdf-extrabias}
	\end{align}
	Finally, 
	squaring (\ref{eqn: KDE-properties-proof-cdf-extrabias}) then taking the expectation,
	and combining the resulting bound with 
	(\ref{eqn: KDE-properties-proof-cdf-biasvar}),
	(\ref{eqn: KDE-properties-proof-cdf-bias}), and
	(\ref{eqn: KDE-properties-proof-cdf-var})
	completes the proof of (\ref{eqn: KDE-properties-cdf}).
	The proof is now complete.
\end{proof}

\begin{lemma}\label{lem: OT-KDE}
	Let $\mu,\nu\in\CP_{2,\AC}(\R)$ satisfy $\support(\mu)=[\underline{a}_{\mu},\overline{a}_{\mu}]$, 
	$\support(\nu)=[\underline{a}_{\nu},\overline{a}_{\nu}]$,
	where 
	$-\infty<\underline{a}_{\mu}<\overline{a}_{\mu}<\nobreak\infty$,
	$-\infty<\underline{a}_{\nu}<\overline{a}_{\nu}<\nobreak\infty$.
	Let
	$f_{\mu}$ and
	$f_{\nu}$ 
	denote the density functions of $\mu$ and $\nu$, respectively, 
	where we assume that 
	there exist $\zeta_{\mu}\ge1$, $\zeta_{\nu}\ge1$ such that 
	$\zeta_{\mu}^{-1}\le f_{\mu}(x)\le \zeta_{\mu}$ $\forall x\in [\underline{a}_{\mu},\overline{a}_{\mu}]$,
	$\zeta_{\nu}^{-1}\le\nobreak f_{\nu}(x)\le\nobreak \zeta_{\nu}$ $\forall x\in\nobreak [\underline{a}_{\nu},\overline{a}_{\nu}]$.
	Next, let $(\Omega,\CF,\PROB)$ be a probability space.
	For any $m\in\N$ and $n\in\N$,
	let 
	$X_1,\ldots,X_m,Y_1,\ldots,Y_n:\Omega\to\R$ be independent random variables 
	where $X_1,\ldots,X_m$ have law $\mu$ 
	and $Y_1,\ldots,Y_n$ have law $\nu$,
	i.e.,
	$X_i\sharp\PROB=\mu$ $\forall 1\le i\le m$,
	$Y_j\sharp\PROB=\nu$ $\forall 1\le j\le n$.
	Moreover, let 
	$\kappa_{\mu},\kappa_{\nu}:\R\to(0,\infty)$ be continuous and satisfy
	$\kappa_{\mu}(x)\ge\kappa_{\mu}(x')$, $\kappa_{\nu}(x)\ge\kappa_{\nu}(x')$ whenever $|x'|\ge|x|$,
	$\int_{\R}\kappa_{\mu}(x)\DIFFX{x}=\int_{\R}\kappa_{\nu}(x)\DIFFX{x}=\nobreak 1$,
	$\int_{\R}x^2\kappa_{\mu}(x)\DIFFX{x}<\nobreak\infty$,
	$\int_{\R}x^2\kappa_{\nu}(x)\DIFFX{x}<\nobreak\infty$.
	Furthermore, for any $h_{\mu}>0$, $h_{\nu}>0$,
	let 
	$F_{\widetilde{\mu}_{m,h_{\mu}}}:[\underline{a}_{\mu},\overline{a}_{\mu}]\to[0,1]$ and 
	$F_{\widetilde{\nu}_{n,h_{\nu}}}:[\underline{a}_{\nu},\overline{a}_{\nu}]\to[0,1]$ be defined as follows:
	\begin{align*}
		F_{\widetilde{\mu}_{m,h_{\mu}}}(x)&:=\frac{1}{m}\sum_{i=1}^{m}
		\frac{{\displaystyle\int_{\underline{a}_{\mu}}^{x}}\kappa_{\mu}\Big(\frac{z-X_i}{h_{\mu}}\Big)\DIFFX{z}}{{\displaystyle\int_{\underline{a}_{\mu}}^{\overline{a}_{\mu}}}\kappa_{\mu}\Big(\frac{z-X_i}{h_{\mu}}\Big)\DIFFX{z}}
		\qquad \forall x\in [\underline{a}_{\mu},\overline{a}_{\mu}],\; \forall h_{\mu}>0,\allowdisplaybreaks\\
		F_{\widetilde{\nu}_{n,h_{\nu}}}(x)&:=\frac{1}{n}\sum_{j=1}^{n}
		\frac{{\displaystyle\int_{\underline{a}_{\nu}}^{x}}\kappa_{\nu}\Big(\frac{z-Y_j}{h_{\nu}}\Big)\DIFFX{z}}{{\displaystyle\int_{\underline{a}_{\nu}}^{\overline{a}_{\nu}}}\kappa_{\nu}\Big(\frac{z-Y_j}{h_{\nu}}\Big)\DIFFX{z}}
		\hspace{6.4pt}\qquad \forall x\in [\underline{a}_{\nu},\overline{a}_{\nu}],\; \forall h_{\nu}>0.
	\end{align*}
	Then, the following statements hold.
	\begin{enumerate}[label=(\roman*),beginpenalty=10000]
		\item\label{lems: OT-KDE-regularity}
		For any $m\in\N$, $n\in\N$, $h_{\mu}>0$, $h_{\nu}>0$,
		it holds $\PROB$-almost surely that
		$F_{\widetilde{\nu}_{n,h_{\nu}}}:[\underline{a}_{\nu},\overline{a}_{\nu}]\to[0,1]$ 
		and 
		$F_{\widetilde{\nu}_{n,h_{\nu}}}^{-1}\circ F_{\widetilde{\mu}_{m,h_{\mu}}}:[\underline{a}_{\mu},\overline{a}_{\mu}]\to[\underline{a}_{\nu},\overline{a}_{\nu}]$
		are diffeomorphisms.
		
		\item\label{lems: OT-KDE-convexity}
		For any $m\in\N$, $n\in\N$, $h_{\mu}>0$, $h_{\nu}>0$,
		it holds $\PROB$-almost surely that
		there exist $0<\underline{\lambda}\le\overline{\lambda}<\infty$
		satisfying
		$\underline{\lambda}\le \big(F_{\widetilde{\nu}_{n,h_{\nu}}}^{-1}\circ F_{\widetilde{\mu}_{m,h_{\mu}}}\big)'(x)\le \overline{\lambda}$ 
		$\forall x\in[\underline{a}_{\mu},\overline{a}_{\mu}]$.
		
		\item\label{lems: OT-KDE-consistency}
		Let $T^{\mu}_{\nu}:[\underline{a}_{\mu},\overline{a}_{\mu}]\to[\underline{a}_{\nu},\overline{a}_{\nu}]$
		denote the OT map from $\mu$ to $\nu$.
		Then, it holds that 
		\begin{align*}
			\hspace{35pt}\EXP\Big[\big\|F_{\widetilde{\nu}_{n,h_{\nu}}}^{-1}\circ F_{\widetilde{\mu}_{m,h_{\mu}}}-T^{\mu}_{\nu}\big\|^2_{\CL^2(\mu)}\Big] 
			= O\big(m^{-1}+n^{-\frac{1}{2}}+h_{\mu}^2+h_{\nu}^2+m^{-1}h_{\mu}^{-2}+n^{-1}h_{\nu}^{-2}\big),
		\end{align*}
		as $m\to\infty$, $n\to\infty$, $h_{\mu}\to0$, $h_{\nu}\to0$.
		The constant term omitted by the big-$O$ notation above depends only on 
		$\underline{a}_{\mu}$, $\overline{a}_{\mu}$, $\zeta_{\mu}$, $\kappa_{\mu}(\cdot)$,
		$\underline{a}_{\nu}$, $\overline{a}_{\nu}$, $\zeta_{\nu}$, $\kappa_{\nu}(\cdot)$,
		and does not depend on $m$, $n$, $h_{\mu}$, $h_{\nu}$.
		In particular, choosing $h_{\mu}=O\big(m^{-\frac{1}{4}}\big)$, 
		$h_{\nu}=O\big(n^{-\frac{1}{4}}\big)$
		leads to 
		$\EXP\Big[\big\|F_{\widetilde{\nu}_{n,h_{\nu}}}^{-1}\circ F_{\widetilde{\mu}_{m,h_{\mu}}}-T^{\mu}_{\nu}\big\|^2_{\CL^2(\mu)}\Big]=O\big(m^{-\frac{1}{2}}+n^{-\frac{1}{2}}\big)$.
	\end{enumerate}
\end{lemma}

\begin{proof}[Proof of Lemma~\ref{lem: OT-KDE}]
	Throughout this proof,
	let us denote the distribution functions of $\mu$ and $\nu$ by 
	$F_{\mu}:[\underline{a}_{\mu},\overline{a}_{\mu}]\to[0,1]$ and 
	$F_{\nu}:[\underline{a}_{\nu},\overline{a}_{\nu}]\to[0,1]$,
	i.e.,
	$F_{\mu}(x):=\int_{\underline{a}_{\mu}}^xf_{\mu}(z)\DIFFX{z}$ 
	$\forall x\in[\underline{a}_{\mu},\overline{a}_{\mu}]$,
	$F_{\nu}(x):=\int_{\underline{a}_{\nu}}^xf_{\nu}(z)\DIFFX{z}$ 
	$\forall x\in[\underline{a}_{\nu},\overline{a}_{\nu}]$.
	Applying statements~\ref{lems: KDE-properties-regularity}, 
	\ref{lems: KDE-properties-convexity}, 
	\ref{lems: KDE-properties-concentrate} 
	of Lemma~\ref{lem: KDE-properties}
	with respect to 
	$\nu\leftarrow \mu$,
	$\underline{a}\leftarrow\underline{a}_{\mu}$,
	$\overline{a}\leftarrow\overline{a}_{\mu}$,
	$\zeta\leftarrow\zeta_{\mu}$,
	$\kappa(\cdot)\leftarrow \kappa_{\mu}(\cdot)$,
	$n\leftarrow m$,
	$h\leftarrow h_{\mu}$
	yields for any $m\in\N$ and any $h_{\mu}>0$ that 
	$F_{\widetilde{\mu}_{m,h_{\mu}}}$ is \mbox{$\PROB$-almost} surely a diffeomorphism,
	that 
	it holds $\PROB$-almost surely that
	there exist $0<\underline{\lambda}_{\mu}\le\overline{\lambda}_{\mu}<\infty$ 
	satisfying 
	$\underline{\lambda}_{\mu}\le F'_{\widetilde{\mu}_{m,h_{\mu}}}(x)\le \overline{\lambda}_{\mu}$ $\forall x\in[\underline{a}_{\mu},\overline{a}_{\mu}]$,
	and that 
	there exists $E_{\mu,m,h_{\mu}}\in\sigma\big((X_i)_{i=1:m}\big)$
	which depends on $m$ and $h_{\mu}$
	such that 
	\begin{align}
		1 - \PROB[E_{\mu,m,h_{\mu}}] 
		&=O\big(m^{-1}h_{\mu}^{-2}\big), \label{eqn: OT-KDE-proof-mu-prob} \\
		\EXP\bigg[\int_{[\underline{a}_{\mu},\overline{a}_{\mu}]}\big(F_{\widetilde{\mu}_{m,h_{\mu}}}(x)-F_{\mu}(x)\big)^2\DIFFM{\mu}{\DIFF x}\INDI_{E_{\mu,m,h_{\mu}}}\bigg]
		&=O\big(m^{-1}+h_{\mu}^2\big),\label{eqn: OT-KDE-proof-mu-cdf}
	\end{align}
	as $m\to\infty$, $h_{\mu}\to0$.
	On the other hand, 
	applying statements~\ref{lems: KDE-properties-regularity}--\ref{lems: KDE-properties-concentrate} 
	of Lemma~\ref{lem: KDE-properties}
	with respect to 
	$\nu\leftarrow \nu$,
	$\underline{a}\leftarrow\underline{a}_{\nu}$,
	$\overline{a}\leftarrow\overline{a}_{\nu}$,
	$\zeta\leftarrow\zeta_{\nu}$,
	$\kappa(\cdot)\leftarrow \kappa_{\nu}(\cdot)$,
	$n\leftarrow n$,
	$h\leftarrow h_{\nu}$
	yields for any $n\in\N$ and any $h_{\nu}>0$ that 
	$F_{\widetilde{\nu}_{n,h_{\nu}}}$ is
	$\PROB$-almost surely
	a diffeomorphism,
	that
	it holds $\PROB$-almost surely that
	there exist $0<\underline{\lambda}_{\nu}\le\overline{\lambda}_{\nu}<\infty$ 
	satisfying 
	$\underline{\lambda}_{\nu}\le F'_{\widetilde{\nu}_{n,h_{\nu}}}(x)\le \overline{\lambda}_{\nu}$ $\forall x\in[\underline{a}_{\nu},\overline{a}_{\nu}]$,
	and that 
	there exists $E_{\nu,n,h_{\nu}}\in\sigma\big((Y_j)_{j=1:n}\big)$
	which depends on $n$ and $h_{\nu}$
	such that 
	\begin{align}
		1 - \PROB[E_{\nu,n,h_{\nu}}] 
		&=O\big(n^{-1}h_{\nu}^{-2}\big), \label{eqn: OT-KDE-proof-nu-prob} \\
		\EXP\Bigg[\sup\Bigg\{\frac{\big(F^{-1}_{\widetilde{\nu}_{n,h_{\nu}}}(u)-F^{-1}_{\widetilde{\nu}_{n,h_{\nu}}}(u')\big)^2}{(u-u')^2}:u,u'\in[0,1],\; u\ne u'\Bigg\}\INDI_{E_{\nu,n,h_{\nu}}}\Bigg]
		&=O(1),\label{eqn: OT-KDE-proof-nu-densmin}\allowdisplaybreaks \\
		\EXP\bigg[\int_{0}^{1}\big(F_{\widetilde{\nu}_{n,h_{\nu}}}^{-1}(u)-F_{\nu}^{-1}(u)\big)^2\DIFFX{u}\bigg]
		&=O\big(n^{-\frac{1}{2}}+h_{\nu}^2\big), \label{eqn: OT-KDE-proof-nu-invcdf}
	\end{align}
	as $n\to\infty$, $h_{\nu}\to0$.
	Let us denote 
	$Q_{n,h_{\nu}}:=\sup\bigg\{\frac{\big(F^{-1}_{\widetilde{\nu}_{n,h_{\nu}}}(u)-F^{-1}_{\widetilde{\nu}_{n,h_{\nu}}}(u')\big)^2}{(u-u')^2}:u,u'\in[0,1],\; u\ne u'\bigg\}$
	$\forall n\in\N$,
	$\forall h_{\nu}>\nobreak0$
	in the remainder of the proof for notational simplicity.
	Note that the property that
	$F'_{\widetilde{\nu}_{n,h_{\nu}}}$
	is $\PROB$-almost surely bounded from above and away from zero on $[\underline{a}_{\nu},\overline{a}_{\nu}]$
	guarantees that 
	$0<\nobreak Q_{n,h_{\nu}}<\nobreak\infty$ 
	holds $\PROB$-almost surely
	for any $n\in\N$ and any $h_{\nu}>0$.

	For any $m\in\N$, $n\in\N$, $h_{\mu}>0$, $h_{\nu}>0$,
	since $F_{\widetilde{\mu}_{m,h_{\mu}}}$ and 
	$F_{\widetilde{\nu}_{n,h_{\nu}}}$ are
	$\PROB$-almost surely
	both diffeomorphisms,
	it follows that 
	$F_{\widetilde{\nu}_{n,h_{\nu}}}^{-1}\circ F_{\widetilde{\mu}_{m,h_{\mu}}}:[\underline{a}_{\mu},\overline{a}_{\mu}]\to[\underline{a}_{\nu},\overline{a}_{\nu}]$
	is $\PROB$-almost surely a diffeomorphism.
	Moreover, 
	since it holds $\PROB$-almost surely that
	$F'_{\widetilde{\mu}_{m,h_{\mu}}}$ 
	is bounded from above and away from zero on $[\underline{a}_{\mu},\overline{a}_{\mu}]$
	and 
	$F'_{\widetilde{\nu}_{n,h_{\nu}}}$
	is bounded from above and away from zero on $[\underline{a}_{\nu},\overline{a}_{\nu}]$,
	it follows that 
	\begin{align*}
		\big(F_{\widetilde{\nu}_{n,h_{\nu}}}^{-1}\circ F_{\widetilde{\mu}_{m,h_{\mu}}}\big)'(x) = \frac{F'_{\widetilde{\mu}_{m,h_{\mu}}}(x)}{F'_{\widetilde{\nu}_{n,h_{\nu}}}\big(F_{\widetilde{\nu}_{n,h_{\nu}}}^{-1}\big(F_{\widetilde{\mu}_{m,h_{\mu}}}(x)\big)\big)} \qquad \forall x\in[\underline{a}_{\mu},\overline{a}_{\mu}]
	\end{align*}
	is also $\PROB$-almost surely bounded from above and away from zero on $[\underline{a}_{\mu},\overline{a}_{\mu}]$.
	The proofs of statements~\ref{lems: OT-KDE-regularity} and \ref{lems: OT-KDE-convexity} are now complete.

	It remains to prove statement~\ref{lems: OT-KDE-consistency}.
	To that end,
	we first combine (\ref{eqn: OT-KDE-proof-mu-prob}) and (\ref{eqn: OT-KDE-proof-nu-prob}) to obtain
	\begin{align}
		1 - \PROB[E_{\mu,m,h_{\mu}} \cap E_{\nu,n,h_{\nu}}]
		&\le \big(1 - \PROB[E_{\mu,m,h_{\mu}}]\big) + \big(1 - \PROB[E_{\nu,n,h_{\nu}}]\big)
		=O\big(m^{-1}h_{\mu}^{-2}+n^{-1}h_{\nu}^{-2}\big),
		\label{eqn: OT-KDE-proof-prob}
	\end{align}
	as $m\to\infty$, $n\to\infty$, $h_{\mu}\to0$, $h_{\nu}\to0$.
	Next, it follows from a classical result about one-dimensional OT maps (see, e.g., \citep[Proposition~1.18]{chewi2024statistical})
	that 
	\begin{align}
		T^{\mu}_{\nu}(x)=F_{\nu}^{-1}\big(F_{\mu}(x)\big) \qquad \forall x\in [\underline{a}_{\mu},\overline{a}_{\mu}].
		\label{eqn: OT-KDE-proof-OTmap}
	\end{align}
	Moreover, observe that we have the upper bound
	\begin{align}
		\begin{split}
			\Big(F_{\widetilde{\nu}_{n,h_{\nu}}}^{-1}\big(F_{\widetilde{\mu}_{m,h_{\mu}}}(x)\big) - F_{\nu}^{-1}\big(F_{\mu}(x)\big)\Big)^2
			&\le (\overline{a}_{\nu}-\underline{a}_{\nu})^2 \\
			&\hspace{-30pt}\forall x\in[\underline{a}_{\mu},\overline{a}_{\mu}],\; \forall m\in\N,\; \forall n\in\N,\; \forall h_{\mu}>0,\; \forall h_{\nu}>0,\; \PROB\text{-a.s.},
		\end{split}
		\label{eqn: OT-KDE-proof-globalUB}
	\end{align}
	as well as 
	\begin{align}
		\begin{split}
			\hspace{20pt}&\hspace{-20pt}\Big(F_{\widetilde{\nu}_{n,h_{\nu}}}^{-1}\big(F_{\widetilde{\mu}_{m,h_{\mu}}}(x)\big) - F_{\nu}^{-1}\big(F_{\mu}(x)\big)\Big)^2 \\
			&\le 2\Big(F_{\widetilde{\nu}_{n,h_{\nu}}}^{-1}\big(F_{\widetilde{\mu}_{m,h_{\mu}}}(x)\big) - F_{\widetilde{\nu}_{n,h_{\nu}}}^{-1}\big(F_{\mu}(x)\big)\Big)^2
			+ 2\Big(F_{\widetilde{\nu}_{n,h_{\nu}}}^{-1}\big(F_{\mu}(x)\big) - F_{\nu}^{-1}\big(F_{\mu}(x)\big)\Big)^2\\
			&\le 2Q_{n,h_{\nu}}\big(F_{\widetilde{\mu}_{m,h_{\mu}}}(x)-F_{\mu}(x)\big)^2 + 2\Big(F_{\widetilde{\nu}_{n,h_{\nu}}}^{-1}\big(F_{\mu}(x)\big) - F_{\nu}^{-1}\big(F_{\mu}(x)\big)\Big)^2 \\
			&\hspace{120.0pt} \forall x\in[\underline{a}_{\mu},\overline{a}_{\mu}],\; \forall m\in\N,\; \forall n\in\N,\; \forall h_{\mu}>0,\; \forall h_{\nu}>0,\; \PROB\text{-a.s.}
		\end{split}
		\label{eqn: OT-KDE-proof-specificUB}
	\end{align}
	On the one hand, since $\sigma\big((X_i)_{i=1:m}\big)$ and $\sigma\big((Y_j)_{j=1:n}\big)$ are independent,
	it follows from 
	(\ref{eqn: OT-KDE-proof-nu-densmin}) and 
	(\ref{eqn: OT-KDE-proof-mu-cdf}) that 
	\begin{align}
		\begin{split}
			\hspace{20pt}&\hspace{-20pt}\EXP\bigg[\int_{[\underline{a}_{\mu},\overline{a}_{\mu}]}Q_{n,h_{\nu}}\big(F_{\widetilde{\mu}_{m,h_{\mu}}}(x)-F_{\mu}(x)\big)^2\DIFFM{\mu}{\DIFF x}\INDI_{E_{\mu,m,h_{\mu}} \cap\, E_{\nu,n,h_{\nu}}}\bigg]\\
			&= \EXP\big[Q_{n,h_{\nu}}\INDI_{E_{\nu,n,h_{\nu}}}\big] \EXP\bigg[\int_{[\underline{a}_{\mu},\overline{a}_{\mu}]}\big(F_{\widetilde{\mu}_{m,h_{\mu}}}(x)-F_{\mu}(x)\big)^2\DIFFM{\mu}{\DIFF x}\INDI_{E_{\mu,m,h_{\mu}}}\bigg]=O\big(m^{-1}+h_{\mu}^2\big),
		\end{split}
		\label{eqn: OT-KDE-proof-specificUB1}
	\end{align}
	as $m\to\infty$, $n\to\infty$, $h_{\mu}\to0$, $h_{\nu}\to0$.
	On the other hand, since $F_{\mu}\sharp\mu$ is equal to the Lebesgue measure on $[0,1]$,
	it follows from (\ref{eqn: OT-KDE-proof-nu-invcdf}) that 
	\begin{align}
		\begin{split}
			\EXP\bigg[\int_{[\underline{a}_{\mu},\overline{a}_{\mu}]}\Big(F_{\widetilde{\nu}_{n,h_{\nu}}}^{-1}\big(F_{\mu}(x)\big) - F_{\nu}^{-1}\big(F_{\mu}(x)\big)\Big)^2\DIFFM{\mu}{\DIFF x}\bigg]
			&=\EXP\bigg[\int_{0}^{1}\big(F_{\widetilde{\nu}_{n,h_{\nu}}}^{-1}(u) - F_{\nu}^{-1}(u)\big)^2\DIFFX{u}\bigg]\\
			&=O\big(n^{-\frac{1}{2}}+h_{\nu}^2\big),
		\end{split}
		\label{eqn: OT-KDE-proof-specificUB2}
	\end{align}
	as $n\to\infty$, $h_{\nu}\to0$.
	Lastly, we combine (\ref{eqn: OT-KDE-proof-prob})--(\ref{eqn: OT-KDE-proof-specificUB2}) to get
	\begin{align*}
		\hspace{20pt}&\hspace{-20pt}\EXP\Big[\big\|F_{\widetilde{\nu}_{n,h_{\nu}}}^{-1}\circ F_{\widetilde{\mu}_{m,h_{\mu}}}-T^{\mu}_{\nu}\big\|^2_{\CL^2(\mu)}\Big]\\
		&= \EXP\bigg[\int_{[\underline{a}_{\mu},\overline{a}_{\mu}]}\Big(F_{\widetilde{\nu}_{n,h_{\nu}}}^{-1}\big(F_{\widetilde{\mu}_{m,h_{\mu}}}(x)\big) - F_{\nu}^{-1}\big(F_{\mu}(x)\big)\Big)^2\DIFFM{\mu}{\DIFF x}\bigg]\\
		&\le 2\EXP\bigg[\int_{[\underline{a}_{\mu},\overline{a}_{\mu}]}Q_{n,h_{\nu}}\big(F_{\widetilde{\mu}_{m,h_{\mu}}}(x)-F_{\mu}(x)\big)^2\DIFFM{\mu}{\DIFF x}\INDI_{E_{\mu,m,h_{\mu}} \cap\, E_{\nu,n,h_{\nu}}}\bigg] \\
		&\qquad + 2\EXP\bigg[\int_{[\underline{a}_{\mu},\overline{a}_{\mu}]}\Big(F_{\widetilde{\nu}_{n,h_{\nu}}}^{-1}\big(F_{\mu}(x)\big) - F_{\nu}^{-1}\big(F_{\mu}(x)\big)\Big)^2\DIFFM{\mu}{\DIFF x}\bigg] + (\overline{a}_{\nu}-\underline{a}_{\nu})^2 \big(1 - \PROB[E_{\mu,m,h_{\mu}} \cap E_{\nu,n,h_{\nu}}]\big)\\
		&= O\big(m^{-1}+n^{-\frac{1}{2}}+h_{\mu}^2+h_{\nu}^2+m^{-1}h_{\mu}^{-2}+n^{-1}h_{\nu}^{-2}\big),
	\end{align*}
	as $m\to\infty$, $n\to\infty$, $h_{\mu}\to0$, $h_{\nu}\to0$.
	Note that the constant terms omitted by all the big-$O$ notations in this proof depend only on 
	$\underline{a}_{\mu}$, $\overline{a}_{\mu}$, $\zeta_{\mu}$, $\kappa_{\mu}(\cdot)$,
	$\underline{a}_{\nu}$, $\overline{a}_{\nu}$, $\zeta_{\nu}$, $\kappa_{\nu}(\cdot)$,
	and do not depend on $m$, $n$, $h_{\mu}$, $h_{\nu}$.
	The proof is now complete.
\end{proof}

\begin{proof}[Proof of Proposition~\ref{prop: special-case-1d}]
	It holds by assumption that 
	the density functions of $\nu_1,\ldots,\nu_K$ are all essentially bounded.
	Moreover, 
	for $t\in\N$ and for $k=1,\ldots,K$,
	the definition of $\widehat{T}_{t,k}$ ensures that it is bounded and has Borel dependency on 
	$(X_{t,k,1},\ldots,X_{t,k,\widehat{M}_{t-1,k}},\allowbreak Y_{t,k,1},\ldots, Y_{t,k,\widehat{N}_{t-1,k}},\widehat{\ITheta}_{t-1,k})$.
	The proof of statement~\ref{props: special-case-1d-regularity} is complete.
	Statements~\ref{props: special-case-1d-uniqueness}--\ref{props: special-case-1d-var} are all direct consequences of the property that 
	$\big(\CP_{2,\AC}(\R),\CW_2\big)$ can be isometrically embedded into a Hilbert space;
	see, e.g., \citep[Proposition~7.14]{chewi2024statistical}.

	Hence, it remains to prove statement~\ref{props: special-case-1d-inequalities}.
	First, notice that the left-hand side of (\ref{eqn: main-convergence-cond2}) is equal to~0 for all $t\in\N$
	due to $\widehat{\mu}_t=\big[{\textstyle\sum_{k=1}^{K}w_k\widehat{T}_{t,k}}\big]\sharp \widehat{\mu}_{t-1}$.
	Next, let us show by induction that, for every $t\in\N_0$,
	the density function $f_{\widehat{\mu}_t}$ of $\widehat{\mu}_t$
	is $\PROB$-almost surely bounded from above and away from zero on $[\underline{a},\overline{a}]$.
	To that end, let us first observe that 
	the density function $f_{\widehat{\mu}_0}=f_{\mu_0}$ of $\widehat{\mu}_0$
	is $\PROB$-almost surely bounded from above and away from zero on $[\underline{a},\overline{a}]$.
	Let us suppose for some $t\in\N$ that 
	the density function $f_{\widehat{\mu}_{t-1}}$ of $\widehat{\mu}_{t-1}$
	is $\PROB$-almost surely bounded from above and away from zero on $[\underline{a},\overline{a}]$.
	For $k=1,\ldots,K$,
	it follows from the definitions of $\widehat{T}_{t,k}$,
	Lemma~\ref{lem: OT-KDE}\ref{lems: OT-KDE-regularity},
	and Lemma~\ref{lem: OT-KDE}\ref{lems: OT-KDE-convexity} 
	with respect to 
	$\mu\leftarrow\widehat{\mu}_{t-1}$,
	$\nu\leftarrow\nu_k$,
	$m\leftarrow\widehat{M}_{t-1,k}$,
	$n\leftarrow\widehat{N}_{t-1,k}$,
	$(X_i)_{i=1:m}\leftarrow (X_{t,k,i})_{i=1:\widehat{M}_{t-1,k}}$,
	$(Y_j)_{j=1:n}\leftarrow (Y_{t,k,j})_{j=1:\widehat{N}_{t-1,k}}$,
	$\kappa_{\mu}\leftarrow \kappa_0$,
	$\kappa_{\nu}\leftarrow \kappa_k$,
	$h_{\mu}\leftarrow h_{t-1,k,0}$,
	$h_{\nu}\leftarrow h_{t-1,k,1}$
	that 
	$\widehat{T}_{t,k}:[\underline{a},\overline{a}]\to[\underline{a}_k,\overline{a}_k]$ is $\PROB$-almost surely a diffeomorphism 
	(here we consider the restriction of $\widehat{T}_{t,k}$ to $[\underline{a},\overline{a}]$)
	where  
	$0<\nobreak\inf_{x\in[\underline{a},\overline{a}]}\big\{\widehat{T}_{t,k}'(x)\big\}\le \sup_{x\in[\underline{a},\overline{a}]}\big\{\widehat{T}_{t,k}'(x)\big\}<\nobreak\infty$
	holds $\PROB$-almost surely.
	Let us denote $\bar{T}_{t}(x):=\sum_{k=1}^{K}w_k\widehat{T}_{t,k}(x)$ $\forall x\in[\underline{a},\overline{a}]$.
	It subsequently holds $\PROB$-almost surely that 
	$\bar{T}_t:[\underline{a},\overline{a}]\to[\underline{a},\overline{a}]$ is a diffeomorphism,
	and that 
	$0<\nobreak\inf_{x\in[\underline{a},\overline{a}]}\big\{\bar{T}'_t(x)\big\}\le \sup_{x\in[\underline{a},\overline{a}]}\big\{\bar{T}'_t(x)\big\}<\nobreak\infty$.
	The change of variable formula for pushforward (see, e.g., \citep[Lemma 5.5.3]{ambrosio2008gradient}) then yields the following expression for the 
	density function $f_{\widehat{\mu}_t}$ of $\widehat{\mu}_t=\bar{T}_t\sharp\widehat{\mu}_{t-1}$:
	\begin{align*}
		f_{\widehat{\mu}_t}(x)=\frac{f_{\widehat{\mu}_{t-1}}\big(\bar{T}_t^{-1}(x)\big)}{\bar{T}'_{t}\big(\bar{T}_t^{-1}(x)\big)} \qquad\forall x\in[\underline{a},\overline{a}].
	\end{align*}
	This shows that 
	$f_{\widehat{\mu}_t}$ is $\PROB$-almost surely bounded from above and away from zero on $[\underline{a},\overline{a}]$.
	Therefore, we conclude by induction that,
	for every $t\in\N_0$,
	the density function $f_{\widehat{\mu}_t}$ of $\widehat{\mu}_t$
	is $\PROB$-almost surely bounded from above and away from zero on $[\underline{a},\overline{a}]$.

	We have now established the required regularity properties in order to apply Lemma~\ref{lem: OT-KDE}  
	on the entire sequence of
	$(\widehat{\mu}_t)_{t\in\N_0}$.
	For each $t\in\N$ and each $k\in\{1,\ldots,K\}$,
	applying Lemma~\ref{lem: OT-KDE}\ref{lems: OT-KDE-consistency}
	with respect to 
	$\mu\leftarrow\nobreak\widehat{\mu}_{t-1}$,
	$\nu\leftarrow\nu_k$,
	$m\leftarrow\widehat{M}_{t-1,k}$,
	$n\leftarrow\widehat{N}_{t-1,k}$,
	$(X_i)_{i=1:m}\leftarrow (X_{t,k,i})_{i=1:\widehat{M}_{t-1,k}}$,
	$(Y_j)_{j=1:n}\leftarrow (Y_{t,k,j})_{j=1:\widehat{N}_{t-1,k}}$,
	$\kappa_{\mu}\leftarrow \kappa_0$,
	$\kappa_{\nu}\leftarrow \kappa_k$,
	$h_{\mu}\leftarrow h_{t-1,k,0}$,
	$h_{\nu}\leftarrow h_{t-1,k,1}$
	yields the bound
	$\EXP\Big[\big\|\widehat{T}_{t,k}-\nobreak T^{\widehat{\mu}_{t-1}}_{\nu_k}\big\|^2_{\CL^2(\widehat{\mu}_{t-1})}\Big|\CF_{t-1}\Big]
	= O\big(\widehat{M}_{t-1,k}^{-1}+\widehat{N}_{t-1,k}^{-\frac{1}{2}}+h_{t-1,k,0}^{2}+h_{t-1,k,1}^{2}+\widehat{M}_{t-1,k}^{-1}h_{t-1,k,0}^{-2}+\widehat{N}_{t-1,k}^{-1}h_{t-1,k,1}^{-2}\big)$
	as $\widehat{M}_{t-1,k}\to\infty$,
	$\widehat{N}_{t-1,k}\to\infty$,
	$h_{t-1,k,0}\to 0$,
	$h_{t-1,k,1}\to 0$.
	Consequently, for any $\beta\in(0,1)$,
	there exist  
	$\widehat{M}_{t-1,k}=O(\beta^{-2t})$,
	$\widehat{N}_{t-1,k}=O(\beta^{-2t})$,
	$h_{t-1,k,0}=O\big(\beta^{\frac{t}{2}}\big)$,
	$h_{t-1,k,1}=O\big(\beta^{\frac{t}{2}}\big)$
	to guarantee 
	$\EXP\Big[\big\|\widehat{T}_{t,k}-T^{\widehat{\mu}_{t-1}}_{\nu_k}\big\|^2_{\CL^2(\widehat{\mu}_{t-1})}\Big|\CF_{t-1}\Big]\le \beta^t$.
	Note that the constants omitted by the big-$O$ notations here
	only depend on $\widehat{\mu}_{t-1}$, $\nu_k$, $\kappa_0(\cdot)$, $\kappa_k(\cdot)$.
	The proof is now complete.
\end{proof}

\subsection{Omitted proofs in Section~\ref{ssec: Caffarelli-setting-concrete}}\label{sapx: proof-Caffarelli-setting-concrete}

In the following, let us establish an inequality which will be used in the subsequent proofs of Proposition~\ref{prop: Caffarelli-density} and Lemma~\ref{lem: regularity-truncation}.

\begin{lemma}\label{lem: truncation-bound}
	The following inequality holds:
	\begin{align*}
		\CW_2(\rho|_{\CX},\rho)^2 &\le \int_{\R^d}2\Big({\textstyle\frac{1-\rho(\CX)}{\rho(\CX)}}+\INDI_{\R^d\setminus\CX}(\BIx)\Big)\|\BIx\|^2\DIFFM{\rho}{\DIFF\BIx} 
		\quad\; \forall \CX\in\big\{\CX'\in\CB(\R^d):\rho(\CX')>0\big\},\; \forall \rho\in\CP_2(\R^d).
	\end{align*}
\end{lemma}

\begin{proof}[Proof of Lemma~\ref{lem: truncation-bound}]
	Let us fix an arbitrary $\rho\in\CP_2(\R^d)$
	and an arbitrary 
	$\CX\in\CB(\R^d)$
	with $\rho(\CX)>0$.
	If $\rho(\R^d\setminus \CX)=\nobreak0$,
	then $\rho|_{\CX}=\rho$.
	Thus, we assume that 
	$\rho(\R^d\setminus\CX)>0$.
	Let us denote 
	$\dot{\mu}:=\rho|_{\CX}$, $\breve{\mu}:=\rho|_{\R^d\setminus\CX}$.
	Notice that $\rho=\rho(\CX)\dot{\mu}+(1-\rho(\CX))\breve{\mu}$.
	Let $\pi_{1}:=[I_d,I_d]\sharp\dot{\mu}$, 
	let $\pi_{2}\in\Pi(\dot{\mu},\breve{\mu})$ be arbitrary, 
	and let $\pi:=\rho(\CX)\pi_{1}+(1-\rho(\CX))\pi_{2}\in\CP(\R^d\times\R^d)$. 
	One may check that $\pi\in\Pi(\dot{\mu},\rho)$. 
	Subsequently, it holds that
	\begin{align*}
		\CW_2(\dot{\mu},\rho)^2 & \leq \int_{\R^d\times\R^d} \|\BIx - \BIy\|^2 \DIFFM{\pi}{ \DIFF \BIx, \DIFF \BIy}\\
		&= \rho(\CX) \int_{\R^d} \|\BIx - \BIx\|^2 \DIFFM{\dot{\mu}}{\DIFF \BIx} + (1 - \rho(\CX))\int_{\R^d\times\R^d} \|\BIx - \BIy\|^2 \DIFFM{\pi_{2}}{\DIFF \BIx, \DIFF \BIy} \allowdisplaybreaks\\
		& \leq (1 - \rho(\CX)) \int_{\R^d\times\R^d} 2\|\BIx\|^2 + 2\|\BIy\|^2 \DIFFM{\pi_{2}}{\DIFF \BIx, \DIFF \BIy} \\
		&= (1 - \rho(\CX)) \int_{\R^d} 2\|\BIx\|^2 \DIFFM{\dot{\mu}}{\DIFF \BIx} + (1 - \rho(\CX)) \int_{\R^d} 2\|\BIy\|^2 \DIFFM{\breve{\mu}}{\DIFF \BIy} \\
		&\le \int_{\R^d}2\Big({\textstyle\frac{1 - \rho(\CX)}{\rho(\CX)}}+\INDI_{\R^d\setminus\CX}(\BIx)\Big)\|\BIx\|^2 \DIFFM{\rho}{\DIFF \BIx}.
	\end{align*}
	The proof is now complete.
\end{proof}

\begin{proof}[Proof of Proposition~\ref{prop: Caffarelli-density}]
	To begin,
	it follows from 
	\citep[Lemma~7.1.10]{ambrosio2008gradient} 
	that $\CW_2(\nu,\rho)\le\frac{\epsilon}{2}$.
	Let us define
	$f_{\eta}(\BIx):=\big(\frac{\pi \epsilon^2}{2d}\big)^{-\frac{d}{2}}\exp\big({-\frac{2d}{\epsilon^2}}\|\BIx\|^2\big)$ $\forall\BIx\in\R^d$,
	which is the density function of $\eta$.
	Subsequently,
	one may check 
	from the definition of $\eta*\nu$
	that 
	$\rho$ admits a density function 
	$f_{\rho}(\BIx):=\int_{\R^d}f_{\eta}(\BIx-\BIy)\DIFFM{\nu}{\DIFF\BIy}$ $\forall\BIx\in\R^d$.
	Since $f_{\eta}$ has bounded derivatives of all orders,
	it follows from Lebesgue's dominated convergence theorem that 
	$f_{\rho}\in\CC^{\infty}(\R^d)\subset\CC^{\LOCAL,q+2,\alpha}(\R^d)$ for any $\alpha\in(0,1)$.
	Moreover, since $f_{\eta}(\BIx)>0$ $\forall\BIx\in\R^d$,
	we get $f_{\rho}(\BIx)>\nobreak0$ $\forall \BIx\in\R^d$.
	Thus, it holds that $\rho\in\CM^q_{\FULL}(\R^d)$.
	Furthermore,
	let us denote $\CX:=\bar{B}(\veczero_d,r)$.
	Since the choice of~$r$ guarantees 
	$\int_{\R^d}2\Big({\textstyle\frac{1 - \rho(\CX)}{\rho(\CX)}}+\INDI_{\R^d\setminus\CX}(\BIx)\Big)\|\BIx\|^2 \DIFFM{\rho}{\DIFF \BIx}\le\nobreak\frac{\epsilon^2}{4}$,
	Lemma~\ref{lem: truncation-bound} yields
	$\CW_2(\rho|_{\CX},\rho)\le \frac{\epsilon}{2}$,
	and we hence get 
	$\CW_2(\rho|_{\CX},\nu)\le \CW_2(\rho|_{\CX},\rho)+\CW_2(\rho,\nu)\le\nobreak\epsilon$.
	The proof is now complete.
\end{proof}

\begin{proof}[Proof of Lemma~\ref{lem: curvature}]
	Let $f_{\mu}$ and $f_{\nu}$ denote the density functions of $\mu$ and $\nu$ which satisfy the conditions in Definition~\ref{def: admissible-measures}.
	Thus, $f_\mu\in\CC^{q,\alpha}(\support(\mu))$ and 
	$f_\nu\in\CC^{q,\alpha'}(\support(\nu))$ 
	for some $\alpha,\alpha'\in(0,1)$. 
	This implies that $f_\mu\in\CC^{q,\alpha''}(\support(\mu))$ and 
	$f_\nu\in\CC^{q,\alpha''}(\support(\nu))$ 
	for $\alpha'':=\alpha\wedge\alpha'$, and hence Caffarelli's global regularity theory (Theorem~\ref{thm: Caffarelli}) implies that the Brenier potential $\varphi^{\mu}_{\nu}$ belongs to $\CC^{q+2,\alpha''}(\support(\mu))$.
	Consequently, 
	the compactness of $\support(\mu)$ implies that 
	there exists $\lambda_{\UBSUB}<\infty$ such that $\nabla^2 \varphi^{\mu}_{\nu}(\BIx)\preceq \lambda_{\UBSUB}\BI_d$ $\forall\BIx\in\support(\mu)$. 
	Moreover, $\varphi^{\mu}_{\nu}$ needs to satisfy the following Monge--Amp{\`e}re type equation as implied by the change of variable formula for pushforward (see, e.g., \citep[Lemma 5.5.3]{ambrosio2008gradient}):
	\begin{align*}
		\det\big(\nabla^2\varphi^{\mu}_{\nu}(\BIx)\big)=\frac{f_{\mu}(\BIx)}{f_{\nu}\big(\nabla\varphi^{\mu}_{\nu}(\BIx)\big)} \qquad\forall \BIx\in\support(\mu).
	\end{align*}
	Since $f_{\nu}$ is bounded from above and $f_{\mu}$ is bounded away from zero on $\support(\mu)$, 
	it follows that $\det\big(\nabla^2\varphi^{\mu}_{\nu}(\BIx)\big)$ is bounded away from zero on $\support(\mu)$.
	Combining this and $\nabla^2\varphi^{\mu}_{\nu}(\BIx)\preceq \lambda_{\UBSUB}\BI_d$ $\forall \BIx\in\support(\mu)$ shows that there exists $\lambda_{\LBSUB}>0$ such that $\nabla^2\varphi^{\mu}_{\nu}(\BIx)\succeq \lambda_{\LBSUB}\BI_d$ $\forall\BIx\in\support(\mu)$.
	The proof is now complete.
\end{proof}

Before we prove Lemma~\ref{lem: regularity-truncation},
let us first prove the following properties of pushforward,
which will be used again in the proof of Proposition~\ref{prop: synthetic-generation} later.

\begin{lemma}\label{lem: regularity-pushforward}
	The following statements hold.
	\begin{enumerate}[label=(\roman*),beginpenalty=10000]
		\item\label{lems: regularity-pushforward-ac}
		If $T=\nabla\varphi\in\CC_{\LIN}(\R^d,\R^d)$ where 
		$\varphi\in\FC^{2}_{\underline{\lambda},\infty}(\R^d)$ for $\underline{\lambda}>\nobreak0$, 
		then it holds that $T\sharp\rho\in\CP_{2,\AC}(\R^d)$ for any $\rho\in\CP_{2,\AC}(\R^d)$.
		
		\item\label{lems: regularity-pushforward-Caffarelli}
		If $T=\nabla\varphi\in\CC_{\LIN}(\R^d,\R^d)$ where 
		$\varphi\in\FC^{\LOCAL,q+2,\alpha}_{\underline{\lambda},\infty}(\R^d)$ for $q\in\N_0$, $\alpha\in(0,1)$, $\underline{\lambda}>0$, 
		then it holds that $T\sharp\rho\in\CM^q_{\FULL}(\R^d)$ for any $\rho\in\CM^q_{\FULL}(\R^d)$.
	\end{enumerate}
\end{lemma}

\begin{proof}[Proof of Lemma~\ref{lem: regularity-pushforward}]
	Let us first prove statement~\ref{lems: regularity-pushforward-ac}.
	It follows from the duality between smooth convex functions and strongly convex functions (see, e.g., \citep[Theorem~26.6]{rockafellar1970convex}) that
	$T$ is a homeomorphism.
	Moreover, since $\varphi$ is twice continuously differentiable on $\R^d$, 
	it holds by the second-order characterization of strongly convex functions (see, e.g., \citep[Theorem~2.1.11]{nesterov2004introconvex}) that $\nabla^2\varphi(\BIx)\succeq \underline{\lambda}\BI_d$ $\forall\BIx\in\R^d$. 
	Let us fix an arbitrary $\rho\in\CP_{2,\AC}(\R^d)$, and let $f_{\rho}$ denote the density function of $\rho$.
	Subsequently, 
	the change of variable formula for pushforward (see, e.g., \citep[Lemma 5.5.3]{ambrosio2008gradient}) yields the following expression for the density function $f_{T\sharp\rho}$ of $T\sharp\rho$:
	\begin{align}
		\label{eqn: regularity-pushforward-changeofvariable}
		f_{T\sharp\rho}(\BIy)=\frac{f_{\rho}\big(T^{-1}(\BIy)\big)}{\det\big(\nabla^2\varphi\big(T^{-1}(\BIy)\big)\big)} \qquad\forall \BIy\in\R^d.
	\end{align}
	Moreover, since $\int_{\R^d}\|\BIy\|^2 \DIFFM{T\sharp\rho}{\DIFF\BIy} = \int_{\R^d}\big\|T(\BIx)\big\|^2\DIFFM{\rho}{\DIFF\BIx}\le \|T\|_{\CC_{\LIN}(\R^d,\,\R^d)}^2\int_{\R^d}\big(1+\|\BIx\|\big)^2\DIFFM{\rho}{\DIFF\BIx}<\infty$, 
	we can conclude that $T\sharp\rho\in\CP_{2,\AC}(\R^d)$.
	The proof of statement~\ref{lems: regularity-pushforward-ac} is complete.

	To prove statement~\ref{lems: regularity-pushforward-Caffarelli}, 
	let us fix an arbitrary $\rho\in\CM^q_{\FULL}(\R^d)$ and denote its density function by $f_{\rho}$.
	It thus holds that $f_{\rho}(\BIx)>0$ $\forall\BIx\in\R^d$,
	and that $f_{\rho}\in\CC^{\LOCAL,q,\alpha'}(\R^d)$ for some $\alpha'\in(0,1)$.
	By replacing either $\alpha$~or~$\alpha'$ with $\alpha\wedge\alpha'$ if necessary,
	we assume without loss of generality that $\alpha=\alpha'$
	and thus $f_{\rho}\in\CC^{\LOCAL,q,\alpha}(\R^d)$.
	Since $\FC^{\LOCAL,q+2,\alpha}_{\underline{\lambda},\infty}(\R^d)\subset\FC^2_{\underline{\lambda},\infty}(\R^d)$ and $\CM^q_{\FULL}(\R^d)\subset\CP_{2,\AC}(\R^d)$, 
	statement~\ref{lems: regularity-pushforward-ac} implies that $T\sharp\rho\in\CP_{2,\AC}(\R^d)$, where
	the density function $f_{T\sharp\rho}$ of $T\sharp\rho$ is given by (\ref{eqn: regularity-pushforward-changeofvariable}).
	Observe that (\ref{eqn: regularity-pushforward-changeofvariable}) shows that $f_{T\sharp\rho}(\BIy)>0$ for all $\BIy\in\R^d$.
	It remains to show the local H{\"o}lder property of $f_{T\sharp\rho}$.
	To that end,
	let $\varphi^*$ denote the convex conjugate of $\varphi$.
	It follows from the duality between smooth convex functions and strongly convex functions (see, e.g., the equivalence between (a) and (e) in \citep[Proposition~12.60]{rockafellar1998variational}) and the inverse function theorem (see, e.g., \citep[Theorem~1A.1]{dontchev2009implicit}) that $T^{-1}$ is continuously differentiable and
	\begin{align}
		\label{eqn: regularity-preservation-proof-inverse}
		\nabla^2\varphi^*(\BIy)=\nabla T^{-1}(\BIy)=\big[\nabla^2 \varphi(T^{-1}(\BIy))\big]^{-1} \qquad \forall\BIy\in\R^d.
	\end{align}
	On the one hand, 
	since $\varphi\in\FC^{\LOCAL,q+2,\alpha}_{\underline{\lambda},\infty}(\R^d)\subset\CC^{\LOCAL,q+2,\alpha}(\R^d)$, 
	it follows from (\ref{eqn: regularity-preservation-proof-inverse}), 
	Fa{\`a} di Bruno's formula (see, e.g., \citep{craik2005prehistory}), 
	and an inductive argument that $\varphi^*\in\CC^{\LOCAL,q+2,\alpha}(\R^d)$.
	Consequently, since $\det(\cdot):\mathbb{S}^d\to\R$ is a polynomial in all entries of the input matrix, we have by (\ref{eqn: regularity-preservation-proof-inverse}) that
	$\frac{1}{\det\big(\nabla^2\varphi\big(T^{-1}(\cdot)\big)\big)}=\det \circ\, \nabla^2\varphi^* \in \CC^{\LOCAL,q,\alpha}(\R^d)$.
	On the other hand,
	since $f_{\rho}\in\CC^{\LOCAL,q,\alpha}(\R^d)$, 
	$T^{-1}=\nabla \varphi^*$, and 
	$\varphi^*\in\CC^{\LOCAL,q+2,\alpha}(\R^d)$, 
	we have by a similar derivation using Fa{\`a} di Bruno's formula and an inductive argument that $f_{\rho}\circ T^{-1}\in\CC^{\LOCAL,q,\alpha}(\R^d)$. 
	Hence, we conclude that $f_{T\sharp\rho}\in\CC^{\LOCAL,q,\alpha}(\R^d)$.
	The proof is now complete.
\end{proof}

\begin{proof}[Proof of Lemma~\ref{lem: regularity-truncation}]
	Throughout this proof, let \ref{setts: Caffarelli-inputs-measures}--\ref{setts: Caffarelli-OTestimator} in Setting~\ref{sett: Caffarelli} hold,
	and let us fix arbitrary 
	$\rho\in\CM^q_{\FULL}(\R^d)$, $\mu\in\CM^q(\R^d)$, $\CX\in\CS^q(\R^d)$,
	$(m_k)_{k=1:K}\subset\N\cap[\underline{m},\infty)$,
	$(n_k)_{k=1:K}\subset\N\cap[\underline{n},\infty)$,
	$(\theta_k)_{k=1:K}\subset\Theta$,
	as well as $\epsilon>0$.
	Moreover, 
	we let $f_{\rho}$ denote the density function of $\rho$,
	denote $\bar{T}:=\sum_{k=1}^{K}w_k\widehat{T}^{\mu,m_k}_{\nu_k,n_k}[\theta_k]$,
	and denote $\dot{\mu}_r:=\rho|_{\CX_r}$, $\dot{T}_r:=\sum_{k=1}^{K}w_k\widehat{T}^{\dot{\mu}_r,m_k}_{\nu_k,n_k}[\theta_k]$ for any $r\in\N$.
	The proof is divided into the following 4 steps.
	\begin{itemize}[beginpenalty=10000]
		\item Step~1: proving statement~\ref{lems: regularity-truncation-preserve}.
		
		\item Step~2: proving statement~\ref{lems: regularity-truncation-pushforward-preserve}.
		
		\item Step~3: showing the existence of $\overline{r}_1(\rho,\epsilon)\in\N$ such that $\CW_2(\dot{\mu}_r,\rho)^2\le \epsilon$
		$\forall r\ge \overline{r}_1(\rho,\epsilon)$.
		
		\item Step~4: showing the existence of $\overline{r}_2(\rho,\nu_1,\ldots,\nu_K,\epsilon)\in\N$
		such that
		$\EXP\big[\CW_2(\dot{T}_r\sharp\dot{\mu}_r,\dot{T}_r\sharp\rho)^2\big]\le\nobreak\epsilon$
		$\forall r\ge \nobreak\overline{r}_2(\rho,\nu_1,\ldots,\nu_K,\epsilon)$,
		and completing the proof of statement~\ref{lems: regularity-truncation-radius}.
	\end{itemize}

	\underline{Step~1.}
	Since $\support(\rho)=\R^d$,
	it holds that 
	$\support(\rho|_{\CX})=\CX\in\CS^q(\R^d)$.
	Moreover, Definition~\ref{def: admissible-measures} guarantees that
	$f_{\rho}(\BIx)>0$ $\forall\BIx\in\R^d$ and 
	$f_{\rho}\in\CC^{\LOCAL,q,\alpha}(\R^d)$ for some $\alpha\in(0,1)$.
	It thus holds that
	$f_{\rho|_{\CX}}:=\frac{f_{\rho}\INDI_{\CX}}{\rho(\CX)}\in\CC^{q,\alpha}(\CX)$ is the density function of $\rho|_{\CX}$.
	Subsequently, the compactness of $\CX$ implies that 
	$0<\inf_{\BIx\in\CX}\big\{f_{\rho|_{\CX}}(\BIx)\big\}\le \sup_{\BIx\in\CX}\big\{f_{\rho|_{\CX}}(\BIx)\big\}<\infty$ 
	and thus $\rho|_{\CX}\in\CM^q(\R^d)$.
	This completes the proof of statement~\ref{lems: regularity-truncation-preserve}.

	\underline{Step~2.}
	For $k=1,\ldots,K$, since $\widehat{T}^{\mu,m_k}_{\nu_k,n_k}[\cdot]$ satisfies the shape condition in Assumption~\ref{asp: OTmap-estimator}\ref{asps: OTmap-estimator-shape},
	it holds \mbox{$\PROB$-almost} surely that there exist 
	$\alpha_k\in(0,1)$, $\underline{\lambda}_k>0$, and 
	$\widehat{\varphi}_k\in\FC^{\LOCAL,q+2,\alpha_k}_{\underline{\lambda}_k,\infty}(\R^d)$
	such that 
	$\nabla\widehat{\varphi}_k=\widehat{T}^{\mu,m_k}_{\nu_k,n_k}[\theta_k]\in\CC_{\LIN}(\R^d,\R^d)$.
	Subsequently, let us denote $\bar{\varphi}:=\sum_{k=1}^{K}w_k\widehat{\varphi}_k$.
	It follows that $\nabla\bar{\varphi}=\bar{T}\in\CC_{\LIN}(\R^d,\R^d)$
	and $\bar{\varphi}\in\FC^{\LOCAL,q+2,\underline{\alpha}}_{\underline{\lambda},\infty}(\R^d)$ where 
	$\underline{\alpha}:=\min_{1\le k\le K}\{\alpha_k\}\in(0,1)$
	and $\underline{\lambda}:=\min_{1\le k\le K}\{\underline{\lambda}_k\}>\nobreak0$.
	Statement~\ref{lems: regularity-truncation-pushforward-preserve}
	then follows from Lemma~\ref{lem: regularity-pushforward}\ref{lems: regularity-pushforward-Caffarelli}.

	\underline{Step~3.}
	Let us define $H_{1,r}(\rho,\epsilon)$ for each $r\in\N$ and 
	$\overline{r}_{1}(\rho,\epsilon)$ as follows:
	\begin{align}
		\begin{split}
		H_{1,r}(\rho,\epsilon)
		&:=\int_{\R^d}2\Big({\textstyle\frac{1 - \rho(\CX_{r})}{\rho(\CX_{r})}}+\INDI_{\R^d\setminus \CX_{r}}(\BIx)\Big)\|\BIx\|^2 \DIFFM{\rho}{\DIFF \BIx}-\epsilon \qquad \forall r\in\N,\\
		\overline{r}_{1}(\rho,\epsilon)
		&:=\min\big\{r\in\nobreak\N:H_{1,\widetilde{r}}(\rho,\epsilon)\le 0 \; \forall \widetilde{r}\ge r\big\}.
		\end{split}
		\label{eqn: regularity-truncation-proof-r1-def}
	\end{align}
	One checks that 
	$H_{1,r}(\rho,\epsilon)$ has a Borel dependency on $(\rho,\epsilon)$
	for each $r\in\N$,
	and hence 
	$\overline{r}_{1}(\rho,\epsilon)$
	also has a Borel dependency on $(\rho,\epsilon)$.
	Since $\bigcup_{r\in\N}\CX_r=\R^d$ and $\rho\in\CP_{2}(\R^d)$ by assumption, it follows from Lebesgue's dominated convergence theorem that
	$\limsup_{r\to\infty}\int_{\R^d}2\Big({\textstyle\frac{1 - \rho(\CX_{r})}{\rho(\CX_{r})}}+\INDI_{\R^d\setminus\CX_{r}}(\BIx)\Big)\|\BIx\|^2 \DIFFM{\rho}{\DIFF \BIx}=\nobreak0$,
	and it hence holds that $\overline{r}_1(\rho,\epsilon)<\infty$.
	Lastly, Lemma~\ref{lem: truncation-bound} and the definition of $\overline{r}_{1}(\rho,\epsilon)$ guarantee that 
	$\CW_2(\dot{\mu}_r,\rho)^2\le\epsilon$ $\forall r\ge\overline{r}_1(\rho,\epsilon)$.
	Step~3 is now complete.

	\underline{Step~4.}
	In this step, we will use the growth condition 
	of $\widehat{T}^{\dot{\mu}_r,m_k}_{\nu_k,n_k}[\cdot]$ in Assumption~\ref{asp: OTmap-estimator}\ref{asps: OTmap-estimator-growth}.
	For every $r\in\N$, 
	let us denote $\breve{\mu}_r:=\rho|_{\R^d\setminus\CX_r}$
	(observe that $\rho(\R^d\setminus\CX_r)>0$).
	Notice that $\rho=\rho(\CX_r)\dot{\mu}_r+(1-\rho(\CX_r))\breve{\mu}_r$. 
	Let $\pi_{r,1}:=[I_d,I_d]\sharp\dot{\mu}_r$, 
	let $\pi_{r,2}\in\Pi(\dot{\mu}_r,\breve{\mu}_r)$ be arbitrary, 
	and let $\pi_r:=\rho(\CX_r)\pi_{r,1}+(1-\rho(\CX_r))\pi_{r,2}\in\CP(\R^d\times\R^d)$. 
	One may check that $\pi_r\in\Pi(\dot{\mu}_r,\rho)$. 
	Moreover, we denote by $\dot{T}_r\otimes\dot{T}_r$ the function $\R^d\times\R^d\ni (\BIx,\BIy)\mapsto \big(\dot{T}_{r}(\BIx),\dot{T}_{r}(\BIy)\big)\in\R^d\times\R^d$. 
	It holds that 
	$\big[\dot{T}_r\otimes\dot{T}_r\big]\sharp\pi_r\in \Pi\big(\dot{T}_r\sharp\dot{\mu}_r,\dot{T}_r\sharp\rho\big)$. 
	Therefore, we get 
	\begin{align}
		\label{eqn: regularity-truncation-proof-control2-step1}
		\begin{split}
			\!\!\CW_2\big(\dot{T}_{r}\sharp\dot{\mu}_r, \dot{T}_{r}\sharp\rho\big)^2 
			&\leq \int_{\R^d\times\R^d} \big\|\BIx - \BIy\big\|^2 \DIFFM{\big[\dot{T}_{r}\otimes\dot{T}_{r}\big] \sharp \pi_r}{\DIFF \BIx, \DIFF \BIy} \\
			&= \rho (\CX_{r}) \int_{\R^d} \big\|\dot{T}_{r}(\BIx) - \dot{T}_{r}(\BIx)\big\|^2 \DIFFM{\dot{\mu}_r}{\DIFF \BIx} \\
			&\qquad + (1 - \rho (\CX_{r}))\int_{\R^d\times\R^d} \big\|\dot{T}_{r}(\BIx) - \dot{T}_{r}(\BIy)\big\|^2 \DIFFM{\pi_{r,2}}{\DIFF \BIx, \DIFF \BIy}\\
			&\leq (1 - \rho (\CX_{r})) \int_{\R^d\times\R^d} 2\big\|\dot{T}_{r}(\BIx)-\dot{T}_{r}(\veczero_d)\big\|^2 + 2\big\|\dot{T}_{r}(\BIy)-\dot{T}_{r}(\veczero_d)\big\|^2 \DIFFM{\pi_{r,2}}{\DIFF \BIx, \DIFF \BIy} \\
			&\leq \int_{\R^d} 2\Big({\textstyle\frac{1 - \rho(\CX_{r})}{\rho(\CX_{r})}}+\INDI_{\R^d\setminus\CX_{r}}(\BIx)\Big)\big\|\dot{T}_{r}(\BIx)-\dot{T}_{r}(\veczero_d)\big\|^2 \DIFFM{\rho}{\DIFF \BIx} \hspace{41pt}\qquad \forall r\in\N.
		\end{split}
	\end{align}
	For $k=1,\ldots,K$, 
	observe that the growth condition 
	of $\widehat{T}^{\dot{\mu}_r,m_k}_{\nu_k,n_k}[\cdot]$ in Assumption~\ref{asp: OTmap-estimator}\ref{asps: OTmap-estimator-growth}
	guarantees $\EXP\big[\|\widehat{T}^{\dot{\mu}_r,m_k}_{\nu_k,n_k}[\theta_k](\BIx)-\widehat{T}^{\dot{\mu}_r,m_k}_{\nu_k,n_k}[\theta_k](\veczero_d)\|^2\big]\le u_0(\nu_k)+u_1(\nu_k)\|\BIx\|^2$ 
	$\forall\BIx\in\R^d$, $\forall r\in\N$, where $u_0(\nu_k)\in\nobreak\R_+$ and $u_1(\nu_k)\in\nobreak\R_+$ only depend on $\nu_k$. 
	It thus holds by the convexity of $\R^d\ni\BIz\mapsto\|\BIz\|^2\in\R$ and Jensen's inequality that
	\begin{align}
		\label{eqn: regularity-truncation-proof-control2-step2}
		\begin{split}
		\EXP\big[\|\dot{T}_r(\BIx)-\dot{T}_r(\veczero_d)\|^2\big]
		&\le \sum_{k=1}^Kw_k \EXP\big[\|\widehat{T}^{\dot{\mu}_r,m_k}_{\nu_k,n_k}[\theta_k](\BIx)-\widehat{T}^{\dot{\mu}_r,m_k}_{\nu_k,n_k}[\theta_k](\veczero_d)\|^2\big]\\
		&\le \Bigg(\sum_{k=1}^{K}w_ku_0(\nu_k)\Bigg) + \Bigg(\sum_{k=1}^{K}w_ku_1(\nu_k)\Bigg)\|\BIx\|^2 \qquad\forall \BIx\in\R^d,\; \forall r\in\N.
		\end{split}
	\end{align}
	Subsequently, taking expectations on both sides of (\ref{eqn: regularity-truncation-proof-control2-step1}) and then applying Fubini's theorem and (\ref{eqn: regularity-truncation-proof-control2-step2}) leads to
	\begin{align}
		\hspace{20pt}&\hspace{-20pt}
		\EXP\big[\CW_2(\dot{T}_r\sharp\dot{\mu}_r, \dot{T}_r\sharp\rho)^2\big]\label{eqn: regularity-truncation-proof-control2-step3} \\
		&\le \int_{\R^d} 2\Big({\textstyle\frac{1 - \rho(\CX_{r})}{\rho(\CX_{r})}}+\INDI_{\R^d\setminus\CX_{r}}(\BIx)\Big)\Big[\big({\textstyle\sum_{k=1}^{K}}w_ku_0(\nu_k)\big)+\big({\textstyle\sum_{k=1}^{K}}w_ku_1(\nu_k)\big)\|\BIx\|^2\Big] \DIFFM{\rho}{\DIFF \BIx}
		\qquad \forall r\in\N.\nonumber
	\end{align}
	Now, let us define $H_{2,r}(\rho,\nu_1,\ldots,\nu_K,\epsilon)$ for each $r\in\N$
	and $\overline{r}_2(\rho,\nu_1,\ldots,\nu_K,\epsilon)$ as follows:
	\begin{align}
		H_{2,r}(\rho,\nu_1,\ldots,\nu_K,\epsilon)&:=
		\int_{\R^d} 2\Big(\hspace{-1pt}{\textstyle\frac{1 - \rho(\CX_{r})}{\rho(\CX_{r})}}+\INDI_{\R^d\setminus\CX_{r}}(\BIx)\hspace{-1pt}\Big)\Big[\big({\textstyle\sum_{k=1}^{K}}w_ku_0(\nu_k)\big)\nonumber \\
		&\hspace{164pt}+\big({\textstyle\sum_{k=1}^{K}}w_ku_1(\nu_k)\big)\|\BIx\|^2\Big] \DIFFM{\rho}{\DIFF \BIx}-\epsilon \quad \forall r\in\N,\nonumber\\
		\overline{r}_2(\rho,\nu_1,\ldots,\nu_K,\epsilon)&:=
		\min\big\{r\in\N: H_{2,\widetilde{r}}(\rho,\nu_1,\ldots,\nu_K,\epsilon)\le\nobreak0\; \forall \widetilde{r}\ge \nobreak r\big\}.
		\label{eqn: regularity-truncation-proof-r2-def}
	\end{align}
	One checks that $H_{2,r}(\rho,\nu_1,\ldots,\nu_K,\epsilon)$
	has a Borel dependency on 
	$(\rho,\nu_1,\ldots,\nu_K,\epsilon)$ for each $r\in\N$,
	and hence 
	$\overline{r}_2(\rho,\nu_1,\ldots,\nu_K,\epsilon)$ 
	also has a Borel dependency on $\rho,\nu_1,\ldots,\nu_K,\epsilon$.
	Same as in Step~3,
	since $\bigcup_{r\in\N}\CX_r=\R^d$ and $\rho\in\CP_2(\R^d)$ by assumption,
	it follows from Lebesgue's dominated convergence theorem that 
	$\limsup_{r\to\infty}\int_{\R^d} 2\big({\textstyle\frac{1 - \rho(\CX_{r})}{\rho(\CX_{r})}}+\INDI_{\R^d\setminus\CX_{r}}(\BIx)\big)\Big[\big({\textstyle\sum_{k=1}^{K}}w_ku_0(\nu_k)\big)+\big({\textstyle\sum_{k=1}^{K}}w_ku_1(\nu_k)\big)\|\BIx\|^2\Big] \DIFFM{\rho}{\DIFF \BIx}=\nobreak0$,
	and it therefore holds that $\overline{r}_2(\rho,\nu_1,\ldots,\nu_K,\epsilon)<\nobreak\infty$.
	Lastly, 
	(\ref{eqn: regularity-truncation-proof-control2-step3})
	and the definition of 
	$\overline{r}_2(\rho,\nu_1,\ldots,\nu_K,\epsilon)$
	guarantee that 
	$\EXP\big[\CW_2(\dot{T}_r\sharp\dot{\mu}_r, \dot{T}_r\sharp\rho)^2\big]\le \epsilon$ $\forall r\ge\overline{r}_2(\rho,\nu_1,\ldots,\nu_K,\epsilon)$ and complete Step~4.
	The proof is now complete.
\end{proof}

\subsection{Omitted proofs in Section~\ref{ssec: Caffarelli-convergence-proof}}\label{sapx: proof-Caffarelli-convergence-proof}

The following measurability-related lemma is used in the proof of Theorem~\ref{thm: Caffarelli-convergence}.

\begin{lemma}[Measurability conditions in Algorithm~\ref{algo: concrete}]\label{lem: concrete-measurability}
	The following statements hold.
	\begin{enumerate}[label=(\roman*),beginpenalty=10000]
		\item\label{lems: concrete-measurability-truncation}%
		Let $q\in\N_0$, let $\CM^q_{\FULL}(\R^d)$ be defined in Definition~\ref{def: admissible-measures},
		and let $\CX\subseteq\R^d$ be the closure of a non-empty open set.
		Then, the function 
		$\CM^q_{\FULL}(\R^d)\ni\rho\mapsto \rho|_{\CX}\in\CP_{2}(\R^d)$ is Borel measurable.

		\item\label{lems: concrete-measurability-pushforward}%
		The function 
		$\CP_{2}(\R^d)\times \CC_{\LIN}(\R^d,\R^d)\ni (\rho,T)\mapsto T\sharp\rho\in\CP_{2}(\R^d)$ is Borel measurable.
		
	\end{enumerate}
\end{lemma}

\begin{proof}[Proof of Lemma~\ref{lem: concrete-measurability}]
	Let us first fix an arbitrary $q\in\N_0$,
	fix an arbitrary $\CX\subseteq\R^d$ that is the closure of a non-empty open set, and prove 
	statement~\ref{lems: concrete-measurability-truncation}.
	Observe that $\rho(\CX)>\nobreak0$ $\forall \rho\in\CM^q_{\FULL}(\R^d)$.
	Since the function $\INDI_{\CX}:\R^d\to\nobreak[0,1]$ is upper semi-continuous, 
	there exists a sequence of bounded continuous functions 
	$(g_n)_{n\in\N}\subset\CC^0(\R^d)$
	such that 
	$g_{n+1}\le g_n$ $\forall n\in\N$
	and 
	$\lim_{n\to\infty}g_n(\BIx)=\INDI_{\CX}(\BIx)$ $\forall\BIx\in\R^d$.
	It thus holds that $g_n\ge 0$ $\forall n\in\N$
	and that $\int_{\R^d}g_n\DIFFX{\rho}>\nobreak0$ $\forall \rho\in\CM^q_{\FULL}(\R^d)$, $\forall n\in\N$.
	Subsequently, 
	let us define 
	$\big(H_n:\CM^q_{\FULL}(\R^d)\to\CP_2(\R^d)\big)_{n\in\N}$ via the Radon--Nikodym derivatives of their images as follows:
	\begin{align*}
		\frac{\DIFF H_n(\rho)}{\DIFF\rho}:= \frac{g_n}{\int_{\R^d}g_n\DIFFX{\rho}}\qquad \forall \rho\in\CM^q_{\FULL}(\R^d),\; \forall n\in\N.
	\end{align*}
	This means that
	\begin{align}
		\begin{split}
			\int_{\R^d}f\DIFFX{H_n(\rho)}=\frac{\int_{\R^d}fg_n\DIFFX{\rho}}{\int_{\R^d}g_n\DIFFX{\rho}}
			\qquad &\forall \rho\in\CM^q_{\FULL}(\R^d),\; \forall n\in\N,\\
			&\forall f\in\CC^0(\R^d) \text{ with } \sup_{\BIx\in\R^d}\bigg\{\frac{|f(\BIx)|}{1+\|\BIx\|^2}\bigg\}<\infty.
		\end{split}
		\label{eqn: concrete-measurability-proof-truncate-limit}
	\end{align}
	It hence follows from 
	(\ref{eqn: concrete-measurability-proof-truncate-limit}) 
	and
	the equivalence between (i) and (iv) in 
	\citep[Theorem~7.12]{villani2003topics} that
	$(H_n)_{n\in\N}$ are all continuous.
	Moreover,
	in view of (\ref{eqn: concrete-measurability-proof-truncate-limit}),
	applying 
	Lebesgue's dominated convergence theorem and
	using again the equivalence between (i) and (iv) in 
	\citep[Theorem~7.12]{villani2003topics}
	yields
	$\lim_{n\to\infty}\CW_2\big(H_n(\rho),\rho|_{\CX}\big)=\nobreak0$ $\forall\rho\in\nobreak\CM^q_{\FULL}(\R^d)$.
	It follows that 
	$\CM^q_{\FULL}(\R^d)\ni\rho\mapsto \rho|_{\CX}\in\CP_{2}(\R^d)$ is
	the point-wise limit of a countable sequence of continuous functions with values in the metric space $(\CP_2(\R^d),\CW_2)$ and is thus Borel measurable.
	This completes the proof of statement~\ref{lems: concrete-measurability-truncation}.

	In the following, we will prove statement~\ref{lems: concrete-measurability-pushforward}
	by showing that 
	$\CP_{2}(\R^d)\times \CC_{\LIN}(\R^d,\R^d)\ni (\rho,T)\mapsto T\sharp\rho\in\CP_{2}(\R^d)$ is continuous (thus Borel measurable).
	Let us fix arbitrary $\rho\in\CP_2(\R^d)$, $T\in\CC_{\LIN}(\R^d,\R^d)$, 
	as well as two sequences 
	$(\rho_n)_{n\in\N}\subset \CP_2(\R^d)$ and
	$(T_n)_{n\in\N}\subset\CC_{\LIN}(\R^d,\R^d)$, 
	with 
	$\lim_{n\to\infty}\CW_2(\rho_n,\rho)=\nobreak0$ and
	$\lim_{n\to\infty}\|T_n-\nobreak T\|_{\CC_{\LIN}(\R^d,\,\R^d)}=\nobreak0$.
	Moreover, for each $n\in\N$, let
	$\pi_n\in\Pi(\rho_n,\rho)$ be an optimal coupling of 
	$\rho_n$ and $\rho$.
	Since $[I_d,I_d]\sharp\rho\in\CP_2(\R^d\times\R^d)$
	is the unique optimal coupling of $\rho$ and itself,
	it follows from \citep[Theorem~5.20]{villani2009optimal}
	that 
	$(\pi_n)_{n\in\N}$ converges weakly to $[I_d,I_d]\sharp\rho$.
	One can then invoke the equivalence between (i) and (iii) in 
	\citep[Theorem~7.12]{villani2003topics}
	to show that 
	$(\pi_n)_{n \in \N}$ also converges in $\CW_2$ to $[I_d,I_d]\sharp\rho$.
	Let $T\otimes T$ denote the function 
	$\R^d\times\R^d\ni (\BIx,\BIy)\mapsto \big(T(\BIx),T(\BIy)\big)\in\R^d\times\R^d$.
	Subsequently, 
	since $(T\otimes T)\sharp \pi_n$ constitutes a suboptimal coupling of $T\sharp\rho_n$ and $T\sharp\rho$ for every $n\in\N$, and
	since 
	$\sup_{(\BIx,\BIy)\in\R^d\times\R^d}\Big\{\frac{\|T(\BIx)-T(\BIy)\|^2}{1+\|\BIx\|^2+\|\BIy\|^2}\Big\}\le 8\|T\|^2_{\CC_{\LIN}(\R^d,\,\R^d)}<\infty$,
	applying the equivalence between (i) and (iv) in 
	\citep[Theorem~7.12]{villani2003topics}
	yields
	\begin{align}
		\begin{split}
			\limsup_{n\to\infty} \CW_2(T\sharp\rho_n,T\sharp\rho)^2
			&\le
			\limsup_{n\to\infty}\int_{\R^d\times\R^d} \big\|T(\BIx)-T(\BIy)\big\|^2 \DIFFM{\pi_n}{\DIFF\BIx,\DIFF\BIy} \\
			&= \int_{\R^d} \big\|T(\BIx)-T(\BIx)\big\|^2\DIFFM{\rho}{\DIFF\BIx}=0.
		\end{split}
		\label{eqn: concrete-measurability-proof-pushforward-step1}
	\end{align}
	On the other hand, 
	since $[T,T_n]\sharp\rho_n$ is a suboptimal coupling of $T\sharp\rho_n$ and $T_n\sharp\rho_n$ for every $n\in\N$,
	it holds that 
	\begin{align}
		\begin{split}
			\CW_2(T\sharp\rho_n,T_n\sharp\rho_n)^2
			&\le \int_{\R^d}\big\|T(\BIx)-T_n(\BIx)\big\|^2\DIFFM{\rho_n}{\DIFF\BIx}\\
			&\le \|T-T_n\|^2_{\CC_{\LIN}(\R^d,\,\R^d)}\int_{\R^d}\big(1+\|\BIx\|\big)^2\DIFFM{\rho_n}{\DIFF\BIx}
			\qquad \forall n\in\N.
		\end{split}
		\label{eqn: concrete-measurability-proof-pushforward-step2}
	\end{align}
	Now, combining (\ref{eqn: concrete-measurability-proof-pushforward-step1}),
	(\ref{eqn: concrete-measurability-proof-pushforward-step2}),
	and using the equivalence between (i) and (iv) in 
	\citep[Theorem~7.12]{villani2003topics}
	leads to 
	\begin{align*}
		\limsup_{n\to\infty}\CW_2(T\sharp\rho,T_n\sharp\rho_n)^2
		&\le \limsup_{n\to\infty} 2\CW_2(T\sharp\rho_n,T\sharp\rho)^2 + 2\CW_2(T\sharp\rho_n,T_n\sharp\rho_n)^2\\
		&\le 2\bigg(\limsup_{n\to\infty}\|T-T_n\|^2_{\CC_{\LIN}(\R^d,\,\R^d)}\bigg)\bigg(\limsup_{n\to\infty}\int_{\R^d}\big(1+\|\BIx\|\big)^2\DIFFM{\rho_n}{\DIFF\BIx}\bigg)
		=0.
	\end{align*}
	This proves the continuity and hence the Borel measurability of the function $\CP_{2}(\R^d)\times \CC_{\LIN}(\R^d,\R^d)\ni (\rho,T)\mapsto T\sharp\rho\in\CP_{2}(\R^d)$.
	The proof is now complete.
\end{proof}

\subsection{Omitted proofs in Section~\ref{ssec: entropic-estimator}}\label{sapx: proof-entropic-estimator}
Before we prove Proposition~\ref{prop: OTmap-estimator-entropic}, let us first establish some intermediate results in the lemma below.

\begin{lemma}\label{lem: logsumexp-properties}
	Let $n\in\N$, 
	$(\BIy_j)_{j=1:n}\subset\R^d$,
	$(g_j)_{j=1:n}\subset\R$,
	and $\gamma>0$ be arbitrary.
	Let the functions
	$\eta_1,\ldots,\eta_n:\R^d\to(0,1]$, 
	$\Beta:\R^d\to(0,1]^n$,
	$\varphi:\R^d\to\R$, 
	$T:\R^d\to\R^d$
	and the matrices
	$\BY\in\R^{d\times n}$,
	$\BSigma\in\R^{d\times d}$
	be defined as follows: 
	\begin{align*}
		\eta_j(\BIx)&:=\frac{\exp\Big({\textstyle\frac{g_j + \langle\BIy_j,\BIx\rangle}{\gamma}}\Big)}{\sum_{k=1}^n\exp\Big({\textstyle\frac{g_{k} + \langle\BIy_{k},\BIx\rangle}{\gamma}}\Big)} 
		& \forall \BIx\in\R^d,\; \forall 1\le j\le n,\allowdisplaybreaks\\
		\Beta(\BIx)&:=\big(\eta_1(\BIx),\ldots,\eta_n(\BIx)\big)^\TRANSP 
		& \forall \BIx\in\R^d,\allowdisplaybreaks\\
		\varphi(\BIx)&:= \gamma\log\left(\sum_{j=1}^n \exp\bigg({\frac{g_j + \langle\BIy_j,\BIx\rangle}{\gamma}}\bigg)\right)  
		& \forall \BIx\in\R^d, \allowdisplaybreaks\\
		T(\BIx)&:= \sum_{j=1}^n \eta_j(\BIx) \BIy_j  
		& \forall \BIx\in\R^d,\allowdisplaybreaks\\
		\BY&:=\left(\begin{smallmatrix}
			| & | &  & | \\
			\BIy_1 & \BIy_2 & \cdots & \BIy_n \\
			| & | &  & |
	  	\end{smallmatrix}\right)\in\R^{d\times n},\allowdisplaybreaks\\
		\BSigma&:=\Bigg(\frac{1}{n}\sum_{j=1}^n\BIy_j\BIy_j^\TRANSP\Bigg) - \Bigg(\frac{1}{n}\sum_{j=1}^n\BIy_j\Bigg)\Bigg(\frac{1}{n}\sum_{j=1}^n\BIy_j\Bigg)^\TRANSP\in\R^{d\times d}.
	\end{align*}
	Then, the following statements hold.
	\begin{enumerate}[label=(\roman*),beginpenalty=10000]
		\item\label{lems: logsumexp-properties-gradient}%
		$\varphi\in\FC^{\infty}_{0,\infty}(\R^d)$
		and 
		$\nabla\varphi=T$.

		\item\label{lems: logsumexp-properties-Hessian-ub}%
		$\nabla^2\varphi(\BIx)
		\preceq \frac{1}{\gamma}e_{\MAXSUB}\big(\BY\big(\diag(\Beta(\BIx))-\Beta(\BIx)\Beta(\BIx)^\TRANSP\big)\BY^\TRANSP\big)\BI_d$
		$\forall \BIx\in\R^d$.

		\item\label{lems: logsumexp-properties-Hessian-lb}%
		$\nabla^2\varphi(\BIx)
		\succeq \frac{e_{\MINSUB}(\BSigma)}{\gamma}\!\exp\!\Big(\hspace{-1pt}{-\frac{1}{\gamma}}\big(\max_{1\le j\le n}\{g_j\}-\min_{1\le j\le n}\{g_j\}+2\max_{1\le j\le n}\{\|\BIy_j\|\}\|\BIx\|\big)\hspace{-1pt}\Big)\BI_{d}$
		$\forall \BIx\in\R^d$.
	\end{enumerate}
\end{lemma}

\begin{proof}[Proof of Lemma~\ref{lem: logsumexp-properties}]
	To begin, observe that 
	$\varphi$ is the composition of a log-sum-exp function and an affine function and thus belongs to $\FC^{\infty}_{0,\infty}(\R^d)$. 
	Moreover, 
	$\nabla\varphi=T$ holds by definition.
	This proves statement~\ref{lems: logsumexp-properties-gradient}.

	To prove statement~\ref{lems: logsumexp-properties-Hessian-ub},
	notice that
	$\nabla \eta_j(\BIx)=\frac{1}{\gamma}\eta_j(\BIx)\BIy_j-\frac{1}{\gamma}\eta_j(\BIx)\big(\sum_{k=1}^n\eta_k(\BIx)\BIy_k\big)$
	$\forall \BIx \in \R^d$, $\forall 1\le j\le n$.
	We hence get
	\begin{align}
		\begin{split}
		\nabla^2\varphi(\BIx) &= \sum_{j=1}^n \nabla \eta_j(\BIx)\BIy_j^\TRANSP 
		= \frac{1}{\gamma}\BY\big(\diag(\Beta(\BIx)) - \Beta(\BIx)\Beta(\BIx)^\TRANSP\big)\BY^\TRANSP 
		\qquad \forall \BIx \in \R^d,
		\end{split}
		\label{eqn: logsumexp-properties-proof-Hessian}
	\end{align}
	which proves statement~\ref{lems: logsumexp-properties-Hessian-ub}.

	To prove statement~\ref{lems: logsumexp-properties-Hessian-lb}, 
	we will establish the following two inequalities:
	\begin{align}
		&\BY\big(\diag(\Beta(\BIx)) - \Beta(\BIx)\Beta(\BIx)^\TRANSP\big)\BY^\TRANSP
		\succeq n e_{\MINSUB}(\BSigma)\min_{1\le j\le n}\big\{\eta_j(\BIx)\big\}\BI_d 
		& \forall \BIx\in\R^d,
		\label{eqn: logsumexp-properties-proof-Hessian-lb-step1}\\
		&\min_{1\le j\le n}\big\{\eta_j(\BIx)\big\}
		\ge \frac{1}{n}\exp\bigg({-\frac{1}{\gamma}}\Big(\max_{1\le j\le n}\{g_j\}-\min_{1\le j\le n}\{g_j\}+2\max_{1\le j\le n}\{\|\BIy_j\|\}\|\BIx\|\Big)\bigg) 
		& \forall \BIx\in\R^d.
		\label{eqn: logsumexp-properties-proof-Hessian-lb-step2}
	\end{align}
	Subsequently, 
	combining (\ref{eqn: logsumexp-properties-proof-Hessian})--(\ref{eqn: logsumexp-properties-proof-Hessian-lb-step2})
	will complete the proof of statement~\ref{lems: logsumexp-properties-Hessian-lb}.

	To prove (\ref{eqn: logsumexp-properties-proof-Hessian-lb-step1}),
	let us fix an arbitrary $\BIx\in\R^d$,
	denote $\eta_{\MINSUB}:=\min_{1\le j\le n}\big\{\eta_j(\BIx)\big\}$,
	define $\BIp:=\frac{1}{1-n\eta_{\MINSUB}}\big(\Beta(\BIx)-\eta_{\MINSUB}\vecone_n\big)$ 
	in the case where $\eta_{\MINSUB}<\frac{1}{n}$,
	and define $\BIp:=\frac{1}{n}\vecone_{n}$ in the case where $\eta_{\MINSUB}=\nobreak\frac{1}{n}$.
	Since $\langle\vecone_n,\BIp\rangle=1$,
	it follows from the convexity of $\R\ni z\mapsto z^2\in\R$ 
	and Jensen's inequality that
	$\diag(\BIp)-\BIp\BIp^\TRANSP\succeq\BO_n$.
	Observe that since $\Beta(\BIx)=\eta_{\MINSUB}\vecone_n+(1-n\eta_{\MINSUB})\BIp$, it holds that 
	\begin{align*}
		\diag\big(\Beta(\BIx)\big)-\Beta(\BIx)\Beta(\BIx)^\TRANSP 
		&= \eta_{\MINSUB}\BI_n 
			+ (1-n\eta_{\MINSUB})\diag(\BIp) 
			- \eta_{\MINSUB}^2\vecone_n\vecone_n^\TRANSP \\
		&\qquad - (1-n\eta_{\MINSUB})^2\BIp\BIp^\TRANSP 
			- \eta_{\MINSUB}(1-n\eta_{\MINSUB})(\BIp\vecone_n^\TRANSP + \vecone_n\BIp^\TRANSP)\\
		&= n\eta_{\MINSUB}\bigg(\frac{1}{n}\BI_n-\frac{1}{n^2}\vecone_n\vecone_n^\TRANSP\bigg) 
			+ (1-n\eta_{\MINSUB})\big(\diag(\BIp)-\BIp\BIp^\TRANSP\big) \\
		&\qquad + n\eta_{\MINSUB}(1-n\eta_{\MINSUB})\bigg(\frac{1}{n}\vecone_n-\BIp\bigg)\bigg(\frac{1}{n}\vecone_n-\BIp\bigg)^\TRANSP\\
		&\succeq n\eta_{\MINSUB}\bigg(\frac{1}{n}\BI_n-\frac{1}{n^2}\vecone_n\vecone_n^\TRANSP\bigg).
	\end{align*}
	It hence follows that 
	\begin{align*}
		\BY\big(\diag(\Beta(\BIx)) - \Beta(\BIx)\Beta(\BIx)^\TRANSP\big)\BY^\TRANSP
		&\succeq n\eta_{\MINSUB}\BY\bigg(\frac{1}{n}\BI_n-\frac{1}{n^2}\vecone_n\vecone_n^\TRANSP\bigg)\BY^\TRANSP
		=n\eta_{\MINSUB}\BSigma
		\succeq ne_{\MINSUB}(\BSigma)\eta_{\MINSUB}\BI_d.
	\end{align*}
	This proves (\ref{eqn: logsumexp-properties-proof-Hessian-lb-step1}).
	
	To prove (\ref{eqn: logsumexp-properties-proof-Hessian-lb-step2}),
	let us fix an arbitrary $\BIx\in\R^d$
	and an arbitrary $j\in\{1,\ldots,n\}$.
	Observe that the Cauchy--Schwarz inequality implies that
	\begin{align*}
		\big(g_k+\langle\BIy_k,\BIx\rangle\big) - \big(g_j+\langle\BIy_j,\BIx\rangle\big)\le \max_{1\le l\le n}\{g_l\}-\min_{1\le l\le n}\{g_l\}+2\max_{1\le l\le n}\{\|\BIy_l\|\}\|\BIx\| 
		\qquad \forall 1\le k\le n,
	\end{align*}
	and we subsequently get 
	\begin{align*}
		\exp\bigg({\frac{g_k+\langle\BIy_k,\BIx\rangle}{\gamma}}\bigg)
		&\le \exp\bigg({\frac{g_j+\langle\BIy_j,\BIx\rangle}{\gamma}}\bigg)\exp\bigg(\frac{1}{\gamma}\big(\max_{1\le l\le n}\{g_l\}-\min_{1\le l\le n}\{g_l\}+2\max_{1\le l\le n}\{\|\BIy_l\|\}\|\BIx\|\big)\bigg) \\
		&\hspace{320pt} \forall 1\le k\le n.
	\end{align*}
	Summing the above inequality over $k=1,\ldots,n$ 
	proves (\ref{eqn: logsumexp-properties-proof-Hessian-lb-step2}).
	The proof is now complete.
\end{proof}

\begin{proof}[Proof of Proposition~\ref{prop: OTmap-estimator-entropic}]
	Throughout this proof, 
	let us define the functions 
	$\widetilde{\varphi}_{\ENTROPIC}[\gamma,\overline{l}]:\R^d\to\R$,
	$\zeta_{\STRONGLYCONVEX}[\psi]:\R_+\to\R$,
	$\varphi_{\STRONGLYCONVEX}[\psi]:\R^d\to\nobreak\R$
	and the matrix 
	$\widehat{\BSigma}\in\R^{d\times d}$ as follows:
	\begin{align*}
		\widetilde{\varphi}_{\ENTROPIC}[\gamma,\overline{l}](\BIx)
		&:= \gamma\log\left(\sum_{j=1}^n \exp\bigg({\frac{\widehat{g}^{(\gamma,\overline{l})}_j + \langle\BIY_j,\BIx\rangle}{\gamma}}\bigg)\right)  
		\hspace{113.8pt} \forall \BIx\in\R^d, \allowdisplaybreaks\\
		\zeta_{\STRONGLYCONVEX}[\psi](z)
		&:=\begin{cases}
			\exp\Big({\textstyle-\frac{1}{2z-\psi^2}}\Big) 
			& \hspace{170.5pt}\forall z\in\big({\textstyle\frac{\psi^2}{2}},\infty\big), \\
			0 
			& \hspace{176.9pt}\forall z\in\big[0,{\textstyle\frac{\psi^2}{2}}\big],
		\end{cases}\allowdisplaybreaks\\
		\varphi_{\STRONGLYCONVEX}[\psi](\BIx)
		&:=\int_{0}^{\frac{\|\BIx\|^2}{2}} \zeta_{\STRONGLYCONVEX}[\psi](z)\DIFFX{z} 
		\hspace{196.4pt}\forall \BIx\in\R^d, \allowdisplaybreaks\\
		\widehat{\BSigma}
		&:=\Bigg(\frac{1}{n}\sum_{j=1}^n\BIY_j\BIY_j^\TRANSP\Bigg) - \Bigg(\frac{1}{n}\sum_{j=1}^n\BIY_j\Bigg)\Bigg(\frac{1}{n}\sum_{j=1}^n\BIY_j\Bigg)^\TRANSP\in\R^{d\times d}.
	\end{align*}

	We divide the proof into the following 8 steps.
	\begin{itemize}[beginpenalty=10000]
		\item Step~1: 
		showing that 
		$\widehat{T}_{\ENTROPIC}[\gamma,\overline{l},\psi]=\nabla\big(\widetilde{\varphi}_{\ENTROPIC}[\gamma,\overline{l}]+\varphi_{\STRONGLYCONVEX}[\psi]\big)$
		and $\widetilde{\varphi}_{\ENTROPIC}[\gamma,\overline{l}],\varphi_{\STRONGLYCONVEX}[\psi]\in\CC^{\infty}(\R^d)$.
		
		\item Step~2: 
		proving statement~\ref{props: OTmap-estimator-entropic-Borel}.
		
		\item Step~3: 
		showing that both
		$\nabla^2\widetilde{\varphi}_{\ENTROPIC}[\gamma,\overline{l}](\BIx)\succeq \frac{e_{\MINSUB}(\widehat{\BSigma})}{\gamma}\exp\big({-\frac{2R_{\nu}(R_{\mu}+2\psi)}{\gamma}}\big)\BI_d$ 
		$\forall \BIx\in\bar{B}(\veczero_d,2\psi)$
		and
		$\nabla^2\widetilde{\varphi}_{\ENTROPIC}[\gamma,\overline{l}](\BIx)\succeq\BO_d$ $\forall\BIx\in\R^d$ 
		hold $\PROB$-almost surely.
		
		\item Step~4: 
		showing that
		$\nabla^2\varphi_{\STRONGLYCONVEX}[\psi](\BIx)\succeq\exp\big({-\frac{1}{3\psi^2}}\big)\BI_d$
		$\forall \BIx\in\R^d\setminus \bar{B}(\veczero_d,2\psi)$
		and 
		$\nabla^2\varphi_{\STRONGLYCONVEX}[\psi](\BIx)\succeq\BO_d$ $\forall\BIx\in\R^d$.

		\item Step~5: 
		proving statement~\ref{props: OTmap-estimator-entropic-shape}.
		
		\item Step~6: 
		proving statement~\ref{props: OTmap-estimator-entropic-growth}.

		\item Step~7:
		proving (\ref{eqn: OTmap-estimator-error}).

		\item Step~8:
		proving statement~\ref{props: OTmap-estimator-entropic-consistency}.
	\end{itemize}

	\underline{Step~1.}
	It follows directly from Lemma~\ref{lem: logsumexp-properties}\ref{lems: logsumexp-properties-gradient} with respect to 
	$(\BIy_j)_{j=1:n}\leftarrow (\BIY_j)_{j=1:n}$,
	$(g_j)_{j=1:n}\leftarrow \big(\widehat{g}^{(\gamma,\overline{l})}_j\big)_{j=1:n}$
	that 
	$\widetilde{T}_{\ENTROPIC}[\gamma,\overline{l}]=\nabla\widetilde{\varphi}_{\ENTROPIC}[\gamma,\overline{l}]$ and
	$\widetilde{\varphi}_{\ENTROPIC}[\gamma,\overline{l}]\in\CC^{\infty}(\R^d)$.
	Moreover, observe that 
	$T_{\STRONGLYCONVEX}[\psi](\BIx)=\zeta_{\STRONGLYCONVEX}[\psi]\big(\frac{1}{2}\|\BIx\|^2\big)\BIx$ $\forall\BIx\in\R^d$
	and $\zeta_{\STRONGLYCONVEX}[\psi]\in\CC^{\infty}(\R_+)$.
	It thus follows from the fundamental theorem of calculus that
	$T_{\STRONGLYCONVEX}[\psi]=\nabla\varphi_{\STRONGLYCONVEX}[\psi]$ 
	and $\varphi_{\STRONGLYCONVEX}[\psi]\in\CC^{\infty}(\R^d)$.
	We conclude that 
	$\widehat{T}_{\ENTROPIC}[\gamma,\overline{l},\psi]=\widetilde{T}_{\ENTROPIC}[\gamma,\overline{l}]+T_{\STRONGLYCONVEX}[\psi]=\nabla\big(\widetilde{\varphi}_{\ENTROPIC}[\gamma,\overline{l}]+\varphi_{\STRONGLYCONVEX}[\psi]\big)$.
	Step~1 is now complete.
	
	\underline{Step~2.}
	Observe from the definitions of 
	$\big(\widehat{f}^{(\gamma,\overline{l})}_i\big)_{i=1:m}$,
	$\big(\widehat{g}^{(\gamma,\overline{l})}_j\big)_{j=1:n}$
	in (\ref{eqn: Sinkhorn-empirical})
	as well as the definition of 
	$\widetilde{T}_{\ENTROPIC}[\gamma,\overline{l}]$
	in (\ref{eqn: OTmap-estimator-baryproj})
	that 
	$\widehat{T}_{\ENTROPIC}[\gamma,\overline{l},\psi]$ has a Borel dependency on 
	$(\BIX_1,\ldots,\BIX_m,\allowbreak\BIY_1,\ldots,\BIY_n,\allowbreak\gamma,\overline{l},\psi)$.
	Moreover, since 
	$\widetilde{T}_{\ENTROPIC}[\gamma,\overline{l}](\BIx)\in\conv\big(\{\BIY_1,\ldots,\BIY_n\}\big)$ $\forall\BIx\in\R^d$, we get
	\begin{align}
		\label{eqn: OTmap-estimator-entropic-proof-convhullbound}
		\big\|\widetilde{T}_{\ENTROPIC}[\gamma,\overline{l}](\BIx)\big\|\le \max_{1\le j\le n}\big\{\|\BIY_j\|\big\} \qquad \forall \BIx\in\R^d.
	\end{align}
	Furthermore, notice that 
	$\zeta_{\STRONGLYCONVEX}[\psi](z)\in[0,1)$ $\forall z\in\R_+$,
	which guarantees $\big\|T_{\STRONGLYCONVEX}[\psi](\BIx)\big\|\le \|\BIx\|$ $\forall\BIx\in\R^d$.
	Combining this with (\ref{eqn: OTmap-estimator-entropic-proof-convhullbound})
	shows that 
	$\widehat{T}_{\ENTROPIC}[\gamma,\overline{l},\psi]\in\CC_{\LIN}(\R^d,\R^d)$.
	The proof of statement~\ref{props: OTmap-estimator-entropic-Borel} is complete.
	
	\underline{Step~3.}
	First, since 
	$\big(\widehat{f}^{(\gamma,\overline{l})}_i\big)_{i=1:m}$,
	$\big(\widehat{g}^{(\gamma,\overline{l})}_j\big)_{j=1:n}$
	satisfy (\ref{eqn: Sinkhorn-empirical}),
	and since 
	$-R_{\mu}R_{\nu}\le\langle\BIX_i,\BIY_j\rangle\le R_{\mu}R_{\nu}$ holds
	\mbox{$\PROB$-almost} surely by the Cauchy--Schwarz inequality,
	it follows that 
	\begin{align*}
		-\gamma\log\Bigg(\frac{1}{m}\sum_{i=1}^{m}\exp\bigg(\frac{\widehat{f}_i^{(\gamma,\overline{l})}}{\gamma}\bigg)\Bigg) - R_{\mu}R_{\nu} 
		&\le \widehat{g}_j^{(\gamma,\overline{l})}
		\le -\gamma\log\Bigg(\frac{1}{m}\sum_{i=1}^{m}\exp\bigg(\frac{\widehat{f}_i^{(\gamma,\overline{l})}}{\gamma}\bigg)\Bigg) + R_{\mu}R_{\nu} 
		\\
		& \hspace{170pt} \forall 1\le j\le n,\; \PROB\text{-a.s.}
	\end{align*}
	It thus holds $\PROB$-almost surely that 
	$\max_{1\le j\le n}\big\{\widehat{g}_{j}^{(\gamma,\overline{l})}\big\}
	-\min_{1\le j\le n}\big\{\widehat{g}_{j}^{(\gamma,\overline{l})}\big\}
	\le 2R_{\mu}R_{\nu}$.
	Next,
	applying Lemma~\ref{lem: logsumexp-properties}\ref{lems: logsumexp-properties-Hessian-lb} with respect to 
	$(\BIy_j)_{j=1:n}\leftarrow (\BIY_j)_{j=1:n}$,
	$(g_j)_{j=1:n}\leftarrow \big(\widehat{g}^{(\gamma,\overline{l})}_j\big)_{j=1:n}$
	leads to
	\begin{align*}
		\nabla^2\widetilde{\varphi}_{\ENTROPIC}[\gamma,\overline{l}](\BIx)
		&\succeq \frac{e_{\MINSUB}(\widehat{\BSigma})}{\gamma}\exp\bigg({-\frac{1}{\gamma}}\bigg(\max_{1\le j\le n}\big\{\widehat{g}_{j}^{(\gamma,\overline{l})}\big\}
		-\min_{1\le j\le n}\big\{\widehat{g}_{j}^{(\gamma,\overline{l})}\big\}
		+2\max_{1\le j\le n}\big\{\|\BIY_j\|\big\}\|\BIx\|\bigg)\bigg)\BI_d\\
		&\succeq \frac{e_{\MINSUB}(\widehat{\BSigma})}{\gamma}\exp\bigg({-\frac{2R_{\nu}\big(R_{\mu}+\|\BIx\|\big)}{\gamma}}\bigg)\BI_d
		\hspace{100pt}\qquad \forall \BIx\in\R^d,\; \PROB\text{-a.s.}
	\end{align*}
	This shows that  
	$\nabla^2\widetilde{\varphi}_{\ENTROPIC}[\gamma,\overline{l}](\BIx)\succeq \frac{e_{\MINSUB}(\widehat{\BSigma})}{\gamma}\exp\big({-\frac{2R_{\nu}(R_{\mu}+2\psi)}{\gamma}}\big)\BI_d$
	$\forall \BIx\in\nobreak\bar{B}(\veczero_d,2\psi)$
	and 
	$\nabla^2\widetilde{\varphi}_{\ENTROPIC}[\gamma,\overline{l}](\BIx)\succeq \BO_d$ $\forall\BIx\in\R^d$
	both hold $\PROB$-almost surely. 
	Step~3 is now complete.
	
	\underline{Step~4.}
	Since $\zeta_{\STRONGLYCONVEX}[\psi]$ is differentiable and increasing on $\R_+$,
	we have 
	$\zeta_{\STRONGLYCONVEX}'[\psi](z)\ge 0$ $\forall z\in\R_+$.
	Consequently, it holds that 
	\begin{align*}
		\nabla^2\varphi_{\STRONGLYCONVEX}[\psi](\BIx)
		&=\zeta_{\STRONGLYCONVEX}'[\psi]\big({\textstyle\frac{1}{2}}\|\BIx\|^2\big)\BIx\BIx^\TRANSP 
		+\zeta_{\STRONGLYCONVEX}[\psi]\big({\textstyle\frac{1}{2}}\|\BIx\|^2\big)\BI_d
		\succeq \zeta_{\STRONGLYCONVEX}[\psi]\big({\textstyle\frac{1}{2}}\|\BIx\|^2\big)\BI_d
		\qquad \forall \BIx\in\R^d.
	\end{align*}
	This shows that 
	$\nabla^2\varphi_{\STRONGLYCONVEX}[\psi](\BIx)\succeq\exp\big({-\frac{1}{3\psi^2}}\big)\BI_d$
	$\forall\BIx\in\R^d\setminus\bar{B}(\veczero_d,2\psi)$
	as well as
	$\nabla^2\varphi_{\STRONGLYCONVEX}[\psi](\BIx)\succeq\BO_d$ 
	$\forall\BIx\in\R^d$.
	Step~4 is now complete.
	
	\underline{Step~5.}
	Combining Step~1, Step~3, Step~4,
	and utilizing the second-order characterization of strongly convex functions 
	(see, e.g., \citep[Theorem~2.1.11]{nesterov2004introconvex}) 
	show that 
	$\nabla\big(\widetilde{\varphi}_{\ENTROPIC}[\gamma,\overline{l}]+\varphi_{\STRONGLYCONVEX}[\psi]\big)=\widehat{T}_{\ENTROPIC}[\gamma,\overline{l},\psi]$
	where 
	$\widetilde{\varphi}_{\ENTROPIC}[\gamma,\overline{l}]+\varphi_{\STRONGLYCONVEX}[\psi]\in\FC^{\infty}_{\underline{\lambda},\infty}(\R^d)$ holds $\PROB$-almost surely.
	Moreover, one checks that 
	$\underline{\lambda}$ has a Borel dependency on 
	$(\mu,\nu,m,n,\allowbreak\BIX_1,\ldots,\BIX_m,\allowbreak\BIY_1,\ldots,\BIY_n,\allowbreak\gamma,\overline{l},\psi)$.
	Furthermore, whenever $n\ge d+1$, 
	since $\nu$ is absolutely continuous with respect to the Lebesgue measure on $\R^d$,
	it holds $\PROB$-almost surely that 
	the (biased) sample covariance matrix $\widehat{\BSigma}$ of $\BIY_1,\ldots,\BIY_n$ 
	is non-singular, 
	and thus $\underline{\lambda}>0$ holds $\PROB$-almost surely.
	The proof of statement~\ref{props: OTmap-estimator-entropic-shape} is now complete.
	
	\underline{Step~6.}
	It follows from (\ref{eqn: OTmap-estimator-entropic-proof-convhullbound})
	that 
	\begin{align*}
		\big\|\widehat{T}_{\ENTROPIC}[\gamma,\overline{l},\psi](\BIx)-\widehat{T}_{\ENTROPIC}[\gamma,\overline{l},\psi](\veczero_d)\big\|^2
		&\le 2\big\|\widetilde{T}_{\ENTROPIC}[\gamma,\overline{l}](\BIx)-\widetilde{T}_{\ENTROPIC}[\gamma,\overline{l}](\veczero_d)\big\|^2
		+ 2\big\|T_{\STRONGLYCONVEX}[\psi](\BIx)-T_{\STRONGLYCONVEX}[\psi](\veczero_d)\big\|^2\\
		&\le 8\max_{1\le j\le n}\big\{\|\BIY_j\|\big\}^2 + 2\zeta_{\STRONGLYCONVEX}[\psi]\big({\textstyle\frac{1}{2}}\|\BIx\|^2\big)^2\|\BIx\|^2\\
		&\le 8R_{\nu}^2 + 2\|\BIx\|^2 
		\hspace{140pt}\qquad \forall\BIx\in\R^d,\; \PROB\text{-a.s.}
	\end{align*}
	This proves statement~\ref{props: OTmap-estimator-entropic-growth}.
	
	\underline{Step~7.}
	We have by the assumption that $\mu,\nu\in\CM^q(\R^d)$ for $q\in\N$ and 
	Lemma~\ref{lem: curvature}
	that the Brenier potentials $\varphi^{\mu}_{\nu}$ and $\varphi^{\nu}_{\mu}$ satisfy $\varphi^{\mu}_{\nu}\in\CC^{q+2,\alpha}(\support(\mu))$ and $\varphi^{\nu}_{\mu}\in\CC^{q+2,\alpha}(\support(\nu))$ for some $\alpha\in(0,1)$,
	and that 
	$\lambda_{\LBSUB}\BI_d\preceq \nabla^2\varphi^{\mu}_{\nu}(\BIx)\preceq \lambda_{\UBSUB}\BI_d$
	$\forall \BIx\in\support(\mu)$
	for some $0<\lambda_{\LBSUB}\le\lambda_{\UBSUB}<\infty$.
	Thus, in view of the property that
	$T_{\STRONGLYCONVEX}[\psi](\BIx)=\veczero_d$ for $\mu$-almost every $\BIx\in\R^d$ whenever $\psi\ge R_{\mu}$,
	we can derive from the proofs of \citep[Theorem~4 \& Theorem~7]{pooladian2021entropic} that 
	there exist 
	$C_1(\mu,\nu),C_2(\mu,\nu),C_3(\mu,\nu),C_4(\mu,\nu),C_5(\mu,\nu),
	\overline{\gamma}(\mu,\nu)\in(0,\infty)$ 
	that have Borel dependencies on $(\mu,\nu)$ such that 
	\begin{align}
		\label{eqn: Pooladian-bound}
		\begin{split}
			\EXP\Big[\big\|\widehat{T}_{\ENTROPIC}[\gamma,\overline{l},\psi]-T^{\mu}_{\nu}\big\|^2_{\CL^2(\mu)}\Big]&=\EXP\Big[\big\|\widetilde{T}_{\ENTROPIC}[\gamma,\overline{l}]-T^{\mu}_{\nu}\big\|^2_{\CL^2(\mu)}\Big] \\
			&\le 
			C_1(\mu,\nu)\gamma^{2}I_0(\mu,\nu)
			+ C_2(\mu,\nu)\gamma^{\frac{\overline{\alpha}(\mu,\nu)}{2}} 
			+ C_3(\mu,\nu)\gamma^{1-\frac{d}{2}}\log(n)n^{-\frac{1}{2}} 
			\\
			& \qquad 
			+ C_4(\mu,\nu)\gamma^{-\frac{d}{2}}\log(m)m^{-\frac{1}{2}} 
			+ C_5(\mu,\nu) \gamma^{-\frac{3}{2}}\overline{l}^{-\frac{1}{2}}\log(m\vee n)^{\frac{1}{2}}\\
			&\hspace{1pt} \qquad \forall m,n\in\N\cap[2,\infty),\; \forall \gamma\in \big(0,\overline{\gamma}(\mu,\nu)\big],\; \forall \overline{l}\in\N,\; \forall \psi\in[R_{\mu},\infty),
		\end{split}
	\end{align}
	where
	\begin{align}
		\overline{\alpha}(\mu,\nu):=(q+2+\alpha)\wedge 4\in(3,4],
		\label{eqn: OTmap-estimator-entropic-proof-consistency-alpha-def}
	\end{align}
	and $I_0(\mu,\nu)\ge 0$ is the integrated Fisher information along the Wasserstein geodesic between $\mu$ and~$\nu$ defined in Appendix~A of \citep{pooladian2021entropic}.\footnote{We refer readers to the proofs of \citep[Theorem~4 \& Theorem~7]{pooladian2021entropic} in \citep[Section~4 \& Section~7]{pooladian2021entropic} for the explicit expressions 
	of 
	$C_1(\mu,\nu)$,
	$C_2(\mu,\nu)$,
	$C_3(\mu,\nu)$,
	$C_4(\mu,\nu)$,
	$C_5(\mu,\nu)$, and
	$\overline{\gamma}(\mu,\nu)$.}
	We remark that the assumptions (A1)--(A3) in \citep{pooladian2021entropic} 
	are satisfied with respect to $\mu$ and~$\nu$
	except for the assumption that the supports of $\mu$ and $\nu$ are the same.
	However, this common support assumption is in fact not needed for any step in the proofs of \citep[Theorem~4 \& Theorem~7]{pooladian2021entropic},
	and the bound (\ref{eqn: Pooladian-bound})
	is derived by following essentially the same steps as in the proofs of \citep[Theorem~4 \& Theorem~7]{pooladian2021entropic},
	up to replacing $P_{n}$ with $P_{m}$ where $m$ and $n$ are not necessarily equal in the proofs of \citep[Proposition~5 \& Theorem~7]{pooladian2021entropic}.

	Next, 
	let $f_{\mu}\in\CC(\support(\mu))$ denote the density function of $\mu$.
	Since $\varphi^{\mu}_{\nu}\in\CC^{q+2,\alpha}(\support(\mu))$ where $q+2\ge 3$, $\nabla^2\varphi^{\mu}_{\nu}$ is $L$-Lipschitz continuous on $\support(\mu)$
	for some $L>0$.
	Thus, in view of the property that 
	$\lambda_{\LBSUB}\BI_d\preceq \nabla^2\varphi^{\mu}_{\nu}(\BIx)\preceq \lambda_{\UBSUB}\BI_d$
	$\forall \BIx\in\support(\mu)$,
	\citep[Proposition~1]{chizat2020faster} implies that 
	$I_0(\mu,\nu)\le 2\lambda_{\LBSUB}^{-1}\big(I_0(\mu,\mu)+\frac{\lambda_{\LBSUB}^{-2}L^2}{3}\big)$,
	where 
	$I_0(\mu,\mu):=\int_{\support(\mu)}\big\|\nabla\log(f_{\mu}(\BIx))\big\|^2f_{\mu}(\BIx)\DIFFX{\BIx}$ 
	is the Fisher information of~$\mu$ (see \citep[equation~(5)]{chizat2020faster}).
	Since the assumption that $\mu\in\CM^q(\R^d)$ for $q\in\N$ guarantees that $f_{\mu}\in\CC^1(\support(\mu))$ and that $f_{\mu}$ is bounded from above and away from zero on $\support(\mu)$,
	it follows that $I_0(\mu,\mu)<\infty$ and thus $I_0(\mu,\nu)<\infty$.
	Subsequently, we can further bound the right-hand side of (\ref{eqn: Pooladian-bound}) as follows:
	\begin{align*}
		\EXP\Big[\big\|\widehat{T}_{\ENTROPIC}[\gamma,\overline{l},\psi]-T^{\mu}_{\nu}\big\|^2_{\CL^2(\mu)}\Big]
		&\le 
		C_3(\mu,\nu)\overline{\gamma}(\mu,\nu)\gamma^{-\frac{d}{2}}\log(n)n^{-\frac{1}{2}}+C_4(\mu,\nu)\gamma^{-\frac{d}{2}}\log(m)m^{-\frac{1}{2}} \\
		&\qquad + \Big(
			C_2(\mu,\nu) 
			+ C_1(\mu,\nu)\overline{\gamma}(\mu,\nu)^{\frac{4-\overline{\alpha}(\mu,\nu)}{2}}I_0(\mu,\nu)
		\Big)\gamma^{\frac{\overline{\alpha}(\mu,\nu)}{2}}\\
		&\qquad + 2C_5(\mu,\nu) \gamma^{-\frac{3}{2}}\overline{l}^{-\frac{1}{2}}\log(m\vee n) \allowdisplaybreaks\\
		&\le C_{\ENTROPIC}(\mu,\nu)\Big[\gamma^{-\frac{d}{2}}\big(\log(m)m^{-\frac{1}{2}}+\log(n)n^{-\frac{1}{2}}\big)\\
		&\qquad\qquad\qquad\;\; +\gamma^{\frac{\overline{\alpha}(\mu,\nu)}{2}}+ \gamma^{-\frac{3}{2}}\overline{l}^{-\frac{1}{2}}\log(m\vee n)\Big] \\
		&\hspace{10pt} \forall m,n\in\N\cap[2,\infty),\; \forall \gamma\in \big(0,\overline{\gamma}(\mu,\nu)\big],\; \forall \overline{l}\in\N,\;\forall\psi\in[R_{\mu},\infty),
	\end{align*}
	where
	\begin{align}
		\begin{split}
			C_{\ENTROPIC}(\mu,\nu)&:=
			C_3(\mu,\nu)\overline{\gamma}(\mu,\nu) 
			\vee C_4(\mu,\nu)
			\vee 2C_5(\mu,\nu) \\
			&\qquad \vee \Big(C_2(\mu,\nu) + C_1(\mu,\nu)\overline{\gamma}(\mu,\nu)^{\frac{4-\overline{\alpha}(\mu,\nu)}{2}}I_0(\mu,\nu) \Big)
			<\infty.
		\end{split}
		\label{eqn: OTmap-estimator-entropic-proof-consistency-Centr-def}
	\end{align}
	This proves (\ref{eqn: OTmap-estimator-error}) and completes Step~7.

	\underline{Step~8.}
	One verifies that 
	$\overline{m}(\mu,\nu,\epsilon)$, $\overline{n}(\mu,\nu,\epsilon)$ 
	have Borel dependencies on $(\mu,\nu,\epsilon)$,
	and that
	$\widetilde{\theta}(\mu,\nu,\allowbreak m,n,\epsilon)$
	has a Borel dependency on $(\mu,\nu,m,n,\epsilon)$.
	Next, let us fix arbitrary $\epsilon>0$, $m\ge \overline{m}(\mu,\nu,\epsilon)$, and $n\ge \overline{n}(\mu,\nu,\epsilon)$.
	Recall that the definitions of 
	$\overline{m}(\mu,\nu,\epsilon)$,
	$\overline{n}(\mu,\nu,\epsilon)$,
	$\widetilde{\gamma}(\mu,\nu,m,n,\epsilon)$,
	and 
	$\widetilde{\overline{l}}(\mu,\nu,m,n,\epsilon)$
	guarantee that
	\begin{align*}
		m&\ge 2,\\
		n&\ge 2,\allowdisplaybreaks\\
		m^{-\frac{\overline{\alpha}(\mu,\nu)}{2(\overline{\alpha}(\mu,\nu)+d)}}\big(2\log(m)+1\big)
		&\le \frac{\epsilon}{2C_{\ENTROPIC}(\mu,\nu)},\allowdisplaybreaks\\
		n^{-\frac{\overline{\alpha}(\mu,\nu)}{2(\overline{\alpha}(\mu,\nu)+d)}}\big(2\log(n)+1\big)
		&\le \frac{\epsilon}{2C_{\ENTROPIC}(\mu,\nu)},\allowdisplaybreaks\\
		\widetilde{\gamma}(\mu,\nu,m,n,\epsilon)
		&=(m\wedge n)^{-\frac{1}{\overline{\alpha}(\mu,\nu)+d}}\in \big(0,\overline{\gamma}(\mu,\nu)\big],\allowdisplaybreaks\\
		\widetilde{\gamma}(\mu,\nu,m,n,\epsilon)^{-\frac{d}{2}}
		&= m^{\frac{d}{2(\overline{\alpha}(\mu,\nu)+d)}} \wedge n^{\frac{d}{2(\overline{\alpha}(\mu,\nu)+d)}},\allowdisplaybreaks\\
		\widetilde{\gamma}(\mu,\nu,m,n,\epsilon)^{\frac{\overline{\alpha}(\mu,\nu)}{2}}
		&\le m^{-\frac{\overline{\alpha}(\mu,\nu)}{2(\overline{\alpha}(\mu,\nu)+d)}} + n^{-\frac{\overline{\alpha}(\mu,\nu)}{2(\overline{\alpha}(\mu,\nu)+d)}},\allowdisplaybreaks\\
		\widetilde{\gamma}(\mu,\nu,m,n,\epsilon)^{-\frac{3}{2}}\widetilde{\overline{l}}(\mu,\nu,m,n,\epsilon)^{-\frac{1}{2}}\log(m \vee n)
		&\le (m\vee n)^{-\frac{\overline{\alpha}(\mu,\nu)}{2(\overline{\alpha}(\mu,\nu)+d)}}\log(m\vee n)\\
		&\le m^{-\frac{\overline{\alpha}(\mu,\nu)}{2(\overline{\alpha}(\mu,\nu)+d)}}\log(m) + n^{-\frac{\overline{\alpha}(\mu,\nu)}{2(\overline{\alpha}(\mu,\nu)+d)}}\log(n).%
	\end{align*}
	Subsequently, substituting these into (\ref{eqn: OTmap-estimator-error}) yields
	\begin{align*}
		&\hspace{-20pt}\EXP\Big[\big\|\widehat{T}_{\ENTROPIC}\big[\widetilde{\theta}(\mu,\nu,m,n,\epsilon)\big]-T^{\mu}_{\nu}\big\|^2_{\CL^2(\mu)}\Big] \\
		&\le C_{\ENTROPIC}(\mu,\nu) 
		\Big(
			\widetilde{\gamma}(\mu,\nu,m,n,\epsilon)^{-\frac{d}{2}}
			\big(
				\log(m)m^{-\frac{1}{2}}
				+\log(n)n^{-\frac{1}{2}}
			\big) 
			+ \widetilde{\gamma}(\mu,\nu,m,n,\epsilon)^{\frac{\overline{\alpha}(\mu,\nu)}{2}} \\
			& \qquad\qquad\qquad\qquad + \widetilde{\gamma}(\mu,\nu,m,n,\epsilon)^{-\frac{3}{2}}\widetilde{\overline{l}}(\mu,\nu,m,n,\epsilon)^{-\frac{1}{2}}\log(m\vee n)
		\Big)\\
		&\le C_{\ENTROPIC}(\mu,\nu) 
		\Big(
			m^{-\frac{\overline{\alpha}(\mu,\nu)}{2(\overline{\alpha}(\mu,\nu)+d)}}\big(2\log(m)+1)
			+ n^{-\frac{\overline{\alpha}(\mu,\nu)}{2(\overline{\alpha}(\mu,\nu)+d)}}\big(2\log(n)+1\big)
		\Big)
		\le \epsilon.
	\end{align*}
	This completes the proof of statement~\ref{props: OTmap-estimator-entropic-consistency}.
	The proof is now complete.
\end{proof}

\subsection{Omitted proofs in Section~\ref{sec: generation-instance}}\label{sapx: proof-generation-instance}

\begin{proof}[Proof of Proposition~\ref{prop: synthetic-generation}]
	To begin, let $(\overline{\lambda}_{\widetilde{k}})_{\widetilde{k}=1:\widetilde{K}}$ satisfy the condition in Line~\ref{alglin: synthetic-generation-smoothness},
	let $(\beta_{-\widetilde{k}}, \beta_{\widetilde{k}})_{\widetilde{k}=1:\widetilde{K}}$ be defined in Line~\ref{alglin: synthetic-generation-multiplier},
	and let us define $\widetilde{\varphi}_1,\ldots,\widetilde{\varphi}_{\widetilde{K}},\widetilde{\varphi}_{-1},\ldots,\widetilde{\varphi}_{-\widetilde{K}},\varphi_1,\ldots,\varphi_K,\varphi^*_1,\ldots,\varphi^*_K:\R^d\to\R$ as follows:
	\begin{align*}
			\widetilde{\varphi}_{\widetilde{k}}(\BIx)&:= \gamma_{\widetilde{k}}\log\Bigg(\sum_{j=1}^{n_{\widetilde{k}}} \exp\bigg({\frac{g_{\widetilde{k},j} + \langle\BIy_{\widetilde{k},j},\BIx\rangle}{\gamma_{\widetilde{k}}}}\bigg)\Bigg) + \frac{\underline{\lambda}_{\widetilde{k}}}{2}\|\BIx\|^2 
			\qquad \forall \BIx\in\R^d, \; \forall 1\le \widetilde{k}\le \widetilde{K},\allowdisplaybreaks\\
			\widetilde{\varphi}_{-\widetilde{k}}(\BIx)&:=\frac{\overline{\lambda}_{\widetilde{k}}}{2}\|\BIx\|^2-\widetilde{\varphi}_{\widetilde{k}}(\BIx) 
			\hspace{142.9pt}\qquad\forall \BIx\in\R^d, \; \forall 1\le \widetilde{k}\le \widetilde{K}, \allowdisplaybreaks\\
			\varphi_k(\BIx)&:=\Bigg(\sum_{i\in\Phi^{-1}(k)}\beta_i \widetilde{\varphi}_{i}(\BIx)\Bigg) + \xi\big\langle{\textstyle\frac{1}{2}\BA_k}\BIx+\BIb_k,\BIx\big\rangle 
			\hspace{35.1pt}\qquad \forall \BIx\in\R^d, \; \forall 1\le k\le K,\allowdisplaybreaks\\
			\varphi^*_k(\BIy)&:=\sup_{\BIx\in\R^d}\big\{\langle\BIy,\BIx\rangle-\varphi_k(\BIx)\big\} 
			\hspace{117.4pt}\qquad\forall \BIy\in\R^d,\;\forall 1\le k\le K.
	\end{align*}

	For $\widetilde{k}=1,\ldots,\widetilde{K}$, it follows from 
	Line~\ref{alglin: synthetic-generation-auxiliary-map}, 
	Lemma~\ref{lem: logsumexp-properties}\ref{lems: logsumexp-properties-gradient}
	and
	Lemma~\ref{lem: logsumexp-properties}\ref{lems: logsumexp-properties-Hessian-ub}
	with respect to 
	$(\BIy_j)_{j=1:n}\leftarrow(\BIy_{\widetilde{k},j})_{j=1:n_{\widetilde{k}}}$,
	$(g_j)_{j=1:n}\leftarrow (g_{\widetilde{k},j})_{j=1:n_{\widetilde{k}}}$,
	$\gamma\leftarrow \gamma_{\widetilde{k}}$
	that 
	$\nabla\widetilde{\varphi}_{\widetilde{k}}=\widetilde{T}_{\widetilde{k}}$ and
	$\widetilde{\varphi}_{\widetilde{k}}\in\FC^{\infty}_{\underline{\lambda}_{\widetilde{k}},\overline{\lambda}_{\widetilde{k}}-\underline{\lambda}_{\widetilde{k}}}(\R^d)$ 
	where 
	$0<\underline{\lambda}_{\widetilde{k}}<\overline{\lambda}_{\widetilde{k}}-\underline{\lambda}_{\widetilde{k}}<\infty$.
	Moreover, Line~\ref{alglin: synthetic-generation-auxiliary-map} implies that
	$\nabla\widetilde{\varphi}_{-\widetilde{k}}=\widetilde{T}_{-\widetilde{k}}$.
	It then follows from the second-order characterization of smooth and strongly convex functions (see, e.g., \citep[Theorem~2.1.6 \& Theorem~2.1.11]{nesterov2004introconvex}) that 
	$\underline{\lambda}_{\widetilde{k}}\BI_d 
	\preceq \nabla^2 \widetilde{\varphi}_{-\widetilde{k}}(\BIx) 
	\preceq (\overline{\lambda}_{\widetilde{k}}-\underline{\lambda}_{\widetilde{k}})\BI_d$ 
	$\forall\BIx\in\R^d$,
	and thus $\widetilde{\varphi}_{-\widetilde{k}}
	\in\FC^{\infty}_{\underline{\lambda}_{\widetilde{k}},\overline{\lambda}_{\widetilde{k}}-\underline{\lambda}_{\widetilde{k}}}(\R^d)$.
	Subsequently, 
	since 
	$(\widetilde{\varphi}_{\widetilde{k}},\widetilde{\varphi}_{-\widetilde{k}})_{\widetilde{k}=1:\widetilde{K}}$ are all infinitely differentiable, smooth, and strongly convex, 
	and since 
	$(\beta_{\widetilde{k}},\beta_{-\widetilde{k}})_{\widetilde{k}=1:\widetilde{K}}\subset (0,\infty)$, 
	$\xi\in[0,1)$,
	$(\BA_{k})_{k=1:K}\subset\mathbb{S}^d_{++}$,
	it holds by definition that
	$\varphi_k\in\FC^{\infty}_{\underline{\zeta}_k,\overline{\zeta}_k}(\R^d)$ for some 
	$0<\underline{\zeta}_k<\overline{\zeta}_k<\infty$ for $k=1,\ldots,K$.
	Hence, since
	Line~\ref{alglin: synthetic-generation-maps} implies $T_k=\nabla\varphi_k$, it holds that $T_k$ is $\overline{\zeta}_k$-Lipschitz continuous and thus belongs to $\CC_{\LIN}(\R^d,\R^d)$.
	Furthermore, 
	since Line~\ref{alglin: synthetic-generation-multiplier} implies $w_{\Phi(\widetilde{k})}\beta_{\widetilde{k}} = w_{\Phi(-\widetilde{k})}\beta_{-\widetilde{k}}$ for $1\le \widetilde{k}\le\widetilde{K}$
	as well as
	$\sum_{\widetilde{k} = 1}^{\widetilde{K}}w_{\Phi(-\widetilde{k})} \beta_{-\widetilde{k}} \overline{\lambda}_{\widetilde{k}}=1-\nobreak\xi$,
	and since $\sum_{k=1}^K w_k \BA_k= \BI_{d}$,
	$\sum_{k=1}^K w_k\BIb_k = \veczero_d$ by the assumptions in Setting~\ref{sett: config-generation},
	we get
	\begin{align}
		\begin{split}
		\sum_{k=1}^K w_k \varphi_k(\BIx)&= \Bigg(\sum_{k = 1}^K w_k \sum_{i\in\Phi^{-1}(k)}\beta_i\widetilde{\varphi}_{i}(\BIx)\Bigg) + \xi\sum_{k = 1}^K\big\langle{\textstyle\frac{1}{2}} w_k \BA_k\BIx+ w_k\BIb_k,\BIx\big\rangle\\
		&= \sum_{\widetilde{k} = 1}^{\widetilde{K}}\Big( w_{\Phi(\widetilde{k})} \beta_{\widetilde{k}} \widetilde{\varphi}_{\widetilde{k}}(\BIx) + w_{\Phi(-\widetilde{k})} \beta_{-\widetilde{k}} \widetilde{\varphi}_{-\widetilde{k}}(\BIx)\Big) + \frac{\xi}{2}\|\BIx\|^2 \\
		&= \frac{1}{2}\Bigg( \sum_{\widetilde{k} = 1}^{\widetilde{K}}w_{\Phi(-\widetilde{k})} \beta_{-\widetilde{k}} \overline{\lambda}_{\widetilde{k}}\Bigg)\|\BIx\|^2 + \frac{\xi}{2}\|\BIx\|^2 
		= \frac{1}{2}\|\BIx\|^2 \qquad \forall \BIx\in\R^d.
		\end{split}
		\label{eqn: synthetic-generation-proof-conjugacy}
	\end{align}

	In the following,
	we will derive some consequences of (\ref{eqn: synthetic-generation-proof-conjugacy}) through 
	Brenier's theorem (Theorem~\ref{thm: Brenier}); see also \citep[Appendix~C.2]{chewi2020gradient}.
	For $k=1,\ldots,K$,
	notice that $[I_d,T_k]\sharp\bar{\mu}\in\Pi(\bar{\mu},T_k\sharp\bar{\mu})$ is the unique optimal coupling between $\bar{\mu}$ and $T_k\sharp\bar{\mu}$ by Theorem~\ref{thm: Brenier},
	and thus 
	\begin{align}
		\label{eqn: synthetic-generation-proof-lowerbound}
		\sum_{k=1}^{K}w_k\CW_2(\bar{\mu},T_k\sharp\bar{\mu})^2=\int_{\R^d}{\textstyle\sum_{k=1}^{K}}w_k\big\|\BIx-T_k(\BIx)\big\|^2\DIFFM{\bar{\mu}}{\DIFF\BIx}=V_{\MINSUB}.
	\end{align}
	On the one hand,
	since $\varphi_k(\BIx)+\varphi^*_k\big(T_k(\BIx)\big)=\langle\BIx,T_k(\BIx)\rangle$ $\forall\BIx\in\R^d$,
	it holds for $[I_d,T_k]\sharp\bar{\mu}$-almost every $(\BIx,\BIy)\in\R^d\times\R^d$ that 
	$\varphi_k(\BIx)+\varphi_k^*(\BIy)=\langle\BIx,\BIy\rangle$.
	Thus, it follows from 
	(\ref{eqn: synthetic-generation-proof-conjugacy})
	that
	\begin{align}
		\begin{split}
			\sum_{k=1}^K w_k \CW_2(\bar{\mu},T_k\sharp\bar{\mu})^2 &= \sum_{k=1}^K w_k \Bigg(\int_{\R^d} \|\BIx\|^2 - 2\varphi_k(\BIx)\DIFFM{\bar{\mu}}{\DIFF\BIx} + \int_{\R^d}\|\BIy\|^2 - 2\varphi^*_k(\BIy) \DIFFM{T_k\sharp\bar{\mu}}{\DIFF\BIy} \Bigg)\\
			&=\sum_{k=1}^K w_k \int_{\R^d}\|\BIy\|^2-2\varphi^*_k(\BIy)\DIFFM{T_k\sharp\bar{\mu}}{\DIFF\BIy}.
		\end{split}\label{eqn: synthetic-generation-proof-optimality}
	\end{align}
	On the other hand, since $\varphi_k(\BIx)+\varphi^*_k(\BIy)\ge \langle\BIx,\BIy\rangle$ $\forall\BIx,\BIy\in\R^d$, it holds that
	\begin{align}
		\begin{split}
			\sum_{k=1}^K w_k \CW_2(\mu,\nu_k)^2 &\ge \sum_{k=1}^K w_k \Bigg(\int_{\R^d}\|\BIx\|^2 - 2\varphi_k(\BIx)\DIFFM{\mu}{\DIFF\BIx} + \int_{\R^d}\|\BIy\|^2 - 2\varphi^*_k(\BIy) \DIFFM{\nu_k}{\DIFF\BIy}\Bigg) \\
			&=\sum_{k=1}^K w_k  \int_{\R^d} \|\BIy\|^2-2\varphi^*_k(\BIy)\DIFFM{\nu_k}{\DIFF\BIy} \qquad\qquad \forall \mu,\nu_1,\ldots,\nu_K\in\CP_2(\R^d).
		\end{split}\label{eqn: synthetic-generation-proof-suboptimality}
	\end{align}
	Combining 
	(\ref{eqn: synthetic-generation-proof-optimality}) 
	and 
	(\ref{eqn: synthetic-generation-proof-suboptimality})
	leads to 
	\begin{align}
		\inf_{\mu\in\CP_2(\R^d)}\Bigg\{\sum_{k=1}^{K}w_k\CW_2(\mu,\nu_k)^2\Bigg\}
		&\ge \sum_{k=1}^K w_k  \int_{\R^d} \|\BIy\|^2-2\varphi^*_k(\BIy)\DIFFM{\nu_k}{\DIFF\BIy} \label{eqn: synthetic-generation-proof-key-inequality}\\
		&\ge \sum_{k=1}^K w_k \CW_2(\bar{\mu},T_k\sharp\bar{\mu})^2 - \Bigg(\sum_{k=1}^K w_k  \int_{\R^d} \big|\|\BIy\|^2-2\varphi^*_k(\BIy)\big|\DIFFM{|\nu_k-T_k\sharp\bar{\mu}|}{\DIFF\BIy}\Bigg)\nonumber\\
		& \hspace{238pt} \forall \nu_1,\ldots,\nu_K\in\CP_2(\R^d).\nonumber
	\end{align}

	Now, let us consider the case where $\mathtt{TRUNCATE}=\mathtt{False}$ and prove statements~\ref{props: synthetic-generation-notruncation} and \ref{props: synthetic-generation-lowerbound}.
	For $k=1,\ldots,K$,
	since 
	$T_k=\nabla\varphi_k\in\CC_{\LIN}(\R^d,\R^d)$ for $\varphi_k\in\FC^{\infty}_{\underline{\zeta}_k,\overline{\zeta}_k}(\R^d)$,
	and since $\bar{\mu}\in\CP_{2,\AC}(\R^d)$ by assumption,
	it follows from Lemma~\ref{lem: regularity-pushforward}\ref{lems: regularity-pushforward-ac}
	that $\nu_k\in\CP_{2,\AC}(\R^d)$.
	Subsequently, the uniqueness of the $\CW_2$-barycenter of $\nu_1,\ldots,\nu_K$ with weights $w_1,\ldots,w_K$ follows from Theorem~\ref{thm: unique barycenter}.
	Moreover, since $\nu_k=T_k\sharp\bar{\mu}$ for $k=1,\ldots,K$
	by Line~\ref{alglin: synthetic-generation-pushforwards},
	(\ref{eqn: synthetic-generation-proof-key-inequality})
	and 
	(\ref{eqn: synthetic-generation-proof-lowerbound})
	show that $\bar{\mu}$ is the unique $\CW_2$-barycenter of $\nu_1,\ldots,\nu_K$ with weights $w_1,\ldots,w_K$
	and that 
	$V_{\MINSUB}=V(\bar{\mu})$.
	This completes the proofs of statements~\ref{props: synthetic-generation-notruncation} 
	and 
	\ref{props: synthetic-generation-lowerbound}.

	Next, let us consider the case where $\mathtt{TRUNCATE}=\mathtt{True}$ and prove statements~\ref{props: synthetic-generation-truncation}--\ref{props: synthetic-generation-error}.
	For $k=1,\ldots,K$,
	since 
	$T_k=\nabla\varphi_k\in\CC_{\LIN}(\R^d,\R^d)$ for $\varphi_k\in\FC^{\infty}_{\underline{\zeta}_k,\overline{\zeta}_k}(\R^d)$,
	and since $\bar{\mu}\in\CM^q_{\FULL}(\R^d)$ by assumption,
	it follows from Lemma~\ref{lem: regularity-pushforward}\ref{lems: regularity-pushforward-Caffarelli} that $T_k\sharp\bar{\mu}\in\CM^q_{\FULL}(\R^d)$.
	Moreover, since $\CY_k\in\CS^q(\R^d)$, 
	it follows from 
	Line~\ref{alglin: synthetic-generation-truncated}
	and Lemma~\ref{lem: regularity-truncation}\ref{lems: regularity-truncation-preserve} that
	$\nu_k\in\CM^q(\R^d)$.
	This proves statement~\ref{props: synthetic-generation-truncation}.
	To prove statement~\ref{props: synthetic-generation-truncation-V-value}, 
	let us 
	define the terms
	$\big(\epsilon_{1,k}(\CY_k),\epsilon_{2,k}(\CY_k)\big)_{k=1:K}\subset [0,\infty)$ and
	$\epsilon(\CY_1,\ldots,\CY_K)\in(0,\infty)$
	for any $\CY_1,\ldots,\CY_K\in\CS^q(\R^d)$ 
	as follows:
	\begin{align}
		\label{eqn: synthetic-generation-truncation-error}
		\begin{split}
			\epsilon_{1,k}(\CY_k)&:= \!\int_{\R^d}2\|\BIx\|^2\Big({\textstyle\frac{1-T_k\sharp\bar{\mu}(\CY_k)}{T_k\sharp\bar{\mu}(\CY_k)}}+\INDI_{\R^d\setminus\CY_k}(\BIx)\Big)\DIFFM{T_k\sharp\bar{\mu}}{\DIFF\BIx} 
			\hspace{48.7pt}\;\; \forall \CY_k\in\CS^q(\R^d),\; \forall 1\le k\le K,\\
			\epsilon_{2,k}(\CY_k)&:= \!\int_{\R^d}\big|\|\BIx\|^2-2\varphi^*_k(\BIx)\big|\Big({\textstyle\frac{1-T_k\sharp\bar{\mu}(\CY_k)}{T_k\sharp\bar{\mu}(\CY_k)}}+\INDI_{\R^d\setminus\CY_k}(\BIx)\Big)\DIFFM{T_k\sharp\bar{\mu}}{\DIFF\BIx} 
			\;\; \forall \CY_k\in\CS^q(\R^d),\;  \forall 1\le k\le K, \\
			\epsilon(\CY_1,\ldots,\CY_K)&:=\!\sum_{k=1}^K w_k\Big(2\CW_2(\bar{\mu},T_k\sharp\bar{\mu})\epsilon_{1,k}(\CY_k)^{\frac{1}{2}} + \epsilon_{1,k}(\CY_k) + \epsilon_{2,k}(\CY_k)\Big) 
			\hspace{20.5pt}\qquad \forall \CY_1,\ldots,\CY_K\in\CS^q(\R^d).
		\end{split}
	\end{align}
	Let us fix arbitrary $\CY_1,\ldots,\CY_K\in\CS^q(\R^d)$.
	On the one hand,
	one can apply Lemma~\ref{lem: truncation-bound} to derive the upper bounds
	$\CW_2(T_k\sharp\bar{\mu},\nu_k)^2\le \epsilon_{1,k}(\CY_k)$ 
	$\forall 1\le k\le K$,
	which then lead to the following inequality:
	\begin{align}
		\begin{split}
			V(\bar{\mu})&\le \sum_{k=1}^Kw_k\Bigg( \CW_2(\bar{\mu},T_k\sharp\bar{\mu})^2 + 2\CW_2(\bar{\mu},T_k\sharp\bar{\mu})\CW_2(T_k\sharp\bar{\mu},\nu_k) + \CW_2(T_k\sharp\bar{\mu},\nu_k)^2 \Bigg)\\
			& \le \Bigg(\sum_{k=1}^Kw_k \CW_2(\bar{\mu},T_k\sharp\bar{\mu})^2\Bigg) + \Bigg(\sum_{k=1}^Kw_k \Big(2\CW_2(\bar{\mu},T_k\sharp\bar{\mu})\epsilon_{1,k}(\CY_k)^{\frac{1}{2}} + \epsilon_{1,k}(\CY_k)\Big)\Bigg).
		\end{split}
		\label{eqn: synthetic-generation-with-truncation-proof-bound1}
	\end{align}
	On the other hand, we have
	\begin{align}
		\begin{split}
			\int_{\R^d}\big|\|\BIy\|^2-2\varphi^*_k(\BIy)\big| \DIFFM{|T_k\sharp\bar{\mu}-\nu_k|}{\DIFF\BIy}
			&= \int_{\R^d}\big|\|\BIy\|^2-2\varphi^*_k(\BIy)\big|\Big|1-{\textstyle\frac{\INDI_{\CY_k}(\BIy)}{T_k\sharp\bar{\mu}(\CY_k)}}\Big| \DIFFM{T_k\sharp\bar{\mu}}{\DIFF\BIy}\\
			&= \epsilon_{2,k}(\CY_k) 
			\hspace{115pt} \forall 1\le k\le K.
		\end{split}
		\label{eqn: synthetic-generation-with-truncation-proof-bound2}
	\end{align}
	Subsequently, combining 
	(\ref{eqn: synthetic-generation-with-truncation-proof-bound1}), 
	(\ref{eqn: synthetic-generation-proof-key-inequality}),
	(\ref{eqn: synthetic-generation-with-truncation-proof-bound2}), and
	(\ref{eqn: synthetic-generation-proof-lowerbound})
	leads to 
	\begin{align*}
		V(\bar{\mu})&\le \inf_{\mu\in\CP_2(\R^d)}\big\{V(\mu)\big\}+\epsilon(\CY_1,\ldots,\CY_K),\allowdisplaybreaks\\
		V_{\MINSUB}&\le \inf_{\mu\in\CP_2(\R^d)}\big\{V(\mu)\big\}+\sum_{k=1}^{K}w_k\epsilon_{2,k}(\CY_k)\le \inf_{\mu\in\CP_2(\R^d)}\big\{V(\mu)\big\}+\epsilon(\CY_1,\ldots,\CY_K),
	\end{align*}
	as well as 
	\begin{align*}
		V_{\MINSUB}&\ge V(\bar{\mu}) - \Bigg(\sum_{k=1}^Kw_k \Big(2\CW_2(\bar{\mu},T_k\sharp\bar{\mu})\epsilon_{1,k}(\CY_k)^{\frac{1}{2}} + \epsilon_{1,k}(\CY_k)\Big)\Bigg)
		\ge \inf_{\mu\in\CP_2(\R^d)}\big\{V(\mu)\big\} - \epsilon(\CY_1,\ldots,\CY_K),
	\end{align*}
	which complete the proof of statement~\ref{props: synthetic-generation-truncation-V-value}.

	Lastly, for $k=1,\ldots,K$, observe that $\varphi_k^*$ is $\underline{\zeta}_k^{-1}$-smooth and $\overline{\zeta}_k^{-1}$-strongly convex by the duality between smooth convex functions and strongly convex functions (see, e.g., the equivalence between (a) and (e) in \citep[Proposition~12.60]{rockafellar1998variational}).
	Hence, $\varphi_k^*$ is bounded from below by some constant and dominated from
	above by some quadratic function, e.g., by
	$\R^d\ni\BIx\mapsto \frac{1}{\underline{\zeta}_k}\|\BIx\|^2+C\in\R$ for sufficiently large $C>0$.
	Consequently, for $k=\nobreak1,\ldots,K$, 
	the properties
	$T_k\sharp\bar{\mu}\in\CP_2(\R^d)$,
	$\bigcup_{r\in\N}\CY_{k,r}=\R^d$, and
	Lebesgue's dominated convergence theorem imply that $\lim_{r\to\infty}\epsilon_{1,k}(\CY_{k,r})=\lim_{r\to\infty}\epsilon_{2,k}(\CY_{k,r})=\nobreak0$,
	and we thus get $\lim_{r\to\infty}\epsilon(\CY_{1,r},\ldots,\CY_{K,r})=\nobreak0$, which proves statement~\ref{props: synthetic-generation-error}.
	The proof is now complete.
\end{proof}

\section{Omitted details of the numerical experiments}\label{apx: numerics}

\subsection{Description of algorithms}\label{sapx: benchmark-descriptions}
Table~\ref{tab: benchmark-summary} summarizes the supplementary details of Algorithm~\ref{algo: concrete} and other benchmark algorithms considered for comparisons in Section~\ref{sec: numerics}, 
including the operations performed in each iteration, the number of iterations, the input and output variables, and the empirical wall-clock run-time until \textit{approximate convergence} across problem instances.
When applying these benchmark algorithms to our experiments, we adopt the model hyperparameters specified in their respective source code.

\begin{table}
\centering
\footnotesize
\begin{adjustbox}{max width=\linewidth, max totalheight=0.78\textheight}
\begin{tabular}{L{2.0cm} L{5.2cm} L{4.5cm} L{6.4cm} L{5.51cm}}
\toprule
\textbf{Method} & \textbf{Summary} & \textbf{Operations performed in each (outermost) iteration} & \textbf{Input, output \& sample sizes} & \textbf{Number of (outermost) iterations \& run-time to approximate convergence} \\
\midrule
Algorithm~\ref{algo: concrete} (ours)
&
Approximates the fixed-point iterative scheme of \citet{alvarez2016fixed} by replacing the true OT maps with modified entropic OT map estimators.
&
Each outermost iteration computes the modified entropic OT map estimators (via $\mathtt{geomloss}$) from the current barycenter candidate to the input measures using samples.
&
\textbf{Input:} streams of iid samples from the input measures;
the sample size used at the first iteration is $10000$ in [SG-2d] and $5000$ in [SG-10d]/[BS-8d], which has already led to approximate convergence. \newline
\textbf{Output:} the computed OT map estimators and the
barycenter candidate represented as 
the pushforward of the initial measure by a composition map followed by truncation.
&
\textbf{[SG-2d]} 1 iteration, 389\,s (0.11\,h)\newline
\textbf{[SG-10d]} 1 iteration, 5389\,s (1.50\,h)\newline
\textbf{[BS-8d]} 1 iteration, 10103\,s (2.81\,h)
\\
\midrule
\citet{cuturi2014fast}
&
Alternates between optimizing the support locations and atom weights of a discrete approximate barycenter with a prescribed number of atoms. The built-in $\mathtt{POT}$ library function is invoked, which removes the line search and weight optimization steps.
&
Each iteration solves exact discrete OT problems between the current barycenter candidate and the input measures, then updates the barycenter candidate via barycentric projection.
&
\textbf{Input:} each input measure discretized into $10000$ atoms (fixed across iterations). \newline
\textbf{Output:} the resulting $10000$-atom equally-weighted discrete barycenter candidate.
&
\textbf{[SG-2d]} 200 iterations, 21045\,s (5.85\,h)\newline
\textbf{[SG-10d]} 200 iterations, 15667\,s (4.35\,h)\newline
\textbf{[BS-8d]} 200 iterations, 5753\,s (1.60\,h)
\\
\midrule
\citet{fan2021scalable}
&
Reformulates the barycenter problem using the semi-dual formulation as a min-max-min problem, and parametrizes the Brenier potentials with input convex neural networks and the barycenter candidate with a generative network.
&
Each outermost iteration draws one minibatch of samples from each input measure and from the latent distribution, and then trains the potential networks and the barycenter generative network with these samples.
&
\textbf{Input:} streams of iid samples from the input measures and the latent distribution, with $100$ samples drawn per minibatch. \newline
\textbf{Output:} the trained potential networks and barycenter generative network; samples from the barycenter candidate are generated by forward-propagating latent samples through the barycenter generative network.
&
\hspace{-1pt}\hspace{1pt}\textbf{[SG-2d]} 100 iterations, 37543\,s (10.43\,h)\newline
\textbf{[SG-10d]} 100 iterations, 71025\,s (19.73\,h)\newline
\textbf{[BS-8d]} 100 iterations, 21145\,s (5.87\,h)
\\
\midrule
\citet{korotin2022wasserstein}%
&
Parametrizes the barycenter with a generative network and uses neural networks to approximate the fixed-point iterative scheme of \citet{alvarez2016fixed}.%
&
Each outermost iteration consists of
drawing one minibatch of samples from each input measure and from the latent distribution,
training the neural network parametrizations of the Brenier potentials and the OT maps,
and training the barycenter generative network.%
&
\textbf{Input:} streams of iid samples from the input measures and the latent distribution, with $1024$ samples drawn per minibatch. \newline
\textbf{Output:} the trained barycenter generative network and the approximated OT maps from the barycenter candidate to the input measures; samples from the barycenter candidate are generated by forward-propagating latent samples through the barycenter generative network.%
&
\textbf{[SG-2d]} 334 iterations, 112573\,s (31.27\,h)\newline
\textbf{[SG-10d]} 182 iterations, 164418\,s (45.67\,h)\newline
\textbf{[BS-8d]} 334 iterations, 65187\,s (18.11\,h)
\\
\midrule
\citet{li2020continuous}
&
Applies stochastic dual ascent to an $\ell_2$-regularized Wasserstein barycenter formulation, which is biased relative to the $\CW_2$-barycenter that we study.
&
Each outermost iteration takes a single Adam step on the regularized Brenier potentials approximated by samples from the input measures and an auxiliary barycenter support measure.
&
\textbf{Input:} streams of iid samples from the input measures, with $8192$ samples drawn per input measure for each outermost iteration. \newline
\textbf{Output:} the approximated Brenier potentials; a barycenter sample is obtained by forward-evaluating an OT map from the gradient of a Brenier potential.
&
\textbf{[SG-2d]} 50000 iterations, 16122\,s (4.48\,h)\newline
\textbf{[SG-10d]} 50000 iterations, 35411\,s (9.84\,h)\newline
\textbf{[BS-8d]} 50000 iterations, 11406\,s (3.17\,h)
\\
\midrule
\citet{kim2025optimal}
&
Reformulates the barycenter problem using the dual formulation of the Wasserstein distance, and adopts a first-order method using gradients on the Wasserstein space and a Hilbert space. The algorithm is restricted to two-dimensional problems.
&
Each outermost iteration uses input densities to update the dual potentials between each input measure and the current barycenter candidate via $\dot H^1$-ascent followed by a projection, then updates the barycenter density by Wasserstein descent.
&
\textbf{Input:} $2048\times2048$ KDE density grids built once from $10000$ samples per input measure. \newline
\textbf{Output:} the computed barycenter density on the same grid; $10000$ samples are subsequently drawn from it for evaluation.
&
\textbf{[SG-2d]} 500 iterations, 2835\,s (0.79\,h)\newline
\textbf{[SG-10d]} not applicable\newline
\textbf{[BS-8d]} not applicable
\\
\bottomrule
\end{tabular}
\end{adjustbox}
\caption{Summary of Algorithm~\ref{algo: concrete} and of the benchmark algorithms considered in Section~\ref{sec: numerics}, together with their empirical iteration counts and wall-clock run-time to approximate convergence. 
For Algorithm~\ref{algo: concrete}, 
we report the run-time of a single iteration 
which has led to near-optimal accuracy; see Figures~\ref{fig: stochastic-FP-SG} and~\ref{fig: stochastic-FP-BS-8d}.}\label{tab: benchmark-summary}
\end{table}

\subsection{Omitted details in Experiment~1}\label{sapx: exp-details-synthetic-generation}
In each problem instance, Algorithm~\ref{algo: concrete} is run for 9 iterations, which produce 10 probability measures $(\widehat{\mu}_t)_{t = 0:9}$.
To initialize Algorithm~\ref{algo: concrete}, we generate $10^4$ samples from $\nu_1, \dots, \nu_K$, and define $\rho_0$ as the multivariate Gaussian distribution parametrized by the empirical mean and covariance of these samples, such that $\rho_0 \in \CM^q_{\FULL}(\R^d)$ for all $q \in \N_0$.
Uniform truncation indices $(\widehat{R}_t)_{t =0: 9}$ are assigned across all iterations and closed spheres $\CX_{\widehat{R}_t} = \bar{B}(\veczero_d, \widehat{R}_t)$ are considered for truncation. 
For each $t \in\{1, \dots, 9\}$ and for each $k \in\{1, \dots, K\}$, we draw an identical number of independent samples from $\widehat{\mu}_{t - 1}$ and $\nu_k$.
In particular, the sample sizes in each experiment follow an exponentially increasing schedule over iterations, which provides more precise OT map estimations when $\widehat{\mu}_t$ becomes near-optimal.\footnote{As shown in Table~\ref{tab: instance-config}, compared with the sample size schedule imposed in [SG-2d], the reduced sample sizes across iterations adopted in [SG-10d] and [BS-8d] reflect a compromise due to the increased computational cost of high-dimensional OT estimation.}
Moreover, we employ the modified entropic OT map estimator $\widehat{T}_{\ENTROPIC}[\cdot]$ with constant entropic regularization parameters $(\widehat{\IGamma}_{t,k})_{t = 0:9,\, k = 1: K}$ across iterations.
The Sinkhorn step is carried out by invoking and running Sinkhorn's algorithm implemented in the $\mathtt{geomloss}$ library \citep{Feydy2019Fast} until the defaulted stopping criterion is reached.\footnote{See \url{https://www.kernel-operations.io/geomloss/index.html}.} 
The concrete setups for the problem instances are formally introduced in Section~\ref{ssec: exp-synthetic-problem-instances}, and 
the concrete values imposed on (selected) parameters and configurations in Setting~\ref{sett: Caffarelli} and Setting~\ref{sett: config-generation} are listed in Table~\ref{tab: instance-config}.
Lastly, the probability density functions of $\bar{\mu}$ and $\varkappa_1, \dots, \varkappa_{\tilde{K}}$ in [SG-2d] are visualized in Figure~\ref{fig: source-auxiliary-pdf}.

\begin{table}
    \centering
	\begin{adjustbox}{max width=\textwidth}
    \begin{tabular}{lrrrrrrrr}
        \toprule
        \multirow{2}{*}{Instance} & \multicolumn{5}{c}{Input configurations} & \multicolumn{3}{c}{Implementation details of Algorithm~\ref{algo: concrete}}\\
        \cmidrule(lr){2-6} \cmidrule(lr){7-9}
        & $d$ & $K$ & $\widetilde{K}$ & $(n_{\widetilde{k}})_{\widetilde{k} = 1 : \widetilde{K}}$ & $(\CY_k)_{k = 1: K}$ & Sample size schedule across $t = 1, \dots, 9$ & $(\CX_{\widehat{R}_t})_{t = 0:9}$ & $(\gamma_{t, k})_{t = 0: 9, k = 1: K}$\\
        \midrule
        {[SG-2d]}  & 2  & 5  & 5  & 1000  & $\bar{B}(\veczero_2, 225)$  & Exp.\ growth ($10000 \rightarrow 160000$)& $\bar{B}(\veczero_2, 225)$ & $5$\\
        {[SG-10d]} & 10 & 10 & 10 & 1000  & $\bar{B}(\veczero_{10}, 1000)$ & Exp.\ growth ($5000 \rightarrow 20000$)& $\bar{B}(\veczero_{10}, 1000)$ & $1$\\
		{[BS-8d]} & 8 & 5 & -- & --  & -- & Exp.\ growth ($5000 \rightarrow 20000$)& $\bar{B}(\veczero_{8}, 15)$ & $10^{-8}$\\
        \bottomrule
    \end{tabular}
\end{adjustbox}
    \caption{Selected parameter configurations and implementation details of Algorithm~\ref{algo: concrete} across problem instances.}\label{tab: instance-config}
\end{table}

\begin{figure}
    \centering
    \includegraphics[width=0.6\textwidth]{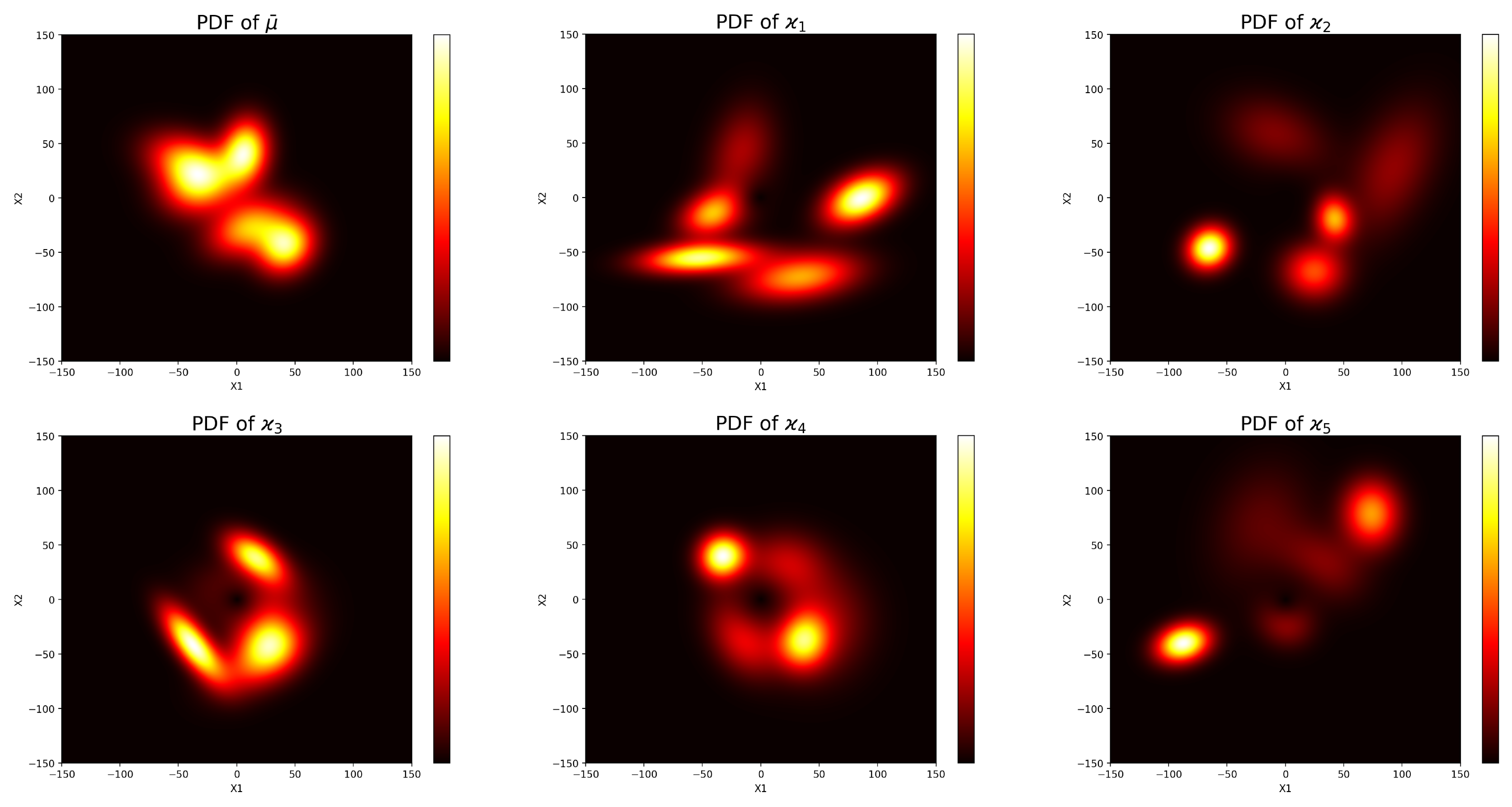}
    \caption{The probability density functions of the ground-truth $\CW_2$-barycenter $\bar{\mu}$ and the auxiliary measures $\varkappa_1, \dots, \varkappa_5$ in the problem instance [SG-2d].}\label{fig: source-auxiliary-pdf}
\end{figure}

\subsection{Omitted details in Experiment~2}\label{sapx: exp-details-bike-sharing}
The summary statistics of the bike-sharing dataset used in [BS-8d] are presented in Table~\ref{tab:bike_summary}.
To generate samples from each posterior of the regression coefficients, 
we followed the data preprocessing pipelines provided in \citep{campbell2018bayesian} and \citep{li2020continuous}, and adopted the stochastic approximation trick introduced in \citep{minsker2014scalable} to appropriately 
rescale the subset posterior densities.
The posterior means of $\Bvartheta \in \R^8$ are presented in Table~\ref{tab:posterior_summary}.
In particular, under a Gaussian prior, it can be shown that the Poisson regression posterior density is dominated by a Gaussian density (up to a multiplicative finite constant), 
therefore every subset posterior distribution lies in $\CP_{2,\AC}(\R^8)$.
Recall that we do not truncate $\nu_1, \dots, \nu_5$ to ensure the representative power of $\bar{\mu}_{\FULL}$ as the presumed $\CW_2$-barycenter for quantitative evaluations.
The concrete setups for the problem instance are formally introduced in Section~\ref{ssec: exp-subset-posterior-aggregation}, and 
the concrete values imposed on (selected) parameters in Setting~\ref{sett: Caffarelli} are listed in Table~\ref{tab: instance-config}.


\begin{table}
  \centering
  \begin{adjustbox}{max width=\textwidth}
  \begin{tabular}{ll>{\raggedleft\arraybackslash}p{1.6cm}>{\raggedleft\arraybackslash}p{1.6cm}>{\raggedleft\arraybackslash}p{1.6cm}>{\raggedleft\arraybackslash}p{1.6cm}>{\raggedleft\arraybackslash}p{1.6cm}}
    \toprule
    \textbf{Variable} & \textbf{Description} & \textbf{Mean} & \textbf{Std} & \textbf{Min} & \textbf{Median} & \textbf{Max} \\
    \midrule
    \multicolumn{7}{l}{\textit{Continuous variables}} \\
    \midrule
    Count & Total bike rentals (response variable) & 189.463 & 181.388 & 1 & 142 & 977 \\
    Temp. & Normalized temperature & 0.497 & 0.193 & 0.020 & 0.500 & 1 \\
    Feels-like Temp. & Normalized feeling temperature & 0.476 & 0.172 & 0 & 0.485 & 1 \\
    Humidity & Normalized humidity & 0.627 & 0.193 & 0 & 0.630 & 1 \\
    Wind Speed & Normalized wind speed & 0.190 & 0.122 & 0 & 0.194 & 0.851 \\
    \midrule
    \multicolumn{7}{l}{\textit{Categorical \& binary variables}} \\
    \midrule
    Hour & Hour of the day (ordinal, 0--23) & \multicolumn{5}{l}{Range: 0--23 \quad Mode: 16} \\
    Season & Meteorological season & \multicolumn{5}{l}{Spring: 24.4\% \quad Summer: 25.4\% \quad Fall: 25.9\% \quad Winter: 24.4\%} \\
    Working Day & Whether the day is a working day & \multicolumn{5}{l}{No: 31.7\% \quad Yes: 68.3\%} \\
    Weather & Weather situation & \multicolumn{5}{l}{Clear: 65.7\% \quad Mist: 26.1\% \quad Light rain/snow: 8.2\% \quad Heavy rain/snow: 0.0\%} \\
    \bottomrule
  \end{tabular}
  \end{adjustbox}
  \caption{Summary statistics of the bike-sharing dataset ($N = 17379$).}\label{tab:bike_summary}
\end{table}

\begin{table}
    \centering
    \begin{adjustbox}{max width=0.7\textwidth}
    \begin{tabular}{lrrrrrrrr}
        \toprule
        \textbf{Posteriors} & \multicolumn{8}{c}{\textbf{Regression coefficients}} \\
        \cmidrule(l){2-9}
        & $\vartheta_1$ & $\vartheta_2$ & $\vartheta_3$ & $\vartheta_4$ & $\vartheta_5$ & $\vartheta_6$ & $\vartheta_7$ & $\vartheta_8$ \\
        \midrule
        $\bar{\mu}_{\FULL}$ & $0.2031$ & $0.4182$ & $0.0250$ & $-0.1265$ & $0.2790$ & $0.0147$ & $-0.1961$ & $0.0272$ \\
        $\nu_1$ & $0.1814$ & $0.4207$ & $0.0218$ & $-0.1128$ & $0.2947$ & $0.0248$ & $-0.2181$ & $0.0111$ \\
        $\nu_2$ & $0.2015$ & $0.4018$ & $0.0362$ & $-0.1328$ & $0.2921$ & $-0.0155$ & $-0.1746$ & $0.0271$ \\
        $\nu_3$ & $0.2130$ & $0.4256$ & $0.0204$ & $-0.1386$ & $0.2751$ & $0.0096$ & $-0.2057$ & $0.0441$ \\
        $\nu_4$ & $0.2127$ & $0.4080$ & $0.0251$ & $-0.1161$ & $0.2654$ & $0.0507$ & $-0.1829$ & $0.0202$ \\
        $\nu_5$ & $0.2085$ & $0.4362$ & $0.0215$ & $-0.1338$ & $0.2657$ & $0.0150$ & $-0.2006$ & $0.0341$ \\
        \bottomrule
    \end{tabular}
    \end{adjustbox}
    \caption{Posterior means of $\Bvartheta \in \R^8$ in [BS-8d]}\label{tab:posterior_summary}.
\end{table}

\section{An instance with multiple Karcher means}\label{apx: multiple-karcher-means}

In this appendix, we examine the sensitivity of Algorithm~\ref{algo: concrete} against suboptimal Karcher means by investigating its performance in a simple problem instance in which the barycenter functional attains multiple Karcher means. 

We focus on an instance provided in \citep[Example~2.2]{backhoff2025stochastic}. 
Specifically, let $\nu_1 \in \CP(\R^2)$ be the uniform measure on $\bar{B}((-60, 20), 2) \cup \bar{B}((60, -20), 2)$ and let $\nu_2 \in \CP(\R^2)$ be the uniform measure on $\bar{B}((60, 20), 2) \cup \bar{B}((-60, -20), 2)$.
Moreover, let $\mu_1$ be the uniform measure on $\bar{B}((0, 20), 2) \cup \bar{B}((0, -20), 2)$, and let $\mu_2$ be the uniform measure on $\bar{B}((60, 0), 2) \cup \bar{B}((-60, 0), 2)$. 
Then, it can be shown that $\mu_1$ and $\mu_2$ are two distinct Karcher means of the barycenter functional $V(\cdot) := \frac{1}{2} \CW_2(\,\cdot\,, \nu_1)^2 + \frac{1}{2} \CW_2(\,\cdot\,, \nu_2)^2$, while $\mu_2$ is the unique $\CW_2$-barycenter. 
We name this problem instance as [MKM-2d].

\begin{figure}[t]
	\centering
	\begin{subfigure}{0.75\textwidth}
		\makebox[\textwidth][c]{\includegraphics[width=\textwidth]{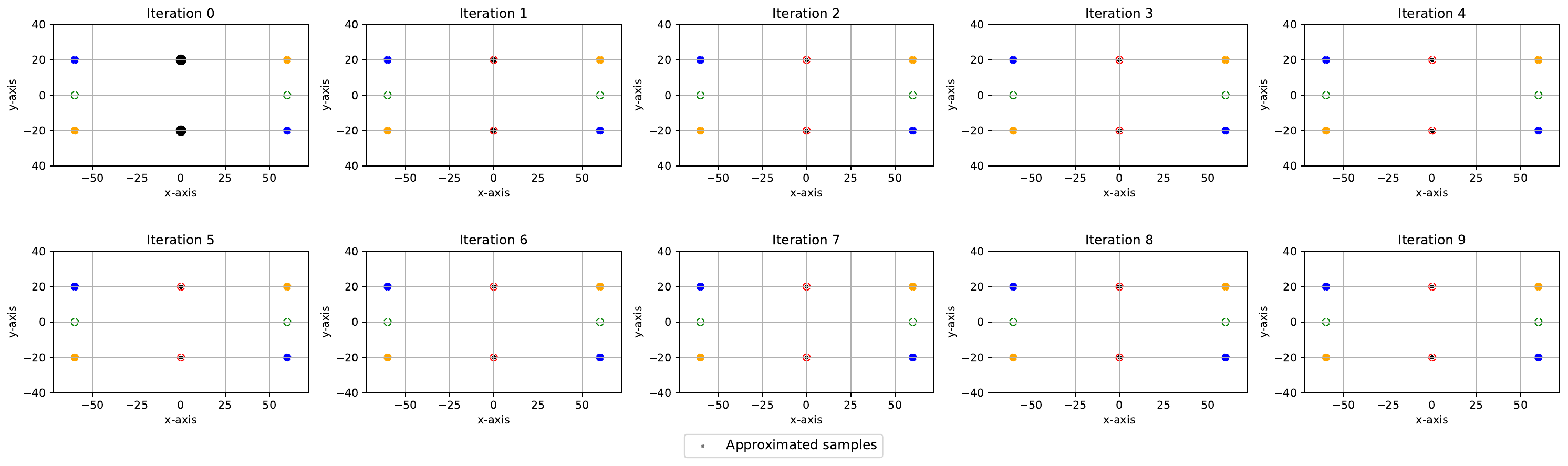}}
		\caption{Trial A}\label{fig: karcher-B}
	\end{subfigure}


	\begin{subfigure}{0.75\textwidth}
		\makebox[\textwidth][c]{\includegraphics[width=\textwidth]{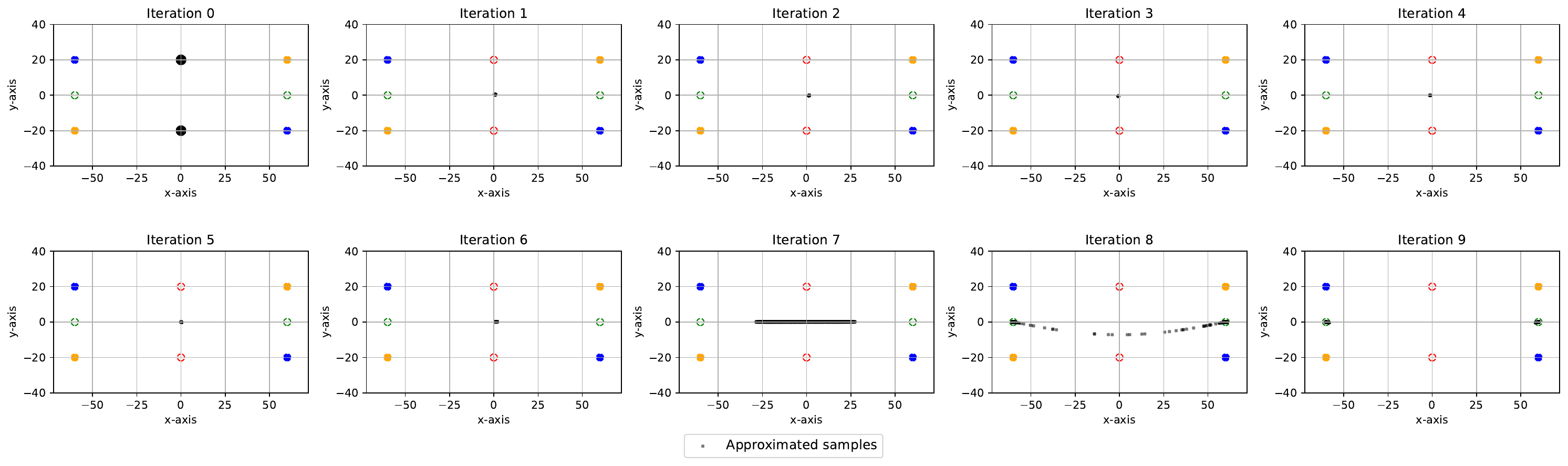}}
		\caption{Trial B}\label{fig: karcher-A}
	\end{subfigure}
	\caption{Performances of Algorithm~\ref{algo: concrete} across the two trials in [MKM-2d].}\label{fig: multiple-karcher-means}
\end{figure}

We conduct two experimental trials in [MKM-2d], indexed by alphabets.
To assess whether our algorithm can escape the suboptimal Karcher mean, the initial measure $\rho_0$ in Algorithm~\ref{algo: concrete} is set as $\mu_1$ in both trials.
The two trials differ in the schedule of the entropic regularization parameter throughout the iterations.
Specifically, we adopt a fixed regularization value of 10 in Trial~A, and adopt a geometrically decreasing schedule from 200000 down to 10 in Trial~B.
Across both trials, we run Algorithm~\ref{algo: concrete} for 9 iterations and 
employ truncation via a fixed closed Euclidean balls centered at the origin with radius 200.

The results for both trials in [MKM-2d] are presented in Figure~\ref{fig: multiple-karcher-means}.
The blue and yellow filled circles denote $\nu_1$ and $\nu_2$, and the red and green hollow circles denote $\mu_1$ and $\mu_2$, respectively. 
In each trial, the black scatter points in the subplot corresponding to the $t$-th iteration represent 2000 independent samples drawn from $\widehat{\mu}_t$ computed by Algorithm~\ref{algo: concrete}.
It is unsurprising that in Trial~A, Algorithm~\ref{algo: concrete} with a small fixed regularization parameter closely mimics the $G$-iteration, and thus remained stuck at $\mu_1$, which is a fixed-point of the $G$-operator
but not the $\CW_2$-barycenter. 
Therefore, 
it is indeed possible for
Algorithm~\ref{algo: concrete} without sufficient regularization to be stuck around a suboptimal Karcher mean.

In contrast, when the OT map estimators were suitably regularized in initial iterations as in Trial~B, Algorithm~\ref{algo: concrete} was able to escape the suboptimal Karcher mean $\mu_1$ and converge to the $\CW_2$-barycenter $\mu_2$ with fewer than 10 iterations. 
This behavior is consistent with the empirical findings of \citet[Section~6]{chizat2025doubly}, who also highlighted the importance of regularization in overcoming suboptimal solutions.
Taken together, these results indicate a degree of practical robustness of our algorithm: although unfavorable initialization may slow down convergence, imposing large regularization parameters in early iterations tends to facilitate convergence to the barycenter. 
We would also like to remark that an exact initialization at a suboptimal Karcher mean is unlikely to arise in practice.


\bibliographystyle{abbrvnat}
\bibliography{reference.bib}


\end{document}

%% file: mathsymbols.tex
\usepackage{amsfonts,mathrsfs,bbm, dsfont}

\newcommand{\Td}{\mathrm{d}}

\newcommand{\BA}{\mathbf{A}}
\newcommand{\BB}{\mathbf{B}}

\newcommand{\BI}{\mathbf{I}}

\newcommand{\BK}{\mathbf{K}}

\newcommand{\BO}{\mathbf{O}}

\newcommand{\BU}{\mathbf{U}}

\newcommand{\BY}{\mathbf{Y}}

\newcommand{\BIb}{{\boldsymbol{b}}}

\newcommand{\BIf}{{\boldsymbol{f}}}
\newcommand{\BIg}{{\boldsymbol{g}}}

\newcommand{\BIp}{{\boldsymbol{p}}}

\newcommand{\BIu}{{\boldsymbol{u}}}
\newcommand{\BIv}{{\boldsymbol{v}}}

\newcommand{\BIx}{{\boldsymbol{x}}}
\newcommand{\BIy}{{\boldsymbol{y}}}
\newcommand{\BIz}{{\boldsymbol{z}}}

\newcommand{\BIX}{{\boldsymbol{X}}}
\newcommand{\BIY}{{\boldsymbol{Y}}}
\newcommand{\BIZ}{{\boldsymbol{Z}}}

\newcommand{\CB}{\mathcal{B}}
\newcommand{\CC}{\mathcal{C}}

\newcommand{\CF}{\mathcal{F}}

\newcommand{\CK}{\mathcal{K}}
\newcommand{\CL}{\mathcal{L}}
\newcommand{\CM}{\mathcal{M}}

\newcommand{\CP}{\mathcal{P}}

\newcommand{\CS}{\mathcal{S}}

\newcommand{\CW}{\mathcal{W}}
\newcommand{\CX}{\mathcal{X}}
\newcommand{\CY}{\mathcal{Y}}
\newcommand{\CZ}{\mathcal{Z}}

\newcommand{\FC}{\mathfrak{C}}

\newcommand{\IGamma}{\varGamma}

\newcommand{\ITheta}{\varTheta}

\newcommand{\IPsi}{\varPsi}

\newcommand{\BSigma}{{\boldsymbol{\Sigma}}}

\newcommand{\Bbeta}{{\boldsymbol{\beta}}}

\newcommand{\Beta}{{\boldsymbol{\eta}}}

\newcommand{\Bvartheta}{{\boldsymbol{\vartheta}}}

\newcommand{\R}{\mathbb{R}} 
\newcommand{\N}{\mathbb{N}} 

\newcommand{\vecone}{\mathbf{1}} 
\newcommand{\veczero}{\mathbf{0}} 
\newcommand{\INDI}{\mathbbm{1}} 
\newcommand{\TRANSP}{\mathsf{T}} 
\newcommand{\DIFF}{\Td} 
\newcommand{\DIFFX}[1]{\,\Td{#1}} 
\newcommand{\DIFFM}[2]{\,{#1}({#2})}
\newcommand{\diag}{\mathrm{diag}} 
\newcommand{\PROB}{\mathbb{P}} 
\newcommand{\EXP}{\mathbb{E}} 
\newcommand{\VAR}{\mathrm{Var}} 
\newcommand{\clos}{\mathrm{cl}} 
\newcommand{\support}{\mathrm{supp}} 
\newcommand{\conv}{\mathrm{conv}} 

\newcommand{\lebesgue}{\mathscr{L}}

\newcommand{\boundary}{\mathrm{bd}}

\DeclareMathOperator*{\maximize}{\mathrm{maximize}} 

\DeclareMathOperator*{\esssup}{ess\,sup}
